\documentclass[10pt]{article}
\usepackage{amsmath}
\usepackage{amssymb}
\usepackage{amscd}
\usepackage{graphicx}
\usepackage{amsthm}       
\newtheorem{thm}{Theorem}[section]
\newtheorem{prop}[thm]{Proposition}
\newtheorem{lem}[thm]{Lemma}
\newtheorem{cor}[thm]{Corollary}  \theoremstyle{definition}
\newtheorem{df}[thm]{Definition}   \theoremstyle{definition}

\newtheorem{rem}[thm]{Remark}                \theoremstyle{plain}
 \theoremstyle{definition}
\newtheorem{ex}[thm]{Example}

\def\CC{\Bbb{C}}
\def\RR{\Bbb{R}}  

\def\ZZ{\Bbb{Z}}
\def\CCI{\hat{\CC}}        \def\NN{\Bbb{N}} 
\def\B1{{\rm\kern.32em\vrule    width.12em       height1.4ex
depth-.05ex\kern-.28em 1}}
\def\G{\Gamma}
\def\g{\gamma }

\def\GN{\Gamma ^{\NN }}

\def\CMX{\text{CM}(Y)}
\def\emCMX{\text{{\em CM}}(Y)}

\def\OCMX{\text{OCM}(Y)}
\def\emOCMX{\text{{\em OCM}}(Y)}

\def\Rat{\text{Rat}}
\def\emRat{\text{{\em Rat}}}
\def\Ratp{\text{Rat}_{+}}
\def\emRatp{\text{{\em Rat}}_{+}}
\def\suppt{\text{supp}\, \tau}

\def\Cpt{\text{Cpt}}
\def\emCpt{\text{\em Cpt}}
\def\Hol{\text{H\"{o}l}}

\def\Min{\text{Min}}
\def\emMin{\text{\em Min}}

\def\MYW{{\frak  M}_{1,c}({\cal Y},\{ {\cal W}_{j}\} _{j=1}^{m})}
\def\supp{\mbox{supp}}
 
\def\supptau{\mbox{supp}\,\tau}

\begin{document}
\title{Negativity of Lyapunov Exponents and Convergence of 
Generic Random Polynomial Dynamical Systems \\ and Random Relaxed Newton's Methods  
\footnote{Date: June 21, 2021. Published in 
Comm. Math. Phys. {\bf 384}, 1513–-1583 (2021). 
2010 MSC:  
37F10, 37H10. Keywords: Random dynamical systems; Random complex dynamics;  
%random iteration, iterated function systems, Markov process, 
Rational semigroups;  
%polynomial semigroups,   
%Julia sets, 
Fractal geometry; Cooperation principle; Noise-induced order; Randomness-induced phenomena; Random relaxed Newton's method. %Phone: +81-(0)6-6850-5307,
%Fax: +81-(0)6-6850-5327.  
%stability, bifurcation
}}

\author{Hiroki Sumi\\  
Course of Mathematical Science, Department of Human Coexistence, \\  
Graduate School of Human and Environmental Studies, Kyoto University\\ 
Yoshida Nihonmatsu-cho, Sakyo-ku, Kyoto, 606-8501, Japan \\ 
{\bf E-mail: sumi@math.h.kyoto-u.ac.jp}\\ 
http://www.math.h.kyoto-u.ac.jp/\textasciitilde sumi/index.html
\date{}
}
\maketitle
\vspace{-8.5mm} 
\begin{abstract}
We investigate i.i.d. random complex dynamical systems generated by 
 probability measures on  finite unions of the loci of holomorphic families 
of rational maps on the Riemann sphere $\CCI .$  
We show that under certain conditions on the families, for a generic system, 
(especially, for a generic random polynomial dynamical system,)  
for all but countably 
many initial values $z\in \CCI $, for almost every sequence of maps  $\gamma =(\gamma _{1},
\gamma _{2},\ldots )$, the Lyapunov exponent of $\gamma $ at $z$ is negative. 
Also, we show that for a generic system, for every initial value $z\in \CCI $, the orbit of the Dirac measure at $z$ 
under 
the iteration of 
the dual map of the transition operator tends to a periodic cycle of measures in the space of 
probability measures on $\CCI $. Note that these are new phenomena in random complex dynamics
 which cannot hold in deterministic complex dynamical systems.  
We apply the above theory  and results of random complex dynamical systems 
to finding roots of any polynomial by random relaxed Newton's methods 
and we show that for any polynomial $g$ of degre two or more, for any initial value $z\in \CC$ which is not a root of $g'$, 
the random orbit starting with $z$ tends to a root of $g$ almost surely, 
which is the virtue of the effect of randomness. 
\end{abstract}
\vspace{-5mm} 
\section{Introduction and the main results}
In this paper, we investigate the independent and identically-distributed (i.i.d.) random dynamics of rational maps on the Riemann sphere $\CCI $ and the dynamics 
of rational semigroups (i.e., semigroups of non-constant rational maps 
where the semigroup operation is functional composition) on $\CCI .$ 

One motivation for research in (complex) dynamical systems is to describe some
 mathematical models in various fields to study nature and  science.   
%For 
% example, the behavior of the population 
% of a certain species can be described by the 
% dynamical system associated with iteration of a polynomial 
% $f(z)= az(1-z)$ 
% such that $f$ preserves the unit interval and 
% the postcritical set in the plane is bounded 
% (cf. \cite{D}). 
%However, when there is a change in the natural environment,  
%some species have 
% several strategies to survive in nature. 
%From this point of view, 
Since nature and any other  environments have a lot of random terms,  
 it is very natural and important not only to consider the dynamics 
of iteration, %where the same survival strategy (i.e., function) is repeatedly applied, 
but also 
to consider random 
 dynamics.  
%where a new strategy might be applied at each time step.  
 Another motivation for research in complex dynamics is Newton's method to 
 find  roots of a complex polynomial,   
which often is expressed as the dynamics of a rational map $f$ on $\CCI $ with 
$\deg (f)\geq 2$, where $\deg (f)$ denotes the degree of $f.$  
% We sometimes use  computers to analyze such dynamics, and since 
% we have some errors at each step of the calculation in the computers, it is quite natural to investigate the 
% random dynamics of rational maps. 
In various fields, we have many mathematical models which are described by 
the dynamical systems associated with polynomial or rational maps. For each 
model, it is natural and important to consider a randomized model, since we always have 
some kind of noise or random terms.  
%in nature.  
Regarding random (complex) dynamics, many researchers in various fields 
(mathematics, physics, chemistry, etc.) have found and investigated 
many kinds of new phenomena in random (complex) dynamics which cannot hold in 
deterministic dynamics. These phenomena arise from the effect of randomness and 
they are called {\bf randomness-induced phenomena} or {\bf noise-induced phenomena} 
(\cite{MT}).  In fact, recently these topics are getting more and more attention in many fields.  

The first study of random complex dynamics was given by J. E. Fornaess and  N. Sibony (\cite{FS}). 
%They mainly investigated random dynamics generated by small perturbations of a single rational map.  
For research on random complex dynamics of quadratic polynomials, 
see \cite{Br1}--\cite{Bu2}, \cite{GQL}.  
For recent research on random complex dynamics and the various randomness-induced phenomena,  
see the author's works \cite{SdpbpIII}--\cite{Sspace}. 
In this paper, we show that a generic random holomorphic dynamical system associated 
with an analytic family of rational maps on $\CCI $ with some mild conditions has  
 {\bf randomness-induced order} even if the noise is multiplicative and we have a common 
 repelling fixed point for any map in the system. 
Some  of the new things of this paper 
%compared to the previous works 
are that (a) we can deal with random holomorphic dynamical systems 
{\bf with multiplicative noise} (Theorems~\ref{t:rcdnkmain1i}, \ref{t:rcdnkmain2i}, \ref{t:rcdnkmain1}, \ref{t:rcdnkmain2}, 
Example~\ref{ex:applyi},  
Section~\ref{Examples})
%\ref{ex:wlz1z}--\ref{ex:pzljgz}) 
and (b) we apply the results to the study of {\bf random relaxed Newton's methods} 
which we introduce in this paper 
to find  roots of {\bf any} polynomial, and we show that 
random relaxed Newton's methods have some much better properties than those of deterministic (relaxed) Newton's method (Theorems~\ref{t:RNM1i}, \ref{t:RNM1}, Corollary~\ref{c:RNM1ione}, Remark~\ref{r:RNM1i}).

  In order to investigate random complex dynamics, it is very important to study the dynamics of 
  associated rational semigroups. 
In fact, 
%it is a very powerful tool to investigate random complex dynamics, 
%since 
random complex dynamics and the dynamics of rational semigroups are related to each other very deeply.   
The first study of dynamics of rational semigroups was 
conducted by
A. Hinkkanen and G. J. Martin (\cite{HM}) 
%who were interested in the role of the
%dynamics of polynomial semigroups (i.e., 
%dynamics of semigroups of non-constant polynomial maps while studying
%various one-complex-dimensional
%moduli spaces for discrete groups,
and
by F. Ren's group (\cite{GR}).   
% who studied 
%such semigroups from the perspective of random dynamical systems.
%Since the Julia set $J(G)$ of a finitely generated rational semigroup 
%$G=\langle h_{1},\ldots, h_{m}\rangle $ has 
%``backward self-similarity,'' i.e.,  
%$J(G)=\bigcup _{j=1}^{m}h_{j}^{-1}(J(G))$ (see \cite[Lemma 0.2]{S4}),  
%the study of the dynamics of rational semigroups can be regarded as the study of  
%``backward iterated function systems,'' and also as a generalization of the study of 
%self-similar sets in fractal geometry.  
For recent work on the dynamics of rational semigroups, 
see the author's papers \cite{S3}--\cite{Sspace}, and 
\cite{JS2, SS, SU1, SU2}. 

To introduce the main idea of this paper,  
we let $G$ be a rational semigroup and denote by $F(G)$  the  {\bf Fatou set of $G$}, which is defined to be  
the maximal open subset of $\CCI $ where $G$ is equicontinuous with respect to the spherical distance on $\CCI $.    
We call $J(G):=\CCI \setminus F(G)$ the {\bf Julia set of $G.$}  
The Julia set is backward invariant under each element $h\in G$, but 
might not be forward invariant. This is a difficulty of the theory of rational semigroups. 
Nevertheless, we utilize this as follows.  
The key to investigating random complex dynamics is to consider the 
following {\bf kernel Julia set of $G$}, which is defined by 
$J_{\ker }(G)=\bigcap _{g\in G}g^{-1}(J(G)).$ This is the largest forward 
invariant subset of $J(G)$ under the action of $G.$ Note that 
if $G$ is a group or if $G$ is a commutative semigroup, 
then $J_{\ker }(G)=J(G).$ 
However, for a general rational semigroup $G$ generated by a family of 
rational maps $h$ with $\deg (h)\geq 2$, it may happen that 
$\emptyset =J_{\ker }(G)\neq J(G) $.   

Let {\bf Rat} be the space of all non-constant rational maps on the Riemann sphere $\CCI $, 
endowed with the distance $\kappa $ which is defined by 
$\kappa (f,g):=\sup _{z\in \CCI }d(f(z),g(z))$, where $d$ denotes the spherical distance on $\CCI .$  
Let {\bf Rat$_{+}$} be the space of all rational maps $g$ with $\deg (g)\geq 2.$ Let 
${\cal P}$ be the space of all polynomial maps $g$ with $\deg (g)\geq 2.$   
Let $\tau $ be a Borel probability measure on Rat with compact support. 
We consider the {\bf i.i.d. random dynamics} on $\CCI $ such that 
at every step we choose a map $h\in \mbox{Rat}$ according to $\tau .$ 
Thus this determines a Markov process 
%with time-homogeneous transition probabilities 
on the state space 
$\CCI $ such that for each $x\in \CCI $ and 
each Borel measurable subset $A$ of $\CCI $, 
the {\bf transition probability} 
$p(x,A)$ from $x$ to $A$ is defined as $p(x,A)=\tau (\{ g\in \Rat \mid g(x)\in A\} ).$ 
Let $G_{\tau }$ be the 
rational semigroup generated by the support of $\tau $, i.e., 
$G_{\tau }=\{ h_{1}\circ \cdots \circ h_{n}\mid n\in \NN , h_{j}\in \mbox{supp}\,\tau  
\mbox{ for all } j\}.$  Moreover, $J_{\ker }(G_{\tau })$ is called the {\bf kernel Julia set of $\tau .$}

For a metric space $X$, let ${\frak M}_{1}(X)$ be the space of all 
Borel probability measures on $X$ endowed with the topology 
induced by weak convergence (thus $\mu _{n}\rightarrow \mu $ in ${\frak M}_{1}(X)$ if and only if 
$\int \varphi d\mu _{n}\rightarrow \int \varphi d\mu $ for each bounded continuous function $\varphi :X\rightarrow \RR $). 
Note that if $X$ is a compact metric space, then ${\frak M}_{1}(X)$ is compact and metrizable. 
For each $\tau \in {\frak M}_{1}(X)$, we denote by supp$\, \tau $ the topological support of $\tau .$  
Let ${\frak M}_{1,c}(X)$ be the space of all Borel probability measures $\tau $ on $X$ such that supp$\,\tau $ is 
compact.     
 
%We define the ``pointwise Fatou set'' $F_{pt}^{0}(\tau )$ of 
%the dynamics of $M_{\tau }^{\ast }$ as the set of all elements $y\in \CCI $  
%satisfying that 
%there exists a neighborhood $B$ of $y$ such that 
% $\{ (M_{\tau }^{\ast })^{n}\circ \Phi :\CCI \rightarrow {\frak M}_{1}(\CCI )\} _{n\in \NN }$ 
%is  equicontinuous at the one point $y\in \CCI $, where 
%$\Phi :\CCI \rightarrow {\frak M}_{1}(\CCI ) $ is the embedding map defined  by 
%$\Phi (y)=\delta _{y}$ (see Definition~\ref{d:manyFJ}). 
%Also, we set $J_{pt}^{0}(\tau ):=\CCI \setminus F_{pt}^{0}(\tau ).$ 

 For each $\tau \in {\frak M}_{1}(\Rat)$, let $\tilde{\tau }:=\otimes _{n=1}^{\infty }\tau \in {\frak M}_{1}((\Rat)^{\NN }).$ 
%For a $\tau \in {\frak M}_{1,c}(\Rat)$, we denote by $U_{\tau }$ 
%the space of all finite linear combinations of unitary eigenvectors of $M_{\tau }:C(\CCI )\rightarrow 
%C(\CCI )$, where an eigenvector  is said to be unitary if the absolute value of 
%the corresponding eigenvalue is equal to one. Moreover, 
%we set ${\cal B}_{0,\tau }:= \{ \varphi \in C(\CCI )\mid M_{\tau }^{n}(\varphi )\rightarrow 0 \mbox{ as }n\rightarrow %\infty \} .$ 
%(Theorem~\ref{t:mtauspec}). 
For a metric space $X$, we denote by $\Cpt(X)$ the space of all 
non-empty compact subsets of $X$ endowed with the Hausdorff metric.  
 For a rational semigroup $G$, we say that a non-empty compact subset $L$ of $\CCI $ is a 
 {\bf minimal set of $(G,\CCI )$} 
if $L=\overline{\cup _{h\in G}\{ h(z)\}}$ for each 
$z\in L.$  
%is minimal in 
%$\{ C\in \Cpt(\CCI ) \mid  \forall g\in G, g(C)%\subset C\} $ 
%with respect to inclusion.  
Moreover, we denote by $\Min (G,\CCI )$
the sets of all minimal sets of $(G,\CCI )$.  
%\{ L \in \Cpt(\CCI )\mid L \mbox{ is a minimal set %for } (G,\CCI )\} .$ 
%For a $\tau \in {\frak M}_{1}(\Rat)$, let $S_{\tau }:=\bigcup _{L\in \Min(G_{\tau },\CCI)}L.$  
%For any $\tau \in {\frak M}_{1}(\Rat)$, 
%for any $L\in \Min(G_{\tau },\CCI )$ and 
%for any $z\in \CCI $, we set 
%$T_{L,\tau }(z)=\tilde{\tau }(\{ \gamma =(\gamma _{1},\gamma _{2},\ldots )\in 
%(\Rat)^{\NN }\mid d(\gamma _{n,1}(z),L)\rightarrow 0\mbox{ as }n\rightarrow \infty \})$.  
%If $L=\{ x\}$, then we set $T_{L,\tau }=T_{x,\tau }.$ 
%For a $\tau \in {\frak M}_{1}(\Rat)$, let $\G _{\tau }:=\mbox{supp}\, \tau (\subset \Rat).$ 
Let $\tau \in {\frak M}_{1,c}(\Rat).$ 
We say that a minimal set $L\in \Min(G _{\tau },\CCI )$  is {\bf attracting 
for $\tau $} if there exist two open subsets $A, B$ of $\CCI $ 
with $\sharp (\CCI \setminus A)\geq 3$ and an $n\in \NN $ such that 
$L\subset B\subset \overline{B}\subset A$ and such that for each $(\gamma _{1},\ldots, \gamma _{n})
\in (\mbox{supp}\,\tau )^{n}$, we have $\gamma _{n}\circ \cdots \circ \gamma _{1}(A)\subset B.$  
In this case, we say that $L$ is an 
{\bf attracting minimal set of $\tau .$} 
%Also, for an element $\tau \in {\frak M}_{1,c}(\Rat)$, 
%if $L\in \Min(\mbox{supp}\,\tau,\CCI )$ is attracting for $\mbox{supp}\,\tau$ then we say that 
%$L$ is attracting for $\tau ,$ that $L$ is an attracting minimal set of $\G $, and 
% that $L$ is an attracting minimal set of $\tau .$  
%\end{df}
%
%\begin{df}
%\label{d:mild} 
Let ${\cal Y}$ be a subset of $\Rat $ endowed with the relative topology from $\Rat. $ 
We say that ${\cal Y}$ is {\bf mild} 
if for each $\tau \in {\frak M}_{1,c}({\cal Y})$, there exists an attracting minimal set of $\tau.$  
For example, any non-empty open subset of ${\cal P}$ is a mild subset of $\Rat.$  

Let ${\cal Y}$ be a closed subset of an open subset of $\Rat$, i.e., there exist an open subset ${\cal V}$ of Rat and a closed subset ${\cal C}$ of Rat such that 
${\cal Y}={\cal V}\cap {\cal C}.$ Let 
${\cal W}=\{ f_{\lambda }\} _{\lambda \in \Lambda }$ be a {\bf holomorphic family of 
rational maps} (see Definition~\ref{d:singdf}) such that $\Lambda $ is a connected complex manifold and 
$\lambda \mapsto f_{\lambda }\in \Rat$ is not 
constant. We say that ${\cal Y}$ is {\bf weakly nice} with respect to ${\cal W}$ 
if ${\cal Y}=\{ f_{\lambda }\in \Rat \mid \lambda \in \Lambda \} $
 (for more general definition, see~\ref{d:weaklynice}).   
In this case, for each $n\in \NN $, we denote by  
$S_{n}({\cal W})$ the set of points $z\in \CCI $ satisfying that 
$(\lambda _{1},\ldots, \lambda _{n})\in \Lambda ^{n}\mapsto 
f_{\lambda _{1}}\circ \cdots \circ f_{\lambda _{n}}(z)$ is constant on $\Lambda ^{n}.$ 
Also, we set $S({\cal W})=\cap _{n=1}^{\infty }S_{n}({\cal W}).$ 
This $S({\cal W})$ is called the {\bf singular set} of ${\cal W}.$ 
Note that 
$\sharp S_{1}({\cal W})<\infty $ and $\sharp S({\cal W})<\infty $ (Lemma~\ref{l:sn1sn}). We say that ${\cal Y}$ is {\bf nice} with respect to ${\cal W}$  if 
${\cal Y}$ is weakly nice with respect to ${\cal W}$ and for each $\tau \in {\frak M}_{1,c}({\cal Y})$, 
for each $L\in \Min(G_{\tau },\CCI )$ with $L\subset S({\cal W})$ and  
for each $z\in L$, either (a) the map $\lambda \mapsto D(f_{\lambda })_{z}$ is non-constant on 
$\Lambda $ or (b) $D(f_{\lambda })_{z}=0$ for all $\lambda \in \Lambda .$ 

For any closed subset ${\cal Y}$ of an open subset of $\Rat$, let ${\cal O}$ be the topology in 
${\frak M}_{1,c}({\cal Y})$ such that the sequence $\{ \tau _{n}\} _{n=1}^{\infty }$ in 
${\frak M}_{1,c}({\cal Y})$ tends to an element $\tau \in {\frak M}_{1,c}({\cal Y})$ with 
respect to the topology ${\cal O}$ if and only if (a) for each bounded continuous function 
$\varphi :{\cal Y}\rightarrow \CC $, $\int \varphi \ d\tau _{n}\rightarrow 
\int \varphi \ d\tau _{n}$ as $n\rightarrow \infty $, and (b) 
$\mbox{supp}\,\tau _{n}\rightarrow  \mbox{supp}\,\tau$ as $n\rightarrow \infty $ 
in Cpt$({\cal Y})$ with respect to the Hausdorff metric. 

Let $C(\CCI )$ be the space of all complex-valued continuous functions on $\CCI $ endowed with 
the supremum norm $\| \cdot \| _{\infty }.$  
Let $M_{\tau }$ be the operator on $C(\CCI )$ 
defined by $M_{\tau }(\varphi )(z)=\int \varphi (g(z)) d\tau (g).$ 
This $M_{\tau }$ is called the {\bf transition operator} of the Markov process induced by $\tau .$ 
Let $M_{\tau }^{\ast }:{\frak M}_{1}(\CCI )\rightarrow {\frak M}_{1}(\CCI )$ 
be the dual of $M_{\tau }$.   
%where ${\frak M}_{1}(\CCI )$ denotes the space of all Borel probability measures on 
%$\CCI  $ endowed with the weak topology. 
This $M_{\tau }^{\ast }$ can be regarded as the ``averaged map'' 
on the extension ${\frak M}_{1}(\CCI )$ of $\CCI $ (see Remark~\ref{r:Phi}). 

We now present the first main result of this paper.

\begin{thm}[For the detailed and more general version, see Theorems~\ref{t:rcdnkmain1}, \ref{t:zfggzcc}]
\label{t:rcdnkmain1i}
Let ${\cal Y}$ be a mild subset of $\emRatp$ and suppose that 
${\cal Y}$ is nice with respect to a holomorphic family  
${\cal W}$ of rational maps. Then 
%the set 
%$$\{ \tau \in {\frak M}_{1,c}({\cal Y})\mid \tau \mbox{ 
%is weakly mean stable} \}$$ 
%is open and dense in $({\frak M}_{1,c}({\cal Y}), {\cal O})$.
%Moreover, 
there exists an open and dense subset ${\cal A}$ of 
$({\frak M}_{1,c}({\cal Y}), {\cal O})$ 
such that for each $\tau \in {\cal A}$, the following {\em (I)} and {\em (II)} hold.
\begin{itemize}
%\item[{\em (I)}] 
%We have $J_{\ker }(G_{\tau })\subset S({\cal W}),\  \sharp J_{\ker }(G_{\tau })<\infty $ and 
%$\sharp \emMin(G_{\tau })<\infty .$ 
%Moreover, each $L\in \emMin(G_{\tau },\CCI )$ with 
%$L\not\subset J_{\ker }(G_{\tau })$ is attracting for $\tau .$ 

\item[{\em (I)}] 
{\em (}{\bf Convergence}{\em)} 
There exist  numbers $l, r\in \NN $,  probability measures 
$\eta _{1},\ldots, \eta _{r}\in {\frak M}_{1}(\CCI )$ and 
functions  $\alpha _{1},\ldots, \alpha _{r}: \CCI \rightarrow [0,1]$ such that 
for each $y\in \CCI $ and for each $\varphi \in C(\CCI )$, we have 
\begin{equation}
\label{eq:mtivpc}
M_{\tau }^{nl}(\varphi )(y)\rightarrow \sum _{i=1}^{r}\alpha _{i}(y)\int \varphi \, d\eta _{i} 
\mbox{ as }n\rightarrow \infty \mbox{ (pointwise convergence) } ,    
\end{equation}
i.e., we have $(M_{\tau }^{\ast })^{nl}(\delta _{y})\rightarrow \sum _{i=1}^{r}
\alpha _{i}(y)\eta _{i}$ as $n\rightarrow \infty $ in ${\frak M}_{1}(\CCI )$ with respect to the 
weak convergence topology. Also, we have 
$(M_{\tau }^{\ast })^{l}(
\sum _{i=1}^{r}\alpha _{i}(y)\eta _{i})=\sum _{i=1}^{r}\alpha _{i}(y)\eta _{i}. $ 
%Moreover, for each $i=1,\ldots, r$, $\mbox{supp}\, \eta _{i}$ is included in an element 
%$L\in \emMin(G_{\tau },\CCI )$ and 
%$\cup _{i=1}^{r}\mbox{supp}\,\eta _{i}=\cup _{L\in \emMin(G_{\tau },\CCI )}L.$ 
%Moreover, these functions $\alpha _{1},\ldots, \alpha _{r}$ are locally constant on $F(G_{\tau %}).$ 
%Furthermore, for each $i=1,\ldots, r$ and for each $y\in F_{pt}^{0}(\tau )$, we have 
%$\lim _{w\in \CCI ,w\rightarrow y}\alpha _{i}(w)=\alpha _{i}(y).$ 
%Also, for each $L\in \emMin(G_{\tau },\CCI )$ and for each $y\in F_{pt}^{0}(\tau )$, we have 
%$\lim_{w\in \CCI ,w\rightarrow y}T_{L,\tau }(w)=T_{L,\tau }(y).$ 
\item[{\em (II)}] 
We have $\sharp J_{\ker }(G_{\tau })<\infty $ and $\sharp \emMin(G_{\tau },\CCI )<\infty .$ 
Moreover, 
for each $y\in \CCI $, there exists a Borel subset $B_{\tau ,y}$ of $(\emRatp)^{\NN}$ 
with $\tilde{\tau }(B_{\tau ,y})=1$ such that for each $\gamma =(\gamma _{1},\gamma _{2},
\ldots, )\in B_{\tau ,y}$,  there exists an element 
$L=L(y,\gamma )\in \emMin(G_{\tau },\CCI )$ for which we have that $d(\gamma _{n}\circ \cdots \circ \gamma _{1}(y), 
L)\rightarrow 0$ as $n\rightarrow \infty . $ 
\end{itemize}
 
\end{thm}
We remark that statements (I)(II) in Theorem~\ref{t:rcdnkmain1i} cannot hold for 
deterministic iteration dynamics of a single $f\in \Ratp$, since the dynamics of 
$f: J(f)\rightarrow J(f)$, where $J(f)$ denotes the Julia set of $f$, is chaotic. 
In fact, it is well-known that  
(a) for a generic $z\in J(f)$, the orbit $\{ f^{n}(z)\} _{n=1}^{\infty }$ is dense in $J(f)$ 
 and for each $l\in \NN $, $\delta _{f^{nl}(y)} $ does not converge to any probability measure on $\CCI $ as $n\rightarrow \infty $, (b) setting $\langle f\rangle := \{ f^{n}\mid 
 n\in \NN \}$, we have that $J_{\ker }(\langle f\rangle )=J(f)$ and $J_{\ker }(\langle f\rangle )$ is uncountable, and (c) there are infinitely many minimal sets of $f$ in $J(f)$, i.e., we have 
 infinitely many periodic cycles of $f$ in $J(f).$ 
Thus Theorem~\ref{t:rcdnkmain1i} deals with some 
randomness-induced phenomena.  

To present the second main theorem, 
for each $\tau \in {\frak M}_{1,c}(\Rat)$ and for each 
$L\in \Min(G_{\tau },\CCI )$ with $\sharp L<\infty$, 
we define the {\bf Lyapunov exponent} of $(\tau ,L)$ 
and denote it by $\chi (\tau, L)$ (see Definition~\ref{d:celyap}). 
 Also, if  ${\cal Y}$ is a weakly nice subset of $\Rat$ with 
 respect to a 
 holomorphic 
family ${\cal W}$ of rational maps, we say that  
 ${\cal Y}$ is {\bf exceptional with respect to ${\cal W}$} if there exists a 
non-empty subset $L$ of $S({\cal W})$ such that for each 
$\tau \in {\frak M}_{1,c}({\cal Y})$, we have $L\in \Min(G_{\tau },\CCI )$ 
and $\chi (\tau, L)=0.$ We say that ${\cal Y}$ is {\bf non-exceptional with 
respect to ${\cal W}$} if ${\cal Y}$ is not exceptional with respect to  ${\cal W}$ 
(For the definition in more general setting, see Definition~\ref{d:exceptional}).  
  
For each sequence $\gamma =(\gamma _{1}, \gamma _{2},\ldots )\in (\Rat )^{\NN }$, and 
for each $m,n\in \NN $ with $m\geq n$, we set $\gamma _{m,n}=
\gamma _{m}\circ \cdots \circ \gamma _{n}$ and 
we denote by $F_{\gamma }$ the set of points $z\in \CCI $ 
satisfying that there exists an open neighborhood of $z$ on which the sequence 
$\{ \gamma _{n,1}\} _{n=1}^{\infty }$ is equicontinuous with respect to 
%$d.$
the spherical distance on $\CCI .$  
This 
$F_{\gamma }$ is called the {\bf Fatou set of the sequence $\gamma $}. Also, we set  
 $J_{\gamma }:=\CCI \setminus F_{\gamma }$ and this $J_{\gamma }$ is called the {\bf Julia set of 
$\gamma .$} 

We now present the second main theorem of this paper.
\vspace{-1mm} 
\begin{thm}[({\bf Negativity of Lyapunov Exponents}) 
For the detailed and more general version, see Theorem~\ref{t:rcdnkmain2}]
\label{t:rcdnkmain2i} 
Let ${\cal Y}$ be a mild subset of $\emRatp$ and 
suppose that ${\cal Y}$ is  nice and non-exceptional with respect to a  
holomorphic family ${\cal W}=\{ f_{\lambda }\} _{\lambda \in \Lambda }$ of rational maps. 
%Suppose that ${\cal Y}$ is non-exceptional with respect to 
%${\cal W}.$ 
%, where 
%${\cal W}_{j}=\{ f_{j,\lambda }\mid \lambda \in \Lambda _{j}\}$, for each $j=1,\ldots ,m. $
%Suppose that $\{ f_{i,\lambda }\mid \lambda \in \Lambda _{i}\} \cap 
%\{ f_{j,\lambda }\mid \lambda \in \Lambda _{j}\} $ is nowhere dense in $\{ f_{i,\lambda }
%\mid \lambda \in \Lambda _{i}\}$ for all $(i,j)$ with $i\neq j.$ 
Then there exists an open and 
dense subset ${\cal A}$ of $({\frak M}_{1,c}({\cal Y}), {\cal O})$ such that 
for each $\tau \in {\cal A}$, all of the following statements {\em (I)} and {\em (II)} hold.
\begin{itemize}
\vspace{-1.1mm} 
%\item[{\em (I)}] 
%Let $H_{+,\tau }=\{ L\in \emMin(G_{\tau }, \CCI )\mid L\subset J_{\ker }(G_{\tau }), \chi (\tau, L)>0\}$ and 
%let $\Omega _{\tau }$ be the set of points 
%$y\in \CCI$ for which  
%$\tilde{\tau }(\{ \gamma \in (\emRatp)^{\NN }\mid \exists n\in \NN \mbox{ s.t. }
%\gamma _{n,1}(y)\in \cup _{L\in H_{+,\tau}}L\} )=0.$ 
%Then %There exists a subset $\Omega _{\tau }$ of $\CCI $ with 
%we have 
%$\Omega _{\tau }=F_{pt}^{0}(\tau ) $, 
%$\sharp (\CCI \setminus \Omega _{\tau })\leq \aleph _{0}$ 
%such that 
%and for each $z\in \Omega _{\tau }$, 
%$\tilde{\tau }(\{ \gamma \in (\emRatp)^{\NN }\mid z\in J_{\gamma }\} )=0.$
%Moreover,   
%{\em Leb}$_{2}(J_{\gamma })=0$ 
%for $\tilde{\tau }$-a.e.$\gamma \in (\emRatp)^{\NN}$. 
%Also,  
%$\cup _{L\in H_{+,\tau }}L\subset J_{pt}^{0}(\tau )= \CCI \setminus \Omega _{\tau }$ and 
%$\sharp J_{pt}^{0}(\tau )\leq \aleph_{0}.$ 
\item[{\em (I)}] 
%Let $\Omega _{\tau }$ be as in {\em (I)}. Then 
There exist a subset $\Omega _{\tau }$ of $\CCI $ with $\sharp (\CCI \setminus \Omega _{\tau })\leq \aleph _{0}$,  
%$\sharp (\CCI \setminus \Omega _{\tau })\leq \aleph _{0}$ and 
 a constant $c_{\tau }<0$ and a constant $\rho _{\tau }\in (0,1)$ 
%and a subset $B_{\tau }$ of $\CCI $ 
%with $\sharp (\CCI \setminus B_{\tau })\leq \aleph _{0}$ 
such that for each $z\in \Omega_{\tau }$, there exists a Borel subset 
$C_{\tau ,z}$ of $(\emRatp)^{\NN }$ with $\tilde{\tau }(C_{\tau ,z})=1$ 
satisfying that for each $\gamma =(\gamma _{1},\gamma _{2},\ldots )\in C_{\tau ,z}$ and 
for each $m\in \NN \cup \{ 0\} $, 
we have the following {\em (a)} and {\em (b)}. 
\begin{itemize}
\item[{\em (a)}]  
\vspace{-2.7mm} 
$$\limsup _{n\rightarrow \infty }\frac{1}{n}\log \| D(\gamma _{n+m,1+m})_{\gamma _{m,1}(z)}\| _{s}\leq c_{\tau }<0.$$
Here, for any $g\in \Rat$ and $z\in \CCI$, 
we denote by $\| Dg_{z}\| _{s}$ the norm of the derivative of $g$ at $z$ with respect to the spherical 
metric.  
\vspace{-1mm} 
\item[{\em (b)}] 
There exist a constant $\delta =\delta (\tau, z,\gamma, m)>0$,  a constant 
$\zeta =\zeta (\tau, z, \gamma, m)>0$ and an element 
$L=L(\tau, z,\gamma )\in \emMin(G_{\tau }, \CCI )$ 
which is either {\em (i)} ``attracting for $\tau $'', or {\em (ii) } ``finite and included in $J_{\ker }(G_{\tau }) $ with 
$\chi (\tau, L)<0$'',  
such that 
\vspace{-0.1mm} 
$$\mbox{{\em diam}}(\gamma _{n+m,1+m}(B(\gamma _{m,1}(z),\delta )))\leq \zeta \rho _{\tau }^{n}\ \mbox{ for all }n\in \NN, $$
\vspace{-1.2mm} 
where  
%Leb$_{2}(B)$ denotes the $2$-dimensional 
%Lebesgue measure of $B$ and 
we set {\em diam}$(B)=\sup _{x,y\in B}d(x,y)$ for any set $B\subset \CCI $, 
and such that 
\vspace{-0.1mm} 
$$d(\gamma _{n+m,1+m}(\gamma _{m,1}(z)), \ L)
%\cup _{L\in \emMin(G_{\tau },\CCI ),\ L \mbox{ is attracting 
%for }\tau } L) 
\leq \zeta \rho _{\tau }^{n} \mbox{\ \  for all }n\in \NN .$$ 
\end{itemize} 
\item[{\em (II)}]
For $\tilde{\tau}$-a.e. $\gamma \in (\emRatp)^{\NN }$, we have that 
$\mbox{{\em Leb}}_{2}(J_{\gamma })=0$, where $\mbox{{\em Leb}}_{2}$ denotes that 
$2$-dimensional Lebesgue measure on $\CCI .$ 
\end{itemize}
 
\end{thm}
\begin{rem}
\label{r:genrcdnkmain}
In Theorems~\ref{t:rcdnkmain1}, \ref{t:rcdnkmain2}, we show more generalized results 
in which we deal with random dynamical systems %generated by 
of 
$\tau \in {\frak M}_{1,c}(\Rat)$ 
%whose support 
such that supp$\,\tau $
is included in a finite union of loci of holomorphic families 
$\{ {\cal W}_{j}\}_{j=1}^{m}$   
of rational maps, and 
%whose support
supp$\,\tau$ meets 
the locus of each ${\cal W}_{j}.$ 
\end{rem}
\vspace{-1.7mm} 
We remark  that statements (I), (II) in Theorem~\ref{t:rcdnkmain2i} cannot hold for 
deterministic iteration dynamics of a single $f\in \Ratp$. In fact the dynamics of 
$f: J(f)\rightarrow J(f)$ is chaotic, 
and we have Ma\~{n}\'{e}'s result $\dim _{H}(\{ z\in \CCI \mid \liminf _{n\rightarrow \infty }\frac{1}{n}
\log \| D(f^{n})_{z}\| _{s}>0\} )>0$, where $\dim _{H}$ denotes the Hausdorff dimension 
with respect to the spherical distance on $\CCI $ (see \cite{Ma}). In particular, 
the set of points $z\in J(f)$ for which $\liminf _{n\rightarrow \infty }\frac{1}{n}
\log \| D(f^{n})_{z}\| _{s}>0$ is uncountable. Also, it is well-known that 
for any open subset $U$ of $\CCI $ with $U\cap J(f)\neq \emptyset$, 
there exists an $N\in \NN $ such that for each $n\in \NN $ with $n\geq N$, we have $f^{n}(U)\supset J(f)$ and 
$\mbox{diam}(f^{n}(U))\geq \mbox{diam}(J(f))>0.$ 
Thus Theorem~\ref{t:rcdnkmain2i} deals with a randomness-induced phenomenon. 
As we see in Theorems~\ref{t:rcdnkmain1i}, \ref{t:rcdnkmain2i}, 
under the assumptions of Theorems~\ref{t:rcdnkmain1i}, \ref{t:rcdnkmain2i}, 
{\bf regarding generic  random complex dynamical systems   
(in particular, regarding generic random polynomial dynamical systems),   
the chaoticity is much weaker than that of deterministic complex dynamical systems.}   
This arises from the effect of randomness and Theorems~\ref{t:rcdnkmain1i}, \ref{t:rcdnkmain2i} 
deal with randomness-induced phenomena.  
Note that {\bf the statements in Theorems~\ref{t:rcdnkmain1i}, 
\ref{t:rcdnkmain2i} are a kind of analogues of 
the conjecture of density of hyperbolic maps} 
(\cite{Mc}) in 
deterministic complex dynamics. 

We remark that in \cite{FS} and \cite{Splms10,Sadv},  regarding random complex dynamical systems, 
results on disappearance of chaos 
were shown.  
In \cite{FS}, 
it was assumed that $S({\cal W})=\emptyset$ and the noise is very small, which implies that  
the systems in the paper have empty kernel Julia sets $J_{\ker }(G_{\tau })$ of corresponding 
rational semigroups.   
In \cite{Sadv}, it was also assumed that $S({\cal W})=\emptyset $ 
(for a holomorphic family ${\cal W}$ of polynomials, it was assumed that $S({\cal W})\setminus \{ \infty \} =\emptyset$)
but 
the range of the noise could be big, and it was shown that the generic systems have 
empty kernel Julia sets, which implies that the chaoticity of the systems is much 
weaker than that of deterministic complex dynamical systems. 
In this paper,   it is important that 
{\bf in Theorems~\ref{t:rcdnkmain1i} and \ref{t:rcdnkmain2i}, 
the set ${\cal A}$ may contain many 
$\tau $ for which 
$J_{\ker }(G_{\tau })\neq \emptyset $ 
and 
%$H_{+,\tau }\neq \emptyset $
there exists an element $L\in \Min(G_{\tau },\CCI )$ which is finite 
and  $\chi (\tau, L)>0$
}   
(Theorem~\ref{t:RNM1i}, 
Corollary~\ref{c:RNM1ione},  Lemma~\ref{l:RNMnnen}, 
Examples~\ref{ex:applyi}, \ref{ex:wlz1z}--
\ref{ex:pzljgz}).  
{\bf Once we have non-empty kernel Julia set, the analysis of the system 
is much more difficult than the cases with empty kernel Julia sets}, even if 
the kernel Julia set is finite.  We need a new framework and more technical arguments 
to study such systems.

We apply the results and the methods in the above to finding 
roots of {\bf any polynomial} $g\in {\cal P}$ by random relaxed Newton's methods as we explained below.  
 Let $g\in {\cal P}.$ 
Let $\Lambda :=\{ \lambda \in \CC \mid |\lambda -1|<1\} $ and 
let $N_{g,\lambda }(z)=z-\lambda \frac{g(z)}{g'(z)}$ for each $\lambda \in \Lambda .$ 
Let ${\cal W}_{g}=\{ N_{g,\lambda }\} _{\lambda \in \Lambda }.$ 
Let ${\cal Y}_{g}:=\{ N_{g,\lambda }\in \Rat \mid \lambda \in \Lambda \}.$  
Then  ${\cal Y}_{g}$ is called the {\bf random relaxed Newton's method set for $g$}  
and  ${\cal W}_{g}$ is called the {\bf random relaxed Newton's method family for $g.$}  
Also, $({\cal Y}_{g}, {\cal W}_{g})$ is called the {\bf random relaxed Newton's method scheme for $g.$}  
Moreover, for each $\tau \in {\frak M}_{1,c}({\cal Y}_{g})$, the random dynamical system on $\CCI $ 
generated by $\tau $ is called a 
{\bf random relaxed Newton's method (or random relaxed Newton's method system) 
for $g. $}  Furthermore, let $Q_{g}:=\{ z_{0}\in \CC \mid g(z_{0})=0\}.$ 

We now present the third main theorem of this paper. 

\begin{thm}[For the details, see Theorem~\ref{t:RNM1}]
\label{t:RNM1i}
Let $g\in {\cal P}.$ Let 
$({\cal Y}_{g}, {\cal W}_{g})$ be the random relaxed Newton's method scheme for $g.$ 
Then 
%${\cal Y}_{g}$ is a mild subset of $\emRatp$, 
%and nice subset of $\emRatp$ with 
%${\cal W}$, 
%the set ${\cal Y}_{g}$ is nice and non-exceptional with respect to ${\cal W}_{g}$ 
%  and $({\cal Y}_{g}, {\cal W}_{g})$ satisfies the assumptions of Theorems~\ref{t:rcdnkmain1i}, 
%\ref{t:rcdnkmain2i}. 
%Moreover, 
%there exists an open and dense subset ${\cal A}$ of 
%${\frak M}_{1,c}({\cal Y}_{g})$ such that all of the following hold. 
%we have the following. 
we have the following {\em (I)} and {\em (II)}. 
\begin{itemize}
\item[{\em (I)}] {\em ({\bf Almost Sure Convergence to a Root of $g$})} Let $\Lambda =\{ \lambda \in \CC \mid |\lambda -1|<1\}.$ 
%Let $z_{0}\in \CC \setminus \Bbb{D}.$ 
Let $\eta \in {\frak M}_{1,c}(\Lambda )$ be an element such that 
int$(\mbox{{\em supp}}\, \eta )\supset \{ \lambda \in \CC \mid |\lambda -1|\leq \frac{1}{2}\} $ and 
$\eta $ is absolutely continuous with respect to the $2$-dimensional Lebesgue measure on $\Lambda .$ 
Here, int{\em(}supp$\,\eta${\em)} denotes the set of 
interior points of supp$\,\eta$ with respect to the topology in 
$\Lambda.$ 
%Under the identification ${\cal Y}\cong \Lambda $, we regard $\eta $ as an element 
%of ${\frak M}_{1,c}({\cal Y}_{1}).$ 
Let $\tilde{\eta }=\otimes _{n=1}^{\infty }\eta 
\in {\frak M}_{1}(\Lambda ^{\NN }).$ 
Then for each $z_{0}\in \CC \setminus 
\{ z\in \CC \mid g'(z)=0 \mbox{ and } g(z)\neq 0\} $,  
there exists a Borel subset $C_{\eta ,z_{0}}$ of 
$\Lambda ^{\NN }$ with $\tilde{\eta }(C_{\eta, z_{0}})=1$ 
such that 
for each  
$(\lambda _{1},\lambda _{2},\ldots ) \in 
C_{\eta ,z_{0}}$,  the sequence 
%\in (\emRatp)^{\NN }$, 
%$\gamma _{n,1}(z_{0})$ 
$\{ N_{g,\lambda _{n}}\circ \cdots \circ N_{g,\lambda _{1}}
(z_{0})\} _{n=1}^{\infty }$ 
tends to a root $x=x(z_{0}, \lambda _{1},\lambda _{2},\ldots) $ of $g$ as $n\rightarrow \infty $ exponentially fast. 

 \ Also, for $\tilde{\eta }$-a.e. $\overline{\lambda }=
(\lambda _{1},\lambda _{2},\ldots )\in \Lambda ^{\NN }$, 
we have that $\mbox{{\em Leb}}_{2}(J_{\gamma (\overline{\lambda })})=0$ and for each $z\in F_{\gamma (\overline{\lambda })}$, 
there exists a root $x=x(\tau , \overline{\lambda }, z)$ of $g$ 
such that $N_{g, \lambda _{n}}\circ \cdots \circ N_{g,\lambda _{1}}(z)\rightarrow x$ as $n\rightarrow \infty $ exponentially fast. 
Here, we set $\gamma (\overline{\lambda })=
(N_{g,\lambda _{1}}, N_{g, \lambda _{2}},\ldots )\in {\cal Y}_{g}^{\NN }$ 
for each $\overline{\lambda }=(\lambda _{1},\lambda _{2},\ldots )
\in \Lambda ^{\NN }.$

\item[{\em (II)}] 
There exists an open and dense subset ${\cal A}$ 
of ${\frak M}_{1,c}({\cal Y}_{g})$ such that we have all of the following {\em (i)(ii)(iii)}.
\begin{itemize}
\item[{\em (i)}] 
Regarding any element $\eta \in {\frak M}_{1,c}(\Lambda )$ 
as in {\em (I)}, we have $\tilde{\eta }\in {\cal A}$, under the 
canonical identification $\Lambda \cong {\cal Y}_{g}.$ 
\item[{\em (ii)}] 
 For each $\tau \in {\cal A}$,   statements {\em (I)(II)} 
in Theorem~\ref{t:rcdnkmain1i} and statements {\em (I)(II)} in Theorem~\ref{t:rcdnkmain2i} 
hold for $\tau .$   
\item[{\em (iii)}] 
Let $\tau \in {\cal A}.$ Then 
$\emMin(G_{\tau },\CCI )
$ is equal to the union of 
$\{ \{ x\} \mid x\in Q_{g}\} \cup \{ \{\infty \} \} $ and 
$\{ L\in \emMin(G_{\tau },\CCI )\mid L\subset \CC \setminus Q_{g}, L\mbox{ is attracting for }\tau \}$.  
Also, for each $x\in Q_{g}$, the minimal set $\{ x\}$ is attracting for $\tau .$  Furthermore, 
$J_{\ker }(G_{\tau })\neq \emptyset .$ 
\end{itemize} 
 
%\item[{\em (II)}] 
%Let $\tau \in {\cal A}.$ 
%Let $\Omega _{\tau }$ be the set defined in Theorem~\ref{t:rcdnkmain2i}. 
%Then $\sharp (\CCI \setminus \Omega _{\tau })\leq \aleph _{0}$ and 
%$$\Omega _{\tau }=\{ y\in \CC 
%\mid \tilde{\tau }(\{ \gamma \in (\emRat)^{\NN }\mid \exists n\in \NN 
%\mbox{ s.t. }\gamma _{n,1}(y)=\infty \})=0\} .$$  
\end{itemize}
\end{thm}
%Note that in Theorem~\ref{t:RNM1i}, we have 
%several kinds of nice effects of noise or randomness, 
%even though $J_{\ker}(G_{\tau })\neq \emptyset $ for each 
%$\tau \in {\cal A}.$ 

We say that a non-constant polynomial $g$ is 
{\bf normalized} if 
$\{ z_{0}\in \CC \mid g(z_{0})=0\} $ is included in  
$\Bbb{D}:=\{ z\in \CC \mid |z|<1\} .$ 
For a given polynomial $g$, sometimes it is not difficult for us to 
find an element $a\in \RR $ with $a>0$ such that $g(az)$ is a normalized polynomial of $z$. 
It is well-known that if $g\in {\cal P}$ is a normalized polynomial, then so is $g'$ 
(see \cite{Ah}). Thus, we obtain the following corollary. 
\begin{cor}
\label{c:RNM1ione}
Let $g\in {\cal P}$ be a normalized polynomial. 
Let $\Lambda =\{ \lambda \in \CC \mid |\lambda -1|<1\}.$ 
Let $z_{0}\in \CC \setminus \Bbb{D}.$ 
Let $\eta \in {\frak M}_{1,c}(\Lambda )$ be an element such that 
int$(\mbox{{\em supp}}\, \eta )\supset \{ \lambda \in \CC \mid |\lambda -1|\leq \frac{1}{2}\} $ and 
$\eta $ is absolutely continuous with respect to the $2$-dimensional Lebesgue measure on $\Lambda .$ %Under the identification ${\cal Y}\cong \Lambda $, we regard $\eta $ as an element 
%of ${\frak M}_{1,c}({\cal Y}_{1}).$ 
Let $\tilde{\eta }=\otimes _{n=1}^{\infty }\eta 
\in {\frak M}_{1}(\Lambda ^{\NN }).$ 
Then for $\tilde{\eta }$-a.e. 
$(\lambda _{1},\lambda _{2},\ldots ) \in 
\Lambda ^{\NN} $, 
%\in (\emRatp)^{\NN }$, 
%$\gamma _{n,1}(z_{0})$ 
$\{ N_{g,\lambda _{n}}\circ \cdots \circ N_{g,\lambda _{1}}
(z_{0})\} _{n=1}^{\infty }$ 
tends to a root $x=x(z_{0}, \lambda _{1},\lambda _{2},\ldots) $ of $g$ as $n\rightarrow \infty .$  Moreover, if, in addition to the assumptions of our theorem, 
we know 
the coefficients of $g$ explicitly,  
then by the following algorithm in which we consider $\deg (g)$-random orbits  
of $z_{0}$ under $\deg (g)$-different random relaxed Newton's methods,  we can find all roots of $g$ almost surely 
with arbitrarily small errors. 
\begin{itemize}
\item[{\em (1)}]  
We first consider the random relaxed Newton's method scheme $({\cal Y}_{g_{1}}, {\cal W}_{g_{1}})$ for 
$g_{1}=g$.  
By Theorem~\ref{t:RNM1i},  for $\tilde{\eta }$-a.e. 
$(\lambda _{1},\lambda _{2},\ldots )
 \in \Lambda ^{\NN }$, 
$\{ N_{g_{1},\lambda _{n}}\circ 
\cdots \circ N_{g_{1},\lambda _{1}}(z_{0})\} _{n=1}^{\infty }$ tends to a root 
$x=x(z_{0}, \lambda _{1},\lambda _{2},\ldots ) $ of $g.$ 
Let $x_{1}$ be one of such 
$x(z_{0}, \lambda _{1},\lambda _{2},\ldots )$ (with arbitrarily small error).  
\item[{\em (2)}] 
Let $g_{2}(z)=g(z)/(z-x_{1}).$ 
By using synthetic division, we regard $g_{2}$ as a polynomial which divides $g_{1}$ 
(with arbitrarily small error). 
Note that $g_{2}$ is also a normalized polynomial. 
  We consider the random relaxed Newton's method scheme $({\cal Y}_{g_{2}}, {\cal W}_{g_{2}})$ 
for $g_{2}.$ 
%As in the first step (replacing $g_{1}$ by $g_{2}$), we find a root $x_{2}$ of $g_{2}$,  
By Theorem~\ref{t:RNM1i},  for $\tilde{\eta }$-a.e. 
$(\lambda _{1},\lambda _{2},\ldots )
 \in \Lambda ^{\NN }$, 
$\{ N_{g_{2},\lambda _{n}}\circ 
\cdots \circ N_{g_{2},\lambda _{1}}(z_{0})\} _{n=1}^{\infty }$ tends to a root 
$x=x(z_{0}, \lambda _{1},\lambda _{2},\ldots ) $ of $g_{2},$  
which is also a root of $g$ (with arbitrarily small error).   
\item[{\em (3)}] 
Let $g_{3}(z)=g_{2}(z)/(z-x_{2})$ and as in the above, we find a root $x_{3}$ of $g$ with 
arbitrarily small error. Continue this method.   
\end{itemize}
%Note that even if we do not 
\end{cor}
We remark that in Theorem~\ref{t:RNM1i} and Corollary~\ref{c:RNM1ione}, we have several kinds of nice effects of 
noise or randomnees, even though 
any system has non-empty kernel Julia set of the corresponding rational semigroup 
(in fact, $\infty $ is a common repelling fixed point of any map in the system and $\{ z_{0}\in \CC \mid g'(z_{0})=0, g(z_{0})\neq 0\} \cup \{ \infty \} $ is included in the kernel Julia set).   
In order to analyze such systems with non-empty kernel Julia sets,   
we need a new framework and much more technical arguments than those of 
\cite{Splms10}, \cite{Sadv}. See the second remark after Remark~\ref{r:genrcdnkmain}. 

 \begin{rem}
\label{r:RNM1i}
{\bf (I)} Regarding the original Newton's method, 
M. Hurley showed in \cite{Hu} that 
for any $k\in \NN $ with $k\geq 3$, there exists a polynomial $g$ of $\deg (g)=k$ 
such that the Newton's method map 
$N_{g}(z)=N_{g,1}(z)=z-g(z)/g'(z)$ for $g$ has $2k-2$ different 
attracting cycles. Thus this $N_{g}$ has $k-2$ attracting cycles which are not zeros of $g.$  
Since attracting cycles are stable under small perturbations, it follows that 
for any $k\geq 3$, 
%the set of elements $g$ for which the Newton's map has attracting cycles which are not zeros of $g$ 
%is a non-empty open subset of ${\cal P}_{k}:=\{ g\in {\cal P}\mid \deg (g)=k\} .$ 
there exist a non-empty open subset 
$U_{k}$ of ${\cal P}_{k}:=
\{ g\in {\cal P}\mid \deg (g)=k\} $ and 
a non-empty open subset $V_{k}$ of $\CCI $ such that for each $(g,z_{0})\in  U_{k}\times 
V_{k}$, $\{ N_{g}^{n}(z_{0})\} _{n=1}^{\infty }$ 
{\bf cannot converge to any root of $g.$} 

{\bf (II)} C. McMullen showed in \cite{Mc87} that for any $k\in \NN , k\geq 4$ and for any $l\in \NN$, 
there exists {\bf NO} rational map $\tilde{N}: {\cal P}_{k}\rightarrow \Rat_{l}:=\{ f\in \Rat \mid 
\deg (f)=l\}$ such that for any $g$ in an open dense subset of $ {\cal P}_{k}$, 
for any $z$ in an open dense subset of $\CCI $, $\tilde{N}(g)^{n}(z)$ tends to a root of $g$ as $n\rightarrow 
\infty .$  

{\bf (III)}  {\bf The essential assumptions on $\eta $ in Theorem~\ref{t:RNM1i} (III)
 and $\eta $ in Corollary~\ref{c:RNM1ione}  of this paper do not depend on $g\in {\cal P}$}.   
By (I)(II), it follows that {\bf the statements of Theorem~\ref{t:RNM1i} and Corollary~\ref{c:RNM1ione} 
cannot hold in the deterministic Newton's method 
and any other deterministic complex analytic iterative schemes to find roots of polynomials.  Thus 
 Theorem~\ref{t:RNM1i} and Corollary~\ref{c:RNM1ione} deal with 
 randomness-induced phenomena}. 

{\bf (IV)} J. Hubbard, D. Schleicher and S. Sutherland showed in \cite{HSS} that 
for each $k\in \NN $,  
there exists a finite subset $B_{k}$ of $\CC $ with $\sharp B_{k}\leq 1.1k (\log k)^{2}$ 
satisfying that  for any normalized polynomial $g$ with $\deg (g)=k$ and for every root $x$ of $g$, 
there exists a point $z_{0}\in B_{k}$ 
  such that $\{ N_{g}^{n}(z_{0})\} _{n=1}^{\infty }$ 
  converges to $x.$ 
%of the same Newton's method map $N_{g}$ for $g.$ 
%Note that $B$ does not depend on $g$ and by this method we can find 
%all roots of $g$ simultaneously by one map

\end{rem}
Note that this is the first paper to investigate random relaxed Newton's method 
systematically. It is important that in Theorem~\ref{t:RNM1i} (III) and 
Corollary~\ref{c:RNM1ione}, the size of the noise is big which enables the system 
to make the minimal set with period greater than 1 collapse. 
However, since the size of the noise is big, it is not enough for us to 
consider the arguments which are similar to those of deterministic dynamics of one map, 
thus we have to develop the theory of random complex dynamical systems 
with noise or randomness of any size as in Theorems~\ref{t:rcdnkmain1i}, \ref{t:rcdnkmain2i}. 
  
As we see  before, in Theorems~\ref{t:rcdnkmain1i} and \ref{t:rcdnkmain2i}, 
the chaoticity of random complex dynamical systems is much weaker than that of 
deterministic dynamical systems. However, 
the random systems may have still a kind of complexity or chaoticity. 
For example, when we consider the function $T_{L,\tau }$ of probability of tending to 
one $L\in \Min(G_{\tau },\CCI )$, then under certain conditions, 
this function is continuous on $\CCI $ and even more, this is $\alpha $-\Hol der continuous on $\CCI $ 
for some $\alpha \in (0,1)$ but there exists an element $\beta \in (0,1)$ such that 
$T_{L,\tau }$ cannot be $\beta $-\Hol der continuous on $\CCI $. This implies that 
the system generated by $\tau $ does not act mildly 
(i.e., the transition operator $M_{\tau }$ of $\tau $ does not act mildly) on the Banach space $C^{\beta }(\CCI )$ 
of $\beta $-\Hol der continuous functions on $\CCI $ endowed with 
$\beta$-\Hol der norm $\| \cdot \| _{\beta }$ 
(e.g., there exists a $\varphi \in C^{\beta }(\CCI )$ such that 
$\| M_{\tau }^{n}(\varphi )\| _{\beta }\rightarrow \infty $ as $n\rightarrow \infty $).  Thus regarding the random (complex) dynamical systems, 
we have the {\bf gradations between  chaos and order} (see \cite{Splms10, Sadv, JS1, 
JSAdv, JS3, Sspace}).  

 In Theorems~\ref{t:nonmild} and \ref{t:rcdnkmain2nomild}, 
 we show the results on random dynamical systems generated by measures $\tau $ on a nice subset 
 ${\cal Y}$ of Rat without assuming ${\cal Y}$ is mild. We show that 
 considering the mild part ${\frak M}_{1,c,mild}({\cal Y})$ 
 (the set of elements $\tau $ which has an attractor, see Definition~\ref{d:m1cmild}), 
 there exists an open and dense subset ${\cal A}$ of ${\frak M}_{1,c,mild}({\cal Y})$ 
 such that for each $\tau \in {\cal A}$,  statements (I)(II) in Theorem~\ref{t:rcdnkmain1i} 
 and statements  (I)(II) in Theorem~\ref{t:rcdnkmain2i}  hold. Also, 
 denoting by  ${\frak M}_{1,c,JF}({\cal Y})$ the set of 
 elements 
 $\tau \in {\frak M}_{1,c}({\cal Y})$ for which $J(G_{\tau })=\CCI $ and either $\Min(G_{\tau },\CCI )=\{ \CCI \} $ 
 or $\cup _{L\in \Min(G_{\tau },\CCI )}L\subset S({\cal W})$, 
 we show that the union of ${\cal A}$ and ${\frak M}_{1,c, JF}({\cal Y})$ 
 is dense in ${\frak M}_{1,c}({\cal Y})$ 
 (Theorems~\ref{t:nonmild} and \ref{t:rcdnkmain2nomild}).  

\begin{ex}
\label{ex:applyi}
We give some examples of ${\cal Y}$ satisfying 
the assumptions in Theorems~\ref{t:rcdnkmain1i}, \ref{t:rcdnkmain2i},  
~\ref{t:rcdnkmain1}, ~\ref{t:rcdnkmain2}. For the details and the proofs, see 
Section~\ref{Examples}.  In the following, ${\cal A}$ denotes the open and dense subset 
of $({\frak M}_{1,c}({\cal Y}), {\cal O})$ or $({\frak M}_{1,c}({\cal Y}, \{ {\cal W}_{j}\} _{j=1}^{m}), {\cal O})$ 
(for the notation, see Definition~\ref{d:weaklynice}) in 
Theorems~\ref{t:rcdnkmain1i}, \ref{t:rcdnkmain2i},  
~\ref{t:rcdnkmain1}, ~\ref{t:rcdnkmain2}.
%Theorem~\ref{t:rcdnkmain2i} or 
%Theorem~\ref{t:rcdnkmain2}. 
As mentioned before, 
{\bf if $J_{\ker}(G_{\tau })\neq \emptyset $, it is much more difficult 
%for us 
to 
show the statements on convergence of measures and 
negativity of Lyapunov exponents  
in Theorems~\ref{t:rcdnkmain1i}, \ref{t:rcdnkmain2i}, \ref{t:rcdnkmain1}, \ref{t:rcdnkmain2} 
than the cases with 
$J_{\ker }(G_{\tau })=\emptyset .$}
%empty kernel Julia sets.}   
\begin{itemize}
\item[(i)] For each $q\in \NN $ with $q\geq 2$, let ${\cal P}_{q}:= \{ f\in {\cal P}\mid \deg (f)=q\} .$ 
Let $(q_{1},\ldots, q_{m})\in \NN ^{m}$ with $q_{1}<q_{2}<\cdots <q_{m}$ and 
let ${\cal W}_{j}=\{ f\} _{f\in {\cal P}_{q_{j}}}, j=1,\ldots, m$ and let 
${\cal Y}=\cup _{j=1}^{m}{\cal P}_{q_{j}}.$ In this case, $S({\cal W}_{j})\setminus \{ \infty \}=\emptyset $ for each $j$ and 
the set $\Omega _{\tau }$ in 
Theorem~\ref{t:rcdnkmain2i} can be taken as $\CCI .$ (Note that this result has been already 
obtained in \cite{Sadv}.)  
\item[(ii)]  Let $q\in \NN $ with $q\geq 2$ and let ${\cal W}=\{ z^{q}+c\} _{c\in \CC } .$ 
Let ${\cal Y} =\{ z^{q}+c\mid c\in \CC \} .$ 
In this case, $S({\cal W})\setminus \{ \infty \} =\emptyset $ and 
the set $\Omega _{\tau }$ in 
Theorem~\ref{t:rcdnkmain2i} can be taken as $\CCI .$ (Note that this result has been already 
obtained in \cite{Sadv}.)  
\item[(iii)] Let ${\cal W}=\{ \lambda z(1-z)\} _{\lambda \in \CC \setminus \{ 0\} } $ and 
let ${\cal Y}=\{ \lambda z(1-z)\in {\cal P}_{2}\mid \lambda \in \CC \setminus \{ 0\} \} .$ 
In this case, $S({\cal W})=\{ 0,1\} \cup \{ \infty \} $ and $S({\cal W})\setminus \{ \infty \} \neq \emptyset .$ 
 There exists a non-empty  open subset ${\cal A}'$ of ${\cal A}$ such that  
for each $\tau \in {\cal A}'$, we have 
%$F_{pt}^{0}(\tau )=\Omega _{\tau }\neq \CCI $ and 
$J_{\ker }(G_{\tau })\neq \emptyset $ and 
$\Omega _{\tau }$ in Theorem~\ref{t:rcdnkmain2i} cannot be equal to $\CCI .$  
We can classify elements $\tau \in {\cal A}$ in terms of averaged behavior 
(Example~\ref{ex:wlz1z}).  
\item[(iv)] Let $f\in {\cal P}$ such that if $z_{0}\in \CC$ and  $f(z_{0})=0$, then $f'(z_{0})\neq 0.$ 
Let ${\cal W}=\{ z+\lambda f(z)\} _{\lambda \in \CC \setminus \{ 0\} }$ and 
let ${\cal Y}=\{ z+\lambda f(z)\in {\cal P}\mid \lambda \in \CC \setminus \{ 0\} \} .$ 
In this case, $S({\cal W})=\{ z_{0}\in \CC \mid f(z_{0})=0\} \cup \{ \infty \} $ and 
$S({\cal W})\setminus \{ \infty \} \neq \emptyset .$ 
Then there exists a non-empty  open subset ${\cal A}'$ of ${\cal A}$ such that  
for each $\tau \in {\cal A}'$, we have 
%$F_{pt}^{0}(\tau )=\Omega _{\tau }\neq \CCI $ and 
$J_{\ker }(G_{\tau })\neq \emptyset $ and 
$\Omega _{\tau }$ in Theorem~\ref{t:rcdnkmain2i} cannot be equal to $\CCI .$ 
%For some $\tau \in {\cal A}$, 
%we have $\Omega _{\tau }\neq \CCI $ and 
%$J_{\ker }(G_{\tau })\neq \emptyset .$  
\item[(v)] Let $n\in \NN $ with $n\geq 2$ and let $w=e^{2\pi i/n}\in \CC .$ 
%For each $i=1,\ldots, n$,
Let ${\cal W}_{i}=
\{ w^{i}(z+\lambda (z^{n}-1))\} _{\lambda \in \CC \setminus \{ 0\} }$ for each $i=1,\ldots, n.$  
Let $i_{1},\ldots, i_{m}\in \{ 1,\ldots, n\} $ with $i_{1}<i_{2}\cdots <i_{m}.$ For these 
$i_{1},\ldots, i_{m}$, 
let ${\cal Y}=\cup _{j=1}^{m}\{ w^{i_{j}}(z+\lambda (z^{n}-1))\in {\cal P}\mid 
\lambda \in \CC \setminus \{ 0\} \} .$ 
Then there exists a non-empty  open subset ${\cal A}'$ of ${\cal A}$ such that  
for each $\tau \in {\cal A}'$, 
%we have $F_{pt}^{0}(\tau )=\Omega _{\tau }\neq \CCI $ and 
we have $J_{\ker }(G_{\tau })\neq \emptyset $ 
and $\Omega _{\tau }$ in Theorem~\ref{t:rcdnkmain2i} 
cannot be equal to $\CCI .$ 

\end{itemize}   

\end{ex}

The strategy to prove Theorems~\ref{t:rcdnkmain1i}, \ref{t:rcdnkmain2i}, \ref{t:rcdnkmain1}, 
\ref{t:rcdnkmain2} is as follows.
Let ${\cal Y}$ be a mild subset of $\Ratp$ and suppose that 
${\cal Y}$ is nice with respect to a holomorphic family ${\cal W}=\{ f_{\lambda }\} 
_{\lambda \in \Lambda }$ of rational maps.  
Let $\tau _{0}\in {\frak M}_{1,c}({\cal Y})$. Then there exists an element $\tau $ which is 
arbitrarily close to $\tau _{0}$ and  
 int$(\mbox{supp}\,\tau)\neq \emptyset $. 
 Here, int (supp$\,\tau )$ denotes the set of interior points 
of supp$\,\tau$ with respect to the topology in ${\cal Y} $ 
which is endowed with the relative topology from Rat.
 We show that for such 
$\tau $,  
we have $J_{\ker }(G_{\tau })\subset S({\cal W})$ and hence $\sharp J_{\ker}(G_{\tau })<\infty $, 
by using Montel's theorem (Lemmas~\ref{l:yrnistfgni}, \ref{l:yrtfaji}).  
In Proposition~\ref{p:jkgfmf}, we develop a theory on   {\bf the systems 
with finite kernel Julia sets} based on careful observations on {\bf limit functions 
on the Fatou sets} by using the {\bf hyperbolic metrics on the Fatou components}
(Lemma~\ref{l:pwjkf}), and we obtain that for each $L\in \Min (G_{\tau },\CCI )$ with 
$L\cap F(G_{\tau })\neq \emptyset $, the dynamics in Fatou components which meet $L$ are 
locally contracting and 
$\sharp \Min(G_{\tau },\CCI )<\infty .$ 
Also, we develop a theory on {\bf bifurcation of minimal sets} under perturbation 
which was initiated  by the author of this paper in \cite{Sadv} in Lemma~\ref{l:nalnotsb}, 
and applying it and enlarging the support of $\tau $ a little bit, we obtain that 
any $L\in \Min(G_{\tau },\CCI )$ with $L\cap F(G_{\tau })\neq \emptyset $ is attracting for $\tau .$  
By the theory of finite Markov chains (\cite{Du}), we see that for such $\tau $ and 
for each $L\in \Min(G_{\tau },\CCI )$ with 
$L\subset J_{\ker }(G_{\tau })$, there exists a canonical invariant measure on $L$  
(Lemmas~\ref{l:fininvset}, \ref{l:nlkfa}, Definition~\ref{d:celyap}). 
It is very important and useful to show that 
for any $y\in \CCI $, letting 
$E_{y}:=\{ \gamma =(\gamma _{1},\gamma _{2},\ldots )\in (\mbox{supp}\,\tau)^{\NN} \mid y\in 
\cap _{n=1}^{\infty }\gamma _{n,1}^{-1}(J(G_{\tau }))\} $, 
$$
\mbox{ for } \tilde{\tau }\mbox{ -a.e. } 
\gamma =(\gamma _{1},\gamma _{2},\ldots )\in E_{y}, \mbox{ we have } 
d(\gamma _{n,1}(y), J_{\ker }(G_{\tau }))\rightarrow 0 \mbox{ as } n\rightarrow \infty ,$$  
by using careful observations on random dynamical systems on general compact metric spaces 
(Lemma~\ref{l:voygvv}). 

We next observe the local dynamics of $G_{\tau }$ 
at each point of $S({\cal W}).$ By enlarging the support of $\tau $ a little bit, 
by some careful arguments, it turns out that 
we may assume that 
each $L\in \Min(G_{\tau },\CCI )$ with $L\subset S({\cal W})$ satisfies 
one of the following conditions 
{\bf (I)--(IV)}. {\bf (I)} ``Uniformly expanding''. {\bf (II)} ``Attracting''. {\bf (III)} 
``There exist a point $z_{1}\in L$ and elements $g_{1},g_{2},g_{3}\in G_{\tau }$ such that 
$g_{1}(z_{1})=z_{1}, \| D(g_{1})_{z_{1}}\| _{s}>1$, $g_{2}(z_{1})=z_{1}, \| D(g_{2})_{z_{1}}\| _{s}<1$, 
$g_{3}(z_{1})=z_{1}$, and $g_{3}$ has a Siegel disk with center $z_{1}$''.  {\bf (IV)} 
``There exists a point $z_{1}\in L$ such that for each $\lambda \in \Lambda $, we have 
$ D(f_{\lambda })_{z_{1}}=0.$ Moreover, there exist a point $z_{2}\in L$ and an element $g\in G_{\tau }$ 
such that $g(z_{2})=z_{2}$ and $\| Dg_{z_{2}}\| _{s}>1$''.  By using some results on rational semigroups 
from \cite{HM}, it turns out that if $L$ is of type (III) or (IV), then $L\subset \mbox{int}(J(G_{\tau })).$ 
Here, int$(J(G_{\tau }))$ denotes 
the set of interior points of $J(G_{\tau })$ with respect 
to the topology in $\hat{\Bbb{C}}.$ 
In particular, for each $z\in F(G_{\tau })$, we have $\overline{G(z)}\cap L=\emptyset .$ 
It turns out that for each $z\in F(G_{\tau })$, if $\overline{G(z)}$ does not meet any 
attracting minimal set of $\tau $, then 
$\overline{G(z)}$ meets 
a minimal set $L$ which is uniformly expanding. Thus $\overline{G(z)} $ meets 
a backward image of $L$ under some element of $G_{\tau }$, which is included in a compact 
subset of $J(G_{\tau }) \setminus S({\cal W}).$ By enlarging the support of $\tau $ a little bit again, 
we obtain that for each $z\in F(G_{\tau })$, $\overline{G(z)}$ meets an attracting minimal set 
of $\tau .$  From these arguments, we can show that this $\tau $ is {\bf weakly mean stable}, 
i.e., there exist a positive integer $n$ and two non-empty open subsets $V_{1,\tau }, V_{2,\tau }$ 
of $\CCI $ with $\overline{V_{1,\tau }}\subset V_{2,\tau }$ and 
$\sharp (\CCI \setminus V_{2,\tau})\geq 3$ such that 
(a) for each $(\gamma _{1},\ldots, ,\gamma _{n})\in 
(\mbox{supp}\,\tau)^{n}$, we have 
$\gamma _{n}\cdots \gamma _{1}(V_{2,\tau })\subset V_{1,\tau }$, 
(b) we have $\sharp D_{\tau }<\infty $, where $D_{\tau }:=\cap _{g\in G_{\tau }}g^{-1}(\CCI \setminus 
V_{2,\tau })$, and (c) for each $L\in \Min (G_{\tau },\CCI )$ with $L\subset D_{\tau }$, 
there exist an element $z\in L$ and an element $g_{z}\in G_{\tau }$ such that 
$z$ is a repelling fixed point of $g_{z}.$ From this fact, we can prove the existence 
of an open and dense subset ${\cal A}$ in Theorems~\ref{t:rcdnkmain1i}, \ref{t:rcdnkmain1}. 
If we assume further that ${\cal Y}$ is non-exceptional 
with respect to ${\cal W}$, then we can show that there exists an open and dense subset  ${\cal A}'$ 
of ${\cal A}$ such that for each $\tau \in {\cal A}'$, (1) for each $L\in \Min(G_{\tau },\CCI )$ 
with $L\subset S({\cal W})$, we have $\chi (\tau, L)\neq 0$, and (2) 
for  each $L\in \Min (G_{\tau },\CCI )$ with $L\subset S({\cal W})$, if $\chi (\tau, L)>0$, then 
for each $z\in L$ and for each $g\in \mbox{supp}\,\tau$, we have $Dg_{z}\neq 0.$ 
Combining this fact and the observations on the local behavior of the systems around the minimal 
sets with non-zero Lyapunov exponents (Lemmas~\ref{l:chimune}--\ref{l:chlpa0}), 
we can prove that each element of $\tau \in {\cal A}'$ satisfies statements (I)(II) 
in Theorem~\ref{t:rcdnkmain2i}.  

By the above arguments, we obtain the following.
\begin{cor}[For more generalized result, see Theorem~\ref{t:rcdnkmain1}]
\label{c:wmsi}
Under the assumptions of Theorem~\ref{t:rcdnkmain1i}, 
the set of weakly mean stable elements $\tau \in ({\frak M}_{1,c}({\cal Y}), {\cal O})$ is open and dense 
in $({\frak M}_{1,c}({\cal Y}), {\cal O})$.  
\end{cor}
Note that weak mean stability is a new concept introduced by the author of this paper, 
and it is crucial to consider 
the density of weakly mean stable elements to prove Theorems~\ref{t:rcdnkmain1i}, \ref{t:rcdnkmain2i}, 
\ref{t:rcdnkmain1}, \ref{t:rcdnkmain2}. We emphasize that weak mean stability 
implies many interesting properties (Lemma~\ref{l:wmsopen}, Theorem~\ref{t:wmsneglyap}). 
We remark that in \cite{Splms10}, the notion {\bf mean stability} (i.e., every minimal set is attracting) 
was introduced by the author 
of this paper and it was proved in \cite{Sadv} that the set of mean stable elements  
$\tau \in {\frak M}_{1,c}({\cal P})$ is open and dense in ${\frak M}_{1,c}({\cal P}).$ 
Mean stability implies weak mean stability, 
but {\bf the converse is not true}. 
In fact, there are many examples of mild and nice sets ${\cal Y}$ for which  
there exists a non-empty open subset ${\cal A}''$ of ${\cal A}$ (where ${\cal A}$ is the set 
in Theorems~\ref{t:rcdnkmain1i}, \ref{t:rcdnkmain2i}) such that 
each $\tau \in {\cal A}''$ is 
{\bf not mean stable (but is weakly mean stable)}. 
For such examples, see Theorems~\ref{t:RNM1i}, \ref{t:RNM1}, Example~\ref{ex:applyi} (iii)--(v) 
and Examples~\ref{ex:wlz1z}--\ref{ex:pzljgz}.  
 
Regarding Theorem~\ref{t:rcdnkmain2i}, 
the simplest system which has the properties described in 
Theorem~\ref{t:rcdnkmain2i}(I)(a)(b) and  (II) (negativity of Lyapunov exponents, etc.) is the system given by random perturbation of a 
hyperbolic polynomial map with small uniform additive noise 
(e.g. random iteration of the maps $z^{2}+c_{n}$, where the 
complex numbers 
$c_{n}$ are chosen from a small disc around $0$ uniformly). 
In this case, a random orbit can easily go away from 
the Julia set of the semigroup associated with the system and tend to one of the common attractors. However, we remark that 
in Theorem~\ref{t:rcdnkmain2i}, we do not assume any kind of 
hyperbolicity, and the size of the noise (or randomness) might be 
very big. In the proof of Theorem~\ref{t:rcdnkmain2i}, we need 
many technical arguments, of which ideas are decribed before.  

 We remark that there have been many studies on random 
 dynamical systems of diffeomorphisms (or homeomorphisms) on manifolds. In \cite{Ba} and \cite{Mal}, the setting and the proofs for random dynamical systems of diffeomorphisms or homeomorphisms are completely different from those in this paper,  
 but   it is interesting to see that the results in \cite{Ba} and \cite{Mal} are, formally, in the same spirit.    

In Section~\ref{Pre}, we give some fundamental notations and 
definitions, and present some basic facts on rational semigroups. 
In Section~\ref{Randomcomplex},  we develope the theory of random complex dynamical systems with possibly non-empty kernel Julia sets. In particular we study various kinds of 
Fatou sets and Julia sets for the iteration of $M_{\tau }^{\ast }$ and the function 
$T_{L,\tau }$ of probability of tending to one $L\in \Min (G_{\tau },\CCI ).$ 
Applying them, we prove several theorems including Theorems~\ref{t:rcdnkmain1i}, \ref{t:rcdnkmain2i} 
and the detailed versions Theorems~\ref{t:rcdnkmain1}, \ref{t:rcdnkmain2}  of them. 
In Section~\ref{Random}, 
we apply Theorems~\ref{t:rcdnkmain1i}, \ref{t:rcdnkmain2i}, \ref{t:rcdnkmain1}, \ref{t:rcdnkmain2} and the other results in 
Section~\ref{Randomcomplex} to 
random relaxed Newton's methods in which we find roots of given polynomials, and 
we show Theorem~\ref{t:RNM1i} and the detailed version Theorem~\ref{t:RNM1}. 
In Section~\ref{Examples}, we give some examples 
to which we can apply the main theorems and we classify elements $\tau \in {\cal A}$ 
for some sets ${\cal Y}.$ In section~\ref{s:list}, we give the  
list of notations of this paper. 

\section{Preliminaries}
\label{Pre} 
In this section, we give some fundamental notations and definitions. 

\noindent {\bf Notation}. 
Let $(X,d)$ be a metric space, $A$ a subset of $X$, and $r>0$. We set 
$B(A,r):= $ $  \{ z\in X\mid d(z,A)<r\} .$ Moreover, 
for a subset $C$ of $\CC $, we set 
$D(C,r):= \{ z\in \CC \mid \inf _{a\in C}|z-a|<r\} .$ 
Moreover, for any topological space $Y$ and for any subset $A$ of $Y$, we denote by int$(A)$ the set of all interior points of $A.$ 
We denote by Con$(A)$ the set of all connected components of $A.$  
\begin{df}
\label{d:cmx}
Let $Y$ be a metric space. 
We set $\CMX := \{ f:Y\rightarrow Y\mid f \mbox{ is continuous}\} $ 
endowed with the compact-open topology.  
Also, we set $\OCMX := \{ f\in \CMX \mid 
f \mbox{ is an open map} \} $ endowed 
with the relative topology from $\CMX .$ 
Moreover, 
we denote by 
$C(Y)$ the space of all continuous functions $\varphi :Y\rightarrow \CC$. 
%When 
If $Y$ is compact, we endow $C(Y)$ with the supremum norm $\| \cdot \| _{\infty }.$ 
%Moreover, for a subset ${\cal F}$ of $C(Y)$, we set 
%${\cal F}_{nc}:=\{ \varphi \in {\cal F}\mid \varphi \mbox{ is not constant}\} .$ 

\end{df}
%\begin{df}
%Let $Y$ be a complex manifold. We set 
%$\HMX := \{ f: Y\rightarrow Y \mid f \mbox{ is holomorphic}\} $ endowed 
%with the compact open topology. Moreover, 
%we set $\NHMX := \{ f\in \HMX \mid f \mbox{ is not constant}\} $ endowed with the 
%compact open topology.  
%\end{df}
\begin{rem}
$\CMX $ and  $\OCMX $,  
%$\HMX $, and $\NHMX $ 
are semigroups with the semigroup operation being 
functional composition. If $Y$ is a compact metric space, then $\CMX$ is a complete separable 
metric space. 
\end{rem}
\begin{df} 
\label{d:rational}
A {\bf rational semigroup} is a semigroup  
generated by a family of non-constant rational maps on 
%the Riemann sphere 
$\CCI $ with the semigroup operation being 
functional composition(\cite{HM,GR}). A 
{\bf polynomial semigroup } is a 
semigroup generated by a family of non-constant 
polynomial maps. 
%\begin{df} 
We set 
Rat : $=\{ h:\CCI \rightarrow \CCI \mid 
h \mbox { is a non-constant rational map}\} $
endowed with the distance $\kappa $ which is defined 
by $\kappa (f,g):=\sup _{z\in \CCI }d(f(z),g(z))$, where $d$ denotes the 
spherical distance on $\CCI .$   
%We set 
%Poly :$=\{ h:\CCI \rightarrow \CCI 
%\mid h \mbox{ is a non-constant polynomial }\} $ endowed with 
%the relative topology from Rat.   
Moreover, we set 
$\Ratp:=\{ h\in \mbox{Rat}\mid \deg (h)\geq 2\} $ endowed with the 
relative topology from Rat. 
%Furthermore, 
Also, we set 
%Poly$_{\deg \geq 2}
%:= \{ g\in \mbox{Poly}\mid \deg (g)\geq 2\} $ 
${\cal P}:= \{ g:\CCI \rightarrow \CCI \mid 
g \mbox{ is a polynomial}, \deg (g)\geq 2\} $
endowed with the relative topology from 
Rat.  
%\end{df}

\end{df}
\begin{rem}(\cite[Theorem 2.8.2, Corollary 2.8.3]{Be})  
\label{r:ratm} 
%For each $d\in \NN $, 
Let Rat$_{m}:=\{ g\in \Rat\mid \deg (g)=m\}$ for each 
$m\in \NN $ and 
%let 
%$d\in \NN $ with $d\geq 2$, 
let ${\cal P}_{m}:=\{ g\in {\cal P}\mid \deg (g)=m\} $ 
for each $m\in \NN$ with $m\geq 2. $  
Then for each $m$, $\Rat_{m}$ (resp. ${\cal P}_{m}$) is a connected component of 
$\Rat $ (resp. ${\cal P}$). Moreover, 
$\Rat_{m}$ (resp. ${\cal P}_{m}$) is open and closed in $\Rat $ (resp. ${\cal P}$) and 
is a finite dimensional complex manifold. 
%Furthermore, 
Also, $h_{n}\rightarrow h$ in ${\cal P}$ if and only if $\deg (h_{n})=\deg (h)$ for each large $n$ and 
the coefficients of $h_{n}$ tend to the coefficients of $h$ appropriately as $n\rightarrow \infty .$  
\end{rem}
\begin{df}
\label{d:FGJG}
Let $Y$ be a compact metric space and 
let $G$ be a subsemigroup of $\CMX. $
%Let $G$ be a rational semigroup. 
%\begin{itemize}
%\item 
The {\bf Fatou set of $G$} is defined to be  
$$F(G):=\\ \{ z\in Y  \mid \exists \mbox{ neighborhood } U \mbox{ of }z  
\mbox{ s.t.} \{ g|_{U}:U\rightarrow \CCI \} _{g\in G} \mbox{ is equicontinuous 
on }  U \} .$$ (For the definition of equicontinuity, see \cite{Be}.) 
%\item 
The {\bf Julia set of $G$} is defined to be 
 $J(G):= \CCI  \setminus F(G).$ 
%\item 
If $G$ is generated by $\{ g_{i}\} _{i=1}^{m}$
(i.e., $G=\{ g_{i_{1}}\circ \cdots \circ g_{i_{n}}\mid n\in \NN, 
i_{1},\ldots, i_{n}\in \{ 1,\ldots, m\} \}$) 
, then 
we write $G=\langle g_{1},g_{2},\ldots ,g_{m}\rangle .$
If $G$ is generated by a subset $\G $ of $\CMX$ 
(i.e., $G$ is equal to the set $\{ h_{1}\circ \cdots \circ h_{n}\mid n\in \NN , 
h_{1},\ldots, h_{n}\in \Lambda \}$), then we write
 $G=\langle \G \rangle .$  
%More generally, 
%if $G$ is generated by $\{ h_{\lambda } : \lambda \in \Lambda \} $, 
%then we write $G=\langle h_{\lambda }: \lambda \in \Lambda \rangle .$ 
%\item 
%For finitely many elements $g_{1},\ldots, g_{m}\in \CMX$, 
%we set  $F(g_{1},\ldots ,g_{m}):=F(\langle g_{1},\ldots ,g_{m}\rangle )$ 
%and $J(g_{1},\ldots ,g_{m}):=J(\langle g_{1},\ldots ,g_{m}\rangle )$.  
%\item 
For a subset $A$ of $Y $, we set 
$G(A):= \bigcup _{g\in G}g(A)$ and 
$G^{-1}(A):= \bigcup _{g\in G}g^{-1}(A).$ 
%\item 
We set $G^{\ast }:= G\cup \{ \mbox{Id}\} $, where 
Id denotes the identity map. 
%\end{itemize}
\end{df}
%By using the method in \cite{HM,GR}, it is easy to see that the following 
%lemma holds. 
\begin{lem}[\cite{HM, GR}] 
\label{ocminvlem}
Let $Y$ be a compact metric space and 
%Let $G$ be a rational semigroup. 
let $G$ be a subsemigroup of  $\emOCMX. $ 
%Then 
Then, for each $h\in G$, $h(F(G))\subset F(G)$ and $h^{-1}(J(G))\subset J(G).$ 
Note that the equality does not hold in general. 
\end{lem}
Regarding the dynamics of rational semigroups $G$, we have the following. 
%Let $G$ be a rational semigroup.  Then we have the following.
%\begin{itemize}
%\item[{\em (a)}] 
$F(G)$ is $G$-forward invariant and $J(G)$ is $G$-backward invariant. 
Here, we say that a set $A\subset \CCI $ is $G$-backward invariant, 
if $g^{-1}(A)\subset A$ for each $g\in G$, and we say that 
$A$ is $G$-forward invariant, if $g(A)\subset A$, for each $g\in G.$ 
If $\sharp (J(G))\geq 3$, then 
$J(G)$ is a perfect set and  
$\sharp (E(G))\leq 2$, where $E(G):=\{ z \in \CCI \mid \sharp G^{-1}(z)<\infty \} .$ 
($E(G)$ is called the exceptional set of $G.$)  
Moreover, if $\sharp J(G)\geq 3$ and 
if $z\in \CCI \setminus E(G)$, then 
$J(G)\subset \overline{G^{-1}(z)}. $ 
In particular, if $\sharp J(G)\geq 3$ and $z\in J(G)\setminus E(G)$, then 
$\overline{G^{-1}(z)}=J(G).$  
Also, if $\sharp (J(G))\geq 3$, then 
$J(G)$ is the smallest closed subset of $\CCI $ 
containing at least three points which is $G$-backward invariant. 
Furthermore, if $\sharp (J(G))\geq 3$, then we have  
  $J(G)=\overline{\{ z\in \CCI \mid  z\mbox{ is a repelling fixed point of some }g\in G\} }
=\overline{\cup _{g\in G}J(g)}.$ 
For the proofs of these results, see \cite{HM} and \cite{St1}. 
We remark that \cite{St1} is a very nice introductory article of rational semigroups. 

The following is the key to investigating random complex dynamics. 
\begin{df}
\label{d:kernelJ}
Let $Y$ be a compact metric space and 
let $G$ be a subsemigroup of $\CMX. $  
%Let $G$ be a rational semigroup. 
We set $J_{\ker }(G):= \bigcap _{g\in G}g^{-1}(J(G)).$ 
This is called the {\bf kernel Julia set of $G.$}  
\end{df}
\begin{rem}
\label{r:kjulia}
Let $Y$ be a compact metric space and 
let $G$ be a subsemigroup of $\CMX. $
%\begin{enumerate}
%\item 
%Let $G$ be a rational semigroup. 
(1) $J_{\ker }(G)$ is a compact subset of $J(G).$ 
%\item 
(2) For each $h\in G$, 
$h(J_{\ker }(G))\subset J_{\ker }(G).$  
%\item 
(3) Let $G$ be a rational semigroup and 
suppose $F(G)\neq \emptyset $. Then 
int$(J_{\ker }(G))=\emptyset .$ For, 
suppose $F(G)\neq \emptyset $ and int$(J_{\ker }(G))\neq 
\emptyset .$ 
Let $A=\mbox{int}(J_{\ker }(G)).$ Then for each $g\in G$, 
we have $g(A)\subset A$ since 
$g(J_{\ker }(G))\subset J_{\ker }(G).$ Moreover, 
since $F(G)\neq \emptyset$, we have 
$\sharp (\CCI \setminus A)\geq 3.$ 
By Montel's theorem, it follows that $\emptyset \neq A
\subset F(G).$ However, this is a contradiction since 
$A\subset J_{\ker }(G)\subset J(G).$   
%\item 
(4) If $G\subset \mbox{OCM}(G)$ and $G$ is generated by a single map or if $G$ is a group, then 
$J_{\ker }(G)=J(G).$ However, for a general rational semigroup $G$, 
it may happen that $\emptyset =J_{\ker }(G)\neq J(G)$ 
(see \cite{Splms10}).  
%\end{enumerate}
\end{rem}
In the rest of this paper we sometimes need some results of 
random complex dynamical systems from \cite{Splms10, Sadv}.

\section{Random complex dynamical systems}
\label{Randomcomplex}
In this section, we develope the theory of random complex dynamical systems 
and prove several theorems including Theorems~\ref{t:rcdnkmain1i}, \ref{t:rcdnkmain2i} 
and the detailed versions Theorems~\ref{t:rcdnkmain1}, \ref{t:rcdnkmain2}  of them. 

%\begin{df}
%For a rational semigroup $G$,  
%We set \\ 
%let $P(G):= 
%\overline{\bigcup _{g\in G}\{ \mbox{all critical values of } 
%g:\CCI \rightarrow \CCI \}}$ where the closure is taken in $\CCI .$  
%This is called the {\bf postcritical set} of $G.$ 
%\end{df} 
%\begin{rem}
%\label{r:pg}
%If $\Gamma \subset \mbox{Rat}$ and $G=\langle \Gamma \rangle $, then 
%$P(G)=\overline{G^{\ast }(\bigcup _{h\in \G }\{ \mbox{all critical values of } h\} )}.$ 
%From this one may know the figure of $P(G)$, in the finitely generated case, using a 
%computer.   
%\end{rem}
%\begin{df}
%\label{d:sh}
%Let $G$ be a rational semigroup. We say that 
%$G$ is {\bf hyperbolic} if $P(G)\subset F(G).$ 
%We say that $G$ is {\bf semi-hyperbolic} if $UH(G)\subset F(G).$ 
%\end{df}
\subsection{Random dynamical systems on general compact metric spaces}
\label{ss:rdsg}
In this subsection we show some results on random dynamical systems on general compact metric 
spaces. 
It is sometimes important to investigate the dynamics of sequences of maps. 
\begin{df}
\label{d:gammamn}
Let $Y$ be a compact metric space. 
For each $\gamma =(\gamma _{1},\gamma _{2},\ldots )\in 
(\CMX)^{\NN }$ and each $m,n\in \NN $ with $m\geq n$, we set
$\gamma _{m,n}=\gamma _{m}\circ \cdots \circ \gamma _{n}$ and we set  
$\gamma _{0,1}=\mbox{Id}$. 
%where id denotes the identity map. 
Also, we set
\vspace{-2mm} 
$$F_{\gamma ,0}:=
\{ z\in Y\mid \{ \gamma _{n,1}\} _{n\in \NN  } 
\mbox{ is  equicontinuous at the one point }z\},$$
\vspace{-7mm}
$$F_{\gamma }:= \{ z\in Y \mid 
\exists \mbox{ neighborhood } U \mbox{ of } z \mbox{ s.t. } 
\{ \gamma _{n,1}\} _{n\in \NN } \mbox{ is equicontinuous on }  
U\}, $$ $J_{\gamma ,0}:=Y\setminus F_{\gamma ,0}$ 
and $J_{\gamma }:= Y  \setminus F_{\gamma }.$ 
The set $F_{\gamma }$ is called the {\bf Fatou set of the sequence $\gamma $} and 
the set $J_{\gamma }$ is called the {\bf Julia set of the sequence $\gamma .$} 
Moreover, we set 
$F^{\gamma, 0}:= 
\{ \gamma \} \times F_{\gamma ,0}
(\subset (\CMX )^{\NN }\times Y )$, 
$F^{\gamma }:= \{ \gamma \} \times F_{\gamma }\ (\subset (\CMX )^{\NN }\times Y )$, 
$J^{\gamma ,0}:= \{ \gamma \} 
\times J_{\gamma ,0}
(\subset (\CMX )^{\NN }\times Y )$   
and $J^{\gamma }:= \{ \gamma \} \times J_{\gamma }\ (\subset (\CMX )^{\NN }\times Y ).$ 
\end{df}
\begin{rem}
Let $\gamma \in (\Rat) ^{\NN }$. Then 
by Montel's theorem, $J_{\gamma ,0}=J_{\gamma }.$ 
Also, if $\gamma \in (\Ratp)^{\NN }$, then by \cite[Theorem 2.8.2]{Be}, $J_{\gamma }\neq \emptyset .$ 
Moreover, if $\Gamma $ is a non-empty compact subset of $\Ratp$ and $\gamma \in \Gamma ^{\NN }$, 
then by \cite{S4}, $J_{\gamma }$ is a perfect set and $J_{\gamma }$ has uncountably many points.  
\end{rem}
\begin{lem}
\label{l:ufgmeas}
Let $Y$ be a compact metric space. 
Let $\Gamma $ be a non-empty closed subset of an open subset of $\emCMX .$ 
Then $\bigcup _{ \gamma \in \GN }F^{\gamma ,0}$, 
$\bigcup  _{\gamma \in \GN }F^{\gamma }$, 
$\bigcup _{\gamma \in \GN }J^{\gamma ,0}$ 
 and 
$\bigcup _{\gamma \in \GN }J^{\gamma }$ are Borel measurable subsets of 
$\GN \times Y $ and 
\vspace{-2mm} 
\begin{equation}
\label{eq:cupfg0}
\bigcup_{\gamma \in \GN}F^{\gamma ,0}
=\{ (\gamma, y)\in \GN \times Y \mid \lim _{m\rightarrow \infty }
\sup _{n\geq 1} \mbox{{\em diam}} \gamma _{n,1}(B(y,\frac{1}{m}))=0\},   
\end{equation}
\vspace{-3mm} 
\begin{equation}
\label{eq:cupfg}
\bigcup _{\gamma \in \GN }F^{\gamma }
=\bigcup _{p\in \NN }\{ (\gamma, y)\in \GN \times Y \mid \lim _{m\rightarrow \infty }
\sup _{n\geq 1} \sup _{y'\in B(y,\frac{1}{p})}\mbox{{\em diam}} \gamma _{n,1}(B(y',\frac{1}{m}))=0\} .
\end{equation}
\end{lem}
\begin{proof}
By the definition of $F^{\gamma }$, we obtain 
(\ref{eq:cupfg0}) and (\ref{eq:cupfg}). 
From (\ref{eq:cupfg0}) and (\ref{eq:cupfg}), 
it follows that 
$\bigcup _{\gamma \in \GN }F^{\gamma ,0}$ and 
$\bigcup _{\gamma \in \GN }F^{\gamma }$ are Borel subsets of 
$\GN \times Y .$ 
Thus $\bigcup _{\gamma \in \GN }J^{\gamma ,0}$ and 
$\bigcup _{\gamma \in \GN }J^{\gamma }$ are also   Borel subsets of 
$\GN \times Y .$ 
\end{proof}

We now give some notations on random dynamics. 
\begin{df}
\label{d:d0}
For a metric space $Y$, we denote by 
${\frak M}_{1}(Y)$ the space of all Borel probability measures on  $Y$ endowed 
with the topology such that 
$\mu _{n}\rightarrow \mu $ in ${\frak M}_{1}(Y)$ if and only if 
for each bounded continuous function $\varphi :Y\rightarrow \CC $, 
$\int \varphi \ d\mu _{n}\rightarrow \int \varphi \ d\mu .$ 
 Note that if $Y$ is a compact metric space, then 
${\frak M}_{1}(Y)$ is a compact metric space with the metric 
$d_{0}(\mu _{1},\mu _{2}):=\sum _{j=1}^{\infty }\frac{1}{2^{j}}
\frac{|\int \phi _{j}d\mu _{1}-\int \phi _{j}d\mu _{2}|}{1+|\int \phi _{j}d\mu _{1}-\int \phi _{j}d\mu _{2}|}$, 
where $\{ \phi _{j}\} _{j\in \NN }$ is a dense subset of $C(Y).$  
Furthermore, for each $\tau \in {\frak M}_{1}(Y)$, 
the topological support  $\suppt $ of $\tau $ is defined as $\suppt:=\{ z\in Y\mid \forall \mbox{ neighborhood } U \mbox{ of }z,\ 
\tau (U)>0\} .$ Note that $\suppt $ is a closed subset of $Y.$ 
Furthermore, 
we set ${\frak M}_{1,c}(Y):= \{ \tau \in {\frak M}_{1}(Y)\mid \suppt 
\mbox{ is a compact subset of }Y\} .$ 

%\begin{df}
For a complex Banach space ${\cal B}$, we denote by ${\cal B}^{\ast }$ the 
space of all continuous complex linear functionals $\rho :{\cal B}\rightarrow \CC $, 
endowed with the weak$^{\ast }$ topology. 
%the norm $\| \rho \| := \sup _{\varphi \in {\cal B}, \| \varphi \| =1} |\rho (\varphi )|.$  
%\end{df}
\end{df}
For any $\tau \in {\frak M}_{1}(\CMX)$, we will consider the i.i.d. random dynamics on $Y $ such that 
at every step we choose a map $g\in \CMX $ according to $\tau $ 
(thus this determines a time-discrete Markov process with time-homogeneous transition probabilities 
on the state space 
$Y $ such that for each $x\in Y $ and 
each Borel measurable subset $A$ of $Y $, 
the transition probability 
$p(x,A)$ from $x$ to $A$ is defined as $p(x,A)=\tau (\{ g\in \CMX \mid g(x)\in A\} )$).  
\begin{df} 
\label{d:ytau}
Let $Y$ be a compact metric space.  
%and let ${\cal F}$ be a subset of $\CMX .$ 
Let $\tau \in {\frak M}_{1}(\CMX).$  
\begin{enumerate}
\item  We denote by $\mbox{supp}\, \tau $ the topological support of $\tau $ (thus $\mbox{supp}\,\tau $ is a 
closed subset of $\CMX $).     
Moreover, we set $X_{\tau }:= (\mbox{supp}\, \tau )^{\NN }$ $
 (=\{ \gamma  =(\gamma  _{1},\gamma  _{2},\ldots )\mid \gamma  _{j}\in \mbox{supp}\,\tau \ (\forall j)\} )$ endowed with the product topology.  
Furthermore, we set $\tilde{\tau }:= \otimes\displaystyle _{j=1}^{\infty }\tau .$ 
This is the unique Borel probability measure 
on $X_{\tau }$ such that for each cylinder set 
$A=A_{1}\times \cdots \times A_{n}\times \mbox{supp}\,\tau\times  \mbox{supp}\,\tau\times \cdots $ in 
$X_{\tau }$, $\tilde{\tau }(A)=\prod _{j=1}^{n}\tau (A_{j}).$ 
 We denote by $G_{\tau }$ the subsemigroup of $\CMX$ generated by 
the subset $\mbox{supp}\,\tau$ of $\CMX .$   
\item 
%We denote by ${\cal M}_{1}(\CCI )$ the space of 
%all Borel probability measures on $\CCI $, endowed with 
%the weak topology. Note that ${\cal M}_{1}(\CCI )$ is 
%a compact metric space. 
%Let $C(Y )$ be the Banach space of all continuous functions on $Y $ endowed with the 
%supremum norm.  
Let $M_{\tau }$ be the operator 
on $C(Y ) $ defined by $M_{\tau }(\varphi )(z):=\int _{\mbox{supp}\,\tau}\varphi (g(z))\ d\tau (g).$ 
$M_{\tau }$ is called the {\bf transition operator} of the Markov process induced by $\tau .$ 
Moreover, let $M_{\tau }^{\ast }: C(Y )^{\ast }
\rightarrow C(Y)^{\ast }$ be the dual of $M_{\tau }$, which is defined as 
$M_{\tau }^{\ast }(\mu )(\varphi )=\mu (M_{\tau }(\varphi ))$ for each 
$\mu \in C(Y)^{\ast }$ and each $\varphi \in C(Y).$ 
Remark: we have $M_{\tau }^{\ast }({\frak M}_{1}(Y))\subset {\frak M}_{1}(Y)$ and  
for each $\mu \in {\frak M}_{1}(Y)$ and each open subset $V$ of $Y$, 
we have $M_{\tau }^{\ast }(\mu )(V)=\int _{\mbox{supp}\,\tau}\mu (g^{-1}(V))\ d\tau (g).$ 
\item  
We denote by $F_{meas }(\tau )$ the 
set of $\mu \in {\frak M}_{1}(Y)$
satisfying that there exists a neighborhood $B$ 
of $\mu $ in ${\frak M}_{1}(Y)$ such that 
the sequence  $\{ (M_{\tau }^{\ast })^{n}|_{B}: 
B\rightarrow {\frak M}_{1}(Y) \} _{n\in \NN }$
is equicontinuous on $B.$
%\item 
We set $J_{meas}(\tau ):= {\frak M}_{1}(Y)\setminus 
F_{meas}(\tau ).$
\item We denote by $F_{meas}^{0}(\tau )$ the 
set of $\mu \in {\frak M}_{1}(Y)$ satisfying that 
$\{ (M_{\tau }^{\ast })^{n}:
{\frak M}_{1}(Y)\rightarrow {\frak M}_{1}(Y)\} _{n\in \NN }$ is 
equicontinuous at the one point $\mu .$ 
Note that 
$F_{meas}(\tau )\subset F_{meas}^{0}(\tau ).$ 
\item 
We set $J_{meas}^{0}(\tau ):= {\frak M}_{1}(\CCI  )\setminus F_{meas}^{0}(\tau ).$ 
\end{enumerate}
\end{df}
\begin{rem}
We have $F_{meas}(\tau )\subset F_{meas}^{0}(\tau )$ and $J_{meas}^{0}(\tau )\subset J_{meas}(\tau ).$ 
\end{rem}
\begin{rem}
\label{r:ggt}
Let $\Gamma $ be a closed subset of an open subset ${\cal U}$ of Rat. Then there exists a 
$\tau \in {\frak M}_{1}({\cal U})$ such that 
supp$\,\tau $ (in the sense of Definition~\ref{d:d0}) is equal to $\Gamma .$ 
%$\mbox{supp}\,\tau=\Gamma .$ 
By using this fact, we sometimes apply the results on random complex dynamics  
to the study of the dynamics of rational semigroups. 
\end{rem}
\begin{df}
\label{d:Phi}
Let $Y$ be a compact metric space. 
Let $\Phi :Y \rightarrow {\frak M}_{1}(Y )$ be the topological embedding 
defined by: $\Phi (z):=\delta _{z}$, where $\delta _{z}$ denotes the 
Dirac measure at $z.$ Using this topological embedding $\Phi :Y \rightarrow {\frak M}_{1}(Y )$, 
we regard $Y $ as a compact subset of ${\frak M}_{1}(Y ).$ 
\end{df}
\begin{rem}
\label{r:Phi}
If $h\in \Rat $ and $\tau =\delta _{h}$, then 
we have $M_{\tau }^{\ast }\circ \Phi = \Phi \circ h$ on $\CCI .$ 
Moreover, for a general $\tau \in {\frak M}_{1}(\Rat)$, 
$M_{\tau }^{\ast }(\mu )=\int h_{\ast }(\mu )d\tau (h)$ for each 
$\mu \in {\frak M}_{1}(\CCI ).$ 
Therefore, for a general $\tau \in {\frak M}_{1}(\Rat)$, 
the map $M_{\tau }^{\ast }:{\frak M}_{1}(\CCI )\rightarrow {\frak M}_{1}(\CCI )$ 
can be regarded as the ``averaged map'' on the extension ${\frak M}_{1}(\CCI )$ of 
$\CCI .$ 
\end{rem}

\begin{df}
\label{d:manyFJ}
Let $Y$ be a compact metric space. 
% and let ${\cal F}$ be a 
%subset of $\CMX .$ 
Let $\tau \in {\frak M}_{1}(\CMX ).$ 
Regarding $Y $ as a compact subset of ${\frak M}_{1}(Y)$ as in Definition~\ref{d:Phi}, we use 
the following notation.   
\begin{enumerate}
\item 
We denote by $F_{pt}(\tau )$ the set of 
$z\in Y $ satisfying that 
there exists a neighborhood $B$ of $z $ in $Y $ such that 
the sequence$\{ (M_{\tau }^{\ast })^{n}|_{B}:B\rightarrow 
{\frak M}_{1}(Y )\} _{n\in \NN }$ is equicontinuous on $B.$ 
%\item 
We set $J_{pt}(\tau ):= Y \setminus F_{pt}(\tau ).$ 
\item Similarly, we denote by 
$F_{pt }^{0}(\tau )$ the set of 
$z\in Y $ such that 
the sequence $\{ (M_{\tau }^{\ast })^{n}|_{Y }: Y 
\rightarrow {\frak M}_{1}(Y)\} _{n\in \NN }$ is 
equicontinuous at the one point $z\in Y.$ 
%Note that $F_{pt }(\tau )\subset F_{pt}^{0}(\tau ).$  
%\item 
We set $J_{pt}^{0}(\tau ):= Y \setminus F_{pt}^{0}(\tau ).$ 
\end{enumerate}
Also, the set $J_{\ker }(G_{\tau })$ is called the 
{\bf kernel Julia set of $\tau .$} 
\end{df}
\begin{rem}
We have $F_{pt}(\tau )\subset F_{pt}^{0}(\tau )$ and  
%$F_{meas}(\tau )\subset F_{meas}^{0}(\tau )$, 
$J_{pt}^{0}(\tau )\subset J_{pt}(\tau )\cap J_{meas}^{0}(\tau )$ (regarding $Y$ as a compact subset of 
${\frak M}_{1}(Y)$ by using the topological embedding $\Phi :Y\rightarrow {\frak M}_{1}(Y)$).
%, and 
%$J_{meas}^{0}(\tau )\subset J_{meas}(\tau ).$ 
\end{rem}
\begin{rem}
If $\tau =\delta _{h}\in {\frak M}_{1}(\Ratp)$ with $h\in  \Ratp $, then 
$J_{pt}^{0}(\tau )$ and $J_{meas}(\tau )$ are uncountable. In fact, 
%using the embedding $\Phi :\CCI \rightarrow {\frak M}_{1}(\CCI )$, 
%we have $\emptyset \neq \Phi (J(h))\subset J_{meas}(\tau ).$     
we have $\emptyset \neq J(h)\subset J_{pt}^{0}(\tau )$ and $J(h)$ is uncountable. 
\end{rem}

\begin{lem}
\label{l:ttaugy0y}
Let $Y$ be a compact metric space. 
Let $\tau \in {\frak M}_{1}(\emCMX). $
Let $y\in Y.$ Suppose 
$\tilde{\tau }(\{ \gamma \in (\emCMX )^{\NN }\mid y\in J_{\gamma ,0}\})=0.$ 
Then $y\in F_{pt}^{0}(\tau ).$  
\end{lem}
\begin{proof}
By (\ref{eq:cupfg}) in Lemma~\ref{l:ufgmeas} and the assumption of our lemma,  
we obtain that for $\tilde{\tau }$-a.e.$\gamma \in (\CMX)^{\NN }$, 
$\lim_{m\rightarrow \infty }\sup _{n\geq 1}\mbox{diam}(\gamma _{n,1}(B(y,\frac{1}{m})))=0.$
Let $\epsilon >0.$ By Egoroff's theorem, 
there exists a Borel subset $A_{1}$ of $X_{\tau }$ with 
$\tilde{\tau }(X_{\tau }\setminus A_{1})<\epsilon $ 
such that 
\begin{equation}
\label{eq:ngeq1d}
\sup _{n\geq 1}\mbox{diam}(\gamma _{n,1}(B(y,\frac{1}{m})))\rightarrow 0
\end{equation} 
as $m\rightarrow \infty $ uniformly on $A_{1}.$ 
Let $\varphi \in C(Y).$ 
Then there exists a $\delta _{1}>0$ such that 
if $d(z_{1},z_{2})<\delta _{1}$ then $|\varphi (z_{1})-\varphi (z_{2})|<\epsilon .$ 
By (\ref{eq:ngeq1d}), 
there exists a $\delta _{2}>0$ such that 
for each $z\in Y$ with $d(z,y)<\delta _{2}$, for each $\gamma \in A_{1}$, 
and for each $n\in \NN $, we have 
$d(\gamma _{n,1}(z),\gamma _{n,1}(y))<\delta _{1}.$  
Therefore for each $z\in Y$ with $d(z,y)<\delta _{2}$, 
we have 
\begin{eqnarray*}
|M_{\tau }^{n}(\varphi )(z)-M_{\tau }^{n}(\varphi )(y)| 
& \leq  & \int _{A_{1}}|\varphi (\gamma _{n,1}(z)-\varphi (\gamma _{n,1}(y))|d\tilde{\tau }(\gamma )
+\int _{X_{\tau }\setminus A_{1}}|\varphi (\gamma _{n,1}(z))-\varphi (\gamma _{n,1}(y))|
d\tilde{\tau }(\gamma )\\ 
& \leq & \tilde{\tau }(A_{1})\cdot \epsilon +2\epsilon \cdot \sup _{a\in \CCI }|\varphi (a)|\\ 
& \leq & \epsilon (1+2\| \varphi \| _{\infty }). 
\end{eqnarray*}
It follows that $y\in F_{pt }^{0}(\tau ).$ 
Thus we have proved our lemma. 
\end{proof}

For a smooth Riemannian real manifold $Y$ with $\dim Y=p$, we denote by 
Leb$_{p}$ the ($p$-dimensional) Lebesgue measure on $Y.$ 

\begin{cor}
\label{c:ttauaegl0}
Let $Y$ be a compact smooth manifold with $\dim (Y)=p$ and let $\tau \in {\frak M}_{1}(\emCMX ).$ 
Suppose that for $\tilde{\tau }$-a.e.$\gamma \in (\emCMX)^{\NN }$, 
Leb$_{p}(J_{\gamma ,0})=0.$ Then Leb$_{p}(J_{pt}^{0}(\tau ))=0.$ 

\end{cor}
\begin{proof}
Under the assumptions of our corollary, Lemma~\ref{l:ufgmeas} and 
Fubini's theorem imply that for Leb$_{p}$-a.e.$y\in Y$, 
we have $\tilde{\tau }(\{ \gamma \in (\CMX)^{\NN }\mid y\in J_{\gamma ,0}\})=0.$ 
By Lemma~\ref{l:ttaugy0y}, it follows that 
for Leb$_{p}$-a.e.$y\in Y$, $y\in F_{pt}^{0}(\tau ).$ 
Thus we have proved our corollary.
\end{proof}

The following lemma is very important and useful to prove many results. 

\begin{lem}
\label{l:voygvv}
Let $Y$ be a compact metric space. 
Let $\tau \in {\frak M}_{1}(\emCMX).$ 
Let $V$ be a non-empty open subset of $Y$. 
Suppose that for each $g\in \mbox{supp}\,\tau$, $g(V)\subset V.$ 
Let $L_{\ker }:=\cap _{g\in G_{\tau }}g^{-1}(Y\setminus V)$. 
Let $y\in Y$ and let 
$E:= \{ \gamma \in X_{\tau }\mid y\in \cap _{j=1}^{\infty }\gamma _{j,1}^{-1}(Y\setminus V)\} $ (Remark: $E$ depends 
on $V$).  
Then for $\tilde{\tau }$-a.e.$\gamma \in E$, we have 
$d(\gamma _{n,1}(y), L_{\ker })\rightarrow 0$ as $n\rightarrow \infty .$ 
\end{lem}
\begin{proof}
For each $\delta >0$ and $n\in \NN $, let 
$A(\delta ,n):=\{ \gamma \in E\mid \gamma _{n,1}(y)\in 
(Y\setminus V)\setminus B(L_{\ker },\delta )\}$
 and 
 $C(\delta ):=\{ \gamma \in E\mid \exists n_{0}\in \NN \mbox{ s.t. } 
 \forall n\geq n_{0}, \gamma _{n,1}(y)\in B(L_{\ker },\delta )\} .$ 
 In order to prove our lemma, 
 it suffices to show that 
\begin{equation}
\label{eq:tauesc}
\tilde{\tau }(E\setminus C(\delta ))=0 \mbox{ for each }\delta >0.
\end{equation} 
Since $E\setminus C(\delta )=\cap _{N=1}^{\infty }\cup _{n=N}^{\infty }A(\delta ,n)$, 
we have 
$$\tilde{\tau }(E\setminus C(\delta ))=\lim _{N\rightarrow \infty }
\tilde{\tau }(\cup _{n=N}^{\infty }A(\delta ,n))
\leq \lim _{N\rightarrow \infty }\sum _{n=N}^{\infty }\tilde{\tau }(A(\delta ,n)).$$ 
Thus, in order to show (\ref{eq:tauesc}), it suffices to prove that 
\begin{equation}
\label{eq:sumttadf}
\sum _{n=1}^{\infty }\tilde{\tau }(A(\delta ,n))<\infty \mbox{ for each }\delta >0.
\end{equation}
In order to prove (\ref{eq:sumttadf}), let $\delta >0$. Then  
for each  $z\in (Y\setminus V)\setminus B(L_{\ker },\delta )$, 
there exists an element $g_{z}\in G$ and a neighborhood $U_{z}$ of $z$ in $Y$ 
such that $g_{z}(\overline{U_{z}})\subset V.$ 
Since $H:=(Y\setminus V)\setminus B(L_{\ker },\delta )$ is compact, 
there exist  finitely many points $z_{1},\ldots, z_{r}\in Y$ such that 
$H\subset \cup _{j=1}^{r}U_{z_{j}}.$  
Since $g(V)\subset V$ for each $g\in \mbox{supp}\,\tau$, we may assume that 
there exists an $l\in \NN $ such that 
for each $j=1,\ldots, r$, there exists an element $\gamma ^{j}=(\gamma _{1}^{j},\ldots, 
\gamma _{l}^{j})\in (\mbox{supp}\,\tau)^{l}$ with 
$g_{z_{j}}=\gamma _{l}^{j}\circ \cdots \circ \gamma _{1}^{j}.$ 
Then for each $j=1,\ldots, r$, there exists a neighborhood $W_{j}$ of $\gamma ^{j}$ in 
$(\mbox{supp}\,\tau)^{l}$ such that 
for each $\alpha =(\alpha _{1},\ldots ,\alpha _{l})\in W_{j}$, 
$\alpha _{l}\circ \cdots \circ \alpha _{1}(\overline{U_{z_{j}}})\subset V.$  
Let $\delta _{0}:=\min _{j=1}^{r}\tau ^{l}(W_{j})>0$, 
where $\tau ^{l}=\otimes _{n=1}^{l}\tau \in {\frak M}_{1}((\CMX )^{l}).$ 
For each $i=0,1,\ldots ,l-1$ and for each 
$n\in \NN $, let 
$$H(\delta , i, n):=\{ \gamma \in (\mbox{supp}\,\tau)^{\NN }\mid 
\gamma _{i+nl,1}(y)\in (Y\setminus V)\setminus B(L_{\ker },\delta ), 
\gamma _{i+(n+1)l,1}(y)\in V\}$$ 
and 
$$I(\delta ,i,n):=\{ \gamma \in (\mbox{supp}\,\tau)^{\NN}\mid 
\gamma _{i+nl,1}(y)\in  (Y\setminus V)\setminus B(L_{\ker },\delta )\}.$$ 
Note that if $n\neq m$ then $H(\delta ,i,n)\cap H(\delta , i,m)=\emptyset .$ 
Let $Q_{1},\ldots ,Q_{s}$ be mutually disjoint Borel subsets of 
$(Y\setminus V)\setminus B(L_{\ker },\delta )$ such that 
$(Y\setminus V)\setminus B(L_{\ker },\delta)=\cup _{p=1}^{s}Q_{p}$ 
and such that for each $p=1,\ldots ,s$ there exists a $j(p)\in \{ 1,\ldots ,r\}$ 
with $Q_{p}\subset U_{z_{j(p)}}.$ 
Then for each $i=0,1\ldots ,l-1$, we have 
\begin{eqnarray*}
\tilde{\tau }(H(\delta ,i,n)) 
 & = & \tau ^{i+(n+1)l}(\coprod _{p=1}^{s}\{ \gamma \in (\mbox{supp}\,\tau)^{i+(n+1)l}
\mid \gamma _{i+nl,1}(y)\in Q_{p}, \gamma _{i+(n+1)l,1}(y)\in V\} )\\ 
& = & \sum _{p=1}^{s}\tau ^{i+(n+1)l}(\{ \gamma \in 
(\mbox{supp}\,\tau)^{i+(n+1)l}
\mid \gamma _{i+nl,1}(y)\in Q_{p}, \gamma _{i+(n+1)l,1}(y)\in V\} )\\ 
& \geq & \sum _{p=1}^{s}\tau ^{i+(n+1)l}(\{ \gamma \in (\mbox{supp}\,\tau)^{i+(n+1)l}
\mid \gamma _{i+nl,1}(y)\in Q_{p}, (\gamma _{i+nl+1}, \ldots, \gamma _{i+(n+1)l})\in W_{j(p)}\} )\\
& = & \sum _{p=1}^{s}\tau ^{i+nl}(\{ \gamma \in 
(\mbox{supp}\,\tau)^{i+nl}\mid \gamma _{i+nl,1}(y)\in Q_{p}\})\cdot \tau ^{l}(W_{j(p)})\\ 
& \geq & \delta _{0}\tilde{\tau }(I(\delta ,i,n)), 
\end{eqnarray*}
where $\coprod$ denotes the disjoint union. 
Therefore
\begin{eqnarray*}
1 & \geq & \tilde{\tau }(\bigcup _{n\in \NN }\{ \gamma \in (\mbox{supp}\,\tau)^{\NN }\mid 
\gamma _{n,1}(y)\in V\} )\\ 
& \geq & \tilde{\tau }(\bigcup _{n=1}^{\infty }H(\delta ,i,n))  
 =  \sum _{n=1}^{\infty }\tilde{\tau }(H(\delta ,i,n))
 \geq  \sum _{n=1}^{\infty }\delta _{0}\tilde{\tau }(I(\delta ,i,n)).  
\end{eqnarray*}
Thus $\sum _{n=1}^{\infty }\tilde{\tau }(I(\delta ,i,n))<\infty $ 
for each $i=0,1,\ldots, l-1.$ 
Hence 
$$\sum _{n=1}^{\infty }\tilde{\tau }(A(\delta ,n))
=\sum _{i=0}^{l-1}\sum _{n=1}^{\infty }\tilde{\tau }(I(\delta ,i,n))<\infty .$$ 
Therefore (\ref{eq:sumttadf}) holds. Thus we have proved our lemma. 
\end{proof}
%\subsection{Systems with hyperbolic kernel Julia sets} 
%\label{ss:syshyp}
%In this subsection, we show a result on random complex dynamical systems with hyperbolic 
%kernel Julia sets. 

\subsection{Minimal sets with finite cardinality and related lemmas}
\label{ss:minimal}
In this subsection, we show some lemmas regarding random dynamical systems having 
 minimal sets with finite cardinality. 
\begin{df}
\label{d:CptY}
For a topological space $Y$, we denote by Cpt$(Y)$ the space of all non-empty compact subsets of $Y$. 
If $Y$ is a metric space, we endow Cpt$(Y)$  
with the Hausdorff metric.
\end{df}
\begin{df}
\label{d:minimal}
Let $X$ be a metric space and let $G$ be a subsemigroup of $\CMX .$ 
%Let $G$ be a rational semigroup. 
Let $Y\in \Cpt (X)$ be such that 
$G(Y)\subset Y.$  
Let $K\in \Cpt(Y).$ 
We say that $K$ is a {\bf minimal set of $(G,Y)$} if 
%$G(K)\subset K$ and 
$K$ is minimal among the space 
$\{ L\in \Cpt(Y)\mid G(L)\subset L\} $ with respect to inclusion. 
Moreover, we denote by $\Min(G,Y)$ 
the set of all minimal sets for $(G,Y).$
%$:= \{ K\in \Cpt(Y)\mid K \mbox{ is a minimal set for } (G,Y)\} .$  
\end{df}
\begin{rem}
\label{r:minimal}
%Let $Y$ be a metric space and let $G$ be a subsemigroup of $\CMX .$
Let $G$ be a rational semigroup. 
By Zorn's lemma, it is easy to see that 
if $K_{1}\in \Cpt(\CCI )$ and $G(K_{1})\subset K_{1}$, then 
there exists a $K\in \Min(G,\CCI )$ with $K\subset K_{1}.$  
Moreover, it is easy to see that 
for each $K\in \Min(G,\CCI )$ and each $z\in K$, 
$\overline{G(z)}=K.$ In particular, if $K_{1},K_{2}\in \Min(G,\CCI )$ with $K_{1}\neq K_{2}$, then 
$K_{1}\cap K_{2}=\emptyset .$ Moreover, 
by the formula $\overline{G(z)}=K$, we obtain that 
for each $K\in \Min(G,\CCI )$, either (1) $\sharp K<\infty $ or (2) $K$ is perfect and 
$\sharp K>\aleph _{0}.$ 
Furthermore, it is easy to see that 
if $\Gamma \in \Cpt(\Rat), G=\langle \Gamma \rangle $, and 
$K\in \Min(G,\CCI )$, then $K=\bigcup _{h\in \Gamma }h(K).$   
\end{rem}
\begin{rem}
\label{r:minsetcorr}
In \cite[Remark 3.9]{Splms10}, 
for the statement ``for each $K\in \Min (G,Y)$, either 
(1) $\sharp K<\infty $ or (2) $K$ is perfect'', 
we should assume that each element $g\in G$ is a finite-to-one map. See \cite[Remark 2.24]{Sadv}. 
%(This is a correction to \cite[Remark 3.9]{Splms10}.) 
\end{rem}
We now show some lemmas on the minimal sets whose cardinalities are finite 
(Lemmas~\ref{l:fininvset}, \ref{l:nlkfa}).  
\begin{df} 
\label{d:chainperiod}
Let $S$ be a finite space and let ${\cal S}=2^{S}.$ 
Let $\{ X_{n}\} $ be a Markov chain on the state space $S$ 
with transition probability $p.$ Suppose that 
$p$ is irreducible. Then 
by \cite[Lemma 6.6.2]{Du}, there exists a positive integer $d$ 
such that for each $x\in S$, the number $d$ is equal 
to the greatest common divisor of $\{ n\in \NN \mid 
p^{n}(x,x)>0\}.$  This $d$ is called the period of this Markov chain. 
\end{df}

\begin{df}
\label{d:Gtaur}
Let $Y$ be a compact metric space. 
Let $\tau \in {\frak M}_{1,c}(\CMX).$ 
For each $r\in \NN $, we set 
$G_{\tau }^{r}:=
\langle \{ g_{1}\circ \cdots \circ g_{r}\mid  g_{1},\ldots, g_{r}\in \mbox{supp}\,\tau\} \rangle.$ 
\end{df}
The following lemma is an easy consequence of 
\cite[Theorem 6.6.4, Lemma 6.7.1]{Du} and some fundamental  
arguments. 
\begin{lem}
\label{l:fininvset} 
Let $Y$ be a compact metric space. 
Let $\tau \in {\frak M}_{1,c}(\emCMX).$ 
Let $K$ be a nonempty finite subset of $Y .$ 
Suppose that $G_{\tau }(K)\subset K.$ 
Let $\{ K_{i}\mid i=1,\ldots, q\}=\emMin (G_{\tau }, K)$ where 
$q=\sharp \emMin (G_{\tau },K).$ For each $i=1,\ldots q$, 
let $p_{i}\in \NN $ be the period of the finite Markov chain with state space $K_{i}$ 
induced by $\tau $ (i.e. the finite Markov chain with state space $K_{i}$ whose 
transition probability $p(x,A)$ from $x\in K_{i}$ to $A\subset K_{i}$ satisfies  
$p(x,A)=\tau (\{ g\in \mbox{supp}\,\tau\mid g(x)\in A\} )$). 
%(For the definition of ``period'', see \cite[p308]{Du}). 
Let $m=\prod _{i=1}^{q}p_{i}\in \NN .$    
%Let $G_{\tau }^{m}:=\langle \{ g_{1}\circ \cdots \circ g_{m}\mid \forall g_{j}\in \mbox{supp}\,\tau\} \rangle .$ 
Let $\{ H_{j}\mid j=1,\ldots, r\}=\emMin(G_{\tau }^{m},K)$ where 
$r=\sharp (\emMin(G_{\tau }^{m},K)).$ 
Then all of the following hold.
\begin{itemize}
\item[{\em (1)}]
Let $i=1,\ldots, q.$ Then $\sharp (\emMin(G_{\tau }^{p_{i}},K_{i}))=p_{i}$.  
%for each 
%$j=1, \ldots, q.$ 
Moreover, there exist 
$K_{i,1},\ldots K_{i,p_{i}}\in \emMin(G_{\tau }^{p_{i}},K_{i})$ 
such that $\{ K_{i,k}\mid k=1,\ldots, p_{i}\}=\emMin(G_{\tau }^{p_{i}},K_{i})$, 
$K_{i}=\cup _{k=1}^{p_{i}}K_{i,k}$ 
 and 
$h(K_{i,k})\subset K_{i,k+1}$ for each $h\in \mbox{supp}\,\tau$, where 
$K_{i,p_{i}+1}:=K_{i,1}.$ Also, for each $k=1,\ldots, p_{i}$ there exists a 
unique element $\omega _{i,k}\in {\frak M}_{1}(K_{i,k})$ such that 
$(M_{\tau }^{\ast })^{p_{i}}(\omega _{i,k})=\omega _{i,k}.$ 
Also, $M_{\tau }^{np_{i}}(\varphi )\rightarrow 
(\int \varphi \ d\omega _{i,k})1_{K_{i,k}}$ in $C(K_{i,k})$ as $n\rightarrow \infty $ 
for each $\varphi \in C(K_{i,k})$, supp$\, \omega _{i,k}=K_{i,k}$ and 
$M_{\tau }^{\ast }\omega _{i,k}=\omega _{i,k+1}$ in ${\frak M} _{1}(K_{i})$ 
for each $k=1,\ldots, p_{i}$, 
where $\omega _{i,p_{i}+1}:=\omega _{i,1}$. 
Here, for each subset $B$ of $Y$, we denote by 
$1_{B}$ the characteristic function of $B.$  
\item[{\em (2)}] 
%and letting $\{ L_{j}\} _{j=1}^{r}$ be the set of all minimal sets of $G_{\tau }$ in $K$,  
We have $r=\sum _{i=1}^{q}p_{i}$ and $\cup _{j=1}^{r}H_{j}=\cup _{i=1}^{q}K_{i}.$ 
Moreover, we have that $\{ H_{j}\mid j=1,\ldots, r\} =\{ K_{i,k}\mid i=1,\ldots, q, k=1,\ldots, p_{i}\} 
=\emMin(G_{\tau }^{nm},K)$ for each $n\in \NN .$  
Moreover, for each $j=1,\ldots, r$, 
there exists a unique Borel probability measure $\eta _{j}$ on $H_{j}$ such that 
$(M_{\tau }^{m})^{\ast }(\eta _{j})=\eta _{j}$.  
Also, $M_{\tau }^{nm}(\varphi )\rightarrow (\int \varphi \ d\eta _{j})\cdot 1_{H_{j}}$ 
in $C(H_{j})$ as $n\rightarrow \infty $ for each $\varphi \in C(H_{j}).$ 
Also, supp$\, \eta _{j}=H_{j}$ for each 
$j=1,\ldots, r.$ Moreover, 
if $H_{j}=K_{i,k}$, then $\eta _{j}=\omega _{i,k}.$ 
\item[{\em (3)}] 
Let $y\in Y $ and let $\Omega $ be a Borel %measurable 
subset of 
$X_{\tau }.$ 
Let $A:=\{ \gamma \in \Omega \mid d(\gamma _{n,1}(y), K)\rightarrow 0 \ (n\rightarrow \infty )\} $
and $A_{j}:= \{ \gamma \in \Omega \mid d(\gamma _{nm,1}(y), H_{j})\rightarrow 0\ (n\rightarrow \infty )
\} $ for each $j=1,\ldots ,r.$ 
Then for each $\varphi \in C(Y)$, we have 
$\int _{A}\varphi (\gamma _{nm,1}(y))d\tilde{\tau }(\gamma )\rightarrow 
\sum _{j=1}^{r}\tilde{\tau }(A_{j})\int \varphi d\eta _{j}$ as $n\rightarrow \infty .$ 

 \end{itemize}

\end{lem}
\begin{proof}
By \cite[Theorem 6.6.4 and Lemma 6.7.1]{Du}, 
it is easy to see that  
statements (1)(2) hold.  
Statement (3) follows from statements (1)(2) and 
some fundamental arguments (or one can show statement (3) directly from statements (1)(2)). 
\end{proof}
\begin{lem}
\label{l:nlkfa}
Let $Y$ be a compact metric space. 
Let $\tau \in {\frak M}_{1,c}(\emCMX). $ 
Let $V$ be a non-empty open subset of $Y.$ 
Suppose that for each $g\in \mbox{supp}\,\tau $, $g(V)\subset V.$ 
Let $L_{\ker }:=\cap _{g\in G_{\tau }}g^{-1}(Y \setminus V).$ 
Suppose that $1\leq \sharp L_{\ker }<\infty .$ 
Let $\{ K_{i}\mid  i=1,\ldots, q\}=\emMin (G_{\tau }, L_{\ker })$ where 
$q=\sharp \emMin (G_{\tau },L_{\ker }).$ For each $i=1,\ldots q$, 
let $p_{i}\in \NN $ be the period of the finite Markov chain with state space $K_{i}$ 
induced by $\tau $.  
Let $m=\prod _{i=1}^{q}p_{i}\in \NN .$    
%Let $G_{\tau }^{m}:=\langle \{ g_{1}\circ \cdots \circ g_{m}\mid \forall g_{j}\in \mbox{supp}\,\tau\} \rangle .$ 
Let $\{ H_{j}\mid j=1,\ldots, r\}=\emMin(G_{\tau }^{m}, L_{\ker })$ where 
$r=\sharp (\emMin(G_{\tau }^{m}, L_{\ker })).$  
Then all of the following hold.
\begin{itemize}
\item[{\em (1)}]
Let $i=1,\ldots, q.$ Then $\sharp (\emMin(G_{\tau }^{p_{i}},K_{i}))=p_{i}$.  
%for each 
%$j=1, \ldots, q.$ 
Moreover, there exist 
$K_{i,1},\ldots K_{i,p_{i}}\in \emMin(G_{\tau }^{p_{i}},K_{i})$ 
such that $\{ K_{i,k}\mid  k=1,\ldots, p_{i}\}=\emMin(G_{\tau }^{p_{i}},K_{i})$, 
$K_{i}=\cup _{k=1}^{p_{i}}K_{i,k}$ 
 and 
$h(K_{i,k})\subset K_{i,k+1}$ for each $h\in \mbox{supp}\,\tau$, where 
$K_{i,p_{i}+1}:=K_{i,1}.$ Also, for each $k=1,\ldots, p_{i}$ there exists a 
unique element $\omega _{i,k}\in {\frak M}_{1}(K_{i,k})$ such that 
$(M_{\tau }^{\ast })^{p_{i}}(\omega _{i,k})=\omega _{i,k}.$ 
Also, $M_{\tau }^{np_{i}}(\varphi )\rightarrow 
(\int \varphi \ d\omega _{i,k})1_{K_{i,k}}$ in $C(K_{i,k})$ as $n\rightarrow \infty $ 
for each $\varphi \in C(K_{i,k})$, supp$\, \omega _{i,k}=K_{i,k}$ and 
$M_{\tau }^{\ast }\omega _{i,k}=\omega _{i,k+1}$ in ${\frak M} _{1}(K_{i})$ 
for each $k=1,\ldots, p_{i}$, 
where $\omega _{i,p_{i}+1}:=\omega _{i,1}$.   
\item[{\em (2)}] 
%and letting $\{ L_{j}\} _{j=1}^{r}$ be the set of all minimal sets of $G_{\tau }$ in $K$,  
We have $r=\sum _{i=1}^{q}p_{i}$ and $\cup _{j=1}^{r}H_{j}=\cup _{i=1}^{q}K_{i}.$
Moreover, we have that $\{ H_{j}\mid j=1,\ldots, r\} =\{ K_{i,k}\mid i=1,\ldots, q, k=1,\ldots, p_{i}\} 
=\emMin(G_{\tau }^{nm},K)$ for each $n\in \NN .$  
Moreover, for each $j=1,\ldots, r$, 
there exists a unique Borel probability measure $\eta _{j}$ on $H_{j}$ such that 
$(M_{\tau }^{m})^{\ast }(\eta _{j})=\eta _{j}$.  
Also, $M_{\tau }^{nm}(\varphi )\rightarrow (\int \varphi \ d\eta _{j})\cdot 1_{H_{j}}$ 
in $C(H_{j})$ as $n\rightarrow \infty $ for each $\varphi \in C(H_{j}).$ 
Also, supp$\, \eta _{j}=H_{j}$ for each 
$j=1,\ldots, r.$ Moreover, 
if $H_{j}=K_{i,k}$, then $\eta _{j}=\omega _{i,k}.$ 
%
%\item[{\em (1)}] 
%We have $r=\sum _{i=1}^{q}p_{i}$ and $\cup _{j=1}^{r}L_{j}=\cup _{i=1}^{q}K_{i}.$   
%and letting $\{ L_{j}\} _{j=1}^{r}$ be the set of all minimal sets of $G_{\tau }$ in $L_{\ker }$,  
%Moreover, for each $j=1,\ldots, r$, 
%there exists a unique Borel probability measure $\omega _{j}$ on $L_{j}$ such that 
%$(M_{\tau }^{m})^{\ast }(\omega _{j})=\omega _{j}$. Also, supp$\,\omega _{j}=L_{j}$ for each 
%$j=1,\ldots, r.$ 
\item[{\em (3)}] 
Let $y\in Y $   
and let $\Omega $ be a Borel subset of 
$X_{\tau }.$ 
Let $A:=\{ \gamma \in \Omega \mid
 y\in \cap _{j=1}^{\infty } \g _{j,1}^{-1}(Y \setminus V)\} $
and $A_{j}:= \{ \gamma \in A \mid d(\gamma _{nm,1}(y),H_{j})\rightarrow 0 \ (n\rightarrow \infty )
\} $ for each $j=1,\ldots ,r.$ 
Then for each $\varphi \in C(Y)$, we have 
$\int _{A}\varphi (\gamma _{nm,1}(y))d\tilde{\tau }(\gamma )\rightarrow 
\sum _{j=1}^{r}\tilde{\tau }(A_{j})\int \varphi d\eta _{j}$ as $n\rightarrow \infty .$  
\end{itemize}
\end{lem}
\begin{proof} 
Let $K=L_{\ker }.$ Then by the assumptions of our lemma, we have that 
$G_{\tau }(K)\subset K$ and $1\leq \sharp K<\infty .$ 
By Lemmas~\ref{l:voygvv} and \ref{l:fininvset}, 
the statement of our lemma holds. 
\end{proof}
\subsection{Invariant measures and Lyapunov exponents}
\label{ss:invariant}
In this subsection, we define invariant measures and the Lyapunov exponents 
for random dynamical systems generated by elements of ${\frak M}_{1}(\Rat).$ 
 Also, 
we show some results on random complex dynamical systems having minimal sets 
with non-zero Lyapunov exponents. 

For a holomorphic map $\varphi :U\rightarrow \CCI $ defined on an open subset 
$U$ of $\CCI $ and for any $z\in U$, we denote by 
$D\varphi _{z}:T_{z}U\rightarrow T_{\varphi (z)}\CCI $ the complex differential map 
of $\varphi $ at $z$, where $T_{z}U$ denotes the complex tangential space of $U$ at $z$ 
and $T_{\varphi (z)}\CCI $ denotes the complex tangential space of $\CCI $ at $\varphi (z)$. Also, 
we denote by $\| D\varphi _{z}\| _{s}$ the norm of $D\varphi _{z}$ with respect to the spherical 
metric on $\CCI .$    

\begin{df}
%[\cite{S7}]
\label{d:sp}
Let $Y$ be a compact metric space and 
let $\Gamma $ be a non-empty subset of  
$\CMX .$   We endow $\Gamma $ with the relative topology from $\CMX .$ 
We define a map $f:\GN \times Y \rightarrow 
\GN \times Y $ as follows: 
For a point $(\gamma  ,y)\in \GN \times Y $ where 
$\gamma =(\gamma _{1},\gamma _{2},\ldots )$, we set 
$f(\gamma  ,y):= (\sigma (\gamma  ), \gamma  _{1}(y))$, where 
$\sigma :\GN \rightarrow \GN $ is the shift map, that is, 
$\sigma (\gamma  _{1},\gamma  _{2},\ldots )=(\gamma  _{2},\gamma  _{3},\ldots ).$ 
The map $f:\GN \times Y \rightarrow \GN \times Y $ 
is called the {\bf skew product associated with the 
generator system }$\G .$ 
Moreover, we use the following notation. 
\begin{enumerate}
\item 
Let 
$\pi : \GN \times \CCI \rightarrow \GN $  
and $\pi _{Y }:\GN \times Y \rightarrow Y $ be 
the canonical projections. Note that 
$\pi ^{-1}\{ \gamma \} =\{ \gamma \} \times \CCI $ for 
each $\gamma \in \Gamma ^{\NN }.$ 
For each 
$\gamma  \in \GN $ and $n\in \NN $, we set 
$f_{\gamma  }^{n}:= f^{n}|_{\pi ^{-1}\{ \gamma  \} } : 
\pi ^{-1}\{ \gamma  \} \rightarrow 
\pi ^{-1}\{ \sigma ^{n}(\gamma  )\} .$  
Moreover, we set 
$f_{\gamma  ,n}:= \gamma  _{n}\circ \cdots \circ \gamma  _{1}.$ 
\item 
%For each $\gamma \in \Gamma ^{\NN }$, 
%we set $J^{\gamma  }:= \{ \gamma  \} \times J_{\gamma  }
%\ (\subset \GN \times Y )$. 
%\item 
%Moreover, we set 
We set $\tilde{J}(f):= \overline{\bigcup _{\gamma  \in \GN }J^{\gamma  }}$, 
where the closure is taken in the product space $\GN \times Y .$ 
Furthermore, we set $\tilde{F}(f):= (\Gamma ^{\NN }\times Y)\setminus \tilde{J}(f).$  
%(Note that $f^{-1}(\tilde{J}(f))=\tilde{J}(f)=f(\tilde{J}(f)).$)  
\item 
For each $\gamma  \in \GN $, we set 
$\hat{J}^{\gamma  ,\Gamma }:= \pi ^{-1}\{  \gamma  \} \cap \tilde{J}(f)$, 
$\hat{F}^{\gamma ,\Gamma }:= \pi ^{-1}(\{ \gamma \} )\setminus \hat{J}^{\gamma ,\G }$,  
$\hat{J}_{\gamma ,\Gamma }:= \pi _{Y }(\hat{J}^{\gamma ,\Gamma })$, 
and $\hat{F}_{\g, \G }:= Y\setminus \hat{J}_{\g ,\G }.$   
%Note that $J_{\gamma }\subset \hat{J}_{\gamma ,\Gamma }.$   
\item When $\Gamma \subset \mbox{Rat}$, 
for each $z=(\gamma  ,y)\in \GN \times \CCI $, 
we set $Df_{z}:=D(\gamma  _{1})_{y}.$ 
%More generally, for each $n\in \NN $ and 
%$z=(\gamma  ,y)\in \GN \times \CCI $, 
%we set $(f^{n})'(z):= (f_{\gamma  ,n})'(y).$ 
\end{enumerate}
\end{df}
\begin{rem}
\label{r:piypic}
Under the above notation, let $G=\langle \Gamma \rangle .$  
Then   
$\pi _{Y }(\tilde{J}(f))\subset J(G)$ and 
$\pi \circ f=\sigma \circ \pi $ on $\Gamma ^{\NN }\times Y .$ 
Note that $\hat{J}^{\gamma, \Gamma }$ is the set of 
accumulation points of fiberwise Julia sets $J^{\gamma'}, \gamma'\in \Gamma ^{\NN }$, in the fiber $\pi ^{-1}\{ \gamma \} $, for each $\gamma \in \Gamma ^{\NN }.$ 
Note also that  $J^{\gamma }\subset \hat{J}^{\gamma, \Gamma }$, $J_{\gamma }\subset \hat{J}_{\gamma, \Gamma }$ for each 
$\gamma \in \Gamma ^{\NN }.$  
Moreover, 
for each $\g \in \Gamma ^{\NN }$, 
$\gamma _{1}(J_{\gamma })\subset J_{\sigma (\gamma )}$, 
$\gamma _{1}(\hat{J}_{\gamma ,\Gamma })\subset \hat{J}_{\sigma (\gamma ),\Gamma }$, 
and $f(\tilde{J}(f))\subset \tilde{J}(f)$.  
%(see Lemma~\ref{genskewprodinvlem1}). 
Furthermore, if $\Gamma \in \Cpt(\Rat)$, then 
for each $\gamma \in \Gamma ^{\NN }$, 
$\gamma _{1}(J_{\gamma })=J_{\sigma (\gamma )}$, $\gamma _{1}^{-1}(J_{\sigma (\gamma )})=J_{\gamma }$, 
$\gamma _{1}(\hat{J}_{\gamma ,\Gamma })=\hat{J}_{\sigma (\gamma ),\Gamma }$, 
$\gamma _{1}^{-1}(\hat{J}_{\sigma (\gamma ),\Gamma })=\hat{J}_{\gamma ,\Gamma }$,  
$f(\tilde{J}(f))=\tilde{J}(f)=f^{-1}(\tilde{J}(f))$, and $f(\tilde{F}(f))=\tilde{F}(f)=f^{-1}(\tilde{F}(f))$   
(see \cite[Lemma 2.4]{S4}). 
%$$\begin{CD}
%\Gamma ^{\NN }\times \CCI @>{f}>>\Gamma ^{\NN }\times \CCI \\ 
%@V{\pi}VV 
%@VV{\pi }V\\ 
%\Gamma ^{\NN }@>>{\sigma }>\Gamma ^{\NN }  
%\end{CD}
%$$
%\item If $\sharp J(G)\geq 3$, then $\pi _{\CCI }(\tilde{J}(f))=J(G).$ 
%\end{enumerate}
\end{rem}

We now define $\tau $-invariant measures, $\tau$-ergodic measures 
and the Lyapunov exponents for $\tau \in {\frak M}_{1}(\Rat).$ 
\begin{df}
\label{d:invmeaserg}
Let $\tau \in {\frak M}_{1}(\Rat).$ 
Let $\rho \in {\frak M}_{1}(\CCI ).$ We say that 
$\rho $ is {\bf $\tau$-invariant} 
if $M_{\tau }^{\ast }(\rho )=\rho .$ 
Moreover, we say that  a $\tau $-invariant measusure 
$\rho $ is {\bf $\tau$-ergodic} if 
$A$ is a Borel subset of $\CCI $ with $\rho (A)>0$ and 
$M_{\tau }(1_{A})(z)=1_{A}(z)$ for $\rho $-a.e.$z\in \CCI $, then 
$\rho (A)=1.$  
For a $\tau $-ergodic measure $\rho $, we set 
$\chi (\tau, \rho ):= \int \log \| Df_{z}\| _{s}d(\tilde{\tau }\otimes \rho )(z)$, 
where $f: X_{\tau }\times \CCI \rightarrow X_{\tau }\times \CCI $ denotes the skew product map 
associated with $\mbox{supp}\,\tau$ (see Definition~\ref{d:sp}).  
This is called the {\bf Lyapunov exponent of $(\tau , \rho ).$}   
\end{df}
\begin{rem}
\label{r:tinvfinv} 
Let $\tau \in {\frak M}_{1}(\Rat).$ 
Let $\rho \in {\frak M}_{1}(\CCI )$ be a $\tau $-invariant measure. 
Let $f:X_{\tau }\times \CCI \rightarrow X_{\tau }\times \CCI 
$ be the skew product map associated with $\mbox{supp}\,\tau.$ 
Then by \cite[Lemma 3.1]{Mo}, the measure $\tilde{\tau }\otimes \rho \in 
{\frak M}_{1}(X_{\tau }\times \CCI )$ is $f$-invariant. 
Also, by \cite[Theorem 4.1]{Mo}, if $\rho $ is $\tau $-ergodic, then 
$\tilde{\tau }\otimes \rho $ is ergodic with respect to $f.$ 
\end{rem}

\begin{df}
\label{d:celyap}
Let $\tau \in {\frak M}_{1}(\Rat).$ 
Let $L\in \Min(G_{\tau },\CCI )$ with $\sharp L <\infty .$ 
Let $m\in \NN $ be the period of the finite Markov chain with state space $L$ induced by $\tau $ 
(see Definition~\ref{d:chainperiod}). 
%\cite[p. 300]{Du}).    
Then by \cite[Theorem 6.6.4 and Lemma 6.7.1]{Du} we have the following. 
%there exists an $m\in \NN $ such that 
\begin{itemize}
\item
$\sharp \Min(G_{\tau }^{m},L)=m$ and setting  
$\{ L_{j}\mid j=1,\ldots, m\}=\Min (G_{\tau }^{m},L)$ we have $L=\cup _{j=1}^{m}L_{j}.$  
\item 
%denoting by $\{ L_{j}\} _{j=1}^{m}$ the set of all minimal sets of $G_{\tau }^{m}$ in $\CCI $, 
%We have that  
Renumbering $L_{1},\ldots, L_{m}$ above,  
for each $j=1,\ldots, m$ there exists a unique $\omega _{L,j}\in {\frak M}_{1}(L_{j})$
such that $M_{\tau }^{mn}(\varphi )\rightarrow (\omega _{L,j}(\varphi ))\cdot 1_{L_{j}}$ 
in $C(L_{j})$ as $n\rightarrow \infty $ for each $\varphi \in C(L_{j})$,  
$(M_{\tau }^{m})^{\ast }(\omega _{L,j})=\omega _{L,j}, \mbox{supp}\, \omega _{L,j}=L_{j}$ and 
$M_{\tau }^{\ast }\omega _{L,j}=\omega _{L,j+1}$ where $\omega _{L,m+1}:=\omega _{L,1}$.  
%and 
\item 
$\omega _{L}:=\frac{1}{m}\sum _{j=1}^{m}\omega _{L,j}$ is $\tau $-ergodic.  
%where denoting by $\varphi _{m}:\Rat ^{m}\rightarrow \Rat $  
%the map $(g_{1},\ldots , g_{m})\rightarrow g_{1}\circ \cdots \circ g_{m}$, 
%we set $\tau _{m}:= (\varphi _{m})_{\ast }(\tau ^{m})\in {\frak M}_{1}(\Rat ).$ 
\end{itemize}
We call ${\omega }_{L}$ the 
{\bf canonical $\tau $-ergodic measure on $L.$} 
By \cite[Lemma 3.1, Theorem 4.1]{Mo}, $\tilde{\tau }\otimes \omega _{L}\in \frak{M}_{1}
(X_{\tau }
\times \CCI )$ is $f$-invariant and ergodic with respect to $f$, 
where $f:X_{\tau }\times \CCI \rightarrow X_{\tau }
\times \CCI $ is the skew product map associated with $\mbox{supp}\,\tau.$ 
We set 
$\chi (\tau, L):=\int \log \| Df _{z}\| _{s}d(\tilde{\tau }\otimes \omega _{L})(z).$ 
This is called the 
{\bf Lyapunov exponent of $(\tau, L).$}  
\end{df}
We now show a lemma and its corollary on $\tau$-invariant and $\tau$-ergodic measures $\mu $ 
with negative Lyapunov exponents (Lemma~\ref{l:chimune} and Corollary~\ref{c:chilne}).  
\begin{lem}
\label{l:chimune}
Let $\tau \in {\frak M}_{1,c}(\emRat). $ 
Let $\mu \in {\frak M}_{1}(\CCI )$ be a $\tau $-invariant and $\tau $-ergodic measure.  
Suppose $\chi (\tau, \mu )<0.$ 
Then for $(\tilde{\tau }\otimes \mu )$-a.e. $(\gamma , z_{0})
\in (\emRat) ^{\NN }\times \CCI $, there exist a constant 
$\delta _{1}=\delta _{1}(\gamma, z_{0})>0$, a constant 
$C=C(\gamma , z_{0})>0$ and a constant $\alpha =\alpha (\gamma, z_{0})\in (0,1)$ such that 
for each $m\in \NN $, we have 
$\mbox{diam}(\gamma _{m,1}(B(z_{0},\delta _{1})))\leq C\alpha ^{m}$. In particular, for $(\tilde{\tau }\otimes \mu )$-a.e. $(\gamma , z_{0})
\in (\emRat) ^{\NN }\times \CCI $, 
 we have $z_{0}\in F_{\gamma }.$ Moreover, 
for $\mu$-a.e. $z_{0}\in \CCI $, we have $z_{0}\in F_{pt}^{0}(\tau ).$ 
\end{lem}
\begin{proof}
For each $r\in \NN $, 
let $\psi _{r}:\mbox{supp}\,\tau\times \CCI \rightarrow \RR $ be 
the function defined by 
$$\psi _{r}(h,y)=\begin{cases}
\log \| Dh_{y}\| _{s}\mbox{ if } \log \| Dh_{y}\| _{s}\geq -r\\ 
-r \mbox{\ \ \ \ \ \ \ \ \ \  \ if } \log \| Dh_{y}\| _{s}<-r.  
\end{cases}$$
Let $\varphi _{r}:X_{\tau } \times \CCI \rightarrow \RR $ be the function 
defined by $\varphi _{r}(\gamma ,y)=\psi _{r}(\gamma _{1},y).$ 
Since $\chi (\tau, \mu )<0$, there exists an $r\in \NN $ such that 
$\int \varphi _{r}(z)d(\tilde{\tau }\otimes \mu )(z)<0.$ 
Let $c_{0}=-\int \varphi _{r}(z)d(\tilde{\tau }\otimes \mu )(z)>0.$ 
By Birkhoff's ergodic theorem, there exists a Borel subset $A$ of 
$X_{\tau }\times \CCI $ with 
$(\tilde{\tau }\otimes \mu ) (A)=1$ such that 
for each $(\gamma ,z_{0})\in A$, 
$\frac{1}{n}\sum _{j=0}^{n-1}\varphi _{r}(f^{j}(\gamma ,z_{0}))\rightarrow -c_{0}$ 
as $n\rightarrow \infty .$ 
Let $\epsilon _{0}\in (0, \frac{1}{4}c_{0}).$ 
Let $(\gamma ,z_{0})\in A.$ 
There exists an $n_{0}\in \NN $ such that 
for each $n\in \NN $ with $n\geq n_{0}$, 
$$\frac{1}{n}\sum _{j=0}^{n-1}\psi _{r}(\gamma _{j+1}, \gamma _{j,1}(z_{0}))
=\frac{1}{n}\sum _{j=0}^{n-1}\varphi _{r}(f^{j}(\gamma , z_{0}))\leq -c_{0}+\epsilon _{0},$$
where $\gamma _{0,1}=Id.$   
Let $\epsilon _{1}\in \RR $ with $0<\epsilon _{1}<\frac{1}{4}c_{0}.$ 
Since $\mbox{supp}\,\tau$ is compact, 
there exists a $\delta >0$ such that 
for each $w\in \CCI $, for each $h\in \mbox{supp}\,\tau$ and for each $z\in B(w,\delta )$, we have 
$$\log \| Dh_{z}\| _{s}\leq \psi _{r}(h,w)+\epsilon _{1}, \mbox{ thus } 
\| Dh_{z}\| _{s}\leq \exp (\psi _{r}(h,w)+\epsilon _{1}).$$
There exists a $\delta _{1}>0$ with $\delta _{1}<\frac{\delta }{2}$ such that 
for each $j=1,\ldots , n_{0}$, 
$\gamma _{j,1}(B(z_{0},\delta _{1}))\subset B(\gamma _{j,1}(z_{0}),\frac{\delta }{2}).$ 
Therefore we obtain 
\begin{eqnarray*}
\gamma _{n_{0},1}(B(z_{0},\delta _{1})) & \subset & 
B(\gamma _{n_{0},1}(z_{0}),\delta _{1}\exp ((\sum _{j=0}^{n_{0}-1}
\psi _{r}(\gamma _{j+1},\gamma _{j,1}(z_{0})))+n_{0}\epsilon _{1}))\\ 
 & \subset &  
B(\gamma _{n_{0},1}(z_{0}),\delta _{1}\exp ((-c_{0}+\epsilon _{0}+\epsilon _{1})n_{0})).
\end{eqnarray*}
Hence we can show that for each $m\in \NN \cup \{ 0\} $, 
\begin{eqnarray*}
\gamma _{n_{0}+m,1}(B(z_{0},\delta _{1})) & \subset & 
B(\gamma _{n_{0}+m,1}(z_{0}), \delta _{1}\exp (\sum _{j=0}^{n_{0}+m-1}
(\psi _{r}(\gamma _{j+1}, \gamma _{j,1}(z_{0}))+(n_{0}+m)\epsilon _{1}))\\ 
 & \subset & B(\gamma _{n_{0}+m,1}(z_{0}),\ \delta _{1}\exp ((-c_{0}+\epsilon _{0}+\epsilon _{1})(n_{0}+m)))\\ 
 & \subset & B(\gamma _{n_{0}+m,1}(z_{0}),\frac{\delta }{2})
\end{eqnarray*}
by induction on $m\in \NN \cup \{ 0\}.$ 
Therefore there exist a constant 
$\delta _{1}=\delta _{1}(\gamma, z_{0})>0$, a constant 
$C=C(\gamma , z_{0})>0$ and a constant $\alpha =\alpha (\gamma, z_{0})\in (0,1)$ such that 
for each $m\in \NN $, we have 
$\mbox{diam}(\gamma _{m,1}(B(z_{0},\delta _{1})))\leq C\alpha ^{m}$. Hence 
$z_{0}\in F_{\gamma }.$ 
Thus for $\mu$-a.e. $z_{0}\in \CCI $, we obtain  that 
$\tilde{\tau }(\{ \gamma \in (\mbox{supp}\,\tau)^{\NN} \mid z_{0}\in J_{\gamma }\})=0.$ 
By Lemma~\ref{l:ttaugy0y}, it follows that 
for $\mu $-a.e.$z_{0}\in \CCI $, $z_{0}\in F_{pt}^{0}(\tau ).$ 
Hence we have proved our lemma.  
\end{proof}
\begin{cor}
\label{c:chilne}
Let $\tau \in {\frak M}_{1,c}(\emRat) $ and let $L\in \emMin(G_{\tau },\CCI )$ with 
$\sharp L<\infty .$ Suppose $\chi (\tau , L)<0.$ 
Then for each $z_{0}\in L$, for $\tilde{\tau }$-a.e.$\gamma \in (\emRat) ^{\NN }$, 
we have $z_{0}\in F_{\gamma }.$ Moreover, $L\subset F_{pt}^{0}(\tau ).$ 

\end{cor}
\begin{proof}
Since $\mbox{supp}\,\omega _{L}=L,$ 
Lemma~\ref{l:chimune} implies the statement of our corollary. 
\end{proof}
We now show some lemmas on minimal sets of $G_{\tau }$ with positive Lyapunov exponents 
(Lemmas~\ref{l:chlpad1d2}--\ref{l:chlpa0}). 
\begin{lem}
\label{l:chlpad1d2}
Let $\tau \in {\frak M}_{1,c}(\emRat).$ 
Let $L\in \emMin(G_{\tau },\CCI )$ with $\sharp L<\infty .$ 
Suppose $\chi (\tau, L)>0.$ Suppose also 
that for each $z_{0}\in L$ and for each $g\in \mbox{supp}\,\tau$, 
$Dg_{z_{0}}\neq 0.$ Let $\alpha >0.$ 
Then  there exist $\delta _{1}>0, \delta _{2}>0$ with $\delta _{2}<\alpha $, and 
a Borel subset $A$ of $(\mbox{supp}\,\tau)^{\NN}$ with $\tilde{\tau }(A)=1$,  
where $\delta _{1}$ and $A$ do not depend on $\alpha $, 
such that 
for each  $z_{0}\in L$, for each $z\in B(z_{0},\delta _{2})\setminus \{ z_{0}\}$ 
and for each $\gamma \in A$, 
there exists an $n_{1}=n_{1}(\gamma, z)\in \NN $ 
with $\gamma _{n_{1},1}(z)\not\in B(L,\delta _{1}).$ 
In particular, for each $z_{0}\in L$, for $\tilde{\tau }$-a.e.$\gamma \in (\mbox{supp}\,\tau)^{\NN }$, 
we have $z_{0}\in J_{\gamma }.$ 
\end{lem}
\begin{proof}
Since $\mbox{supp}\,\tau$ is compact, there exists 
a $\delta >0$ such that 
for each $w_{0}\in L$ and for each $g\in \mbox{supp}\,\tau$, 
$g:B(w_{0},5\delta )\rightarrow \CCI $ is injective.
Let $u:=\min \{ d(a,b)\mid a,b\in L, a\neq b\} >0$
 (if $\sharp L=1$ then let $u=1$).  
Let $0<\epsilon <\frac{1}{4}\chi (\tau, L).$ 
Then there exists a $\delta _{1}>0$ with $\delta _{1}<\min \{ \frac{u}{2},\delta \} $ 
such that 
\begin{itemize}
\item[(i)] 
for each $w_{0}\in L $ and for each $g\in \mbox{supp}\,\tau$, we have 
$g(B(w_{0},\delta _{1}))\subset B(g(w_{0}),\frac{u}{2})$, and 
\item[(ii)] 
for each $w_{0}\in L$ and for each $g\in \mbox{supp}\,\tau$, 
there exists an inverse branch $g_{w_{0}}^{-1}:B(g(w_{0}),2\delta _{1})
\rightarrow B(w_{0},\delta )$ of $g$ with $g_{w_{0}}^{-1}(g(w_{0}))=w_{0}$ 
such that for each $w\in B(g(w_{0}),2\delta _{1})$, we have 
\begin{equation}
\label{eq:lndgt}
\log \| D(g_{w_{0}}^{-1})_{w}\|_{s} 
\leq \log \| D(g_{w_{0}}^{-1})_{g(w_{0})}\| _{s}+\epsilon .
\end{equation}
\end{itemize} 
%Let $0<\epsilon <\frac{1}{4}\chi (L).$ 
By Birkhoff's ergodic theorem, 
there exists a Borel subset $A$ of $X_{\tau }$ with 
$\tilde{\tau }(A)=1$ such that 
for each $(\gamma, z_{0})\in A\times L$, 
there exists an $n_{0}=n_{0}(\gamma, z_{0})\in \NN $ 
such that for each $n\in \NN $ with $n\geq n_{0}$ 
we have 
\begin{equation}
\label{eq:enchl-e}
|\frac{1}{n}\sum _{j=0}^{n-1}\log \| D(\g _{j+1})_{\g _{j,1}(z_{0})}\| _{s}-\chi (\tau , L)|
<\epsilon, \mbox{ thus } 
e^{n(\chi (\tau, L)-\epsilon )}\leq \| D(\gamma _{n,1})_{z_{0}}\| _{s}\leq e^{n(\chi (\tau, L)+\epsilon )}.
\end{equation}
Let $\delta _{2}:=\frac{1}{2}\min \{ \alpha, \delta _{1}\} >0.$ 
Let $z_{0}\in L.$ 
Let $z\in B(z_{0},\delta _{2})\setminus \{ z_{0}\} .$ Let $\gamma \in A.$ 
We now prove the following claim.\\ 
Claim 1. There exists an $n\in \NN $ such that 
$\gamma _{n,1}(z)\not\in B(L,\delta _{1}).$ 

To prove this claim, suppose that 
 for each $n\in \NN $, $\gamma _{n,1}(z)\in B(L,\delta _{1}).$ 
Let $n_{0}=n_{0}(\gamma, z_{0})$ be the number defined above. 
Let $m\in \NN $ with $m\geq n_{0}.$ 
Then 
we have 
$\gamma _{m}(\gamma _{m-1,1}(z))=\gamma _{m,1}(z), $ 
$\gamma _{m}((\gamma _{m})^{-1}_{\gamma _{m-1,1}(z_{0})}(\gamma _{m,1}(z)))=\gamma _{m,1}(z)$, 
$\gamma _{m-1,1}(z)\in B(\gamma _{m-1,1}(z_{0}),5\delta), $ 
$(\gamma _{m})^{-1}_{\gamma _{m-1,1}(z_{0})}(\gamma _{m,1}(z))\in B(\gamma _{m-1,1}(z_{0}),5\delta )$, 
and $\gamma _{m}:B(\gamma _{m-1,1}(z_{0}),5\delta )\rightarrow \CCI $ 
is injective. 
Hence $$\gamma _{m-1,1}(z)=({\gamma }_{m})^{-1}_{\gamma _{m-1,1}(z_{0})}(\gamma _{m,1}(z)).$$   
Similarly, it is easy to see that for each $j=1,\ldots ,m$, 
\begin{equation}
\label{eq:gm-j1z}
\gamma _{m-j,1}(z)=(\gamma _{m-j+1})^{-1}_{\gamma _{m-j,1}(z_{0})}(\gamma _{m-j+1,1}(z)).
\end{equation}
Combining (\ref{eq:lndgt}), (\ref{eq:enchl-e}), (\ref{eq:gm-j1z}), 
we obtain that 
\begin{eqnarray*}
d(z,z_{0}) & \leq & \delta _{1}\exp (-\sum _{j=1}^{m}\log \| D(\g _{j})_{\gamma _{j-1,1}(z_{0})}\| _{s}
+m\epsilon _{1})\\ 
 & = & \delta _{1}\| D(\gamma _{m,1})_{z_{0}}\| _{s}^{-1}\cdot e^{m\epsilon }\\ 
 & \leq & \delta _{1}e^{-m(\chi (\tau, L)-\epsilon )}\cdot e^{m\epsilon } \\ 
 & = & \delta _{1}e^{-m(\chi (\tau, L)-2\epsilon )}. 
\end{eqnarray*}
Since the above inequality holds for any $m\in \NN $, it follows that 
$z=z_{0}.$ However, this is a contradiction. Therefore Claim 1 holds. 

%By Claim 1, the statement of our lemma holds. 
We now let $z_{0}\in L$ and $\gamma \in A.$ 
Suppose $z_{0}\in F_{\gamma }.$ Then 
there exists a number $\delta _{3}>0$ with $\delta _{3}<
\delta _{2}$ such that 
for each $z\in B(z_{0}, \delta _{3})\setminus \{ z_{0}\}$ and 
for each $n\in \NN $, we have 
$d(\gamma _{n, 1}(z), \gamma _{n,1}(z_{0}))<\delta _{1}.$ 
Since $G_{\tau }(L)\subset L$, it implies that 
$d(\gamma _{n, 1}(z), L)<\delta _{1}$ for each $n\in \NN .$  
However, this contradicts Claim 1. Thus we have 
$z_{0}\in J_{\gamma }.$ 

Hence, for each $z_{0}\in L$, 
for $\tilde{\tau }$-a.e. $\gamma \in (\mbox{supp}\,\tau)^{\NN }$, 
we have $z_{0}\in J_{\gamma }.$ 
\end{proof}
\begin{lem}
\label{l:chilp}
Let $\tau \in {\frak M}_{1,c}(\emRat). $ Let 
$L\in \emMin(G_{\tau },\CCI )$ with $\sharp L<\infty .$ 
Suppose $\chi (\tau, L)>0.$ 
Suppose also that for each $x\in L$ and for each 
$g\in \supptau$, we have 
$Dg_{x}\neq 0.$ 
Let $y\in \CCI .$ 
Let $B=\{ \gamma \in X_{\tau }\mid d(\gamma _{n,1}(y),L)\rightarrow 0 
\mbox{ as }n\rightarrow \infty \} .$ 
Then for $\tilde{\tau }$-a.e. $\gamma \in B$, there exists a number $n_{0}=n_{0}(\gamma, y)\in \NN $ 
such that for each $n\in \NN $ with $n\geq n_{0}$, we have $\gamma _{n,1}(y)\in L.$ 
\end{lem}
\begin{proof}
Suppose that there exists a Borel subset $B_{0}$ of $B$ with $\tilde{\tau }(B_{0})>0$ 
such that for each $\gamma \in B_{0}$ and for each $n\in \NN $, $\gamma _{n,1}(y)\not\in L.$ 
Since $\tilde{\tau }$ is invariant under the shift map 
$\sigma :X_{\tau }\rightarrow X_{\tau }$, 
Lemma~\ref{l:chlpad1d2}  implies that for $\tilde{\tau }$-a.e. $\gamma \in B_{0}$, 
$\limsup _{n\rightarrow \infty }d(\gamma _{n,1}(y),L)>0.$ However, this is a contradiction. 
Hence the statement of our lemma holds.  
\end{proof}
\begin{lem}
\label{l:tyaca0}
Let $\tau \in {\frak M}_{1,c}(\emRat). $ 
Let $y\in \CCI .$ 
Then there exists a subset $A$ of $\CCI $ 
with $\sharp (\CCI \setminus A)\leq \aleph _{0}$
such that for each $x\in A$, $\tau (\{ g\in \emRat \mid g(x)=y\} )=0.$ 
\end{lem}
\begin{proof}
For each finite subset $F=\{ x_{1},\ldots ,x_{n}\} $ of $\CCI $  
such that $x_{1},\ldots ,x_{n}$ are mutually distinct, 
let $B_{F}:=\{ g\in \Rat \mid g(x_{i})=y, \mbox{ for each }i=1,\ldots, n\}.$ 
Since $\mbox{supp}\, \tau $ is compact, 
there exists an $N\in \NN $ such that 
for each $g\in \mbox{supp}\,\tau$, $\deg (g)\leq N.$ 
Hence, if $\sharp F>N$, then 
$\tau (B_{F})=\tau (B_{F}\cap \mbox{supp}\,\tau)=\tau (\emptyset )=0.$  
 For each $k\in \ZZ $ with $0\leq k\leq N$, 
let ${\cal F}_{k}=\{ F\subset \CCI \mid \sharp F=N+1-k, \tau (B_{F})>0\}.$ Note that ${\cal F}_{0}=\emptyset $ 
from the above argument.  
We now prove the following claim. \\ 
Claim 1. Let $k\in \ZZ $ with $0\leq k<N.$ 
If $\sharp {\cal F}_{k} \leq \aleph _{0}$, 
then $\sharp {\cal F}_{k+1} \leq \aleph _{0}.$  

To prove this claim, let $0\leq k<N$ and 
suppose  we have that $\sharp {\cal F}_{k} \leq \aleph _{0}$.    
Let 
${\cal H}$ be the set $\{ H\in {\cal F}_{k+1}\mid \exists F\in {\cal F}_{k} \mbox{ such that }H\subset F\} .$ 
Then $\sharp {\cal H}\leq \aleph _{0}.$ 
Moreover, for each $H_{1},H_{2}\in {\cal F}_{k+1}\setminus {\cal H}$ with 
$H_{1}\neq H_{2}$, we have 
\begin{equation}
\label{eq:tbh1c}
\tau (B_{H_{1}}\cap B_{H_{2}})=0.
\end{equation}
For, let $x\in H_{2}\setminus H_{1}$ and let $F=H_{1}\cup \{ x\} .$ 
Then $\sharp F=N+1-k$ and $H_{1}\subset F.$ 
Since $H_{1}\not\in {\cal H}$, we have $F\not\in {\cal F}_{k}.$ Hence $\tau (B_{F})=0.$ 
Since $B_{H_{1}}\cap B_{H_{2}}\subset B_{F}$, 
%it follows that 
%$\tau (B_{H_{1}}\cap B_{H_{2}})=0.$  
(\ref{eq:tbh1c}) holds. 
By (\ref{eq:tbh1c}), $\sharp ({\cal F}_{k+1}\setminus {\cal H})\leq \aleph _{0}.$ 
Therefore $\sharp {\cal F}_{k+1}\leq \aleph _{0}.$ Thus we have prove Claim 1. 

By Claim 1, we obtain that 
$\sharp \{ H\subset \CCI \mid \sharp H=1, \tau (B_{H})>0\} \leq \aleph _{0}.$ 
Hence the statement of our lemma holds. 
\end{proof}
\begin{lem}
\label{l:tcac}
Let $\tau \in {\frak M}_{1,c}(\emRat).$ 
Let $C$ be a non-empty finite subset of $\CCI .$ 
Then there exists a subset $A_{C}$ of $\CCI $ with 
$\sharp (\CCI \setminus A_{C})\leq \aleph _{0}$ 
such that 
for each $x\in A_{C}$, 
$$\tilde{\tau }(\{ \gamma \in X_{\tau }\mid \exists n\in \NN 
\mbox{ such that } \gamma _{n,1}(x)\in C\} )=0.$$  
\end{lem}
\begin{proof}
%For each $n\in \NN $ and each $y\in C$, 
Let $D_{y,n}=\{ x\in \CCI \mid 
\tau ^{n}(\{ (\gamma _{1},\ldots ,\gamma _{n})\in 
(\mbox{supp}\,\tau )^{n}\mid 
\gamma _{n}\cdots \gamma _{1}(x)=y\} )>0\} $ 
for each $y\in C$ and each $n\in \NN, $  
where $\tau ^{n}=\otimes _{j=1}^{n}\tau \in {\frak M}_{1,c}((\mbox{supp}\,\tau)^{n}).$   
By using the argument in the proof of Lemma~\ref{l:tyaca0}, 
we can show that $\sharp D_{y,n}\leq \aleph _{0}.$ 
Let $A_{C}=\CCI \setminus (\cup _{y\in C,n\in \NN }D_{y,n}).$ 
Then $\sharp (\CCI \setminus A_{C})\leq \aleph _{0}.$ 
For each $x\in A_{C}$, we have 
\begin{eqnarray*}
& & \tilde{\tau }(\{ \gamma \in X_{\tau }\mid \exists n\in \NN 
\mbox{ such that }\gamma _{n,1}(x)\in C\} )\\ 
& \leq & \tilde{\tau }(\cup _{n\in \NN ,y\in C}\{ \gamma \in X_{\tau }
\mid \gamma _{n,1}(x)=y\})\\ 
& \leq & \sum _{n\in \NN ,y\in C}\tilde{\tau }(\{ \gamma \in X_{\tau }
\mid \gamma _{n,1}(x)=y\} )\\ 
& = & \sum _{n\in \NN , y\in C}\tilde{\tau }(\{ (\gamma _{1},\ldots, \gamma _{n})
\in (\mbox{supp}\,\tau)^{n}\mid \gamma _{n}\cdots \gamma _{1}(x)=y\} 
\times \prod _{j=n+1}^{\infty }\mbox{supp}\,\tau)\\ 
& = & \sum _{n\in \NN ,y\in C}\tau ^{n}(\{ (\gamma _{1},\ldots, \gamma _{n})\in (\mbox{supp}\,\tau)^{n}
\mid \gamma _{n}\cdots \gamma _{1}(x)=y\})=0. \\ 
\end{eqnarray*}
Thus the statement of our lemma holds. 
\end{proof}
%\begin{lem}
%\label{l:tlfclpy}
%Let $\tau \in {\frak M}_{1,c}(\emRat).$ 
%Let $L\in \Min(G_{\tau },\CCI )$ with $\sharp L<\infty .$ 
%Suppose that $\chi (L)>0.$ Suppose also that 
%for each $g\in \mbox{supp}\,\tau $ and for each $x\in L$, $Dg_{x}\neq 0.$ 
%Then 
%$$\sharp \left\{ y\in \CCI \mid \tilde{\tau }(\{ \gamma \in (\emRat) ^{\NN }\mid 
%d(\gamma _{n,1}(y),L)\rightarrow 0 \mbox{ as } n\rightarrow \infty \} )>0\right\} 
%\leq \aleph _{0}.$$ 
%\end{lem}
%\begin{proof}
%\end{proof}

\begin{lem}
\label{l:chlpa0}
Let $\tau \in {\frak M}_{1,c}(\emRat).$ 
Let $L\in \emMin(G_{\tau },\CCI )$ with $\sharp L<\infty .$ 
Suppose that $\chi (\tau, L)>0$ and for each $x\in L$ and 
for each $g\in \mbox{supp}\,\tau$, $Dg_{x}\neq 0.$ 
Then for each $y\in \CCI $, we have 
$$\tilde{\tau }(\{ \gamma \in X_{\tau }\mid 
d(\gamma _{n,1}(y),L)\rightarrow 0 \mbox{ as }n\rightarrow \infty \} )
=\tilde{\tau }(\{ \gamma \in X_{\tau }\mid \exists n\in \NN 
\mbox { such that }\gamma _{n,1}(y)\in L\} )$$ and  
$$ \sharp \{ y\in \CCI \mid \tilde{\tau }(\{ \gamma \in X_{\tau } \mid 
d(\gamma _{n,1}(y),L)\rightarrow 0 \mbox{ as }n\rightarrow \infty \} )>0\} \leq \aleph _{0}.$$
\end{lem}
\begin{proof}
Lemma~\ref{l:chilp} implies that for each $y\in \CCI $, 
$$\tilde{\tau }(\{ \gamma \in X_{\tau }\mid 
d(\gamma _{n,1}(y),L)\rightarrow 0 \mbox{ as }n\rightarrow \infty \} )
=\tilde{\tau }(\{ \gamma \in X_{\tau }\mid \exists n\in \NN 
\mbox { such that }\gamma _{n,1}(y)\in L\} ).$$ 
Hence 
\begin{eqnarray*}
\ & \ &  \{ y\in \CCI 
\mid \tilde{\tau }(\{ \gamma \in X_{\tau }\mid d(\gamma _{n,1}(y),L)\rightarrow 0
\mbox{ as }n\rightarrow \infty \} )>0\} \\ 
& = & \{ y\in \CCI \mid \tilde{\tau }(\{ \gamma \in X_{\tau }
\mid \exists n\in \NN \mbox { such that } \gamma _{n,1}(y)\in L\} )>0\} 
\subset \CCI \setminus A_{L},
\end{eqnarray*}  
where $A_{L}$ is the set for $L$ coming from Lemma~\ref{l:tcac}. 
Since $\sharp (\CCI \setminus A_{L})\leq \aleph _{0}$, the statement of our lemma 
holds. 
\end{proof}
\subsection{Systems with finite kernel Julia sets} 
\label{ss:finitejker}
In this subsection, 
we show a theorem on the random dynamical systems generated by 
elements $\tau \in {\frak M}_{1,c}(\Rat)$ with $J_{\ker }(G_{\tau })<\infty .$ 
\begin{thm}
\label{t:jkfcln0}
Let $\tau \in {\frak M}_{1,c}(\emRat). $
Suppose we have all of the following. 
\begin{itemize}
\item[{\em (i)}] 
$\sharp J_{\ker }(G_{\tau })<\infty .$

\item[{\em (ii)}] 
For each $L\in \emMin(G_{\tau }, J_{\ker }(G_{\tau }))$, 
we have $\chi (\tau, L)\neq 0.$
\item[{\em (iii)}] 
For each $L\in \emMin(G_{\tau }, J_{\ker }(G_{\tau }))$ with  
$\chi (\tau, L)>0$, for each $g\in \mbox{supp}\,\tau$ and  for each $x\in L$, we have 
$Dg_{x}\neq 0.$ 
\end{itemize}
Let $H_{+}=\{ L\in \emMin(G_{\tau }, J_{\ker }(G_{\tau }))\mid  \chi (\tau, L)>0\}$ and 
we denote by $\Omega $ the set of 
points $y\in \CCI$ for which  
$\tilde{\tau }(\{ \gamma \in X_{\tau }\mid \exists n\in \NN \mbox{ s.t. }
\gamma _{n,1}(y)\in \cup _{L\in H_{+}}L\} )=0.$  
Then, $\sharp (\CCI \setminus \Omega )\leq \aleph _{0}$ and 
%there exists a subset $\Omega $ of $\CCI $ with 
%$\sharp (\CCI \setminus \Omega )\leq \aleph _{0}$ such that 
for each $z\in \Omega $, 
$\tilde{\tau }(\{ \gamma \in X_{\tau }\mid z\in J_{\gamma }\} )=0.$
Moreover, 
for $\tilde{\tau }$-a.e.$\gamma \in (\emRat)^{\NN}$, 
$\mbox{{\em Leb}}_{2}(J_{\gamma })=0.$ Furthermore,  
$J_{pt}^{0}(\tau )\subset \CCI \setminus \Omega $ and 
$\sharp (J_{pt}^{0}(\tau ))\leq \aleph _{0}.$
\end{thm}
\begin{proof}
%Suppose that the statement `` for $\tilde{\tau }$-a.e.$\gamma \in (\Rat)^{\NN}$, 
%$\mbox{{\em Leb}}_{2}(J_{\gamma })=0$'' is not true. 
%Then, since $\mbox{Leb}_{2}(J_{\sigma (\gamma )})=\mbox{Leb}_{2}(\gamma _{1}(J_{\gamma }))
%=\mbox{Leb}_{2}(J_{\gamma })$, since 
%$\gamma \mapsto \mbox{Leb}_{2}(J_{\gamma })$ is Borel measurable 
%(see Lemma~\ref{l:ufgmeas}), and since $(\sigma, \tilde{\tau })$ is 
%ergodic, it follows that 
%\begin{equation}
%\label{eq:fttael2}
%\mbox{ for } \tilde{\tau }\mbox{-a.e.}\gamma , \mbox{Leb}_{2}(J_{\gamma })>0. 
%\end{equation}
%Let $J_{small}(f):=\cup _{\gamma \in \mbox{supp}\,\tau^{\NN }}J^{\gamma} (\subset \mbox{supp}\,\tau^{\NN }
%\times \CCI ).$ 
%  By Lemma~\ref{l:voygvv},    
% we obtain that there exists  a Borel subset $A_{1}$ of $\mbox{supp}\,\tau^{\NN }$ 
% with $\tilde{\tau }(A_{1})=1$ such that for each $\gamma \in A_{1}$, 
% for $\tilde{\tau }\otimes \mbox{Leb}_{2}$-a.e.$(\gamma, z)\in J_{small}(f)$, 
% we have 
% $d(\gamma _{n,1}(z), \tilde{S}_{\tau })\rightarrow 0$ as $n\rightarrow \infty $, 
% where $\tilde{S}_{\tau }=\cup _{L\in \Min(G_{\tau },\CCI ),L\subset J_{\ker}(G_{\tau })}L.$  
Under the assumptions of our theorem,  
Lemma~\ref{l:chlpa0} implies that 
\begin{equation}
\label{eq:tgdgbl0}
\Omega =\{ y\in \CCI \mid 
\tilde{\tau }(\{ \gamma \in X_{\tau }\mid 
d(\gamma _{n,1}(y), \cup _{L\in H_{+}}L)\rightarrow 0 
\mbox{ as } n\rightarrow \infty \} )=0\} .
\end{equation} 
From (\ref{eq:tgdgbl0}) and Lemma~\ref{l:chlpa0}, 
%there exists a subset $\Omega $ of $\CCI $ with $\sharp (\CCI \setminus \Omega )\leq 
%\aleph _{0}$ such that 
it follows that $\sharp (\CCI \setminus \Omega )\leq \aleph _{0}$.  
%and 
%for each $y\in \Omega $, 
%\begin{equation}
%\label{eq:tgdgbl0}
%\tilde{\tau }(\{ \gamma \in X_{\tau }\mid d(\gamma _{n,1}(y),\cup _{L\in H_{+}}L)
%\rightarrow 0 \mbox{ as }n\rightarrow \infty \})=0.
%\end{equation}
Let $z\in \Omega .$ 
Let $C_{z}=\{ \gamma \in X_{\tau }\mid z\in J_{\gamma }\}.$ 
Suppose $\tilde{\tau }(C_{z})>0.$ 
 Let $H_{-}$ be the set of all $ L\in \Min(G_{\tau }, J_{\ker }(G_{\tau }))$ with $\chi(\tau, L)<0.$ 
By Lemma~\ref{l:voygvv} (with $V=F(G_{\tau })$), 
Remark~\ref{r:piypic} and the assumptions of our theorem, we have that 
for $\tilde{\tau }$-a.e.$\gamma \in C_{z}$, 
$d(\gamma _{n,1}(z), \cup _{L\in H_{+}\cup H_{-}}L)\rightarrow 0$ as $n\rightarrow \infty .$ 
Combining this with (\ref{eq:tgdgbl0}), 
we obtain that 
\begin{equation}
\label{eq:dgn1zclb}
\mbox{ for } \tilde{\tau } \mbox{ -a.e. } \gamma \in C_{z},  
d(\gamma _{n,1}(z), \cup _{L\in H_{-}}L)\rightarrow 0\mbox{ as }n\rightarrow \infty .
\end{equation} 
% By Lemma~\ref{l:}
%Combining this with Lemma XXX,  it follows that 
%for $\tilde{\tau }\otimes \mbox{Leb}_{2}$-a.e.$(\gamma, z)\in J_{small}(f)$, 
%we have $d(\gamma _{n,1}(z), \cup _{L\in B}L)\rightarrow 0$ as $n\rightarrow \infty .$ 
%Hence there exists a Borel subset $D$ of $\CCI $ with 
%Leb$_{2}(D)=\mbox{Leb}_{2}(\CCI )$ 
%such that for each $z\in D$, for $\tilde{\tau }$-a.e.$\gamma \in 
%C_{z}:=\{ \gamma \in \mbox{supp}\,\tau^{\NN }\mid z\in J_{\gamma }\} $, 
%we have 
%\begin{equation}
%\label{eq:dgn1zclb}
%d(\gamma _{n,1}(z),\cup _{L\in B}L)\rightarrow 0 \mbox{ as }n\rightarrow \infty . 
%\end{equation}
%By (\ref{eq:fttael2}), there exists an element $z\in D$ with $\tilde{\tau }(C_{z})>0.$ 
Let $0<\epsilon <\frac{1}{2}\tilde{\tau }(C_{z}).$ 
By Corollary~\ref{c:chilne}, 
for each $z_{0}\in \cup _{L\in H_{-}}L$, 
for $\tilde{\tau }$-a.e. $\gamma $, we have $z_{0}\in F_{\gamma }.$ 
Combining this with the argument to deduce (\ref{eq:ngeq1d}) in the proof of Lemma~\ref{l:ttaugy0y}, 
we obtain that there exist a Borel subset $A_{1}$ of $X_{\tau }$ 
with $\tilde{\tau }(A_{1})\geq 1-\epsilon $ 
and a $\delta >0$  
such that for each $z_{0}\in \cup _{L\in H_{-}}L$, for each $\gamma \in A_{1}$, 
we have $\sup _{n\geq 1}\mbox{diam}\gamma _{n,1}(B(z_{0},\delta ))\leq \frac{1}{10}\mbox{diam}\CCI .$
In particular, 
\begin{equation}
\label{eq:bz0dsf}
\mbox{ for each }z_{0}\in \cup _{L\in H_{-}}L \mbox{ and for each } \gamma \in A_{1},  
B(z_{0},\delta )\subset F_{\gamma }. 
\end{equation}
By (\ref{eq:dgn1zclb}) and Egoroff's theorem, there exist a Borel subset 
$A_{2}$ of $C_{z}$ with $\tilde{\tau }(A_{2})\geq \tilde{\tau }(C_{z})-\epsilon $ 
and an $n_{0}\in \NN $ such that 
for each $\gamma \in A_{2}$, 
\begin{equation}
\label{eq:gn01zbclb}
\gamma _{n_{0},1}(z)\in B(\cup _{L\in H_{-}}L, \delta ).
\end{equation} 
By (\ref{eq:bz0dsf}) and (\ref{eq:gn01zbclb}), 
we obtain $A_{2}\cap \sigma ^{-n_{0}}(A_{1})=\emptyset .$ 
Therefore $\tilde{\tau }(A_{2})\leq \tilde{\tau }(X_{\tau }\setminus 
\sigma ^{-n_{0}}(A_{1}))\leq \epsilon .$ 
Combining this with that $\tilde{\tau }(A_{2})\geq \tilde{\tau }(C_{z})-\epsilon $, 
we obtain that $\tilde{\tau }(C_{z})\leq 2\epsilon .$ 
However, this is a contradiction because $\epsilon <\frac{1}{2}\tilde{\tau }(C_{z}).$ 
Thus, we have proved that 
for each $z\in \Omega $, $\tilde{\tau }(C_{z})=0.$ 
By Fubini's theorem, it follows that 
for $\tilde{\tau }$-a.e.$\gamma $, Leb$_{2}(J_{\gamma })=0.$ 
%Combing this with Corollary~\ref{c:ttauaegl0}, 
Moreover, by Lemma~\ref{l:ttaugy0y}, 
we obtain that $J_{pt}^{0}(\tau )\subset \CCI \setminus \Omega $ and  
$\sharp J_{pt}^{0}(\tau )\leq \aleph_{0}.$ 
Thus we have proved our theorem. 
\end{proof}
%\begin{thm}
%\label{t:sjkfcln}
%Let $\tau \in {\frak M}_{1,c}(\emRat).$ Suppose that 
%$\sharp J_{\ker }(G_{\tau })<\infty ,$ $\sharp J(G_{\tau })\geq 3$ and that for each $L\in \emMin(G_{\tau }, %J_{\ker }(G_{\tau }))$,  
%$\chi (\tau ,L)<0.$ Then we have $F_{pt}^{0}(\tau )=\CCI , F_{meas}(\tau )={\frak M}_{1}(\CCI )$ 
%and all statements in \cite[Theorem 3.15]{Splms10} hold. Moreover, 
%there exists a negative constant $c<0$ such that for each $z\in \CCI $, 
%there exists a Borel subset ${\cal A}_{z}$ of $\mbox{supp}\,\tau^{\NN } $ 
%with $\tilde{\tau }({\cal A}_{z})=1$ satisfying that 
%for each $\gamma =(\gamma _{1},\gamma _{2},\ldots )\in {\cal A}_{z}$, 
%we have $$\limsup _{n\rightarrow \infty }\frac{1}{n}\log \| D(\gamma _{n,1})_{z}\| _{s}
%\leq c<0.$$    
%\end{thm}
%\begin{proof}
%We modify the proof of Theorem~\ref{t:jkfcln0}. 
%By the assumption of our theorem, we can take $\Omega $ in 
%the proof of Theorem~\ref{t:jkfcln0} as $\Omega =\CCI .$ 
%By using the argument of Theorem~\ref{t:jkfcln0}, we see that 
%for each $z\in \CCI $, $\tilde{\tau }(\{ \gamma \in \mbox{supp}\,\tau^{\NN }\mid z\in J_{\gamma }\})=0.$ 
%By Lemma~\ref{l:ttaugy0y}, it follows that for each $z\in \CCI $, $z\in F_{pt}^{0}(\tau ).$ 
%Therefore $F_{meas}(\tau )=\CCI.$  

%\end{proof}
\subsection{Random dynamical systems generated by measures on weakly nice sets} 
\label{ss:rdswn}
In this subsection, we show several results (including Theorems~\ref{t:rcdnkmain1i}, \ref{t:rcdnkmain2i} 
and their detailed and more generalized version Theorems~\ref{t:rcdnkmain1}, \ref{t:rcdnkmain2}) 
regarding random complex dynamical systems generated by measures on weakly nice subsets of 
$\Rat .$ 

We now consider holomorphic families of rational maps. 
\begin{df}
\label{d:singdf}
Let $\Lambda $ be a complex manifold. 
Let ${\cal W}=\{ f_{\lambda }\} _{\lambda \in \Lambda }$ be a family of rational maps on $\CCI .$ 
We say that ${\cal W}$ is a 
{\bf holomorphic family of rational maps} if 
$(z,\lambda )\in \CCI \times \Lambda \mapsto f_{\lambda }(z)\in \CCI $ is 
holomorphic on $\CCI \times \Lambda .$ 
{\bf Throughout the paper, we always assume that $\Lambda $ is connected}. 
If ${\cal W}=\{ f_{\lambda }\} _{\lambda \in \Lambda }$ is a holomorphic family of rational maps 
and each $f_{\lambda }$ is a polynomial, then we say that ${\cal W}$ is a {\bf holomorphic family of 
polynomial maps}.  
We say that a holomorphic family ${\cal W}=\{ f_{\lambda }\} _{\lambda \in \Lambda }$ 
of rational maps is {\bf non-constant} if $\lambda \in \Lambda \mapsto f_{\lambda }\in \Rat $ 
is non-constant. 

For each $n\in \NN $, we set 
$$S_{n}({\cal W})=\{ z\in \CCI \mid (\lambda _{1},\ldots, \lambda _{n})\in \Lambda ^{n}
\mapsto f_{\lambda _{1}}\circ \cdots \circ f_{\lambda _{n}}(z)\mbox{ is constant on }\Lambda ^{n}\} .$$
Moreover, we set $S({\cal W}):=\cap _{n=1}^{\infty }S_{n}({\cal W}).$ 
Each point of $S({\cal W})$ is called a 
{\bf singular point of ${\cal W}$} and 
the set $S({\cal W})$ is called the 
{\bf singular set of ${\cal W}.$} 

\end{df}
\begin{lem}
\label{l:sn1sn}
Let ${\cal W}=\{ f_{\lambda }\} _{\lambda \in \Lambda }$ be a holomorphic family of 
rational maps. Then 
$S_{n+1}({\cal W})=\cap _{\lambda _{n+1}\in \Lambda }f_{\lambda _{n+1}}^{-1}(S_{n}({\cal W}))$ 
and  
$S({\cal W})=\cap _{n=1}^{\infty }\cap _{(\lambda _{1},\ldots ,\lambda _{n})\in 
\Lambda ^{n}}(f_{\lambda _{1}}\circ \cdots \circ f_{\lambda _{n}})^{-1}(S_{1}({\cal W})).$ 
Moreover, if, in addition to the assumption, ${\cal W}$ is non-constant, 
then $\sharp S_{1}({\cal W})<\infty $ and $\sharp S_{n}({\cal W})<\infty $ for each $n\in \NN .$ 
\end{lem}
\begin{proof}
We may assume that ${\cal W}$ is non-constant. 
We first show that $\sharp S_{1}({\cal W})<\infty .$ 
Suppose that $\sharp S_{1}({\cal W})=\infty .$ 
Then there exist a sequence $\{ z_{n}\} $ in $S_{1}({\cal W})$  and a point $z_{\infty }\in \CCI $ 
such that $z_{n}\rightarrow z_{\infty }$ and $z_{n}\neq z_{\infty }$ for each $n\in \NN .$ 
By conjugating the family ${\cal W}$ by an element of Aut$(\CCI )$, we may assume that 
$z_{\infty }\in \CC .$ Let $b\in \Lambda .$ 
Then there exist an open connected neighborhood $\Lambda _{0}$ of $b$ in $\Lambda $ 
and an open connected neighborhood $U$ of $z_{\infty }$ in $\CC $ such that 
$f_{\lambda }(z)\in \CC $ for all $\lambda \in \Lambda _{0}$ and all $z\in U.$ 
We may suppose that $\Lambda _{0}\subset \CC ^{r}$ where $r=\dim \Lambda \in \NN .$ 
Let $n\in \NN $, $(i_{1},\ldots, i_{n})\in (\{1,\ldots  ,r\} )^{n}$ and $z\in U.$ 
Let $g(z)=\frac{\partial ^{n}f_{\lambda }(z)}{\partial \lambda _{i_{1}}\cdots \partial \lambda _{i_{n}}}
|_{\lambda =b}$ for each $z\in U.$ 
Then $g:U\rightarrow \CC $ is holomorphic in $U$ and 
$g(z_{j})=0$ for each large $j.$ Hence $g(z)=0$ for all $z\in U.$ 
Therefore for each $z\in U$, the function $\lambda \mapsto f_{\lambda }(z)\in \CC $ is constant 
on $\Lambda _{0}.$ 
Thus, for each $z\in U$, the function $\lambda \mapsto f_{\lambda }(z)\in \CCI $ 
is constant on $\Lambda .$ 
Hence $U\subset S_{1}({\cal W}).$ 
Therefore 
\begin{equation}
\label{eq:zinfinint}
z_{\infty }\in \mbox{int}(S_{1}({\cal W})).
\end{equation}
In particular, int$(S_{1}({\cal W}))\neq \emptyset .$ 
We now suppose that $\CCI \neq \mbox{int}(S_{1}({\cal W})).$ 
Then $\partial (\mbox{int}(S_{1}({\cal W})))\neq \emptyset .$ 
If we take any $w_{0}\in \partial (\mbox{int}(S_{1}({\cal W})))$, then 
by the argument of the proof of (\ref{eq:zinfinint}), 
we obtain $w_{0}\in \mbox{int}(S_{1}({\cal W})).$ However, 
this contradicts $w_{0}\in \partial (\mbox{int}(S_{1}({\cal W}))).$ 
Therefore, we must have that $\CCI = \mbox{int}(S_{1}({\cal W})).$ 
Hence, the function $\lambda \mapsto f_{\lambda }\in \Rat $ is constant on $\Lambda .$ 
However, this contradicts to the assumption that ${\cal W}$ is non-constant. 
Thus, we have that $\sharp S_{1}({\cal W})<\infty .$ 

 It is easy to see that $S_{n+1}({\cal W})\subset \cap _{\lambda _{n+1}\in \Lambda }
 f_{\lambda _{n+1}}^{-1}(S_{n}({\cal W})).$ Since 
 $\sharp S_{1}({\cal W})<\infty ,$ it follows that 
 $\sharp S_{n}({\cal W})<\infty .$ We now prove 
 $\cap _{\lambda _{n+1}\in \Lambda }
 f_{\lambda _{n+1}}^{-1}(S_{n}({\cal W}))\subset S_{n+1}({\cal W}).$ 
 Let $z\in \cap _{\lambda _{n+1}\in \Lambda }
 f_{\lambda _{n+1}}^{-1}(S_{n}({\cal W})).$ 
 Then for each  $\lambda _{n+1}\in \Lambda $, we have  
$f_{\lambda _{n+1}}(z)\in S_{n}({\cal W}).$ Since $\sharp S_{n}({\cal W})<\infty $ 
and $\Lambda $ is connected, it follows that $\sharp \{ f_{\lambda _{n+1}}(z)\in S_{n}({\cal W}) 
\mid \lambda _{n+1}\in \Lambda \} =1.$ 
Therefore we obtain that the cardinality of the set 
$\{ f_{\lambda _{1}}\circ \cdots \circ f_{\lambda _{n+1}}(z)\mid 
(\lambda _{1}\ldots ,\lambda _{n+1})\in \Lambda ^{n+1}\} $ is equal to $1.$ 
In particular, $z\in S_{n+1}({\cal W}).$ Thus we have proved our lemma.  
\end{proof}
\begin{cor}
\label{c:falflam}
Let ${\cal W}=\{ f_{\lambda }\} _{\lambda \in \Lambda }$ be a holomorphic family of 
rational maps. Then $f_{\lambda }(S({\cal W}))\subset S({\cal W})$ for all $\lambda \in \Lambda .$ 
\end{cor}
We now define weakly nice subsets of Rat.
\begin{df}
\label{d:weaklynice}
%Let 
We say that a subset ${\cal Y}$ of $\Rat $ is {\bf weakly nice} 
(with respect to holomorphic families $\{ {\cal  W}_{j}\} _{j=1}^{m}$ of rational maps) 
if 
there exist an open subset ${\cal U}$ of $\Rat $ and 
finitely many non-constant holomorphic families 
${\cal W}_{j}=\{ f_{j,\lambda }\} _{\lambda \in \Lambda _{j}}, 
j=1,\ldots, m$, of rational maps such that for each $j=1,\ldots, m$, $\{ f_{j,\lambda }\mid \lambda \in \Lambda _{j}\} $ is a closed subset of ${\cal U}$ and 
${\cal Y}=\cup _{j=1}^{m}\{ f_{j,\lambda }\mid \lambda \in \Lambda _{j}\} .$

 Moreover, for a weakly nice set ${\cal Y}$ with respect to holomorphic families 
 $\{ {\cal W}_{j}\} _{j=1}^{m}$ of rational maps, we set 
 $${\frak M}_{1}({\cal Y},  \{ {\cal W}_{j}\} _{j=1}^{m})
 := \{ \tau \in {\frak M}_{1}({\cal Y})\mid \mbox{supp}\,\tau \cap  \{ f_{j,\lambda }\mid \lambda 
 \in \Lambda _{j}\} \neq \emptyset \ (\forall j=1,\ldots, m)\} $$ 
 and 
 $${\frak M}_{1,c}({\cal Y},  \{ {\cal W}_{j}\} _{j=1}^{m}):=
  {\frak M}_{1,c}({\cal Y})\cap {\cal M}_{1}({\cal Y},  \{ {\cal W}_{j}\} _{j=1}^{m}).$$  
  Here, for the notation ``supp$\,\tau$'', see Definition~\ref{d:d0} 
  (setting $Y={\cal Y}$). (Thus supp$\,\tau$ is a closed subset of ${\cal Y}.$)  
  Also, each point of $\cap _{j=1}^{m}S({\cal W}_{j})$ is called a {\bf singular point of 
  $({\cal Y}, \{ {\cal W}_{j}\} _{j=1}^{m})$} and 
  the set $\cap _{j=1}^{m}S({\cal W}_{j})$ is called the {\bf singular set of 
  $({\cal Y}, \{ {\cal W}_{j}\} _{j=1}^{m})$}.  
\end{df}
\begin{df}[\cite{Splms10, Sadv}] 
\label{d:topologyO}
Let ${\cal Y}$ be a closed subset of an open subset of $\Rat .$ 
Let ${\cal O}$ be the topology in ${\frak M}_{1,c}({\cal Y})$ such that 
the sequence $\{ \tau _{n}\} _{n=1}^{\infty } $ in ${\frak M}_{1,c}({\cal Y})$ 
tends to an element $\tau \in {\frak M}_{1,c}({\cal Y})$ 
with respect to the topology ${\cal O}$ if and only if 
(a) for each bounded $\varphi \in C({\cal Y})$, $\int \varphi d\tau _{n}\rightarrow 
\int \varphi d\tau $ as $n\rightarrow \infty $, and 
(b) supp$\,\tau_{n}\rightarrow \mbox{supp}\,\tau $ as $n\rightarrow \infty $ in 
Cpt$({\cal Y})$ with respect to the Hausdorff topology. 
We call ${\cal O}$ the wH-topology in ${\frak M}_{1,c}({\cal Y}).$ 
\end{df}
\begin{rem}
\label{r:Oandws}
In general the topology ${\cal O}$ is really different from 
the weak-$^{\ast }$ topology. 
For example, let ${\cal Y}=\{ z^{2}+c\mid c\in \CC \} 
\cong \CC $ and let $\tau _{n}=(1-\frac{1}{n})\delta _{0}
+\frac{1}{n}\delta _{1}$ for each $n\in \NN $, and let $\tau =\delta _{0}$, 
where $\delta _{z}$ denotes the Dirac measure concentrated 
at $z\in  \CC \cong {\cal Y}.$ 
Then $\tau _{n}\rightarrow \tau $ as $n\rightarrow \infty $ 
with respect to the weak-$^{\ast }$ topology,  
but $\tau _{n}\not \rightarrow \tau $ as $n\rightarrow \infty $ 
with respect to the topology ${\cal O}.$  

\end{rem}

By the definition of weakly nice subsets, it is easy to see the following lemma. 
\begin{lem}
\label{l:m1cywjo}
Let ${\cal Y}$ be a weakly nice subset of $\emRat$ with respect to some holomorphic families 
$\{ {\cal W}_{j}\} _{j=1}^{m}$ of rational maps.  
Then ${\frak M}_{1,c}({\cal Y}, \{ {\cal W}_{j}\} _{j=1}^{m})$ is closed in 
${\frak M}_{1,c}({\cal Y})$ with respect to the topology ${\cal O}.$ 
\end{lem}
The following lemma is easy to show but it is one of the keys to proving many results. 
\begin{lem}
\label{l:yrnistfgni}
Let ${\cal Y}$ be a weakly nice subset of $\emRat$ with respect to some 
holomorphic families $\{ {\cal  W}_{j}\} _{j=1}^{m}$ of rational maps. 
We endow ${\cal Y}$ with the relative topology from Rat. 
Let $\tau \in {\frak M}_{1}({\cal Y},\{ {\cal W}_{j}\} _{j=1}^{m}).$ 
Suppose that $\mbox{int}(\mbox{supp}\,\tau )\neq \emptyset $ with respect to 
the topology in ${\cal Y}$ and $F(G_{\tau })\neq \emptyset .$ 
Then $J_{\ker }(G_{\tau })\subset S({\cal W}_{j})$ for some $j=1,\ldots, m$ and 
$\sharp J_{\ker }(G_{\tau })<\infty .$   
\end{lem}
\begin{proof}
Let ${\cal W}_{j}=\{ f_{j,\lambda } \} _{\lambda \in \Lambda _{j}}$ for each $j.$ 
Then there exists an element $j\in \{ 1,\ldots, m\} $ such that 
$\mbox{int}(\mbox{supp}\,\tau )\cap \{ f_{j,\lambda }\mid \lambda \in \Lambda _{j}\} \neq \emptyset .$ 
Suppose $J_{\ker }(G_{\tau })\setminus S({\cal W}_{j})\neq \emptyset .$ 
Let $z_{0}\in  J_{\ker }(G_{\tau })\setminus S({\cal W}_{j}).$ 
Then there exists an element $n\in \NN $ such that the map  
$(\lambda _{1},\ldots, \lambda _{n})\in \Lambda _{j}^{n}\mapsto f_{j,\lambda _{1}}\circ \cdots 
\circ f_{j,\lambda _{n}}(z_{0})\in \CCI $ is non-constant on $\Lambda _{j}^{n}.$ 
Moreover, we have that $z_{0}\in J_{\ker }(G_{\tau })$ 
and $G_{\tau }(J_{\ker }(G_{\tau }))\subset J_{\ker }(G_{\tau }).$ 
Hence, combining the open mapping theorem for holomorphic mappings  and the assumption 
``int(supp$\,\tau)\neq \emptyset$'' 
implies 
%It implies 
that $\mbox{int}(J_{\ker }(G_{\tau }))\neq \emptyset .$ 
However, this contradicts to the assumption $F(G_{\tau })\neq \emptyset $ and 
%Montel's theorem. 
Remark~\ref{r:kjulia}(3). 
Thus we must have that 
$J_{\ker }(G_{\tau })\subset S({\cal W}_{j}).$ 
Since $\sharp (S_{n}({\cal W}_{j}))<\infty $ (see Lemma~\ref{l:sn1sn}), 
it follows that $\sharp J_{\ker }(G_{\tau })<\infty .$ 
\end{proof}
\begin{lem}
\label{l:yrtfaji}
Let ${\cal Y}$ be a weakly nice subset of $\emRat$ with respect to some 
holomorphic families $\{ {\cal  W}_{j}\} _{j=1}^{m}$ of rational maps, where ${\cal W}_{j}=\{ f_{j,\lambda }\} _{\lambda 
\in \Lambda _{j}}$, $j=1,\ldots ,m.$   
Let $\tau \in {\frak M}_{1}({\cal Y}, \{ {\cal W}_{j}\} _{j=1}^{m}).$ 
Suppose that for each $j=1,\ldots, m$, we have 
$\mbox{int}(\mbox{supp}\,\tau \cap \{ f_{j,\lambda }\mid \lambda \in \Lambda _{j}\} )\neq \emptyset $ with respect to 
the topology in $\{ f_{j,\lambda }\mid \lambda \in \Lambda _{j}\} $ (which is endowed with the relative topology from Rat),  and that 
$F(G_{\tau })\neq \emptyset .$ 
Then $J_{\ker }(G_{\tau })\subset \cap _{j=1}^{m}S({\cal W}_{j}).$    
\end{lem}
\begin{proof}
By using the argument in the proof of Lemma~\ref{l:yrnistfgni}, 
it is easy to see that our lemma holds. 
\end{proof}
\begin{lem}
\label{l:lcbwj}
Let ${\cal Y}$ be a weakly nice subset of $\emRat$ with respect to some holomorphic families 
$\{ {\cal W}_{j}\} _{j=1}^{m}$ of rational maps. 
Let $\tau \in {\frak M}_{1}({\cal Y}, \{ {\cal W}_{j}\} _{j=1}^{m}).$ 
Let $L\in \emMin(G_{\tau },\CCI )$ such that $L\subset \cap _{j=1}^{m}S({\cal W}_{j}).$ 
Then for each $\rho \in {\frak M}_{1}({\cal Y}, \{ {\cal W}_{j}\} _{j=1}^{m})$, we have 
$L\in \emMin (G_{\rho },\CCI ).$ 
\end{lem}
\begin{proof}
Let $z\in L.$ 
Let ${\cal W}_{j}=\{ f_{j,\lambda }\} _{\lambda \in \Lambda _{j}}$ for each $j.$  
Let $\rho \in  {\frak M}_{1}({\cal Y}, \{ {\cal W}_{j}\} _{j=1}^{m}). $ 
Let $h\in \mbox{supp}\,\rho. $ Then there exist an $i\in \{ 1,\ldots ,m\} $ 
and an element $\lambda _{0}\in \Lambda _{i} $ such that $h=f_{i,\lambda _{0}}.$ 
Since we have $\mbox{supp}\,\tau \cap \{ f_{i,\lambda }\mid \lambda \in \Lambda _{i}\} \neq \emptyset $, 
there exists an element $\lambda _{1}\in \Lambda _{i}$ such that 
$f_{i,\lambda _{1}}\in \mbox{supp}\,\tau .$ Since $L\in \Min(G_{\tau }, \CCI )$, 
we have $f_{i,\lambda _{1}}(z)\in L.$ Moreover, since $L\subset S({\cal W}_{i})$ (which follows from the assumption 
$L\subset \cap _{j=1}^{m}S({\cal W}_{j})$),  
we have that $h(z)=f_{i,\lambda _{0}}(z)=f_{i,\lambda _{1}}(z)\in L.$ 
Hence $h(L)\subset L.$ Therefore $L\in \Min(G_{\rho }, \CCI ).$  
\end{proof} 
\begin{df}
\label{d:smin}
Let ${\cal Y}$ be a weakly nice subset of $\Rat $ with respect to some holomorphic families 
$\{ {\cal W}_{j}\} _{j=1}^{m}$ of rational maps. 
Let $\tau \in {\frak M}_{1,c}({\cal Y}, \{ {\cal W}_{j}\} _{j=1}^{m})$. 
Then we set 
\vspace{-2mm} 
$$S_{\min }(\{ {\cal W}_{j}\} _{j=1}^{m})=\cup _{L\in \Min(G_{\tau },\CCI ), 
L\subset \cap _{j=1}^{m}S({\cal W}_{j})}L.$$ 
\vspace{-2mm} 
Note that this definition does not depend on the choice of 
$\tau \in {\frak M}_{1,c}({\cal Y}, \{ {\cal W}_{j}\} _{j=1}^{m})$ due to 
%Definition~\ref{d:smin}. 
Lemma~\ref{l:lcbwj}. 
\end{df}

We now give the definition of attracting minimal sets which was introduced by the author 
in \cite{Sadv}. 
\begin{df}
\label{d:attminset}
Let $\Gamma \in \mbox{Cpt}(\Rat).$ 
We say that a minimal set $L\in \Min(\langle \Gamma \rangle ,\CCI )$ 
%for $\Gamma $ 
is {\bf attracting} 
(for $\Gamma )$ if there exist two open subsets $A, B$ of $\CCI $ 
with $\sharp (\CCI \setminus A)\geq 3$ and an $n\in \NN $ such that 
$L\subset B\subset \overline{B}\subset A$ and such that for each $(\gamma _{1},\ldots, \gamma _{n})
\in \Gamma ^{n}$, we have $\gamma _{n}\circ \cdots \circ \gamma _{1}(A)\subset B.$ 
In this case, we say that $L$ is an 
{\bf attracting minimal set of $\G .$}  
Also, for an element $\tau \in {\frak M}_{1,c}(\Rat)$, 
if $L\in \Min(G_{\tau },\CCI )$ is attracting for $\mbox{supp}\,\tau$ then we say that 
$L$ is {\bf attracting for $\tau ,$} that $L$ is an {\bf attracting minimal set of $\supptau $}, and 
 that $L$ is an {\bf attracting minimal set of $\tau .$}  
\end{df}
\begin{df}
\label{d:mild} 
Let ${\cal Y}$ be a subset of $\Rat $ endowed with the relative topology from $\Rat. $ 
We say that ${\cal Y}$ is {\bf mild} 
if for each $\Gamma \in \Cpt({\cal Y})$, there exists an attracting minimal set of $\Gamma .$  
\end{df}
We give some examples of mild sets. 
\begin{ex}[Examples of mild sets]
\label{ex:mild} 
\ 
\begin{itemize} 
\item[(a)]  Any non-empty open subset ${\cal U}$ of 
${\cal P}$ is a mild set. For, for each $\Gamma \in \Cpt({\cal U})$, 
the set $\{\infty \}$ is an attracting  minimal set of $\Gamma .$ 
Also, for any $\Lambda \in \Cpt({\cal P})$, there exists an open subset 
${\cal V}$ of $\Rat$ with ${\cal V}\supset \Lambda $ such that ${\cal V}$ is mild.  
\item[(b)] Let $\Lambda \in \Cpt(\Rat)$ such that 
$\Lambda $ is an attracting   
minimal set of $\Lambda .$ 
%$(\langle \Lambda \rangle , \CCI )$ which is attracting for $\Lambda .$ 
Then there exists an open subset ${\cal U}$ of $\Rat $ with ${\cal U}\supset \Lambda $ 
such that ${\cal U}$ is mild.  
\item[(c)] Let $a\in \CCI $ and let ${\cal Y}=\{ f\in \Rat \mid a \mbox{ is an attracting fixed point of } f\}.$ 
Then ${\cal Y}$ is a mild subset  of Rat. 
\end{itemize} 
\end{ex}
We now give the definition of mean stability which is introduced by the author in 
\cite{Splms10}. 
\begin{df}
\label{d:meanstable}
Let $\Gamma \in \Cpt(\Rat).$ Let $G=\langle \G \rangle .$ We say that 
$\G $ is {\bf mean stable} if there exist non-empty open subsets 
$U$ and $V$ of $F(G)$ and a number $n\in \NN $ such that all of the following hold. 
\begin{itemize}
\item[(a)]
$\overline{V}\subset U$ and $\overline{U}\subset F(G).$ 
\item[(b)] 
For each $\gamma \in \GN , \gamma _{n,1}(\overline{U})\subset V.$ 
\item[(c)] 
For each $z\in \CCI $, there exists an element $g\in G$ such that 
$g(z)\in U.$ 
\end{itemize}
Also, if $\G $ is mean stable, we say that $G$ is mean stable (this notion does not 
depend on the choice of $\G \in \Cpt(\Rat)$ with $\langle \G \rangle =G$). 
Moreover, for an element $\tau \in {\frak M}_{1,c}(\Rat )$, if $\mbox{supp}\,\tau$ is 
mean stable, then we say that $\tau $ is mean stable. 
\end{df}
\begin{rem}
\label{r:ms} 
If $\tau \in {\frak M}_{1,c}(\Rat)$ is mean stable and 
$J(G_{\tau })\neq \emptyset $, then the random dynamical system 
generated by $\tau $ has many nice properties (e.g. $J_{\ker }(G_{\tau })=\emptyset$, 
stability of the limit state functions under the perturbation, 
negativity of Lyapunov exponent for any point of $z\in \CCI $ for $\tilde{\tau }$-a.e. 
$\gamma $ etc., see \cite{Splms10,Sadv}).   

Moreover, if $\Gamma \in \mbox{Cpt}(\Rat)$ and 
$\sharp J(\langle \Gamma \rangle )\geq 3$, then 
$\Gamma $ is mean stable if and only if 
each $L\in \Min(\langle \Gamma \rangle, \CCI )$ is attracting 
for $\Gamma $ (see \cite[Remark 3.7]{Sadv}). 
\end{rem}

We now give a result of the density of mean stable elements. 
Recall that an element $g\in \mbox{Aut}(\CCI )$ is called loxodromic 
if $g$ has exactly two fixed points $a,b\in \CCI $ and the modulus of multiplier 
of $(g,a)$ is strictly larger than $1$ and the modulus of multiplier of $(g,b)$ is 
strictly less than $1.$ 
\begin{lem}
\label{l:ypwntln}
Let ${\cal Y}$ be a mild subset of $\emRat$ and suppose that 
${\cal Y}$ is weakly nice  with respect to some  
holomorphic families $\{ {\cal W}_{j}\} _{j=1}^{m}$ 
of rational maps, 
where ${\cal W}_{j}=\{ f_{j,\lambda }\} _{\lambda \in \Lambda _{j}}$ 
for each $j=1,\ldots, m. $ 
%Suppose that ${\cal Y}\subset {\cal P}.$ 
Suppose that 
for each $\tau \in {\frak M}_{1,c}({\cal Y},\{ {\cal W}_{j}\} _{j=1}^{m})$ and 
for each $L\in \emMin(G_{\tau },\CCI )$, we have 
$L\cap (\cup _{j=1}^{m}S({\cal W}_{j}) \cap J(G_{\tau }))=\emptyset .$ 
Suppose also that 
%for each $\tau \in {\cal M}_{1,c}({\cal Y}, \{ {\cal W}_{j}\} _{j=1}^{m})$, 
for each $z\in S_{\min }(\{ {\cal W}_{j}\} _{j=1}^{m})$ 
and for each $j=1,\ldots, m$, 
either {\em (a)} the map $\lambda \mapsto 
D(f_{j, \lambda })_{z}$ is nonconstant on $\Lambda _{j}$, 
or {\em (b)} $D(f_{j, \lambda})_{z}=0$ for all $\lambda 
\in \Lambda _{j}$. 
Then ${\cal A}:=\{ \tau \in  {\frak M}_{1,c}({\cal Y},\{ {\cal W}_{j}\} _{j=1}^{m})\mid \tau 
\mbox{ is mean stable} \} $ is open and dense in 
$ {\frak M}_{1,c}({\cal Y},\{ {\cal W}_{j}\} _{j=1}^{m})$ with respect to the topology ${\cal O}.$ 

\end{lem}
 \begin{proof}
 By \cite[Lemma 3.62]{Splms10}, ${\cal A}$ is open in 
 $ {\frak M}_{1,c}({\cal Y},\{ {\cal W}_{j}\} _{j=1}^{m})$ with respect to the topology ${\cal O}.$ 
To prove the density of ${\cal A}$, 
let $\rho \in {\frak M}_{1,c}({\cal Y},\{ {\cal W}_{j}\} _{j=1}^{m}).$ 
Then there exists an element $\rho _{0}\in 
{\frak M}_{1,c}({\cal Y},\{ {\cal W}_{j}\} _{j=1}^{m})$ which is arbitrarily close to $\rho $ with 
respect to ${\cal O}$ such that for each $j\in \{ 1,\ldots, m\} $, 
int$(\mbox{supp}\,\rho _{0}\cap \{ f_{j,\lambda }\mid \lambda \in \Lambda _{j}\} )\neq \emptyset $ with respect to the topology 
in $\{ f_{j,\lambda }\mid \lambda \in \Lambda _{j}\} $, 
where ${\cal W}_{j}=\{ f_{j,\lambda }\} _{\lambda \in \Lambda _{j}}.$ 
For each $j$ we endow $\{ f_{j, \lambda }\mid \lambda \in \Lambda _{j}\}$ 
with the relative topology from Rat.  
By Lemma~\ref{l:yrtfaji} and the assumption of our lemma, 
we obtain $J_{\ker }(G_{\rho _{0}})=\emptyset $. 
Since ${\cal Y}$ is mild, each $g\in \mbox{supp}\,\rho _{0}\cap \mbox{Aut}(\CCI )$ is loxodromic.   
%By \cite[Theorem 1.8]{Sadv} and its proof, if we enlarge supp$\rho _{0}$ a little bit, and 
%take an element $\rho _{1}\in {\frak M}_{1,c}({\cal Y},\{ {\cal W}_{j}\} _{j=1}^{m})$ which is 
%close to $\rho _{0}$, then $\rho _{1}$ is mean stable. 
Let $\rho _{1} \in {\frak M}_{1,c}({\cal Y}, \{ {\cal W}_{j}\}_{j=1}^{m})$ be an element such that $\rho _{1}$ is close enough to $\rho _{0}$ with respect to the topology ${\cal O}$ and 
$\mbox{supp}\,\rho _{0}\cap \{ f_{j,\lambda }\mid \lambda 
\in \Lambda _{j}\} \subset \mbox{int}
(\mbox{supp}\rho _{1}\cap \{ f_{j,\lambda }\mid \lambda \in \Lambda _{j}\})$ with respect to the topology in 
$\{ f_{j, \lambda }\mid \lambda \in \Lambda _{j}\} $ for each $j.$ 
Then by the assumptions of our lemma, 
Lemma~\ref{l:sn1sn}, 
\cite[Lemmas 3.8, 3.16, and Theorem 3.26]{Sadv} and their proofs, we obtain that each $L\in \Min(G_{\rho _{1}}, \CCI )$ 
is attracting for $\rho _{0}$. From \cite[Remark 3.7]{Sadv}, it follows that 
$\rho _{1}$ is mean stable. 

Thus ${\cal A}$ is dense in ${\frak M}_{1,c}({\cal Y},\{ {\cal W}_{j}\} _{j=1}^{m}).$    
 \end{proof}
\begin{df}
\label{d:exceptional}
Let ${\cal Y}$ be a weakly nice subset of $\Rat $ with respect to some 
holomorphic families 
$\{ {\cal W}_{j}\} _{j=1}^{m}$ of rational maps. 
We say that ${\cal Y}$ is 
{\bf exceptional with respect to $\{ {\cal W}_{j}\} _{j=1}^{m}$}  if 
there exists a non-empty subset $L$ of $\cap _{j=1}^{m}S({\cal W}_{j})$ such that 
for each $\tau \in {\frak M}_{1,c}({\cal Y}, \{ {\cal W}_{j}\} _{j=1}^{m})$, 
we have that $L\in \Min(G_{\tau },\CCI )$ and $\chi (\tau, L)=0.$ 
We say that ${\cal Y}$ is {\bf non-exceptional with respect to $\{ {\cal W}_{j}\} _{j=1}^{m}$} if 
 ${\cal Y}$ is not exceptional with respect to $\{ {\cal W}_{j}\} _{j=1}^{m}$.  
\end{df}
\begin{prop}
\label{p:neadabc}
Let  ${\cal Y}$ be a  mild subset of $\emRat$ and suppose that 
${\cal Y}$ is weakly nice and non-exceptional with respect to some holomorphic families 
$\{ {\cal W}_{j}\} _{j=1}^{m}$ of rational maps. 
%Suppose that ${\cal Y}$ is non-exceptional with respect to 
%$\{ {\cal W}_{j}\} _{j=1}^{m}.$ 
%Suppose that  
%$({\cal Y}, \{ {\cal W}_{j}\} _{j=1}^{m})$ is not exceptional. 
Then 
there exists a dense subset ${\cal A} $ of 
the topological space $({\frak M}_{1,c}({\cal Y}, \{ {\cal W}_{j}\} _{j=1}^{m}), {\cal O})$ 
such that all of the following {\em (a)(b)} hold. 
\begin{itemize}
\item[{\em (a)}] 
For each $\tau \in {\cal A}$ and for each $L\in \emMin(G_{\tau },\CCI )$ with 
$L\subset \cap _{j=1}^{m}S({\cal W}_{j})$, we have $\chi (\tau, L)\neq 0.$

\item[{\em (b)}] 
Let $\tau \in {\cal A}.$ Then 
$\sharp J_{\ker }(G_{\tau })<\infty $ and  $J_{\ker }(G_{\tau })\subset \cap _{j=1}^{m}
S({\cal W}_{j})$.
Moreover, 
setting $H_{+}:=\{ L\in \emMin(G_{\tau }, J_{\ker }(G_{\tau }))\mid  \chi (\tau, L)>0\}$ and 
 denoting by $\Omega $ the set of 
 points $ y\in \CCI$ for which  
$\tilde{\tau }(\{ \gamma \in X_{\tau }\mid \exists n\in \NN \mbox{ s.t. }
\gamma _{n,1}(y)\in \cup _{L\in H_{+}}L\} )=0$, we have that  
%
%there exists a subset $\Omega $ of $\CCI $ with 
$\sharp (\CCI \setminus \Omega )\leq \aleph _{0}$ and  
for each $z\in \Omega $, 
$\tilde{\tau }(\{ \gamma \in (\mbox{supp}\,\tau)^{\NN }\mid z\in J_{\gamma }\} )=0.$
Moreover, for $\tilde{\tau }$-a.e. $\gamma \in (\emRat)^{\NN }$, 
we have {\em Leb}$_{2}(J_{\gamma })=0.$  
Furthermore, 
$J_{pt}^{0}(\tau )\subset \CCI \setminus \Omega $ and 
$\sharp J_{pt}^{0}(\tau )\leq \aleph _{0}.$

\end{itemize}
 \end{prop} 
\begin{proof}
For each $j=1,\ldots, m$, let ${\cal W}_{j}=
\{ f_{j,\lambda }\} _{\lambda \in \Lambda _{j}}$ and 
we endow $\{f_{j,\lambda }\mid 
\lambda \in \Lambda _{j}\} $ with the relative 
topology from Rat. 
%By Lemma~\ref{l:ypwntln}, \cite[Propositions 4.7, 4.8]{Splms10} and \cite[Remark 3.5]{Sadv}, 
 Suppose  that for each $\tau \in {\frak M}_{1,c}({\cal Y}, 
 \{ {\cal W}_{j}\}_{j=1}^{m})$ and for each $L\in 
 \Min(G_{\tau }, \CCI )$, we have $L\not\subset 
\cap _{j=1}^{m}S({\cal W}_{j})$. 
Let ${\cal A} $ be the set of elements 
$\rho \in {\frak M}_{1,c}({\cal Y}, \{ {\cal W}_{j}\} _{j=1}^{m})$ 
satisfying that int$(\supp\,\rho\cap 
\{ f_{j,\lambda }\mid \lambda \in \Lambda _{j}\})\neq 
\emptyset $ with respect to the topology in 
$\{ f_{j,\lambda }\mid \lambda \in \Lambda _{j}\}$ for 
all $j=1,\ldots, m.$ Then ${\cal A}$ 
is dense in ${\frak M}_{1,c}({\cal Y}, \{ {\cal W}_{j}\} _{j=1}^{m})$ 
and 
satisfies 
(a).  Let $\tau \in {\cal A}.$  By Lemma~\ref{l:yrtfaji},  
the fact that $G_{\tau }(J_{ker }(G_{\tau }))
\subset J_{\ker}(G_{\tau })$ and the above assumption,  
we have $J_{\ker }(G_{\tau })=\emptyset .$ 
From  \cite[Theorem 1.5]{Splms10}, it follows that Leb$_{2}(J_{\gamma })=0$ 
for $\tilde{\tau }$-a.e.$\gamma \in \mbox{Rat}^{\NN} $ and $J_{pt}^{0}(\tau )=\emptyset. $ Hence ${\cal A}$ satisfies (b). 

Thus 
 we may assume that 
there exist a $\tau \in {\frak M}_{1,c}({\cal Y}, \{ {\cal W}_{j}\})$ and 
an $L\in \Min(G_{\tau },\CCI )$ such that 
$L\subset \cap _{j=1}^{m}S({\cal W}_{j})).$ 
For such $L$, 
 Lemma~\ref{l:lcbwj} implies that  for each $\rho \in \MYW, $ we have 
$L\in \Min(G_{\rho },\CCI )$ and $L\subset \cap _{j=1}^{m}S({\cal W}_{j}).$ 
Let $$\{ L_{1},\ldots, L_{r}\} :=\{ K\subset \cap _{j=1}^{m}S({\cal W}_{j})\mid K\in 
\Min(G_{\rho },\CCI ) \mbox{ for each } \rho \in \MYW \} .$$ 
Here, note that the right hand side of the above is a finite set 
since $\cap _{j}S({\cal W}_{j})$ is a finite set 
(see Lemma~\ref{l:sn1sn}). 

Since $({\cal Y}, \{ {\cal W}_{j}\} _{j=1}^{m})$ is non-exceptional 
with respect to $\{ {\cal W}_{j}\} _{j=1}^{m}$, 
for each $k=1,\ldots, r$ there exists a $\tau _{k}\in \MYW $ such that 
$\chi (\tau _{k}, L_{k})\neq 0.$ 
Let ${\cal W}_{j}=\{ f_{j,\lambda }\mid \lambda \in \Lambda _{j}\}$ for each $j=1.\ldots, m.$ 
We consider the following two cases. \\ 
Case (I). For each  $k=1,\ldots, r$, for each $z\in L_{k}$ and for each 
$j=1,\ldots, m$, there exists a $\lambda \in \Lambda _{j}$ such that 
$D(f_{j,\lambda })_{z}\neq 0.$ \\ 
Case (II). There exist a $k\in \{ 1,\ldots, r\}$, a point $z\in L_{k}$ and an element $j\in \{ 1,\ldots, m\} $ 
such that for each $\lambda \in \Lambda _{j}$, $D(f_{j,\lambda })_{z}=0.$ \\ 
Suppose that  we have Case (I). We now prove the following claim. \\ 
Claim 1. For each $k$ 
there exists an element $\rho _{k}\in \MYW $ which is arbitrarily close to 
$\tau _{k}$ such that $\sharp \mbox{supp}\rho _{k}<\infty $, 
such that for each $g\in \mbox{supp}\rho _{k}$ and for each $z\in L_{k}$, we have 
$Dg_{z}\neq 0$, and such that  $\chi (\rho _{k}, L_{k})\neq 0.$ 

To prove this claim, for each $\tau \in \MYW $ and for each $L\in \Min(G_{\tau },\CCI )$, 
let $\mu _{\tau , L}$ be the  
canonical $\tau $-ergodic measure on $L$ (see Definition~\ref{d:celyap}).  
Let $k\in \{ 1,\ldots,r\} .$ 
We now consider the following two cases.\\ 
Case (I)(a). $\chi (\tau _{k},L_{k})\neq -\infty .$ 
Case (I)(b). $\chi (\tau _{k},L_{k})=-\infty .$\\ 
Suppose we have Case (I)(a).  
Let $B_{k}:=\{ g\in {\cal Y}\mid Dg_{z}=0 \mbox{ for some } z\in L_{k}\} .$ 
Since $\chi (\tau _{k}, L_{k})=\int _{L_{k}}\int _{{\cal Y}}\log \| Dg_{z}\| _{s}d\tau _{k}(g) 
d\mu _{\tau _{k},L_{k}}(z)$, we obtain that $\tau _{k}(B_{k})=0.$ 
Let $C_{k,n}$ be the set of elements $g\in {\cal Y}$ 
with $\kappa (g, B_{k})\geq 1/n. $ 
Then $\int _{L_{k}}\int _{C_{k,n}}\log \| Dg_{z}\| _{s}d\tau _{k}(g)d\mu _{\tau _{k},L_{k}}(z)
\rightarrow \chi (L_{k},\tau _{k})$ as $n\rightarrow \infty $ and 
$\frac{\tau _{k}|_{C_{k,n}}}{\tau _{k}(C_{k,n})}\rightarrow \tau _{k}$ as $n\rightarrow \infty $ 
in 
${\frak M}_{1.c}({\cal Y},{\cal O})$. Modifying $\frac{\tau _{k}|_{C_{k,n}}}{\tau _{k}(C_{k,n})}$, 
we obtain $\rho _{k}$ which is arbitrarily close to $\tau _{k}$ such that 
$\sharp \mbox{supp}\rho _{k}<\infty $, 
such that for each $g\in \mbox{supp}\rho _{k}$ and for each $z\in L_{k}$, we have 
$Dg_{z}\neq 0$, and such that  $\chi (\rho _{k}, L_{k})\neq 0.$ 

 We now suppose that we have Case (I)(b). 
 Let $\alpha _{n}(g,z)=\max \{ \log \| Dg_{z}\| _{s} , -n\} $ for each $n\in \NN .$ 
 Since $\chi (\tau _{k}, L_{k})=-\infty $, we have 
 $\int _{L_{k}}\int _{{\cal Y}}\alpha _{n}(g,z)d\tau _{k}(g)d\mu _{\tau _{k},L_{k}}(z)
 \rightarrow -\infty $ as 
 $n\rightarrow \infty $. 
 Hence for each $M<0$ there exists an $n\in \NN $ such that 
  $ \int _{L_{k}}\int _{{\cal Y}}\alpha _{n}(g,z)d\tau _{k}(g)d\mu _{\tau _{k},L_{k}}(z)<M$. 
  Therefore there exists a $\rho _{k}\in \MYW$ 
  which is arbitrarily close to $\tau _{k}$ such that 
  $\sharp \mbox{supp}\rho _{k}<\infty $, 
such that for each $g\in \mbox{supp}\rho _{k}$ and for each $z\in L_{k}$, we have 
$Dg_{z}\neq 0$, and such that 
$\int _{L_{k}}\int _{{\cal Y}}\alpha _{n}(g,z)d\rho _{k}(g)d\mu _{\rho _{k},L_{k}}(z)<\frac{M}{2}.$ 
 Hence $\chi (\rho _{k},L_{k})\leq \int _{L_{k}}\int _{{\cal Y}}\alpha _{n}(g,z)
 d\rho _{k}(g)d\mu _{\rho _{k},L_{k}}(z)<\frac{M}{2}.$ 
 Thus we have proved Claim 1. 
 
 For each $n\in \NN $, let 
 $$D_{k,n}:=\{ ((\lambda _{ji})_{i=1,\ldots, n})_{j=1,\ldots, m}\in \prod _{j=1}^{m}\Lambda _{j}^{n} \mid 
 D(f_{j,\lambda _{ji}})_{z}=0 \mbox{ for some }z\in L_{k}\} .$$ 
 Moreover, let 
 $$E_{k,n}:= \{ (p_{ij})_{i=1,\ldots ,n, j=1,\ldots ,m}\in (0,1)^{nm}\mid \sum _{i,j}p_{i,j}=1\} 
 \times ((\prod _{j=1}^{m}\Lambda _{j}^{n})\setminus D_{k,n})$$
 and let 
 $\alpha _{k,n}: E_{k,n}\rightarrow \RR $ be the function defined by 
 $\alpha _{k,n}((p_{ij}),(\lambda _{ji}))=\chi (\sum _{j=1}^{m}\sum _{i=1}^{n}p_{ij}\delta _{f_{j,\lambda _{ji}}}, L_{k}).$
 Then $(\prod _{j=1}^{m}\Lambda _{j}^{n})\setminus D_{k,n}$ is connected and 
 $\alpha _{k,n}:E_{k,n}\rightarrow \RR $ is real-analytic. 
 Hence claim 1 implies the following claim. \\ 
 Claim 2.  
There exists an $n_{0}\in \NN $ such that for each $n\in \NN $ with $n\geq n_{0}$, 
  the function $\alpha _{k,n}:E_{k,n}\rightarrow \RR $ is not identically equal to zero 
  in any open subset of $E_{k,n}.$  
  
  We now let $\zeta \in \MYW $ be an arbitrary element. 
  Then there exists an element $\zeta _{0}\in \MYW $ 
  arbitrarily close to $\zeta $ 
  such that $\sharp \mbox{supp}\, \zeta _{0} <\infty $ and 
  such that for each $g\in \mbox{supp}\,\zeta _{0}$, for each 
  $k$, and for each $z\in L_{k}$, we have $Dg_{z}\neq 0.$ 
  We may assume that 
  %$\zeta _{0}\in D_{n,k}$ for some $n\geq n_{0}.$ 
  for some $n\geq n_{0}$ there exists an element 
  $(((p_{ij})_{i=1,\ldots,n})_{j=1,\ldots, m}, ((\lambda _{ji})_{i=1,\ldots, n})_{j=1,\ldots, m})
  \in \cap _{k=1}^{r}E_{k,n}$ such that 
  $\zeta _{0}=\sum _{j=1}^{m}\sum _{i=1}^{n}p_{ij}\delta _{f_{j,\lambda _{ji}}}.$ 
  By claim 2, there exists a $\zeta _{1}$ close to $\zeta _{0}$ such that 
  for each $g\in \mbox{supp}\, \zeta _{1}$, for each $k$, and for each $z\in L_{k}$, 
  we have $Dg_{z}\neq 0$, and such that for each $k$, 
  $\chi (\zeta _{1},L_{k})\neq 0.$  
  By enlarging the support of $\zeta _{1}$, we obtain 
  an element $\zeta _{2}\in \MYW $ which is close to $\zeta _{1}$ 
  such that for each $g\in \mbox{supp}\, \zeta _{2}$, for each $k$, and for each $z\in L_{k}$, 
  we have $Dg_{z}\neq 0$, such that for each $k$, 
  $\chi (\zeta _{2},L_{k})\neq 0$, and 
  such that 
  for each $j=1,\ldots, m$, int$(\mbox{supp}\,\zeta _{2}\cap 
  \{ f_{j,\lambda }\mid \lambda \in {\cal W}_{j}\})\neq \emptyset $ 
  in the space $\{ f_{j,\lambda }\mid \lambda \in {\cal W}_{j}\} $ which is endowed with the relative topology from Rat. 
  By Lemma~\ref{l:yrtfaji}, we obtain that $J_{\ker }(G_{\zeta _{2}})\subset \cap 
  _{j=1}^{m}S({\cal W}_{j}).$ In particular, $\sharp J_{\ker }(G_{\zeta _{2}})<\infty $ 
  by Lemma~\ref{l:sn1sn}. 
By Theorem~\ref{t:jkfcln0}, 
denoting by $H_{+}$ the set of elements 
$L\in \Min(G_{\zeta _{2}}, J_{\ker }(G_{\zeta _{2} }))$ 
with $\chi (\zeta _{2}, L)>0$ and 
 denoting by $\Omega $ the set of elements 
 $y\in \CCI $ for which  
$\tilde{\zeta }_{2} (\{ \gamma \in X_{\zeta _{2}}\mid \exists n\in \NN \mbox{ s.t. }
\gamma _{n,1}(y)\in \cup _{L\in H_{+}}L\} )=0$,   
we have that 
%there exists a subset $\Omega $ of $\CCI $ with 
$\sharp (\CCI \setminus \Omega )\leq \aleph _{0}$ 
%such that 
and for each $z\in \Omega $, 
$\tilde{\zeta }_{2}(\{ \gamma \in X_{\zeta _{2} }\mid z\in J_{\gamma }\} )=0.$
Moreover, 
for $\tilde{\zeta }_{2}$-a.e.$\gamma \in (\Rat)^{\NN}$, 
$\mbox{Leb}_{2}(J_{\gamma })=0.$ Furthermore, 
$J_{pt}^{0}(\tau )\subset \CCI \setminus \Omega $ 
and $\sharp J_{pt}^{0}(\tau )\leq \aleph _{0}.$   
%$\mbox{Leb}_{2}(J_{pt}^{0}(\zeta _{2}))=0.$  

We now suppose that we have Case (II). 
Let 
$$I:=\{ k\in \{ 1,\ldots, r\} \mid \exists z\in L_{k} \ \exists j\in \{ 1,\ldots, m\} 
\mbox{ such that for each }\lambda \in \Lambda _{j}, D(f_{j,\lambda })_{z}=0\} .$$
%$k\in \{ 1,\ldots, r\} $ and suppose that 
%there exist a $z\in L_{k}$ and an element $j\in \{ 1,\ldots, m\} $ 
%such that for each $\lambda \in \Lambda _{j}$, $D(f_{j,\lambda })_{z}=0.$ 
We modify the argument in Case (I). 
Namely, we can choose $\zeta _{1}$ and $\zeta _{2}$ in 
the argument of Case (I) so that $\chi (\zeta _{1}, L_{k})=\chi (\zeta _{2}, L_{k})=-\infty $ 
for any $k\in I.$ For any $k\not\in I$, we use the same argument in that of Case (I). 
Thus we have proved our proposition. 
  \end{proof} 
\begin{lem}
\label{l:wiwjnod}
Under the assumptions of Proposition~\ref{p:neadabc}, 
%suppose that $\{ f_{i,\lambda }\mid \lambda \in \Lambda _{i}\} 
%\cap \{ f_{j,\lambda }\mid \lambda \in \Lambda _{j}\} $ is 
%nowhere dense in $\{ f_{i,\lambda }\mid \lambda \in \Lambda _{i}\}$ 
%for all $i, j$ with $i\neq j$, where 
%${\cal W}_{j}=\{ f_{j,\lambda }\} _{\lambda \in \Lambda _{j}}$ for all $j.$ 
%Then  
there exists an open dense subset ${\cal A} $ of 
the topological space $({\frak M}_{1,c}({\cal Y}, \{ {\cal W}_{j}\} _{j=1}^{m}), {\cal O})$ 
such that  all of the following hold.
\begin{itemize}
\item[{\em (i)}] 
For each $\tau \in {\cal A}$ and for each $L\in \emMin(G_{\tau },\CCI )$ with 
$L\subset \cap _{j=1}^{m}S({\cal W}_{j})$, we have $\chi (\tau, L)\neq 0.$ 
\item[{\em (ii)}] 
For each $\tau \in {\cal A}$ and for each $L\in \emMin(G_{\tau },\CCI )$ with 
$L\subset \cap _{j=1}^{m}S({\cal W}_{j})$, if $\chi (\tau, L)>0$, then 
for each $z\in L$ and for each $g\in G_{\tau }$, we have $Dg_{z}\neq 0.$ 

\end{itemize}
  \end{lem} 
 \begin{proof}
 %Let ${\cal A}_{0}:= \{ \tau \in \MYW \mid \mbox{supp}\,\tau \cap \cup _{(i,j):i\neq j}
 %{\cal W}'_{i}\cap {\cal W}'_{j}=\emptyset \} $, 
 %where ${\cal W}'_{j}:=\{ f_{j,\lambda }\mid \lambda \in \Lambda _{j}\}$ for all $j.$  
 %Since $\{ f_{i,\lambda }\mid \lambda \in \Lambda _{i}\} 
%\cap \{ f_{j,\lambda }\mid \lambda \in \Lambda _{j}\} $ is 
%nowhere dense in $\{ f_{i,\lambda }\mid \lambda \in \Lambda _{i}\}$ 
%for all $i, j$ with $i\neq j$, we obtain that 
%${\cal A}_{0}$ is open and dense in $\MYW $. 
%Combining this with the argument in the proof of Proposition~\ref{p:neadabc}, 
%we easily see that the statement of our lemma holds. 
Let ${\cal W}_{j}=\{ f_{j,\lambda }\} _{\lambda \in \Lambda _{j}}$ for all $j.$  
We use the arguments in the proof of Proposition~\ref{p:neadabc}. 
%As in the proof of Proposition~\ref{p:neadabc}, 
We may assume that 
there exists a $\tau \in {\frak M}_{1,c}({\cal Y}, \{ {\cal W}_{j}\} _{j=1}^{m})$ and 
an $L\in \Min(G_{\tau },\CCI )$ such that $L\subset \cap _{j=1}^{m}S({\cal W}_{j}).$ 
%\cap J(G_{\tau }).$ 
Let $L_{1},\ldots, L_{r}$ be as in the proof of Proposition~\ref{p:neadabc}. 
Let $\zeta \in {\frak M}_{1,c}({\cal Y}, \{ {\cal W}_{j}\} _{j=1}^{m})$. 
Let $\zeta _{0}\in  {\frak M}_{1,c}({\cal Y}, \{ {\cal W}_{j}\} _{j=1}^{m})$ 
with $\sharp \supp\,\zeta _{0}<\infty $ which is arbirarily close to $\zeta $ with respect to ${\cal O}.$  
We classify the elements $k$ of $ \{1,\ldots, r\} $ into the following two types (I) and (II). 

Type (I). There exist an element $i=1,\ldots, m$ and an element $z_{0}\in L_{k}$ 
such that 
%$\G_{\zeta _{0}} \cap \{ f_{i,\lambda }\mid \lambda \in \Lambda _{i}\} 
%\neq \emptyset $ and  
$D(f_{i,\lambda })_{z_{0}}=0$ for all $\lambda \in \Lambda _{i}.$

 Type (II). Not type (I). 
  
  Note that if $k$ is of type (I), then $\chi (\zeta _{0}, L_{k})=-\infty .$ 
  Note also that if $k$ is of type (II), then perturbing $\zeta _{0}$ if necessary, 
  we may assume that for each $g\in \supp\,\zeta _{0}$ and for each 
  $z\in L_{k}$, we have $Dg_{z}\neq 0.$ 
  Therefore, by using the arguments in the proof of Proposition~\ref{p:neadabc}, 
  we can take $\zeta _{1}\in {\frak M}_{1,c}({\cal Y}, 
  \{ {\cal W}_{j}\} _{j=1}^{m})$ with $\sharp \supp\,\zeta _{1}<\infty $ 
  which is arbitrarily close to $\zeta _{0}$ such that the following hold. 
\begin{itemize}
\item[(a)] $\chi (\zeta _{1},L_{k})=-\infty $ for any $k$ of type (I). 
\item[(b)] For any $k$ of type (II),  for any $z\in L_{k}$ and 
  for any $g\in \supp\,\zeta _{1}$, we have $Dg_{z}\neq 0$. 
  \item[(c)] For any $k$ of type (II) and  for any $z\in L_{k}$, we have 
  $\chi (\zeta _{1},L_{k})\neq 0.$
  \end{itemize}  
  Hence for any $\zeta _{2}\in   {\frak M}_{1,c}({\cal Y}, \{ {\cal W}_{j}\} _{j=1}^{m}) $ 
  which is close enough to $\zeta _{1}$, we have the following. 
\begin{itemize}
\item[(a)']    $\chi (\zeta _{2}, L_{k})<0$ for any $k$ of type (I).     
 
 \item[(b)'] For any $k$ of type (II),  for any $z\in L_{k}$ and 
  for any $g\in \supp\,\zeta _{2}$, we have $Dg_{z}\neq 0$. 
 \item[(c)'] For any $k$ of type (II) and  for any $z\in L_{k}$, we have 
  $\chi (\zeta _{2},L_{k})\neq 0.$
 \end{itemize} 
 Thus we have proved our lemma. 
    \end{proof}
 \begin{df}
 For a topological space $X$, we denote by 
 Con$(X)$ the set of connected components of $X$.
 \end{df}
 \begin{df}
 \label{d:unitary}
 Let $\tau \in {\frak M}_{1}(\Rat).$ For an element $L\in \Min(G_{\tau },\CCI )$, 
 we denote by $U_{\tau ,L}$ the space of all finite linear combinations of 
 unitary eigenfunctions of $M_{\tau }:C(L)\rightarrow C(L)$, where we say that 
 an element $\varphi \in C(L)\setminus \{ 0\} $ is a unitary eigenfunction of $M_{\tau }:C(L)\rightarrow 
 C(L)$ if there exists an element $\alpha \in \CC $ with $|\alpha |=1$ such that 
 $M_{\tau }(\varphi )=\alpha \varphi $ in $L.$  Also, 
we say that an element $\alpha \in \CC $ with $|\alpha |=1$ is 
a unitary eigenvalue of $M_{\tau }:C(L)\rightarrow C(L)$ if 
there exists an element $\varphi \in C(L)\setminus \{ 0\} $ such that 
$M_{\tau }(\varphi )=\alpha \varphi .$ Moreover, 
we denote by  $U_{\tau, L,\ast }$ the set of unitary eigenvalues of 
$M_{\tau }:C(L)\rightarrow C(L).$ 
 \end{df}
 \begin{df}
 \label{d:limitfunction}
 Let $U$ be an open subset of $\CCI $ and let $\{ \varphi _{n}:U\rightarrow \CCI \} _{n=1}^{\infty }$ 
 be a sequence of holomorphic maps from $U$ to $\CCI .$ We say that 
 a map $\psi :U\rightarrow \CCI $ is a 
{\bf limit function of $\{ \varphi _{n}\} _{n=1}^{\infty }$} 
 if there exists a subsequence $\{ \varphi _{n_{j}}\} _{j=1}^{\infty }$ of 
 $\{ \varphi _{n}\}_{n=1}^{\infty }$ such that $\varphi _{n_{j}}\rightarrow \psi $ 
 as $j\rightarrow \infty $ locally uniformly on $U.$ 
 \end{df}
 
 The following lemma is very important to analyze the random dynamical system generated by 
 $\tau \in {\frak M}_{1,c}(\Rat)$ with $\sharp J_{\ker }(G_{\tau })<\infty .$ 
 The proof is based on careful observations of limit functions on Fatou components of  
 $G_{\tau }$ by using the hyperbolic metrics on the Fatou components of $G_{\tau }.$

\begin{lem}
\label{l:pwjkf}
Let $\tau \in {\frak M}_{1,c}(\emRat )$ and suppose 
$\sharp J(G_{\tau })\geq 3.$ 
Let $L\in \emMin(G,\CCI )$ with $L\cap F(G_{\tau })\neq \emptyset .$ 
Let $\Omega _{L}:=\cup _{U\in \mbox{{\em Con}}(F(G_{\tau })),U\cap L\neq \emptyset }U.$ 
Suppose that $\sharp ((\partial \Omega _{L})\cap J_{\ker }(G_{\tau }))<\infty .$ 
Then we have the following {\em (I)(II)(III)}. 
\begin{itemize}
\item[{\em (I)}] 
There exists a Borel subset ${\cal A}$ of $X_{\tau }$ with 
$\tilde{\tau }({\cal A})=1$ such that 
for each $\gamma =(\gamma _{1},\gamma _{2},\ldots )\in {\cal A}$ 
and 
for each $z\in \Omega _{L}$, 
there exists a $\delta =\delta (z,\gamma )>0$ satisfying that 
$d(\gamma _{n,1}(z),L)\rightarrow 0$ and 
diam$(\gamma _{n,1}(B(z,\delta )))$ $\rightarrow 0$ as 
$n\rightarrow \infty .$

\item[{\em (II)}] 
We have 
$C(L)=U_{\tau,L}\oplus \{ \varphi \in C(L)\mid M_{\tau }^{n}(\varphi )\rightarrow 0 \mbox{ as }
n\rightarrow \infty \} $ in the Banach space $C(L)$ endowed with 
the supremum norm and $\dim _{\CC }U_{\tau ,L}<\infty .$ 
Moreover, setting $r_{L}:=\dim _{\CC }U_{\tau, L}$, 
we have $\sharp \emMin(G_{\tau }^{r_{L}},L)=r_{L}$. 
Also, there exist $ L_{1},\ldots, L_{r_{L}}\in \emMin (G_{\tau }^{r_{L}},L)$
such that $\{ L_{j}\mid j=1,\ldots, r_{L}\}=\emMin(G_{\tau }^{r_{L}},L)$, 
$L=\cup _{j=1}^{r_{L}}L_{j}$ and 
$h(L_{j})=L_{j+1}$ for each $h\in \mbox{supp}\,\tau$, where 
$L_{r_{L}+1}:=L_{1}.$ Moreover, 
for each $j=1,\ldots, r_{L}$, 
there exists a unique element 
$\omega _{L,j}\in {\frak M}_{1}(L_{j})$ such that  $(M_{\tau }^{r_{L}})^{\ast }(\omega _{L,j})=\omega _{L,j}.$ 
Also, for each $j=1,\ldots, r_{L}$, we have 
$M_{\tau }^{nr_{L}}(\varphi )\rightarrow (\int \varphi \ d\omega _{L,j})\cdot 1_{L_{j}}$ in 
the Banach space $C(L_{j})$ endowed with the supremum norm as $n\rightarrow \infty $
for each $\varphi \in C(L_{j})$, supp$\,\omega _{L,j}=L_{j}$ and 
$M_{\tau }^{\ast }(\omega _{L,j})=\omega _{L,j+1}$ in ${\frak M} _{1}(L)$ where 
$\omega _{L,r_{L}+1}=\omega _{L,1}.$ 
Also, we have $U_{\tau ,L,\ast }=
\{ \alpha \in \CC \mid \alpha ^{r_{L}}=1\}$ and 
for each $\alpha \in U_{\tau ,L,\ast }$, we have 
$\dim _{\CC }\{ \varphi \in C(L)\mid M_{\tau }\varphi =\alpha \varphi \} =1.$ 
\item[{\em (III)}]
The function $T_{L,\tau }:\CCI \rightarrow [0,1]$ of probability of tending to $L$ is 
locally constant on $F(G_{\tau }).$  
Here, we set 
$T_{L, \tau }(z)=
\tilde{\tau }(\{ \gamma =(\gamma _{1}, \gamma _{2},\ldots )
\in X_{\tau}\mid d(\gamma _{n,1}(z), L)\rightarrow 0 \ 
\mbox{ as } n\rightarrow \infty \} )$ for each $z\in \hat{\Bbb{C}}.$ 
\end{itemize}
\end{lem}
\begin{proof}
Let $\Omega =\Omega _{L}.$ 
Let $U\in \mbox{Con}(\Omega)$.  
Let $a\in U\cap L.$ To prove item (I), it suffices to prove that 
there exists a Borel subset ${\cal A}_{a}$ of $X_{\tau }$ with $\tilde{\tau }({\cal A}_{a})=1$ 
such that for each $\gamma \in {\cal A}_{a}$, 
each limit function of the sequence $\{ \gamma _{n,1}\} _{n=1}^{\infty }$ around $a$ is constant. 
Since $L\cap F(G_{\tau })\neq \emptyset $, we have $L\cap J_{\ker }(G_{\tau })=\emptyset .$ 
Hence there exists a $\delta >0$ such that 
\begin{equation}
\label{eq:bjkgtdc}
B(J_{\ker }(G_{\tau }),\delta )\cap \overline{G_{\tau }(B(a,\delta ))}=\emptyset .
\end{equation}
Since we are assuming $\sharp (J(G_{\tau }))\geq 3$, we have that 
$J(G_{\tau })$ is perfect. Since we are assuming 
$\sharp (\partial \Omega \cap J_{\ker }(G_{\tau }))<\infty $, taking $\delta $ so small, 
we may assume that 
\begin{equation}
\label{eq:cbpocj}
J(G_{\tau }) \setminus B((\partial \Omega )\cap J_{\ker }(G_{\tau }),\delta )\neq \emptyset .
\end{equation}
Here, if $(\partial \Omega)\cap J_{\ker }(G_{\tau })=\emptyset $, then we set $B((\partial \Omega )\cap 
J_{\ker }(G_{\tau }, \delta )=\emptyset .$ 
By (\ref{eq:bjkgtdc}) and (\ref{eq:cbpocj}), taking $\delta $ so small, 
we may assume that for each $U_{0}\in \mbox{Con}(F(G_{\tau }))$ with $L\cap U_{0}\neq \emptyset $, 
we have 
\begin{equation}
(\partial U_{0})\setminus B((\partial \Omega )\cap J_{\ker }(G_{\tau }),\delta )\neq \emptyset .
\end{equation}

For each $z\in (\partial \Omega)\setminus B((\partial \Omega )\cap J_{\ker }(G_{\tau }),\delta ) $, 
there exist an element $g_{z}\in G_{\tau }$ and an open disk neighborhood $V_{z}$ of $z$ in $\CCI $ 
such that $g_{z}(\overline{V_{z}})\subset F(G_{\tau }).$  
Since $(\partial \Omega )\setminus (B((\partial \Omega)\cap J_{\ker }(G_{\tau }),\delta )$ is compact, 
there exists a finite set $\{ z_{1},\ldots, z_{p}\} $ in $(\partial \Omega  )\setminus 
B((\partial \Omega )\cap J_{\ker }(G_{\tau }),\delta )$ such that 
\begin{equation}
\label{eq:cj1pvzj}
\cup _{j=1}^{p}V_{z_{j}}\supset (\partial \Omega )\setminus B((\partial \Omega )\cap 
J_{\ker }(G_{\tau }),\delta ) \mbox{ and } g_{z_{j}}(\overline{V_{z_{j}}})\subset F(G_{\tau }) \mbox{ for each }j. 
\end{equation}
For each $j=1,\ldots, p$ there exists an element 
$\alpha ^{j}=(\alpha _{1}^{j},\ldots, \alpha _{k(j)}^{j})\in 
(\mbox{supp}\,\tau)^{k(j)}$ for some $k(j)\in \NN $ 
such that $g_{z_{j}}=\alpha _{k(j)}^{j}\circ \cdots \circ \alpha _{1}^{j}.$ 
Since $G_{\tau }(F(G_{\tau }))\subset F(G_{\tau })$, we may assume that 
there exists a $k\in \NN $ such that for each $j=1,\ldots, p$, we have $k(j)=k.$ 
For each $j=1,\ldots, p$, let $W_{j}$ be a compact neighborhood of $\alpha ^{j}$ in $(\mbox{supp}\,\tau)^{k}$ 
such that for each $\beta =(\beta _{1},\ldots, \beta _{k})\in W_{j}$, 
we have $\beta _{k}\circ \cdots \circ \beta _{1}(\overline{V_{z_{j}}})\subset F(G_{\tau }).$ 
Also, for each $j=1,\ldots, p$, let 
$B_{j}:=\cup _{B\in \mbox{Con}(\Omega ),B\cap V_{z_{j}}\neq \emptyset }B.$ 
Let $n\in \NN $ and let 
$c_{q}=1/q$ for each $q\in \NN .$   
Let $(i_{1},\ldots ,i_{l})$ be a finite sequence of positive integers with 
$i_{1}< \cdots <i_{l}.$ Let $q>0.$   
We denote by $A_{q,j}(i_{1},\ldots ,i_{l})$ the set of elements 
$\gamma \in X_{\tau }$ which satisfies all of the following (a) and (b). 
\begin{itemize}
\item[(a)]
$\gamma _{kt,1}(a)\in (\CCI \setminus B(\partial \Omega), c_{q}))\cap B_{j}$ if 
$t\in \{ i_{1},\ldots ,i_{l}\}.$  
\item[(b)]
$\gamma _{kt,1}(a)\not\in (\CCI \setminus B(\partial \Omega, c_{q}))\cap B_{j}$ if 
$t\in \{ 1,\ldots ,i_{l}\} \setminus \{ i_{1},\ldots ,i_{l}\}.$  
%\item[(c)]
%$(\gamma _{ki_{s}+1}, \ldots ,\gamma _{ki_{s}+k})\not\in V_{j}$ for each $s=n,n+1,\ldots ,l.$ 
\end{itemize}
Moreover, when $l\geq n$,  
we denote by $B_{q,j,n}(i_{1},\ldots ,i_{l})$ the set of elements $\gamma \in X_{\tau }$ 
which satisfies items (a) and (b) above and the following (c). 
\begin{itemize}
\item[(c)]
$(\gamma _{ki_{s}+1}, \ldots ,\gamma _{ki_{s}+k})\not\in W_{j}$ for each $s=n,n+1,\ldots ,l.$
\end{itemize}
 Furthermore, we denote by   
$C_{q,j,n}(i_{1},\ldots ,i_{l})$ the set of elements $\gamma \in X_{\tau }$ 
which satisfies items (a) and (b) above and the following (d). 
 \begin{itemize}
\item[(d)]
$(\gamma _{ki_{s}+1}, \ldots ,\gamma _{ki_{s}+k})\not\in W_{j}$ for each $s=n,n+1,\ldots ,l-1.$
\end{itemize}
Furthermore, 
for each $q$, $j$, $n$, $l$ with $l\geq n$, 
let $B_{q,j,n,l}:= \bigcup _{i_{1}<\cdots <i_{l}}B_{q,j,n}(i_{1},\ldots ,i_{l}).$ 
Let ${\cal D}:= \bigcup _{q=1}^{\infty }\bigcup _{j=1}^{p}\bigcup _{n\in \NN }\bigcap _{l\geq n}B_{q,j,n,l}.$ 
We show the following claim. 

Claim 1. Let $\gamma \in X_{\tau }$ be such that there exists a non-constant limit function of the sequence  
$\{ \gamma _{n,1}|_{U}:U\rightarrow \CCI \} _{n=1}^{\infty }$. Then $\gamma \in {\cal D}.$ 

To show this claim, let $\gamma \in X_{\tau }$ be 
an element such that there exists a non-constant limit function of 
$\{ \gamma _{n,1}|_{U}: U\rightarrow \CCI \} _{n=1}^{\infty }.$ Then there exists a $q\in \NN $, a $j\in \{ 1,\ldots ,p\} $,  
and a strictly increasing sequence $\{ i_{l}\} _{l=1}^{\infty }$ in $\NN $
such that $\gamma \in \bigcap _{l=1}^{\infty }A_{q,j}(i_{1},\ldots ,i_{l})$ and 
any subsequence of $\{ \gamma _{ki_{l},1}|_{U}:U\rightarrow \CCI \} _{l=1}^{\infty }$ 
does not converge to a constant map. 
 Suppose that there exists 
a strictly increasing sequence $\{ l_{p}\} _{p=1}^{\infty }$ in $\NN $ such that 
for each $p\in \NN $, 
$(\gamma _{ki_{l_{p}}+1},\ldots ,\gamma _{ki_{l_{p}}+k})\in W_{j}.$ 
Since $\sharp J(G_{\tau })\geq 3$, for each $A\in \mbox{Con}(F(G_{\tau }))$, we can take the hyperbolic metric on $A.$ 
From the definition of $W_{j}$ and \cite[Pick Theorem]{Mi}, we obtain that there exists a constant 
$0<\alpha <1 $ such that 
for each $p\in \NN $ and for each $a'$ in a small 
neighborhood $U_{a}$ of $a$, we have  
$\| (\gamma _{ ki_{l_{p}}+k}\cdots \gamma _{ki_{l_{p}+1}})'(\g _{ki_{l_{p}},1}(a'))\| _{h}\leq \alpha $, 
where for each $g\in G_{\tau }$ and for each $z\in F(G_{\tau })$, $\| g'(z) \| _{h}$ denotes the norm of the derivative of $g$ at $z$ 
measured from the  
hyperbolic metric on the element of $\mbox{Con}(F(G_{\tau }))$ containing 
%$\gamma _{ki_{l_{p}},1}(a)$ 
$z$
to that on the element of $\mbox{Con}(F(G_{\tau }))$ containing 
%$\gamma _{ki_{l_{p}}+k,1}(a)$. 
$g(z).$ 
Hence, for each $a'\in U_{a}$, 
$\| (\gamma _{ki_{l_{p}},1})'(a')\| _{h}\rightarrow 0$ as $p\rightarrow \infty $. 
However, this is a contradiction, since 
$\{ \gamma _{ki_{l_{p}},1}|_{U}\} _{p=1}^{\infty }$ does not converge to a constant map. 
Therefore, $\gamma \in {\cal D}.$ Thus, we have proved claim 1. 

 Let $\eta := \max _{j=1}^{p}(\otimes _{s=1}^{k}\tau )
 ((\mbox{supp}\,\tau)^{k}\setminus W_{j})\ (<1).$  
Then we have for each $(l,n)$ with $l\geq n,$  
\begin{align*}
\tilde{\tau }(B_{q,j,n}(i_{1},\ldots ,i_{l+1}))
& \leq \tilde{\tau }(C_{q,j,n}(i_{1},\ldots ,i_{l+1})\cap 
\{ \gamma \in X_{\tau }\mid (\gamma _{ki_{l+1}+1},\ldots ,\gamma _{ki_{l+1}+k})\not\in W_{j}\} )\\ 
& \leq \tilde{\tau }(C_{q,j,n}(i_{1},\ldots ,i_{l+1}))\cdot \eta .  
\end{align*}
Hence, for each $l$ with $l\geq n$, 
\begin{align*}
\tilde{\tau }(B_{q,j,n,l+1})
= & \tilde{\tau }(\bigcup _{i_{1}<\cdots <i_{l+1}}B_{q,j,n}(i_{1},\ldots ,i_{l+1})) 
=  \sum _{i_{1}<\cdots <i_{l+1}}\tilde{\tau }(B_{q,j,n}(i_{1},\ldots ,i_{l+1}))\\ 
\leq & \sum _{i_{1}<\cdots <i_{l+1}}\eta \tilde{\tau }(C_{q,j,n}(i_{1},\ldots ,i_{l+1})) 
=  \eta \tilde{\tau }(\bigcup _{i_{1}<\cdots <i_{l+1}}C_{q,j,n}(i_{1},\ldots ,i_{l+1}))\leq 
\eta \tilde{\tau }(B_{q,j,n,l}).
\end{align*}
%Let $\{ c_{i}\} _{i=1}^{\infty }$ be a decreasing seqence of postive numbers which converges to $0.$ 
Therefore $\tilde{\tau }({\cal D})\leq \sum _{q=1}^{\infty }\sum _{j=1}^{p}\sum _{n\in \NN }
\tilde{\tau }(\bigcap _{l\geq n}B_{q,j,n,l})=0.$ 
Hence we have proved item (I) of our lemma. 

We now prove item (II). Since $L\cap J_{\ker }(G_{\tau })=\emptyset $, 
Lemmas~\ref{l:ttaugy0y} and \ref{l:voygvv} imply that $L\subset F_{pt}^{0}(\tau ).$ 
Thus we obtain that for each $\varphi \in C(L)$, 
$\{ M_{\tau }^{n}(\varphi )|_{L}\} _{n=1}^{\infty }$ is equicontinuous on $L.$ 
Combining this with item (I) and the argument in \cite[page 83-87]{Splms10}, 
we easily see that item (II) holds. Note that 
in \cite[page 83-87]{Splms10}, we work on the system 
on the whole $\CCI $, but here, we work on the system on $L$ 
and $\Omega _{L}$ by using the argument based on the 
hyperbolic metric on each connected component of $\Omega _{L}$ which is amost the same  as the argument given in \cite[page 83-87]{Splms10}. 

We now prove (III). 
Let $U$ be a connected component of $F(G_{\tau })$ and 
let $x_{1},x_{2}\in U.$ 
Let ${\cal H}_{i}:=\{ \gamma \in X_{\tau }\mid 
d(\gamma _{n,1}(x_{i}),L)\rightarrow 0 \mbox{ as }n\rightarrow \infty \}$ for each $i=1,2.$ 
Then $T_{L,\tau }(x_{i})=\tilde{\tau }({\cal H}_{i}).$ 
Let ${\cal I}_{i}:= \{ \gamma \in {\cal H}_{i}\mid 
\exists n\in \NN \mbox{ such that }\gamma _{n,1}(x_{i})\in \Omega \} .$ 
Let ${\cal A}$ be as in (I). Since $\tilde{\tau }({\cal A})=1$ and 
$\tilde{\tau }$ is $\sigma$-invariant, 
we have $\tilde{\tau }({\cal I}_{i})=\tilde{\tau }({\cal I}_{i}\cap \cap _{n=1}^{\infty }
\sigma ^{-n}({\cal A})).$  By (I), we have 
${\cal I}_{1}\cap \cap _{n=1}^{\infty }\sigma ^{-n}({\cal A})
={\cal I}_{2}\cap \cap _{n=1}^{\infty }\sigma ^{-n}({\cal A}).$ 
Hence $\tilde{\tau }({\cal I}_{1})=\tilde{\tau }({\cal I}_{2}).$ 
Let $\gamma \in {\cal H}_{1}\setminus {\cal I}_{1}.$ 
Then $d(\gamma _{n,1}(x_{1}),L\cap J(G_{\tau }))\rightarrow 0$ as $n\rightarrow \infty $ and 
every limit function of $\{ \gamma _{n,1}\} _{n=1}^{\infty }$ on $U$ should be constant. Therefore 
$\gamma \in {\cal H}_{2}\setminus {\cal I}_{2}.$ Thus 
${\cal H}_{1}\setminus {\cal I}_{1}\subset {\cal H}_{2}\setminus {\cal I}_{2}$. 
Similarly, we have ${\cal H}_{2}\setminus {\cal I}_{2}\subset {\cal H}_{1}\setminus {\cal I}_{1}$. 
Hence 
${\cal H}_{1}\setminus {\cal I}_{1}= {\cal H}_{2}\setminus {\cal I}_{2}$. 
From these arguments, 
it follows that $T_{L,\tau }(x_{1})=T_{L,\tau }(x_{2}).$ Therefore  
$T_{L,\tau }$ is locally constant on $F(G_{\tau }).$ 

Thus, we have completed the proof of Lemma~\ref{l:pwjkf}.     
\end{proof}
\begin{df}
\label{d:period} 
Under the assumptions of Lemma~\ref{l:pwjkf}, we call 
the number $r_{L}$ the period of $(\tau, L).$ 
\end{df}
\begin{rem}
\label{r:cofullcor}
The above argument in the proof of item (I) generalizes the argument in the proof of 
\cite[Lemma 5.2]{Splms10}. 
In the proof of \cite[Lemma 5.2]{Splms10}, 
in order to make the argument more precise, 
``and $\{ \gamma _{ki_{l},1}|_{U}:U\rightarrow \CCI \}_{l=1}^{\infty }$ converges to a non-constant 
map.'' (\cite[page 81,line -5]{Splms10}) should be 
``and any subsequence of 
$\{ \gamma _{ki_{l},1}|_{U}:U\rightarrow \CCI \}_{l=1}^{\infty }$ does not converge to a constant 
map.'' and 
``converges to a non-constant map.'' (\cite[page 82, line 4]{Splms10}) should be 
``does not converge to a constant map.'' 
 Also, in the proof of \cite[Lemma 5.3]{Splms10}, 
 the definition of $E_{n,m}$ should be 
 ``$E_{n,m}:=
 \{ \gamma \in {\cal A}\mid 
 \gamma _{ik,1}(a_{0})
 \in \cup _{j=1}^{p}V_{z_{j}}\cap 
 B(\partial J(G_{\tau }), b), i=n,\ldots, m\} $, where the number 
 $b$ is equal to $\min \{ d(u,v)\mid u\in \partial J(G_{\tau }), 
 v\in \cup _{j=1}^{p}\cup _{(\gamma _{1},\ldots, \gamma _{k})
 \in W_{j}}\gamma _{k}\cdots \gamma _{1}(V_{z_{j}})\}>0$''. 
 (This is a correction to the  proofs in \cite{Splms10}.)    
\end{rem}
The following is an important result on random dynamical systems generated by $\tau \in {\frak M}_{1,c}(\Rat)$ 
with $\sharp J_{\ker }(G_{\tau })<\infty .$ In the proof we use 
the no-wandering-domain theorem (\cite{Su}) and 
the well-known fact that the number of periodic non-repelling 
cycles of one element $f\in \Ratp$ is finite 
(its sofisticated result is known as the Fatou-Shishikura inequality (\cite{Sh})). 
\begin{prop}
\label{p:jkgfmf}
Let $\tau \in {\frak M}_{1,c}(\emRat )$ with $\sharp J(G_{\tau })\geq 3.$ 
Suppose $\sharp J_{\ker }(G_{\tau })<\infty .$ 
Then $1\leq \sharp \emMin(G_{\tau })<\infty $ and 
for each $L\in \emMin(G_{\tau })$ with $L\cap F(G_{\tau })\neq \emptyset $, 
statements {\em (I)(II)(III)} of Lemma~\ref{l:pwjkf} hold. 
\end{prop}  
\begin{proof}
By Lemma~\ref{l:pwjkf}, for each $L\in \Min(G_{\tau })$ with $L\cap F(G_{\tau })\neq \emptyset $, 
statements (I)(II)(III) of Lemma~\ref{l:pwjkf} hold. 
Also, by Lemma~\ref{l:pwjkf}, we obtain that 
\begin{equation}
\label{eq:slmglw}
\mbox{for each }W\in \mbox{Con}(F(G_{\tau })), 
\sharp \{ L\in \Min(G_{\tau },\CCI )\mid L\cap W\neq \emptyset \} \leq 1.
\end{equation}
We now suppose that $\sharp \Min(G_{\tau },\CCI )=\infty .$ 
Then, since we are assuming $\sharp J_{\ker }(G_{\tau })<\infty $, 
we obtain that 
\begin{equation}
\label{eq:slmgtci}
\sharp \{ L\in \Min(G_{\tau },\CCI )\mid L\cap F(G_{\tau })\neq \emptyset \} =\infty .
\end{equation}
We now show the following claim.\\ 
Claim 1. Let $\{ L_{j}\} _{j=1}^{\infty }$ be a sequence in $\Min(G_{\tau },\CCI)$ 
consisting of mutually distinct elements 
such that for each $j$, $L_{j}\cap F(G_{\tau })\neq \emptyset $. 
Moreover, let $\{ w_{j}\} _{j=1}^{\infty }$ be a sequence in $\CCI $ such that 
$w_{j}\in L_{j}$ for each $j$ and $\{ w_{j}\} $ tends to a point $w_{\infty }\in \CCI .$ 
Then $w_{\infty }\in J_{\ker }(G_{\tau }).$   

 To show this claim, suppose that $w_{\infty }\not\in J_{\ker }(G_{\tau }).$ 
Then there exists an element $\alpha \in G_{\tau }$ such that $\alpha (w_{\infty })\in F(G_{\tau }).$ 
Let $U\in \mbox{Con}(F(G_{\tau }))$ with $\alpha (w_{\infty })\in U.$ 
Then for each large $j$, we have $U\cap L_{j}\neq \emptyset .$ 
However, this contradicts (\ref{eq:slmglw}). Hence, Claim 1 holds. 

Let $\{ L_{j}\} _{j=1}^{\infty }$ be a sequence in $\Min(G_{\tau },\CCI)$ 
consisting of mutually distinct elements 
such that for each $j$, $L_{j}\cap F(G_{\tau })\neq \emptyset $. 
For each $j$, let $z_{j}\in L_{j}\cap F(G_{\tau })$ be a point. 
%Since $L_{j}\cap F(G_{\tau })\neq \emptyset $, we have 
%$L_{j}\cap J_{\ker }(G{\tau })=\emptyset $ for each $j.$ 
%Hence for each $\beta \in G_{\tau }$, 
%we have 
%\begin{equation}
%\label{eq:ocnbnlj}
%\overline{\cup _{n\in \NN }\beta ^{n}(L_{j})}\cap J_{\ker }(G_{\tau })=\emptyset . 
%\end{equation}
Since $\sharp J(G_{\tau })\geq 3$, 
by the density of repelling fixed points in the Julia set (\cite{St1}), either there exists a loxodromic 
element of Aut$(\CCI )\cap G$ or  there exists an element in $G$ of degree 
two or more. 

Suppose that there exists a loxodromic element 
$g\in \mbox{Aut}(\CCI )\cap G$.  
Let $a_{g}$ be the attracting fixed point of $g.$ 
Then for each $j$, we have $g^{n}(z_{j})\rightarrow a_{g}$ as $n\rightarrow \infty .$ 
This implies $a_{g}\in L_{j}$ for each $j.$ However, this is a contradiction. 

Suppose that there exists an element $g\in G$ with $\deg (g)\geq 2.$ 
By the no-wandering-domain theorem (\cite{Su}), we have that for each $z\in F(G_{\tau })$, 
$g^{n}(z)$ tends to one of the following cycles. (I) attracting cycle. (II) parabolic cycle. 
(III) Siegel disc cycle. (IV) Hermann ring cycle. 
Moreover, by the Fatou-Shishikura inequality (\cite{Sh}), 
the number of those cycles for one element $g$ is finite. 
Suppose that there exist a subsequence $\{ z_{j_{k}}\} $ of $\{ z_{j}\} $   
and a sequence $\{ n_{k}\} $ in $\NN $ such that 
$g^{n_{k}}(z_{j_{k}})$ tends to an attracting or parabolic cycle $c_{g}$ of $g.$ 
Then $c_{g}\in L_{j_{k}}$ for each large $k\in \NN $ and this is a contradiction. 
Therefore,  there exist a subsequence $\{ z_{j_{k}}\} $ of $\{ z_{j}\} $   
and a sequence $\{ n_{k}\} $ in $\NN $ such that 
$g^{n_{k}}(z_{j_{k}})$ belongs to a Siegel disk cycle or Hermann ring cycle of $g$ for each $k.$  
By taking a higher iterate of $g$, we may assume that the period of the cycle is one. 
Also, by renaming $g^{n_{k}}(z_{j_{k}})$ as $z_{k}$, we may assume that 
there exists a $B\in \mbox{Con}(F(g))$ which is either Siegel disk or Hermann ring of $g$ 
such that for each $j$, we have $z_{j}\in B\cap L_{j}\cap F(G_{\tau }).$    
Note that each $B\cap L_{j}\cap F(G_{\tau })$ is a union of analytic Jordan curves in 
$B.$ Let $$D:= \{ z\in \CCI \mid \mbox{ for each }\delta >0, \sharp \{ j\in \NN \mid B(z,\delta )\cap 
B\cap L_{j}\cap F(G_{\tau })\neq \emptyset \}  =\infty \} .$$  
Then by Claim 1, we have $D\subset J_{\ker }(G_{\tau }).$ 
Since we are assuming $\sharp J_{\ker }(G_{\tau })<\infty $, 
Claim 1 implies that for any connected component $A$ of $\partial B$, 
we cannot have that $A\subset D.$ Thus 
it follows that there exists a point $z_{\infty }\in J_{\ker }(G_{\tau }) $ such that 
$D=\{ z_{\infty }\} .$ Therefore $B$ is a Siegel disk for $g$ and $z_{\infty }$ is the center 
of $B$, i.e. $z_{\infty }\in B$ and $g(z_{\infty })=z_{\infty }.$ 
Let $0<\epsilon <\frac{1}{2}\min \{ d(a,b)\mid a,b\in J_{\ker }(G_{\tau }), a\neq b\} .$ Here, 
if $\sharp J_{\ker }(G_{\tau })=1$, then 
we set $\min \{ d(a,b)\mid a,b\in J_{\ker }(G_{\tau }), a\neq b\} =1.$ 
For each $j\in \NN $, let $C_{j}$ be the connected component of 
$\CCI \setminus (B\cap L_{j})$ such that $z_{\infty }\in C_{j}.$ 
Let $e_{0}:=\max\{ \| Dh_{z}\| _{s}\mid h\in \mbox{supp}\,\tau, z\in \CCI \} .$ 
We may assume that for each $j$, $\max _{a\in C_{j}}d(z_{\infty }, a)<\frac{1}{2}e_{0}^{-1}\epsilon .$ 
%We take $\epsilon $ so small that 
%$$e_{0}^{-1}\epsilon \leq \epsilon < \frac{1}{2}\min \{ d(a,b)\mid a,b\in J_{\ker }(G_{\tau }), %a\neq b\}.$$
Since $z_{\infty }\in J_{\ker }(G_{\tau })\subset J(G_{\tau })$, 
it follows that for each $j$ there exists an element $h_{j}\in G_{\tau }$ such that
\begin{equation}
\label{eq:macjdhj}
\max_{a\in C_{j}}d(h_{j}(z_{\infty }), h_{j}(a))\geq \frac{1}{2}e_{0}^{-1}\epsilon . 
\end{equation}
We may assume that fixing the generator system $\mbox{supp}\,\tau$ of $G_{\tau }$, 
the word length of $h_{j}$ is the minimum among the word lengths of elements 
of $G_{\tau }$ satisfying the same property as that of $h_{j}.$ Then, 
by (\ref{eq:macjdhj}) and the minimality of the word length of $h_{j}$, 
it follows that 
\begin{equation}
\label{eq:macjdhjzi}
\max_{a\in C_{j}}d(h_{j}(z_{\infty }),h_{j}(a))<\epsilon , \mbox{ for each } j.
\end{equation} 
Let $a_{j}\in C_{j}$ be a point such that 
$d(h_{j}(z_{\infty }),h_{j}(a_{j}))=\max _{a\in C_{j}}d(h_{j}(z_{\infty }), h_{j}(a)).$ 
Then $a_{j}\in \partial C_{j}\subset L_{j}$ for each $j.$ 
Hence, setting $u_{j}=h_{j}(a_{j})$, we have 
\begin{equation}
\label{eq:ujinlj}
u_{j}\in L_{j} \mbox{ for each }j.
\end{equation}
By (\ref{eq:macjdhj}) and (\ref{eq:macjdhjzi}), we have 
\begin{equation}
\label{eq:he0-1edhj}
\frac{1}{2}e_{0}^{-1}\epsilon \leq d(h_{j}(z_{\infty }),u_{j})<\epsilon .
\end{equation}
Since $h_{j}(z_{\infty })\in J_{\ker }(G_{\tau })$ and by the way of the choice of $\epsilon $, 
(\ref{eq:he0-1edhj}) implies that $\{ u_{j}\} _{j=1}^{\infty }$ cannot accumulate in 
any point of $J_{\ker }(G_{\tau }).$ Combining this with (\ref{eq:ujinlj}) and Claim 1, 
we obtain a contradiction. Hence, $\sharp \Min(G_{\tau },\CCI )<\infty .$ 

Thus we have proved our proposition. 
\end{proof}         
The following is an important and interesting object in random dynamics. 
\begin{df}
\label{d:probabilityoftending}
%Let $Y$ be a compact metric space and let $A$ be a subset of $Y.$ 
Let $A$ be a subset of $\CCI .$ 
Let $\tau \in {\frak M}_{1}(\Rat).$ For each $z\in \CCI $,  
we set 
$T_{A,\tau }(z):= \tilde{\tau }(\{ \gamma =(\gamma _{1},\gamma _{2},\ldots )\in X_{\tau }\mid 
d(\gamma _{n,1}(z),A)\rightarrow 0 \mbox{ as } n\rightarrow \infty \}).$ 
This is the {\bf probability of tending to $A$ 
regarding the random orbits 
starting with the initial value $z\in \CCI .$}    
For any $a\in \CCI $,  we set $T_{a,\tau }:=T_{\{ a\} ,\tau }.$ 
\end{df}
We now prove the following theorem regarding the systems with finite kernel Julia sets.  
\begin{thm}
\label{t:zfggzcc}
Let $\tau \in {\frak M}_{1,c}(\emRat ) $ with $\sharp J(G_{\tau })\geq 3.$ 
Suppose that $\sharp J_{\ker }(G_{\tau })<\infty $ and 
for each $z\in F(G_{\tau })$, we have 
$\overline{G_{\tau }(z)}\cap (\bigcup _{L\in \emMin(G_{\tau },\CCI ), L\not\subset J_{\ker }(G_{\tau })}L)
\neq \emptyset .$ 
Then we have the following.
\begin{itemize}
\item[{\em (i)}] 
$\sharp \emMin(G_{\tau },\CCI )<\infty .$ 
Moreover, 
for each $L\in \emMin(G_{\tau },\CCI )$, 
we have 
\vspace{-2mm}
$$C(L)=U_{\tau,L}\oplus \{ \varphi \in C(L)\mid M_{\tau }^{n}(\varphi )\rightarrow 0 \mbox{ as }
n\rightarrow \infty \} $$ 
\vspace{-0mm} 
in the Banach space $C(L)$ endowed with 
the supremum norm and $\dim _{\CC }U_{\tau ,L}<\infty .$ 
Moreover, 
for each $L\in \emMin(G_{\tau },\CCI )$, let $r_{L}=\dim _{C}(U_{\tau ,L}).$ Then 
$\sharp \emMin(G_{\tau }^{r_{L}},L)=r_{L}.$ Also, 
there exist $L_{1},\ldots, L_{r_{L}}\in \emMin(G_{\tau }^{r_{L}},L)$ such that 
$\{ L_{j}\mid  j=1,\ldots, r_{L}\}=\emMin(G_{\tau }^{r_{L}},L)$, 
$L=\cup _{j=1}^{r_{L}}L_{j}$ and $h(L_{i})\subset L_{i+1}$ for each $h\in \mbox{supp}\,\tau$, 
where $L_{r_{L}+1}:=L_{1}.$  
\item[{\em (ii)}] 
For each $L\in \emMin(G_{\tau}, \CCI )$, 
for each $j=1,\ldots, r_{L}$, 
there exists a unique element 
$\omega _{L,j}\in {\frak M}_{1}(L_{j})$ such that  $(M_{\tau }^{r_{L}})^{\ast }(\omega _{L,j})=\omega _{L,j}.$ 
Also, for each $j=1,\ldots, r_{L}$, we have 
$M_{\tau }^{nr_{L}}(\varphi )\rightarrow (\int \varphi \ d\omega _{L,j})\cdot 1_{L_{j}}$ in 
the Banach space $C(L_{j})$ endowed with the supremum norm as $n\rightarrow \infty $
for each $\varphi \in C(L_{j})$, supp$\,\omega _{L,j}=L_{j}$ and 
$M_{\tau }^{\ast }(\omega _{L,j})=\omega _{L,j+1}$ in ${\frak M} _{1}(L)$ where 
$\omega _{L,r_{L}+1}=\omega _{L,1}.$ 
Moreover, $U_{\tau ,L,\ast }=
\{ a\in \CC \mid a ^{r_{L}}=1\}$ and 
for each $a \in U_{\tau ,L,\ast }$, we have 
$\dim _{\CC }\{ \varphi \in C(L)\mid M_{\tau }\varphi =a \varphi \} =1.$ 

%For each $L\in \Min(G_{\tau },\CCI )$ and for each $j=1,\ldots, r_{L}$, 
%there exists a unique element $\eta _{L,j}\in {\frak M}_{1}(L_{j})$ such that 
%$(M_{\tau }^{\ast })^{r_{L}}(\eta _{L,j})=\eta _{L,j}.$ 
\item[{\em (iii)}]
Let $l:=\prod _{L\in \emMin (G_{\tau },\CCI )}r_{L}. $  
For each $L\in \emMin(G_{\tau },\CCI )$, for each $j=1,\ldots, r_{L}$ and 
for each $y\in \CCI $, let 
$\alpha (L_{j},y)=\tilde{\tau }(\{ \gamma \in 
(\mbox{supp}\,\tau)^{\NN }\mid 
d(\gamma _{nl,1}(y), L_{j})\rightarrow 0 \mbox{ as }n\rightarrow \infty .\}). $ 
Then for each $y\in \CCI $ and for each $\varphi \in C(\CCI )$, we have 
\begin{equation}
\label{eq:mtnlvy}
M_{\tau }^{nl}(\varphi )(y)\rightarrow \sum _{L\in \emMin(G_{\tau },\CCI )}
\sum _{j=1}^{r_{L}}\alpha (L_{j},y)\int \varphi \ d\omega _{L,j} \mbox{ as }n\rightarrow \infty 
\mbox{ (pointwise convergence)},
\end{equation} 
i.e. we have 
\begin{equation}
\label{eq:mtanld}
(M_{\tau }^{\ast })^{nl}(\delta _{y}) \rightarrow \sum _{L\in \emMin(G_{\tau },\CCI )}
\sum _{j=1}^{r_{L}}\alpha (L_{j},y)\omega _{L,j} \mbox{ as }n\rightarrow \infty \mbox{ in }{\frak M}_{1}(\CCI )
\end{equation}   
with respect to the weak convergence topology. 
Also, 

\vspace{1mm}
\hspace{5mm} $(M_{\tau }^{\ast })^{l}(\sum _{L\in \emMin(G_{\tau },\CCI )}
\sum _{j=1}^{r_{L}}\alpha (L_{j},y)\omega _{L,j})=\sum _{L\in \emMin(G_{\tau },\CCI )}
\sum _{j=1}^{r_{L}}\alpha (L_{j},y)\omega _{L,j}.$ 
\item[{\em (iv)}]
For each $z\in \CCI $ there exists a Borel subset $A_{z}$ of $(\mbox{supp}\,\tau)^{\NN }$ 
with $\tilde{\tau }(A_{z})=1$ such that 
for each $\gamma =(\gamma _{1},\gamma _{2},\ldots )\in A_{z}$, 
there exists an element $L=L(z, \gamma )\in \emMin(G_{\tau },\CCI )$ satisfying that 
$d(\gamma _{n,1}(z), L)\rightarrow 0$ as $n\rightarrow 0.$ 
%we have $d(\gamma _{n,1}(z), \cup _{L\in \emMin(G_{\tau },\CCI )}L)\rightarrow 0$ as 
%$n\rightarrow \infty .$ 
\item[{\em (v)}] 
There exists a Borel subset ${\cal A}$ of $\emRat ^{\NN} $ with $\tilde{\tau }({\cal A})=1$ 
such that for each $L\in \emMin(G_{\tau },\CCI )$ with $L\not\subset J_{\ker }(G_{\tau })$, 
for each point $z\in \Omega _{L}:=\cup _{U\in \mbox{{\em Con}}(F(G_{\tau })):U\cap L\neq \emptyset }U$ 
and for each $\gamma \in {\cal A}$,  
there exists a $\delta =\delta (z,\gamma )>0$ such that 
diam$(\gamma _{n,1}(B(z,\delta )))\rightarrow 0$  and 
$d(\gamma _{n,1}(z),L)\rightarrow 0$ as $n\rightarrow \infty .$  
In particular, for each $L\in \emMin(G_{\tau },\CCI )$ with $L\not\subset J_{\ker }(G_{\tau })$ and    
for each $j=1,\ldots, r_{L}$,  if $y\in \Omega _{L,j}
:=\cup _{U\in \mbox{{\em Con}}(F(G_{\tau })):U\cap L{j}\neq \emptyset }U $, then  
$\alpha (L_{j},y)=1$,  and if $y\in \Omega _{L', i}$ with $(L',i)\neq (L,j)$ then 
$\alpha (L_{j},y)=0.$

\item[{\em (vi)}]
Let $L\in \emMin(G_{\tau },\CCI )$ and let $j=1,\ldots, r_{L}.$ 
Then  
the functions $T_{L,\tau }:\CCI \rightarrow [0, 1]$ and 
$\alpha (L_{j},\cdot ):\CCI \rightarrow [0,1]$ are locally constant on $F(G_{\tau }).$  
\item[{\em (vii)}] 
Let $L\in \emMin(G_{\tau },\CCI )$ and let $j=1,\ldots, r_{L}.$ 
Then for each $y\in F_{pt}^{0}(\tau )$, we have 

\vspace{2mm}
\hspace{15mm} 
$\lim _{z\in \CCI, z\rightarrow y}T_{L,\tau }(z)=T_{L,\tau }(y)$ and 
$\lim _{z\in \CCI, z\rightarrow y}\alpha (L_{j},z)=\alpha (L_{j}, y).$ 
\end{itemize}

\end{thm}
\begin{proof}
Let ${\cal L}:=\bigcup _{L\in \Min(G_{\tau },\CCI ), L\not\subset J_{\ker }(G_{\tau })}L.$ 
Then, by  the assumption of our theorem, we have ${\cal L}\neq \emptyset .$ 
Moreover, we have 
${\cal L}=\bigcup _{L\in \Min(G_{\tau },\CCI ), L\cap F(G_{\tau })\neq \emptyset }L.$ 
Let $$V:=\bigcup \{ U\in \mbox{Con}(F(G_{\tau }))\mid \exists L\in \Min(G_{\tau },\CCI ) 
\mbox{ with }L\cap U\neq \emptyset \} .$$ 
Then $G_{\tau }(V)\subset V.$ Moreover, by the assumptions of our theorem, we obtain 
$\cap _{g\in G_{\tau }}g^{-1}(\CCI \setminus V)=J_{\ker }(G_{\tau }).$ Therefore, 
the statements (i)--(v) of our theorem follow from  Lemma~\ref{l:nlkfa},  
Proposition~\ref{p:jkgfmf} and Lemma~\ref{l:voygvv}. 

We now prove statement (vi). 
Let $L\in \Min(G_{\tau },\CCI )$ and $j=1,\ldots, r_{L}.$ 
If $L\not\subset J_{\ker }(G_{\tau })$, 
then 
by Lemma~\ref{l:pwjkf} (III) and its proof,  
the functions $T_{L, \tau }$ and $\alpha (L_{j},\cdot )$ are locally constant on $F(G_{\tau }).$ 
If $L\subset J_{\ker }(G_{\tau })$, then for any $U\in \mbox{Con}(F(G_{\tau }))$, 
for any $x,y\in U$ and  
for any $\gamma \in X_{\tau }$ with $d(\gamma _{n,1}(x), L)\rightarrow 0\ (n\rightarrow \infty)$, 
 we have that any limit function of $\{ \gamma _{n,1}\} _{n=1}^{\infty }$ is 
constant on $U.$ Hence $d(\gamma _{n,1}(y),L)\rightarrow 0$ as $n\rightarrow \infty $. 
This argument implies that $T_{L,\tau }$ is constant on $U$ for any $U\in \mbox{Con}(F(G_{\tau })).$ 
Therefore $T_{L,\tau }$ is locally constant on $F(G_{\tau }).$ 
By the same method as above, we can show that 
$\alpha (L_{j}, \cdot )$ is locally constant on $F(G_{\tau }).$  

We now prove (vii). 
Let $L\in \Min (G_{\tau },\CCI )$ and let $j=1,\ldots, r_{L}.$ 
Since $\sharp \Min(G_{\tau },\CCI )<\infty $, there exists an element 
$\varphi _{L}\in C(\CCI )$ such that 
$\varphi _{L}|_{L}=1$ and  
$\varphi _{L}|_{L'}=0$ for any $L'\in \Min(G_{\tau },\CCI )$ with 
$L'\neq L.$ By statement (iv), we have 
$T_{L,\tau }(x)=\lim _{n\rightarrow }M_{\tau }^{n}(\varphi _{L})(x)$ for any $x\in \CCI .$ 
 Thus for any $y\in F_{pt}^{0}(\tau )$, 
we have $\lim _{z\in \CCI ,z\rightarrow y}T_{L,\tau }(z)=T_{L,\tau }(y).$ 
Similarly, we can show that for any $y\in F_{pt}^{0}(\tau )$, 
$\lim _{z\in \CCI ,z\rightarrow y}\alpha (L_{j},z)=\alpha (L_{j},y).$ 

Thus we have proved our theorem. 
\end{proof}
We now prove the following theorem, which is a generalization of \cite[Theorem 3.15]{Splms10}. 
\begin{thm}
\label{t:sjkfcln}
Let $\tau \in {\frak M}_{1,c}(\emRat).$ Suppose that 
$\sharp J_{\ker }(G_{\tau })<\infty ,$ $\sharp J(G_{\tau })\geq 3$ and that for each $L\in \emMin(G_{\tau }, J_{\ker }(G_{\tau }))$,  
$\chi (\tau ,L)<0.$ Then we have $F_{pt}^{0}(\tau )=\CCI , F_{meas}(\tau )={\frak M}_{1}(\CCI )$, 
{\em Leb}$_{2}(J_{\gamma })=0$ for $\tilde{\tau }$-a.e.$\gamma \in (\emRat) ^{\NN }$,  
and all statements in {\em \cite[Theorem 3.15 (1)--(3),(4a),(5)(6),(8)--(16), (19)(20)]{Splms10}} hold for 
$\tau .$ Moreover, 
%there exists a negative constant $c<0$ such that 
for each $z\in \CCI $, 
there exists a Borel subset ${\cal A}_{z}$ of $X_{\tau } $ 
with $\tilde{\tau }({\cal A}_{z})=1$ satisfying that 
for each $\gamma =(\gamma _{1},\gamma _{2},\ldots )\in {\cal A}_{z}$ and for each 
$m\in \NN \cup \{ 0\},$  
we have $$\lim_{n\rightarrow \infty }\| D(\gamma _{n+m,1+m})_{\gamma _{m,1}(z)}\| _{s}=0 
.$$    
Also, if, in addition to the assumptions of our theorem, 
each $L\in \Min(G_{\tau }, \CCI )$ with $L\not\subset J_{\ker }(G_{\tau })$ is attracting 
for $\tau $, then there exist a constant $c_{\tau }<0$ and a constant $\rho _{\tau }\in (0,1)$ such that 
for each $z\in \CCI $, there exists a Borel subset ${\cal A}_{z}$ of $X_{\tau }$ 
with $\tilde{\tau }({\cal A}_{z})=1$ such that for each $\gamma \in {\cal A}_{z}$ and 
for each $m\in \NN \cup \{ 0\},$ we have the following {\em (a)} and {\em (b)}. 
\begin{itemize}
\item[{\em (a)}] 
$$\limsup _{n\rightarrow \infty }\frac{1}{n}\log \| D(\gamma _{n+m,1+m})_{\gamma _{m,1}(z)}\| _{s}\leq c_{\tau }<0.$$ 
\item[{\em (b)}] 
There exist a constant $\delta =\delta (\tau, z,\gamma, m)>0$, a constant 
$\zeta =\zeta (\tau, z, \gamma, m)>0$ 
%and 
%an attracting minimal set $L=L(\tau, z, \gamma )$ 
%of $\tau $ 
and a minimal set $L=L(\tau, z,\gamma )$ of $\tau $ 
which is either {\em (i)} ``attracting for $\tau $'', or {\em (ii) } ``finite and included in $J_{\ker }(G_{\tau }) $ with 
$\chi (\tau, L)<0$'',  
such that 
$$\mbox{diam}(\gamma _{n+m,1+m}(B(\gamma _{m,1}(z),\delta )))\leq \zeta \rho _{\tau }^{n}\ \mbox{ for all }n\in \NN, $$
and 
$$d(\gamma _{n+m,1+m}(\gamma _{m,1}(z)), \ L)
%\cup _{L\in \emMin(G_{\tau },\CCI ), L \mbox{ is attracting 
%for }\tau } L) 
\leq \zeta \rho _{\tau }^{n} 
\mbox{\ \  for all }n\in \NN .$$ 
\end{itemize}
\end{thm}
\begin{proof}
We modify the proof of Theorem~\ref{t:jkfcln0}. 
By the assumption of our theorem, the set $\Omega $ in 
the proof of Theorem~\ref{t:jkfcln0} is equal to $\CCI .$ 
By Theorem~\ref{t:jkfcln0}, we see that 
for each $z\in \CCI $, $\tilde{\tau }(\{ \gamma \in X_{\tau }\mid z\in J_{\gamma }\})=0$ 
and $F_{pt}^{0}(\tau )=\CCI .$  
%By Lemma~\ref{l:ttaugy0y}, it follows that for each $z\in \CCI $, $z\in F_{pt}^{0}(\tau ).$ 
Therefore by \cite[Lemma 4.2(6)]{Splms10}, $F_{meas}(\tau )={\frak M}_{1}(\CCI).$  

Let $\delta _{1}>0$ be a small number. Let $\epsilon >0$ be an arbitrarily small number. 
Then by the argument in the proof of Lemma~\ref{l:ttaugy0y},  
there exist a $\delta _{2}>0$ with $\delta _{2}<\delta _{1}$ and a Borel subset $A_{\epsilon }$ of 
$X_{\tau } $ with $\tilde{\tau }(A_{\epsilon })\geq 1-\epsilon $ 
such that for each $L\in \Min(G_{\tau },J_{\ker }(G_{\tau }))$,
  for each $z\in L$, and for each $\gamma =(\gamma _{1},\gamma _{2}
 \ldots )\in A_{\epsilon }$, we have 
 diam$(\gamma _{n,1}(B(z,\delta _{2})))\leq \delta _{1}.$ 
 For this $\delta _{2}$, 
by the argument in the proof of Lemma~\ref{l:ttaugy0y} again, 
there exist a $\delta _{3}>0$ and a Borel subset 
 $B_{\epsilon }$ of $X_{\tau }$ with $\tilde{\tau }(B_{\epsilon })\geq 1-\epsilon $  
  such that  for each $L\in \Min(G_{\tau },J_{\ker }(G_{\tau }))$,
  for each $z\in L$, and for each $\gamma =(\gamma _{1},\gamma _{2}
 \ldots )\in B_{\epsilon }$, we have 
 diam$(\gamma _{n,1}(B(z,\delta _{3})))\leq \delta _{2}.$ 
 
 Let $I_{1}:=\{ L\in \Min(G_{\tau },\CCI )\mid L\subset J_{\ker }(G_{\tau })\} $ 
 and $I_{2}:= \{ L\in \Min(G_{\tau },\CCI )\mid L\cap F(G_{\tau })\neq \emptyset \} .$ 
 Note that $I_{1}\cup I_{2}=\Min(G_{\tau },\CCI ).$ For each $L\in I_{2}$, 
 let $W_{L}:= \cup _{U\in \mbox{Con}(F(G_{\tau })),U\cap L\neq \emptyset }U.$ 
 Then for each $z\in \CCI $, there exists an element $g_{z}\in G_{\tau }$ 
 such that $g_{z}(z)\in B(\cup _{L\in I_{1}}L,\delta _{3})\cup \cup _{L\in I_{2}}W_{L}.$ 
Let $\delta _{z}>0$ be a number such that 
$g_{z}(\overline{B(z,\delta _{z})})\subset B(\cup _{L\in I_{1}}L,\delta _{3})\cup \cup _{L\in I_{2}}W_{L}.$ 
Since $\CCI $ is compact, there exist finitely many points $z_{1},\ldots, z_{p}\in \CCI $ 
such that $\CCI =\cup _{j=1}^{p}B(z_{j},\delta _{z_{j}}).$ Note that $G_{\tau }(W_{L})\subset W_{L}$ for each 
 $L\in I_{2}.$ Thus if $g_{z}(z)\in W_{L}$ for some $L\in I_{2}$, then 
 for each $g\in G_{\tau }$, we have $gg_{z}(z)\in W_{L}.$ Moreover, 
 for each $L\in I_{1}$, for each $z\in L$, for each $\gamma \in B_{\epsilon }$ and for each $n\in \NN $, 
 we have $\gamma _{n,1}(B(z,\delta _{3}))\subset B(\cup _{L\in I_{1}}L,\delta _{2})$. 
 Hence, considering $\alpha _{j}\circ g_{z_{j}}$ for some $\alpha _{j}\in G_{\tau }$ for each $j$, 
we have the following claim. 

\noindent  \underline{Claim 1.} There exist an $l\in \NN $ 
and $p$ elements $h_{z_{1}},\ldots, h_{z_{p}}\in G_{\tau }$ such that 
each $h_{z_{j}}$ is  a composition of $l$ elements of $\mbox{supp}\,\tau$  
and such that 
 for each $j=1,\ldots, p$,  we have $h_{z_{j}}(\overline{B(z,\delta _{z_{j}})})\subset B(\cup _{L\in I_{1}}L,\delta _{2})
 \cup \cup _{L\in I_{2}}W_{L}.$ 

   For each $j=1,\ldots, p$, let   $(\gamma _{1}^{j},\ldots, \gamma _{l}^{j})\in (\mbox{supp}\,\tau)^{l}$ 
be an element   such that $h_{z_{j}}=\gamma _{l}^{j}\circ \cdots \circ \gamma _{1}^{j}.$  
   For each $j=1,\ldots, p$, let $V_{j}$ be a neighborhood 
   of $(\gamma _{1}^{j},\ldots, \gamma _{l}^{j})\in 
   (\mbox{supp}\,\tau)^{l}$ such that 
   for each $(\alpha _{1},\ldots, \alpha _{l})\in V_{j}$, we have
   $\alpha _{l}\cdots \alpha _{1}(\overline{B(z_{j},\delta _{z_{j}})})
   \subset B(\cup _{L\in I_{1}}L,\delta _{2})\cup \cup _{L\in I_{2}}W_{L}.$ 
   Let $\Omega _{1},\ldots, \Omega _{t}$ be the measurable partition of $\CCI $ such that 
   each $\Omega _{i}$ is a finite intersection of elements of $\{ B(z_{j},\delta _{z_{j}})\} _{j=1}^{p} .$ 
   For each $i=1,\ldots, t$, let $\varphi (i)\in \{ 1,\ldots, p\} $ be an element 
   such that $\Omega _{i}\subset B(z_{\varphi (i)},\delta _{z_{\varphi (i)}}).$ 
   For each $z\in \CCI $, let $i(z)\in \{ 1,\ldots,  t\} $ be the unique element such that 
   $z\in \Omega_{i(z)}.$  Let $j(z)=\varphi (i(z))\in \{ 1,\ldots ,p\} .$  
   For each $n\in \NN $ and each $z\in \CCI $, let 
   $C_{n,z}$ be the set of elements $\gamma =(\gamma _{1},\gamma _{2},\ldots )\in X_{\tau }$ 
satisfying the following. 
\begin{itemize}
\item 
$(\gamma _{1},\ldots, \gamma _{l})\not\in V_{j(z)}, 
(\gamma _{l+1},\ldots, \gamma _{2l})\not\in V_{j(\gamma _{l,1}(z))}, 
\ldots, 
(\gamma _{(n-2)l+1},\ldots, \gamma _{(n-1)l})\not\in V_{j(\gamma _{(n-2)l,1}(z))}, $ and 
$(\gamma _{(n-1)l+1},\ldots ,\gamma _{nl})\in V_{j(\gamma _{(n-1)l,1}(z))}$. 
\end{itemize} 
Similarly, let $D_{n,z}:=\{ \gamma \in C_{n,z}\mid 
(\gamma _{nl+1},\gamma _{nl+2},\ldots )\not\in A_{\epsilon }\}. $ 
Moreover, let $E_{z}$ be the set of elements $\gamma =(\gamma _{1},\gamma _{2},\ldots )
\in X_{\tau }$ satisfying that 
for each $n\in \NN , (\gamma _{(n-1)l+1},\ldots, \gamma _{nl})\not\in 
V_{j(\gamma _{(n-1)l,1}(z))}.$ 
Then for each $z\in \CCI $ we have 
$$\{ \gamma \in X_{\tau }
\mid \gamma _{n,1}(z)\not\in B(\cup _{L\in I_{1}}L,\delta _{1})\cup \cup _{L\in I_{2}}W_{L} 
\mbox{ for infinitely many } n\in \NN 
\} 
\subset \cup _{n=1}^{\infty }D_{n,z}\cup  E_{z}.$$ 
It is easy to see that $\tilde{\tau }(E_{z})=0.$ 
Moreover, 

\vspace{2mm} 
$\tilde{\tau }(\cup _{n=1}^{\infty }D_{n,z})
=\sum _{n=1}^{\infty }\tilde{\tau }(D_{n,z})
=\sum _{n=1}^{\infty }\tilde{\tau }(C_{n,z})\cdot \tilde{\tau }
(X_{\tau }\setminus A_{\epsilon })
= \tilde{\tau }(\cup _{n=1}^{\infty }C_{n,z})\cdot \tilde{\tau }
(X_{\tau }\setminus A_{\epsilon })
\leq \epsilon. $
 
\vspace{1mm} 
\noindent  Hence 
 \vspace{-2mm} 
\begin{equation}
\label{eq:ttggn1e}
\tilde{\tau }(\{ \gamma \in X_{\tau }
\mid \gamma _{n,1}(z)\not\in B(\cup _{L\in I_{1}}L,\delta _{1})\cup \cup _{L\in I_{2}}W_{L} 
\mbox{ for infinitely many } n\in \NN \} )\leq \epsilon .
\end{equation} 
Since $\epsilon, \delta _{1}$ are arbitrary, combining (\ref{eq:ttggn1e}) and Lemma~\ref{l:pwjkf} 
implies that for each $z\in \CCI $, there exists a Borel subset 
$Q_{z}$ of $X_{\tau }$ with $\tilde{\tau }(Q_{z})=1$ such that 
for each $\gamma \in Q_{z}$, we have 
\begin{equation}
\label{eq:dgn1zcl0} 
d(\gamma _{n,1}(z), \cup _{L\in \Min(G_{\tau },\CCI )}L)\rightarrow 0 \mbox{ as }n\rightarrow \infty .
\end{equation}
  By the result 
$F_{meas}(\tau )={\frak M}_{1}(\CCI )$, (\ref{eq:dgn1zcl0}), 
Lemma~\ref{l:fininvset} and its proof, 
Lemma~\ref{l:chimune} and its proof, 
Theorem~\ref{t:jkfcln0}, 
Lemma~\ref{l:pwjkf},  the proof of 
\cite[Theorem 3.15]{Splms10} and \cite[Theorem 3.14]{Splms10}, the statement of our 
theorem holds. 
\end{proof}
\begin{rem}
\label{r:sjknm}
Under the assumptions of Theorem~\ref{t:sjkfcln}, 
suppose that $J_{\ker }(G_{\tau })\neq \emptyset .$ 
Then $\tau $ is not mean stable. Also, $\tau $ does not satisfy the assumptions of 
\cite[Theorem 3.15]{Splms10}, 
although most of the statements of \cite[Theorem 3.15]{Splms10} hold for $\tau .$ 
Note that we have many examples $\tau \in {\frak M}_{1,c}(\Rat)$ with 
$J_{\ker }(G_{\tau })\neq \emptyset $ satisfying the assumptions of Theorem~\ref{t:sjkfcln}. 
See Section~\ref{Examples}, Example~\ref{ex:wlz1z}.  
\end{rem}
We now give the definition of nice sets and strongly nice sets of $\Rat.$ 
\begin{df}
\label{d:stronglynice}
Let ${\cal Y}$ be a weakly nice subset of $\Rat $ with respect to some holomorphic families 
$\{ {\cal W}_{j}\} _{j=1}^{m}$ of rational maps, where 
${\cal W}_{j}=\{ f_{j,\lambda }\mid \lambda \in \Lambda _{j}\}$ for each $j=1,\ldots, m.$   

\begin{itemize}
\item 
We say that ${\cal Y}$ is {\bf nice} (with respect to holomorphic families 
 $\{ {\cal W}_{j}\} _{j=1}^{m}$ of rational maps) if 
for each $z\in S_{\min }(\{ {\cal W}_{j}\} _{j=1}^{m})$ (see Definition~\ref{d:smin}) and 
 for each $j=1,\ldots ,m$, either (a) the map $\lambda \mapsto D(f_{j,\lambda })_{z}$ 
 is non-constant on $\Lambda _{j}$ or (b)   $D(f_{j,\lambda })_{z}= 0$ 
 for all  $\lambda \in \Lambda _{j}.$ 

\item  We say that  a finite sequence $\{ z_{i}\} _{i=1}^{n}$ of points of 
$\CCI $ 
is a {\bf peripheral cycle} for $({\cal Y}, \{ {\cal W}_{j}\} _{j=1}^{m})$ if 
there exists a $\Gamma \in \Cpt({\cal Y})$ such that 
both of the following (a)(b) hold.  
\begin{itemize}
\item[(a)] 
%There exists a $\Gamma \in \Cpt{{\cal Y}}$ such that  
$\{ z_{i}\mid i=1,\ldots, n\} \subset 
  (\cup _{j=1}^{m}S_{1}({\cal W}_{j}))\setminus 
\cup_{L\in \Min(\langle \Gamma \rangle ,\CCI ), L\subset \cup _{j=1}^{m}S_{1}({\cal W}_{j})}L$. 
\item[(b)]   
There exists a finite sequence $\{ \gamma _{i}\} _{i=1}^{n}$ of elements of $\Gamma $ 
such that for each $i=1,\ldots, n$, there exists a number 
$j_{i}\in \{ 1,\ldots, n\} $ satisfying  that for each $i=1,\ldots, n$, we have 
$\gamma _{i}\in \{ f_{j_{i},\lambda }\mid \lambda \in \Lambda _{j_{i}}\} $, 
$z_{i}\in S_{1}({\cal W}_{j_{i}})$  and 
$\gamma _{i}(z_{i})=z_{i+1}$ where $z_{n+1}:=z_{1}.$ 
\end{itemize} 
\item We say that ${\cal Y}$ is {\bf strongly nice}  with respect to 
$\{ {\cal W}_{j}\} _{j=1}^{m}$ 
if 
${\cal Y}$ is nice with respect to  
 $\{ {\cal W}_{j}\} _{j=1}^{m}$ and there exists no 
 peripheral cycle for $({\cal Y}, \{ {\cal W}_{j}\} _{j=1}^{m})$.  
%for each $\Gamma \in \Cpt({\cal Y})$, for each 
%$z\in (\cup _{j=1}^{m}S_{1}({\cal W}_{j}))\setminus 
%\cup_{L\in \Min(\langle \Gamma \rangle ,\CCI ), L\subset \cup _{j=1}^{m}S_{1}({\cal W}_{j})}L$, 
%we {\bf never} have the following situation.
%\begin{itemize}
%\item There exist a finite sequence $\{ z_{i}\} _{i=1}^{n}$ of points of $\CCI $ with $z_{1}=z$   
%and a finite sequence $\{ \gamma _{i}\} _{i=1}^{n}$ of elements of $\emRat$ 
%such that for each $i=1,\ldots, n$, there exists a number 
%$j_{i}\in \{ 1,\ldots, n\} $ satisfying  that for each $i=1,\ldots, n$, we have 
%$\gamma _{i}\in \{ f_{j_{i},\lambda }\mid \lambda \in \Lambda _{j_{i}}\} $, 
%$z_{i}\in S_{1}({\cal W}_{j_{i}})$  and 
%$\gamma _{i}(z_{i})=z_{i+1}$ where $z_{n+1}:=z_{1}=z.$ 

%\end{itemize}

\end{itemize}

 \end{df} 
 \begin{df}
 \label{d:strbf}
 Let ${\cal Y}$ be a weakly nice subset of $\Ratp $ with respect to some holomorphic families 
$\{ {\cal W}_{j}\} _{j=1}^{m}$ of rational maps, where
 ${\cal W}_{j}=\{ f_{j,\lambda }\mid \lambda \in \Lambda _{j}\}$ for each $j=1,\ldots, m.$   
Let $\Gamma \in \mbox{Cpt}({\cal Y})$ such that $\Gamma \cap \{ f_{j,\lambda }\mid \lambda 
\in \Lambda _{j}\} \neq \emptyset $ for each $j=1,\ldots, m.$ 
Let $L\in \Min(\langle \Gamma \rangle ,\CCI )$ with $L\neq \CCI .$ 
Let $g\in \Gamma $ and $j\in \{ 1,\ldots, m\} .$ 
We say that $g$ is a 
{\bf strict bifurcation element for $(\Gamma, L)$ 
with corresponding suffix $j$}
if one of the following statements (a)(b) holds. 
\begin{itemize}
\item[(a)] $g\in \{ f_{j,\lambda }\mid \lambda \in \Lambda _{j}\} $ and 
there exists a point $z\in (L\cap J(\langle \Gamma \rangle ))\setminus S_{1}({\cal W}_{j})$  
such that  
 $g(z)\in L\cap J(\langle \Gamma \rangle ).$  

\item[(b)] 
$g\in \{ f_{j,\lambda }\mid \lambda \in \Lambda _{j}\} $ and 
there exist an open subset $U$ of $\CCI $ 
%with $(U\cap L)\setminus S_{1}({\cal W}_{j})\neq \emptyset $ 
and finitely many 
elements $\gamma _{1},\ldots, \gamma _{n-1}\in \Gamma $ such that 
$g\circ \gamma _{n-1}\cdots \gamma _{1}(U)\subset U$,  
\ $U$ is a subset of a Siegel disk or a Hermann 
ring of $g\circ \gamma _{n-1}\circ \cdots \circ \gamma _{1}$, and $(\gamma _{n-1}\circ \cdots \circ \gamma _{1}(U)
\cap L)\setminus S_{1}({\cal W}_{j})\neq \emptyset .$   
  \end{itemize}  
 
 \end{df}
 
\begin{lem}
\label{l:strbif}
 Let ${\cal Y}$ be a weakly nice subset of $\emRatp $ with respect to some 
 holomorphic families 
$\{ {\cal W}_{j}\} _{j=1}^{m}$ of rational maps, where 
${\cal W}_{j}=\{ f_{j,\lambda }\mid \lambda \in \Lambda _{j}\}$ for each $j=1,\ldots, m.$   
Suppose there exists no peripheral cycle for $({\cal Y}, \{ {\cal W}_{j}\} _{j=1}^{m})$.  
Let $\Gamma \in \mbox{{\em Cpt}}({\cal Y})$ such that $\Gamma \cap \{ f_{j,\lambda }\mid \lambda 
\in \Lambda _{j}\} \neq \emptyset $ for each $j=1,\ldots, m.$ 
Let $L\in \emMin(\langle \Gamma \rangle ,\CCI )$ with $L\neq \CCI .$ 
Suppose  that $J_{\ker }(\langle \G \rangle )\subset \cap _{j=1}^{m}S({\cal W}_{j})$ and   
$\sharp L=\infty $. 
%and $L\not\subset \cap _{j=1}^{m}S({\cal W}_{j}).$ 
Suppose also that $L$ is not attracting for 
$\Gamma $.  
Then there exists an element $(g,j)\in \Gamma \times \{ 1,\ldots, m\} $ 
such that $g$ is a strict bifurcation element for $(\Gamma , L)$ with corresponding suffix $j.$ 
Moreover, if $(h,i)\in \Gamma \times \{ 1,\ldots, m\}  $ and $h$ is a strict bifurcation element 
for $(\Gamma ,L)$ with corresponding suffix $i$, then 
$h\in \partial (\Gamma \cap \{ f_{i,\lambda }\mid \lambda \in \Lambda _{i}\})$ 
with respect to the topology in 
$ \{ f_{i,\lambda }\mid \lambda \in \Lambda _{i}\}$. Here, 
we endow $\{ f_{i, \lambda }\mid \lambda \in \Lambda _{i}\}$ 
with the relative topology from Rat.  
 \end{lem} 
\begin{proof}
Let $G=\langle \Gamma \rangle .$ 
\cite[Lemma 3.8]{Sadv} implies that 
we have one of the following two situations (I)(II). 
\begin{itemize}
\item[(I)] 
There exist an element $(g,j)\in \G \times \{ 1,\ldots, m\} $ with 
$g\in \Gamma \cap \{ f_{j,\lambda }\mid \lambda \in \Lambda _{j}\} $ 
and a point $z_{0}\in L\cap J(G)$ 
such that $g(z_{0})\in L\cap J(G).$ 
\item[(II)] 
There exist an open subset $U$ of $\CCI $ with $U\cap L\neq \emptyset $ and 
finitely many elements $\gamma _{1},\ldots, \gamma _{r}\in \Gamma $ 
such that $\gamma _{r}\circ \cdots \gamma _{1}(U)\subset U$ and $U$ is a subset of a 
Siegel disk or a Hermann ring of $\gamma _{r}\circ \cdots \circ \gamma _{1}.$ 

\end{itemize}  
 Suppose we have case (II).  
 Since $\sharp L=\infty $, 
 by using \cite[Remark 3.9]{Splms10} and \cite[Remark 2.24]{Sadv}  
 we obtain that $\sharp (U\cap L)=\infty .$ 
%Hence $U\cap L$ contains an analytic Jordan curve.  
 Let $j\in \{ 1,\ldots, m\} $ with $\gamma _{1}\in \{ f_{j,\lambda }\mid \lambda \in \Lambda _{j}\} .$ 
 Since $\sharp S_{1}({\cal W}_{j})<\infty $ (Lemma~\ref{l:sn1sn}), it follows that $\gamma _{1}$ is a 
 strict bifurcation element with corresponding suffix $j.$ 
 %If $\gamma _{r}\in \mbox{int}(\Gamma \cap \{ f_{j,\lambda }\mid \lambda \in \Lambda _{j}\})$ 
 % then since $\sharp \cup _{j=1}^{m}S_{1}({\cal W}_{j})<\infty $,  
 %we obtain that $\mbox{int}(L)\cap J(G)\neq \emptyset .$ 
 %This implies that $L=\CCI .$ However, this is a contradiction.  
 
 Suppose we have case (I). 
Then there exist a sequence $\{ \gamma _{i}\} _{i=1}^{\infty }$ in $\Gamma $ 
with $\gamma _{i}\in \{ f_{j_{i},\lambda }\mid \lambda \in \Lambda _{j_{i}}\}$, 
$j_{i}\in \{ 1,\ldots, m\} $ and a point $z_{0}\in L\cap J(G)$ such that 
$\gamma _{i}\cdots \gamma _{1}(z_{0})\in L\cap J(G)$ for each $i.$ 
 We now consider the following two cases (a)(b). 
 \begin{itemize}
 \item[(a)] 
 There exists an $i\in \NN $ such that $\gamma _{i}\cdots \gamma _{1}(z_{0})\not\in S_{1}({\cal W}_{j_{i+1}}).$ 
 
 \item[(b)] 
For each $i\in \NN $,  $\gamma _{i}\cdots \gamma _{1}(z_{0})\in S_{1}({\cal W}_{j_{i+1}}).$ 
 \end{itemize}
Suppose we have case (a). 
Then $\gamma _{i+1}$ is a strict bifurcation element with corresponding suffix $j_{i+1}.$ 
%If $\gamma _{i+1}\in \mbox{int}(\Gamma _{\tau _{1}}\cap 
%\{ f_{j_{i+1},\lambda }\mid \lambda \in \Lambda _{j_{i+1}}\} $, 
%then $\mbox{int}(L)\cap J(G_{\tau _{1}})\neq \emptyset .$ It implies that 
%$L=\CCI .$ However, this is a contradiction. Hence, 
%$\gamma _{i+1}\in \partial (\Gamma _{\tau _{1}}\cap 
%\{ f_{j_{i+1},\lambda }\mid \lambda \in \Lambda _{j_{i+1}}\} ).$ Moreover, 
%$\gamma _{i+1}$ is a bifurcation element for $L.$ 

 Suppose we have case (b). 
 Since $\sharp L=\infty $ and $\sharp \cup _{j=1}^{m}S_{1}({\cal W}_{j})<\infty $ 
(Lemma~\ref{l:sn1sn}),  
 we have that $L\not\subset \cup _{j=1}^{m}S_{1}({\cal W}_{j}).$ 
 %We consider the following two cases $(\alpha) (\beta ).$  
 %\begin{itemize}
 %\item[($\alpha $)] $L\not\subset \cup _{j=1}^{m}S_{1}({\cal W}_{j}).$ 
 %\item[($\beta $)] $L\subset \cup _{j=1}^{m}S_{1}({\cal W}_{j}).$ 
 %\end{itemize} 
 %Suppose we have case ($\alpha $). 
 Then for each $i\in \NN $, 
 we have $$\gamma _{i}\cdots \gamma _{1}(z_{0})\in 
 (\cup _{j=1}^{m}S_{1}({\cal W}_{j}))\setminus \cup _{K\in \Min(G, \CCI ), 
 K\subset \cup _{j=1}^{m}S_{1}({\cal W}_{j})}K.$$ 
 Since we are assuming case (b) and since $\sharp \cup _{j=1}^{m}S_{1}({\cal W}_{j})<\infty $, 
 there exist two elements $i,j\in \NN $ with $j>i$ such that 
 $\gamma _{j}\cdots \gamma _{i}\cdots \gamma _{1}(z_{0})=\gamma _{i}\cdots \gamma _{1}(z_{0}).$ 
 This contradicts to the assumption that there exists no peripheral cycle 
for $({\cal Y}, \{ {\cal W}_{j}\} _{j=1}^{m})$.  
 %is strongly nice with $\{ {\cal W}_{j}\} _{j=1}^{m}.$ 
 
 %Spoose we have case ($\beta $). 
 %Since $\sharp \cup _{j=1}^{m}S_{1}({\cal W}_{j})<\infty $, we obtain $\sharp L<\infty .$ 
 %Moreover, $z_{0}\in J(G_{\tau _{1}})$ and $L=\overline{G_{\tau _{1}}(z)}$ for each $z\in L.$ 
 %Therefore $L\subset J(G_{\tau _{1}}).$ It implies that 
 %$L\subset J_{\ker }(G_{\tau _{1}})\subset \cap _{j=1}^{m}S({\cal W}_{j}).$   
 % However, this contradicts to the assumption that 
 % $L\not\subset \cap _{j=1}^{m}S({\cal W}_{j}).$ 
  
We now suppose that $(h,i)\in \Gamma \times \{ 1,\ldots, m\} $ is a strict 
bifurcation element for $(\Gamma, L)$ with corresponding suffix $i.$ 
Supoose that $h\in \mbox{int}(\Gamma \cap \{ f_{i,\lambda }\mid \lambda\in \Lambda _{i}\} )$ 
with respect to the topology in $\{ f_{i,\lambda }\mid \lambda \in \Lambda _{i}\} .$  
Then for each $z\in \CCI \setminus S_{1}({\cal W}_{i})$, we have that 
$\mbox{int}(\langle \Gamma \rangle (z))\neq \emptyset .$ 
Hence it is easy to see that $\mbox{int}(L)\cap J(G)\neq \emptyset .$ 
It implies that $L=\CCI .$ However, this contradicts the assumption of our lemma. 
Hence $h\in \partial (\Gamma \cap \{ f_{i,\lambda }\mid \lambda\in \Lambda _{i}\} ).$

Thus we have proved our lemma. 
%We now consider the following two cases (A)(B). 
%\begin{cases}
%\item[(A)]
%There exists a point $z_{0}\in (L\cap J(G))\setminus S_{1}({\cal W}_{i})$ such that 
%$h(z_{0})\in L\cap J(G).$ 
%
%\item[(B)]
%There exist an open set $U$ of $\CCI $ with $L\cap 
%There exist  finitely many elements $\gamma _{1},\ldots, \gamma _{r}\in \Gamma $ 
%and 
%
%\end{cases}
\end{proof}

\begin{lem}
\label{l:nalnotsb}
 Let ${\cal Y}$ be a  weakly nice subset of $\emRatp$ with respect 
 to some 
 holomorphic families 
$\{ {\cal W}_{j}\} _{j=1}^{m}$ of rational maps, where 
${\cal W}_{j}=\{ f_{j,\lambda }\}_{ \lambda \in \Lambda _{j}} , j=1,\ldots, m.$ 
Suppose that there exists no peripheral cycle for $({\cal Y}, \{ {\cal W}_{j}\} _{j=1}^{m})$.  
Let $\rho \in {\frak M}_{1,c}({\cal Y},\{ {\cal W}_{j}\} _{j=1}^{m})$ and 
suppose that the interior of $\supp\,\rho \cap \{ f_{j,\lambda }\mid \lambda \in \Lambda _{j}\}$ is not empty  
with respect to the topology in $\{ f_{j,\lambda }\mid \lambda \in \Lambda _{j}\}$ (which is endowed with the 
relative topology from Rat) for each 
$j=1,\ldots, m.$ Suppose also that $F(G_{\rho })\neq \emptyset .$ 
Then we have the following. 
\begin{itemize}
\item[{\em (i)}] 
$J_{\ker }(G_{\rho })\subset \cap _{j=1}^{m}S({\cal W}_{j})$, 
$\sharp J_{\ker }(G_{\rho })<\infty $ and
$\sharp \emMin(G_{\rho })<\infty .$ 

\item[{\em (ii)}] 
Let $L\in \emMin(G_{\rho},\CCI )$ with $L\not\subset \cap _{j=1}^{m}S({\cal W}_{j})$.  
Suppose that $L$ is not attracting for $\rho .$ 
Then there exists an element 
$(g, j)\in \supp\, \rho \times \{ 1,\ldots, m\} $ 
such that $g$ is a strict bifurcation element  for 
$(\supp\,\rho , L)$ with 
corresponding suffix $j$. 
Moreover, if $(h,i)\in  \supp\,\rho \times \{ 1,\ldots, m\} $ such that 
 $h$ is a strict bifurcation element  for $(\supp\, \rho , L)$ with 
corresponding suffix $i$, then 
$h$ belongs to the boundary of $\supp\, \rho \cap \{ f_{i,\lambda }\mid \lambda \in \Lambda _{i}\}$, 
where the boundary of  $\supp\,\rho \cap \{ f_{i,\lambda }\mid \lambda \in \Lambda _{i}\}$ 
is taken with respect to the topology in 
$\{ f_{i,\lambda }\mid \lambda \in \Lambda _{i}\} .$ 

\item[{\em (iii)}] 
Suppose that there exists an element $L_{0}\in \emMin(G_{\rho },\CCI )$ which is 
attracting for $\rho .$ Then 
there exists an open neighborhood $V$ of $\rho $ in 
$({\frak M}_{1,c}({\cal Y},\{ {\cal W}_{j}\} _{j=1}^{m}),{\cal O})$ 
such that for each $\rho _{1}\in V$ satisfying that 
$\supp\,\rho\cap \{ f_{j,\lambda }\mid \lambda \in \Lambda _{j}\} 
\subset \mbox{int}(\supp\, \rho _{1}\cap \{ f_{j,\lambda }\mid 
\lambda \in \Lambda _{j}\})$ with respect to the topology in 
$\{ f_{j,\lambda }\mid \lambda \in \Lambda _{j}\} $ for each $j=1,\ldots, m$, 
we have the following {\em (a)(b)(c)(d)}.
\begin{itemize}
\item[{\em (a)}] 
$\sharp \emMin(G_{\rho _{1}},\CCI )=\\ 
\sharp (\{ L'\in \emMin (G_{\rho },\CCI )\mid 
L'\subset \cap _{j=1}^{m}S({\cal W}_{j}) \} ) $ \\ 
$+ \sharp \{ L'\in \emMin(G_{\rho },\CCI )\mid L'\not\subset \cap _{j=1}^{m}S({\cal W}_{j}) \mbox{ and }L' \mbox{ is attracting for }\rho \}.$  
\item[{\em (b)}] 
For each $L\in \emMin(G_{\rho _{1}},\CCI )$ there exists a unique 
$L'\in \emMin(G_{\rho },\CCI )$ with $L'\subset L$ such that  
either ``$L'\subset \cap _{j=1}^{m}S({\cal W}_{j})$''  or  
``$L'\not\subset \cap _{j=1}^{m}S({\cal W}_{j}) \mbox{ and }L' \mbox{ is attracting for }\rho $''.   
\item[{\em (c)}]
In item {\em (b)},  
if $L'\subset \cap _{j=1}^{m}S({\cal W}_{j})$, then $L=L'.$ 
If $L'\not\subset \cap _{j=1}^{m}S({\cal W}_{j})$ and $L'$ is attracting for $\rho $, 
then $L$ is attracting for $\rho _{1}.$ 
\item[{\em (d)}] 
Each $L\in \emMin(G_{\rho _{1}},\CCI )$ with 
$L\not\subset \cap _{j=1}^{m}S({\cal W}_{j})$ is attracting for $\rho _{1}.$ 
\end{itemize}
\item[{\em (iv)}] 
Suppose that each element $L_{0}\in \emMin(G_{\rho },\CCI )$ is not attracting for $\rho .$ 
Let $\rho _{1}\in {\frak M}_{1,c}({\cal Y}, \{ {\cal W}_{j}\} _{j=1}^{m})$ be an element such that 
$\supp\,\rho \cap \{ f_{j,\lambda}\mid \lambda \in \Lambda _{j}\}\subset $ $
\mbox{int}(\supp\,\rho _{1}\cap \{ f_{j,\lambda }\mid \lambda \in \Lambda _{j}\})$ 
with respect to the topology in $\{ f_{j.\lambda }\mid \lambda \in \Lambda _{j}\}$ 
for each $j=1,\ldots, m.$ Then we have the following.   
\begin{itemize}
\item[{\em (a)}] 
If there exists an element $L\in \emMin(G_{\rho },\CCI )$ with $L\subset 
\cap _{j=1}^{m}S({\cal W}_{j})$, then 
$\emMin(G_{\rho _{1}},\CCI )=\{ L\in \emMin(G_{\rho },\CCI )\mid 
L\subset \cap _{j=1}^{m}S({\cal W}_{j})\} .$ 

\item[{\em (b)}] 
If there exists no $L\in \emMin(G_{\rho },\CCI )$ with $L\subset 
\cap _{j=1}^{m}S({\cal W}_{j})$, then 
$\emMin(G_{\rho _{1}},\CCI )=\CCI $ and $J(G_{\rho _{1}})=\CCI .$ 
\end{itemize}
\end{itemize}
  \end{lem} 
\begin{proof}
By Lemma~\ref{l:yrtfaji}, we obtain that $J_{\ker }(G_{\rho})\subset \cap _{j=1}^{m}S({\cal W}_{j}).$ 
 Thus by Lemma~\ref{l:sn1sn}, $\sharp J_{\ker }(G_{\rho})<\infty .$ 
 From  Proposition~\ref{p:jkgfmf}, it follows that 
 $\sharp \Min(G_{\rho},\CCI )<\infty .$  
Thus statement (i) holds. 

To prove statement (ii), 
since $L\not\subset \cap _{j=1}^{m}S({\cal W}_{j})$ and 
since $\mbox{int}(\Gamma _{\rho }\cap \{ f_{j,\lambda }\mid \lambda \in \Lambda _{j}\})\neq \emptyset $ 
with respect to the topology in $\{ f_{j,\lambda }\mid \lambda \in \Lambda _{j}\}$ for each 
$j=1,\ldots, m$, we obtain that $\sharp L=\infty .$ 
Moreover, since int$(\supp\,\rho \cap \{ f_{j, \lambda }\mid \lambda \in \Lambda _{j}\})\neq \emptyset $ 
for each $j$ and since 
$J(G_{\rho })$ is perfect (see \cite{HM}) and 
$J(G_{\rho })\setminus \cup _{j=1}^{m}S_{1}({\cal W}_{j})\neq \emptyset $, 
we have int$(J(G_{\rho }))\neq \emptyset .$  Combining this with the assumption 
$F(G_{\rho })\neq \emptyset $,  we obtain that $\CCI $ cannot be a minimal set for $(G_{\rho },\CCI ).$ 
Thus statement (ii) follows from Lemma~\ref{l:strbif}. 

   To prove statement (iii),  
  let $V$ be a small open neighborhood $V$ of $\rho $ in 
$({\frak M}_{1,c}({\cal Y},\{ {\cal W}_{j}\} _{j=1}^{m}),{\cal O})$ 
 and let $\rho _{1}\in V$ such that 
$\supp\,\rho \cap \{ f_{j,\lambda }\mid \lambda \in \Lambda _{j}\} 
\subset \mbox{int}(\supp\, \rho _{1}\cap \{ f_{j,\lambda }\mid 
\lambda \in \Lambda _{j}\})$.  Taking $V$ small enough, we have that 
for each $\rho '\in V$, $F(G_{\rho '})\neq \emptyset .$    
   By Zorn's lemma, for each $L\in \Min(G_{\rho _{1}},\CCI )$ there exists an element 
   $L'\in \Min(G_{\rho },\CCI)$ with $L'\subset L.$ 
   If $L'\not\subset \cap _{j=1}^{m}S({\cal W}_{j})$ and $L'$ is not attracting for 
   $\rho $, then statement (ii) (for $\rho $ and $\rho _{1}$) implies a contradiction. 
%   that there exists a bifurcation element $g\in \Gamma _{\rho }$ for $(L', \rho )$ with 
%   $g\in \cup _{j=1}^{m}\partial (\Gamma _{\rho }\cap \{ f_{j,\lambda }\mid 
%   \lambda \in \Lambda _{j}\} ).$ 
%   Then $g\in \cap _{j=1}^{m}\mbox{int}(\Gamma _{\rho _{1}}\cap 
%   \{ f_{j,\lambda }\mid \lambda \in \Lambda _{j}\}).$ 
%   By the argument in the above (we consider cases (I)(II) etc.), 
%   we obtain that $\mbox{int}(L)\cap J(G_{\rho _{1}})\neq \emptyset .$ 
%   It implies that $L=\CCI .$ However, this is a contradiction.  
   Hence either $L'\subset \cap _{j=1}^{m}S({\cal W}_{j})$ or $L' $ is attracting for 
   $\rho .$ If $L'\subset \cap _{j=1}^{m}S({\cal W}_{j})$, then Lemma~\ref{l:lcbwj} implies 
   that $L'=L.$ Suppose $L'\not\subset \cap _{j=1}^{m}S({\cal W}_{j})$ and 
   $L'$ is attracting for $\rho .$ Then taking $V$ so small, 
   \cite[Remark 3.6, Lemma 5.2]{Sadv} implies that $L$ is attacting for $\rho _{1}$ and 
   there is no $L''\in \Min (G_{\rho },\CCI)$ with $L''\neq L'$ such that $L''\subset L.$ 
Also, by Lemma~\ref{l:lcbwj} again, for any $K\in \Min(G_{\rho }, \CCI )$ with 
$K\subset \cap _{j=1}^{m}S({\cal W}_{j})$, we have $K\in \Min(G_{\rho _{1}},\CCI )$. 
Moreover, by \cite[Lemma 5.2]{Sadv} again, for any $K\in \Min(G_{\rho }, \CCI )$ with 
$K\not\subset \cap _{j=1}^{m}S({\cal W}_{j})$ which is attracting for $\rho $, 
there exists a unique element $\tilde{K}\in \Min(G_{\rho _{1}}, \CCI )$ close to $K$, 
and this $\tilde{K}$ is attracting for $\rho _{1}.$  
   From these arguments, statement (iii) follows.  
   
We now prove statement (iv). 
Suppose that each $L_{0}\in \Min(G_{\rho },\CCI )$ is not attracting for $\rho .$ 
Let $\rho _{1}\in {\frak M}_{1,c}({\cal Y}, \{ {\cal W}_{j}\} _{j=1}^{m})$ be an element such that 
$\supp\,\rho \cap \{ f_{j,\lambda}\mid \lambda \in \Lambda _{j}\}\subset $ $
\mbox{int}(\supp\,\rho _{1}\cap \{ f_{j,\lambda }\mid \lambda \in \Lambda _{j}\})$ 
%with respect to the topology in $\{ f_{j.\lambda }\mid \lambda \in \Lambda _{j}\}$ 
for each $j=1,\ldots, m.$ 
Let $L\in \Min(G_{\rho _{1}},\CCI )$. Suppose that 
$L\neq \CCI $ and $L\not\subset \cap _{j=1}^{m}S({\cal W}_{j}).$ 
Then $\emptyset \neq \mbox{int}(L)$, 
$\sharp (\CCI \setminus \mbox{int}(L))\geq 3$ and $G_{\rho _{1}}(\mbox{int}(L))\subset 
\mbox{int}(L).$ Hence 
$\emptyset \neq \mbox{int}(L)\subset F(G_{\rho _{1}}).$ 
Also, $L$ is not attracting for $\rho _{1}$ (otherwise 
by Zorn's lemma there exists an element $L_{0}\in \Min(G_{\rho }, L)$ which is  attracting for $\rho $). 
By applying statement (ii) for $\rho $ and $\rho _{1}$, we obtain a contradiction. 
Thus either $L=\CCI $ or $L\subset \cap _{j=1}^{m}S({\cal W}_{j}). $ 
If $L=\CCI $, then since int$(J(G_{\rho _{1}}))\neq \emptyset $ (see the argument in the proof of (ii)), 
we obtain that $F(G_{\rho _{1}})= \emptyset .$ Hence statements (a) and (b) in (iv) hold.   

   Thus we have proved our lemma. 
\end{proof}      
\begin{df}
\label{df:wms}
%Let ${\cal Y}$ be a weakly nice subset of $\emRat$ with holomorphic families 
%$\{ {\cal W}_{j}\} _{j=1}^{m}$ of rational maps. 
Let $\Gamma \in \Cpt(\Rat).$ 
We say that $\Gamma $ is {\bf weakly mean stable} if 
there exist a positive integer $n$ and two non-empty open subsets 
$V_{1,\Gamma }, V_{2,\Gamma }$ of $\CCI $ 
with $\overline{V_{1,\Gamma }}\subset V_{2,\Gamma }$ and $\sharp (\CCI \setminus V_{2,\Gamma })\geq 3$ 
such that the following three conditions hold.
\begin{itemize}
\item[(a)] 
For each $(\gamma _{1},\ldots, \gamma _{n})\in \Gamma ^{n}$, 
$\gamma _{n}\circ \cdots \circ \gamma_{1}(V_{2,\Gamma })\subset V_{1,\Gamma }.$ 

\item[(b)] 
Let $D_{\Gamma }:=\cap _{g\in \langle \Gamma \rangle }g^{-1}(\CCI \setminus V_{2,\Gamma }).$ 
Then $\sharp D_{\Gamma }<\infty .$ 

\item[(c)] 
For each $L\in \Min(\langle \Gamma\rangle ,D_{\Gamma })$ there exist an element $z\in L$ and an element 
$g_{z}\in \langle \Gamma \rangle $ 
such that $z$ is a repelling fixed point of $g_{z}.$

\end{itemize}
Moreover, we say that $\tau \in {\frak M}_{1,c}(\Rat)$ is weakly mean stable 
if supp$\,\tau $ is weakly mean stable. If $\tau \in {\frak M}_{1,c}(\Rat)$ is 
weakly mean stable, then we set 
$V_{i,\tau }=V_{i,\mbox{supp}\,\tau}$ and $D_{\tau }=D_{\mbox{supp}\,\tau}.$ 
 \end{df} 
\begin{rem}
\label{r:mswms}
If $\Gamma \in \Cpt(\Rat)$ is mean stable and 
$\sharp J(\langle \Gamma \rangle )\geq 3$, then 
$\Gamma $ is weakly mean stable. Moreover, 
if $\Gamma \in \Cpt(\Rat)$  is weakly mean stable and 
$D_{\Gamma }=\emptyset $, then $\Gamma $ is mean stable. 
\end{rem} 

 \begin{lem}
\label{l:wmsopen}
% Let ${\cal Y}$ be a weakly nice subset of $\emRat$ with holomorphic families 
%$\{ {\cal W}_{j}\} _{j=1}^{m}$ of rational maps. 
Let ${\cal A}:=\{ \Gamma \in \emCpt(\emRat)\mid 
\Gamma \mbox{ is weakly mean stable} \}.$ Then 
${\cal A}$ is open in $\emCpt(\emRat).$ 
In particular, the set ${\cal A}':=\{ \tau \in {\frak M}_{1,c}(\emRat)\mid 
\tau \mbox{ is weakly mean stable}\} $ is open in $({\frak M}_{1,c}(\emRat), {\cal O}).$  
  \end{lem} 
\begin{proof}
Let $\Gamma \in {\cal A}.$ For this mean stable $\Gamma $, 
let $V_{1,\Gamma },V_{2,\Gamma },n$  as in Definition~\ref{df:wms}. 
Let $V_{1,\Gamma }'$ be an open subset of $\CCI $ such that 
$\overline{V_{1,\Gamma }}\subset V_{1,\Gamma }'\subset \overline{V_{1,\Gamma }'}\subset 
V_{2,\Gamma }.$ 
 Then, 
 since the topology in Rat is the compact-open topology, 
  there exists a neighborhood ${\cal U}$ of $\Gamma $ in 
 %$({\frak M}_{1,c}(\Rat),{\cal O})$ 
 $\Cpt(\Rat)$ 
 such that for each $\Lambda \in {\cal U}$ and for each $(\gamma _{1},
 \ldots, \gamma _{n})\in \Lambda^{n}$, we have 
 $\gamma_{n}\circ \cdots \circ \gamma _{1}(V_{2,\Gamma })\subset V_{1,\Gamma }'.$   
 
 If $D_{\Gamma }=\emptyset $ (i.e., $\Gamma $ is mean 
 stable), then by \cite[Lemma 5.7]{Sadv}, there exists an 
 open neighborhood ${\cal B}$ of $\Gamma $ in 
 $\Cpt(\Rat)$ such that for each $\Lambda\in {\cal B}$, 
 $\Lambda $ is mean stable and weakly mean stable. Thus, 
 we may assume that $D_{\Gamma }\neq \emptyset .$ 
 
 For each $L\in \Min(\langle \Gamma \rangle ,D_{\G })$, let $z_{L}\in L$ and $g_{L}\in 
 \langle \G \rangle $ such that 
$z_{L}$ is a repelling fixed point of $g_{L}.$ 
%Let $M=\max \{ \max \{ \| Dh_{z}\| _{s}\mid h\in \Gamma _{\rho }, \rho \in {\cal U}, z\in \CCI \}, 1\} . %$ 
%Let $\epsilon >0$ be a small number 
%such that $3M\epsilon <\min \{ d(a,b)\mid a,b\in \cup _{L\in R_{\tau }}L, a\neq b\}.$ 
%such that 
%for each $L\in R_{\tau }$ and for each $z\in B(z_{L}, \epsilon )\setminus \{ z_{L}\} $ 
%there exists an element $n\in \NN $ such that $g_{L}^{n}(z)\in B(z_{L},2M\epsilon )\setminus %B(z_{L},2\epsilon).$ 
Let $\epsilon >0$ be a small number. 
By considering linearizing coordinate for $g_{L}$ at $z_{L}$ and 
the fundamental region for $g_{L}$ near $z_{L},$  
it is easy to see that for each $L\in \Min(\langle \G \rangle , D_{\G })$ 
there exist small simply connected open neighborhoods 
$H_{L,\G ,1}, H_{L,\G ,2}$ of $z_{L}$ with $\overline{H_{L,\G,2}}\subset 
H_{L,\G ,1}$ such that for each $z\in B(z_{L},\epsilon )\setminus 
\{ z_{L}\} $ there exists an element $n\in \NN $ such that 
$g_{L}^{n}(z)\in H_{L,\G ,1}\setminus H_{L,\G ,2}.$ 
 
Shrinking ${\cal U}$ if necessary, we may assume that 
for each $\Lambda \in {\cal U}$ and  for each $L\in \Min(\langle \G \rangle  , D_{\Gamma })$
 there exist $z_{L,\Lambda }\in B(z_{L},\frac{\epsilon }{2}) $ and  $g_{L,\Lambda }\in 
 \langle \Lambda \rangle $ 
such that $z_{L, \Lambda }$ is a repelling fixed point of $g_{L,\Lambda }$ and such that  
$g_{L,\Lambda }\rightarrow g_{L}$ and $z_{L,\Lambda }\rightarrow z_{L}$ 
as $\Lambda \rightarrow \G .$ Since the linearizing coordinate for a repelling 
fixed point 
is continuous 
on $\Rat $, if ${\cal U}$ is small enough, then 
for each $\Lambda \in {\cal U}$, for each $L\in \Min(\langle \G \rangle , D_{\G })$ there 
exist two small simply connected open neighborhoods $H_{L,\Lambda, 1},  
H_{L,\Lambda ,2}$ of $z_{L,\Lambda }$ with $\overline{H_{L,\Lambda ,2}}\subset 
H_{L,\Lambda ,1}$ such that the following hold.
\begin{enumerate}
\item For each $z\in B(z_{L,\Lambda },\epsilon )\setminus 
\{ z_{L,\Lambda }\}$ there exists an element $n\in \NN $ 
%such that 
with 
$g_{L,\Lambda }^{n}(z)\in H_{L,\Lambda ,1}\setminus H_{L,\Lambda ,2}.$

\item 
There exist two small numbers $\epsilon _{1},\epsilon _{2}>0$ 
with $\epsilon _{1}<\epsilon _{2}<\frac{1}{3}\min \{ d(a,b)\mid a,b\in D_{\G }, a\neq b\} $ (if $\sharp D_{\G}=1$ then 
we set $\min \{ d(a,b)\mid a,b\in D_{\G}, a\neq b\}=1$) 
such that 
for each $\Lambda \in {\cal U}$ and for each $L\in \Min(\langle \G \rangle, D_{\G })$, 
\begin{equation}
\label{eq:bze1h}
B(z_{L},\epsilon _{1})\subset H_{L,\Lambda ,2}, \ H_{L,\Lambda ,1}\subset 
B(z_{L},\epsilon _{2}).
\end{equation} 

\end{enumerate} 
For each $w\in D_{\G }$, let $L_{w}\in \Min(\langle \G \rangle, D_{\G })$ 
be an element such that $\langle \G \rangle (w)\cap L_{w}\neq \emptyset .$ 
Moreover, let $h_{w}\in \langle \G \rangle $ such that $h_{w}(w)=z_{L_{w}}.$ 
Taking ${\cal U}$ small enough, there exists a $\delta >0$ with 
\begin{equation}
\label{eq:deltale1}
\delta <\epsilon _{1}, \delta <\frac{1}{3}\min \{ d(a,b)\mid a,b\in D_{\G },a\neq b\}.
\end{equation}
such that for each $\Lambda \in {\cal U} $ and for each $w\in D_{\G }$, 
there exists an element $h_{w,\Lambda }\in \langle \Lambda \rangle $  close to $h_{w}$ 
such that 
\begin{equation}
\label{eq:hwrhob}
h_{w,\Lambda }(B(w,\delta ))\subset B(z_{L_{w}}, \frac{\epsilon }{2})\subset 
B(z_{L_{w},\Lambda },\epsilon ).
\end{equation}

Let $K_{\delta }=\CCI \setminus B(D_{\G }, \delta ).$  
Then for each $z\in K_{\delta }$ there exists an element $\alpha_{z}\in \langle \G \rangle $ 
such that $\alpha_{z}(z)\in V_{2,\G }.$ Since $K_{\delta }$ is compact, 
there exist a finite set $\{ z_{1},\ldots, z_{q}\} $ in $K_{\delta }$, 
a number $\epsilon _{0}>0$ and elements $\beta_{1},\ldots, \beta_{q}\in \langle \G \rangle $ 
such that 
\begin{equation}
\label{eq:kdeltasub}
K_{\delta }\subset \cup _{j=1}^{q}B(z_{j},\epsilon _{0})
\end{equation}
 and $\beta _{j}(\overline{B(z_{j},\epsilon _{0})})\subset V_{2,\G }$ for all 
$j=1,\ldots, q.$ 
Hence shrinking ${\cal U}$  if necessary, 
%  $({\frak M}_{1,c}({\cal Y}, \{ {\cal W}_{j}\} _{j=1}^{m}),{\cal O})$ such that 
we have that 
for each $\Lambda \in {\cal U}$, there exist elements 
$\beta_{1,\Lambda },\ldots ,\beta_{q,\Lambda }\in \langle \Lambda \rangle $ 
such that 
\begin{equation}
\label{eq:betajrho}
\beta_{j,\Lambda }(B(z_{j},\epsilon _{0}))\subset V_{2,\G},  
 \mbox{ for all }j=1,\ldots, q.
\end{equation} 
We now let $\Lambda \in {\cal U}$ and let $z_{0}\in \cap _{g\in \langle \Lambda \rangle }
g^{-1}(\CCI \setminus V_{2,\G }).$ 
Then by (\ref{eq:kdeltasub}) and (\ref{eq:betajrho}), we have $z_{0}\not\in K_{\delta }.$ 
Thus $z_{0}\in B(D_{\G },\delta ).$ 
Moreover, by  (\ref{eq:bze1h}) and (\ref{eq:deltale1}), we have 
$H_{L,\Lambda ,2}\setminus H_{L,\Lambda ,1}\subset K_{\delta }$ for all 
$L\in \Min(\langle \G \rangle , D_{\G })$ and for all $\Lambda \in {\cal U}.$ 
Combining this with (\ref{eq:hwrhob}) (\ref{eq:kdeltasub}) 
(\ref{eq:betajrho}), we obtain that 
taking an element $w\in D_{\G }$ with $d(z_{0},w)<\delta $, we have 
$h_{w,\Lambda }(z_{0})=z_{L_{w},\Lambda }$ for all $\Lambda \in {\cal U}.$ It follows that 
\begin{equation}
\label{eq:capggrho}
\bigcap _{g\in \langle \Lambda \rangle }g^{-1}(\CCI \setminus V_{2,\G })
\subset \bigcup _{L\in \Min(\langle \G \rangle, D_{\G }), 
w\in D_{\G }}
h_{w,\Lambda }^{-1}(z_{L,\Lambda }) \ \ \mbox{ for all }\Lambda \in {\cal U}. 
\end{equation}
Since the right hand side of the above is a finite set, we obtain that 
$\sharp D_{\Lambda }<\infty $, where 
$D_{\Lambda }:=\cap _{g\in \langle \Lambda \rangle }g^{-1}(\CCI \setminus V_{2,\G }).$ 
Moreover, by (\ref{eq:capggrho}), 
we have that for each $K\in \Min(\langle \Lambda \rangle , D_{\Lambda })$ there exist 
an element $z\in K$ and an element $\zeta_{z}\in \langle \Lambda \rangle $ such that 
$z$ is a repelling fixed point of $\zeta _{z}.$ Thus 
$\Lambda $ is weakly mean stable. 
Hence we have proved our lemma.  
\end{proof}   
 
\begin{lem}
\label{l:wmsdtjk}
Let $\G \in \emCpt(\emRat)$ be weakly mean stable. 
Let $D_{\G }$ be as in Definition~\ref{df:wms}. Then 
$D_{\G }=J_{\ker }(\langle \G \rangle )$, $\sharp (J_{\ker }(\langle \G \rangle))<\infty $, 
and for each $z\in F(\langle \G \rangle )$, we have 
\vspace{-2mm} 
$$\overline{\langle \G \rangle (z)}\cap 
\left(\bigcup _{L\in \emMin(\langle \G \rangle ,\CCI ), 
L\not\subset J_{\ker }(\langle \G \rangle )}L\right)\neq \emptyset .$$ 
In particular, if $\tau \in {\frak M}_{1,c}(\emRat)$ is weakly mean stable and
 $\sharp J(G_{\tau })\geq 3$, 
then statements (i)-(vii) in Theorem~\ref{t:zfggzcc} hold 
for $\tau $.  

\end{lem}
\begin{proof}
By definition of $D_{\G }$, we have 
$\langle \G \rangle (D_{\G })\subset D_{\G }.$ 
Also, by condition (c) in Definition~\ref{df:wms}, 
we have $D_{\G }\subset J(\langle \G \rangle ).$ 
Thus $D_{\G }\subset J_{\ker }(\langle \G \rangle ).$  
Let $V_{2,\G }$ be as in Definition~\ref{df:wms}. 
Then $V_{2,\G }\subset F(\langle \G \rangle ).$  
Since $\langle \G \rangle (J_{\ker }(\langle \G \rangle ))\subset 
J_{\ker }(\langle \G \rangle ) 
\subset J(\langle \G \rangle )\subset \CCI \setminus V_{2,\G }$,  
we obtain $J_{\ker }(\langle \G \rangle )\subset D_{\G }.$ 
Hence we have  $D_{\G }=J_{\ker }(\langle \G \rangle ).$  
By definition of weakly mean stable elements again, 
we  have $\sharp D_{\G }<\infty .$ Thus $\sharp J_{\ker }(\langle \G \rangle )<\infty .$ 
Let $V_{2,\G }$ be as in Definition~\ref{df:wms} for $\G .$  
Let $z\in F(\langle \G \rangle ).$ Since 
$z\not\in J_{\ker }(\langle \G \rangle )$ and 
$D_{\G }=J_{\ker }(\langle \G \rangle )$, 
it follows that $\langle \G \rangle (z)\cap V_{2,\G }\neq \emptyset .$ 
Thus $\overline{\langle \G \rangle (z)}\cap 
 (\bigcup _{L\in \Min(\langle \G \rangle ,\CCI ), 
 L\not\subset J_{\ker }(\langle \G \rangle )}L)\neq \emptyset .$ 
If
$\tau \in {\frak M}_{1, c}(\Rat)$ is weakly mean stable and  $\sharp J(G_{\tau })\geq 3$, 
then combining the above argument and Theorem~\ref{t:zfggzcc} implies that statements (i)--(vii) in Theorem~\ref{t:zfggzcc} 
hold for $\tau. $   
\end{proof} 
\begin{lem}
\label{l:unifea}
Let $\G \in \emCpt(\emRat).$ Let $G=\langle \G \rangle .$ 
Let $L\in \emMin(G,\CCI )$ with $\sharp L<\infty .$ 
Then we have the following. 
\begin{itemize}
\item[{\em (i)}] 
Suppose that for each $z\in L$ and for each $g\in G$ with $g(z)=z$, we have 
$\| Dg_{z}\|_{s}>1.$ Then 
there exist a constant $C_{1}>0$ and a constant $\alpha >1$ such that 
for each $\gamma \in \GN $, for each $n\in \NN $ and for each $z\in L,$ we have  
$\| D(\gamma _{n,1})_{z}\| _{s}\geq C_{1}\alpha ^{n}.$  

\item[{\em (ii)}] 
Suppose that for each $z\in L$ and for each $g\in G$ with $g(z)=z$, we have 
$\| Dg_{z}\|_{s}<1.$ Then 
there exist a constant $C_{2}>0$ and a constant $\beta <1$ such that 
for each $\gamma \in \GN $, for each $n\in \NN $ and for each $z\in L,$ we have  
$\| D(\gamma _{n,1})_{z}\| _{s}\leq C_{2}\beta ^{n}.$ 
\end{itemize}

\end{lem}
\begin{proof}
We first prove statement (i). 
We show the following claim. \\ 
Claim 1. Under the assumptions of our lemma and statement (i), 
let $k\in \NN $ with $1\leq k\leq \sharp L.$ 
Then there exist a constant $A_{k}>0$ and a constant $\alpha _{k}>1$ 
such that for any subset $H\subset L$ with $\sharp H=k$, 
for any $n\in \NN $,  for any $z\in H$ and for any $\gamma \in \GN $, 
if  $\gamma _{j,1}(z)\in H$ for each $j=1,\ldots,  n$, 
then $\| D(\gamma _{n,1})_{z}\| _{s}\geq A_{k}\alpha _{k}^{n}.$ 

To prove this claim, we use the induction on $k.$ 
Apparently, the statement of the conclusion of the claim holds for $k=1.$  
Suppose that the statement of the conclusion of the claim holds for $k$, where 
$1\leq k<\sharp L.$ 
Let $u\in \NN $ with $u\geq \sharp L +1$ such that 
for each $u'\in \NN $ with $u'\geq u$, we have 
\begin{equation}
\label{eq:mwldg}
(\min _{w\in L}\| Dg_{w}\| _{s})A_{k}\alpha _{k}^{u'}\geq 2.
\end{equation} 
For this $u$, let 
\begin{equation}
\label{eq:bmdrr}
B:=\min \{ \| D(\rho _{r}\circ \cdots \circ \rho _{1})_{w}\| _{s}
\mid w\in L, r\leq u, (\rho _{1},\ldots, \rho _{r})\in \G ^{r}, 
\rho _{r}\circ \cdots \circ \rho _{1}(w)=w\}>1.
\end{equation}
Also, let $v\in \NN $ be a large number such that 
\begin{equation}
\label{eq:mwldgws}
(\min _{w\in L}\| Dg_{w}\| _{s})A_{k}\alpha _{k}^{uv}
\cdot (\min \{ \| D(\rho _{r}\circ \cdots \circ \rho _{1})_{w}\| _{s}\mid 
w\in L, r\leq u, (\rho _{1},\ldots, \rho _{r})\in \G^{r}\})>2.
\end{equation}  
Let $p\in \NN $ be a large number such that 
\begin{equation}
\label{eq:bpcmd}
B^{p}\cdot \min \{ \| D(\rho _{r}\circ \cdots \circ \rho _{1})_{w}\| _{s}\mid 
w\in L, r\leq u, (\rho _{1},\ldots, \rho _{r})\in \G ^{r}\}>2.
\end{equation}
Let $n\in \NN $ with $n>puv.$ 
Let $H\subset L$ with $\sharp H=k+1.$  
Let $\gamma \in \GN , z\in H$ and suppose 
that $\gamma _{j,1}(z)\in H$ for each $j=1,\ldots ,n.$ 
Let $j_{1},\ldots, j_{m}\in \NN $ with 
$1\leq j_{1}<j_{2}<\cdots <j_{m}\leq n$ such that 
$\gamma _{j_{i},1}(z)=z$ for each $i=1,\ldots, m$ and 
$\gamma _{l,1}(z)\neq z$ for each $l\in \{ 1,\ldots, n\} \setminus \{ j_{i}\mid i=1,\ldots, m\}.$ 
Also, let $j_{0}:=0$ and $j_{m+1}=n.$  
(If there is no $j\in \NN $ such that $\gamma _{j,1}(z)=z$, then 
we set $j_{0}=0, m=0, j_{1}=n.$) 
We now want to show that $\| D(\gamma _{n,1})_{z}\| _{s}\geq 2.$ In order to do that, 
we consider the following three cases 1,2,3. 

Case 1. $j_{m+1}-j_{m}> u.$ In this case, by the definition of $\{ j_{i}\}$, assumptions of our lemma and 
(\ref{eq:mwldg}), we obtain that $\| D(\gamma _{n,1})_{z}\| _{s} \geq 2.$ 

Case 2. $j_{m+1}-j_{m}\leq u$ and there exists an element $q\in \NN \cup \{ 0\} $ 
with $0\leq q\leq m-1$ such that $j_{q+1}-j_{q}>uv.$  In this case, 
by the definition of $\{ j_{i}\}$, assumptions of our lemma and (\ref{eq:mwldgws}), 
we obtain that $\| D(\gamma _{n,1})_{z}\| _{s}\geq 2.$ 

Case 3.  $j_{m+1}-j_{m}\leq u$ and for each $i\in \NN \cup \{ 0\} $ 
with $0\leq i\leq m-1$, $j_{i+1}-j_{i}\leq uv.$ 
In this case, we have $puv<n=\sum _{i=0}^{m}(j_{i+1}-j_{i})\leq (m+1)uv.$ 
Hence $m\geq p.$ Combining this with the definition of $\{ j_{i}\}$ and 
(\ref{eq:bpcmd}), we obtain that 
$$\| D(\gamma _{n,1})_{z}\| _{s}\geq 
B^{m}\cdot \min \{ \| D(\rho _{r}\circ \cdots \circ \rho _{1})_{w}\| _{s}
\mid r\leq u, (\rho _{1},\ldots, \rho _{r})\in \Gamma ^{r}\} \geq 2.$$ 
From these arguments, the induction step for $k+1$ is complete. Thus 
we have proved Claim 1. By Claim 1, statement (i) of our lemma holds. 

By the similar method to the above, we can show that statement (ii) of our lemma holds. 

Thus we have proved our lemma. 
\end{proof}
We now prove the following theorem, which is one of the main results of this paper. 
\begin{thm}
\label{t:rcdnkmain1}
Let ${\cal Y}$ be a mild subset of $\emRatp$ and suppose that 
${\cal Y}$ is strongly nice  with respect to some holomorphic families 
$\{ {\cal W}_{j}\} _{j=1}^{m}$ of rational maps. Then 
the set 
$$\{ \tau \in {\frak M}_{1,c}({\cal Y},\{ {\cal W}_{j}\} _{j=1}^{m})
\mid \tau \mbox{ is weakly mean stable}\}$$ 
is open and dense in $({\frak M}_{1,c}({\cal Y},\{ {\cal W}_{j}\} _{j=1}^{m}), {\cal O})$. 
Moreover, 
there exists the largest  open and dense 
subset ${\cal A}$ of $({\frak M}_{1,c}({ \cal Y},\{ {\cal W}_{j}\} _{j=1}^{m}), {\cal O})$ such that 
for each $\tau \in {\cal A}$, all of the following statements {\em (i)--(v)} hold.

\begin{itemize}
\item[{\em (i)}] 
$\tau $ is weakly mean stable.
\item[{\em (ii)}] 
Let $D_{\tau }$ be as in Definition~\ref{df:wms} for $\tau .$ Then 
$\sharp J_{\ker }(G_{\tau })<\infty $, $D_{\tau }=J_{\ker }(G_{\tau })\subset \cap _{j=1}^{m}S({\cal W}_{j})$ and $\sharp \emMin(G_{\tau },\CCI )<\infty .$ 
\item[{\em (iii)}] 
For each $L\in \emMin(G_{\tau },\CCI )$ with $L\not\subset J_{\ker }(G_{\tau }) $, we have that 
$L$ is attracting for $\tau .$ 
\item[{\em (iv)}] 
For each $z\in F(G_{\tau })$, 
we have that 
$\overline{G_{\tau }(z)}\cap (\cup _{L\in \emMin(G_{\tau },\CCI ),  
L\not\subset J_{\ker}(G_{\tau })}L)\neq \emptyset .$ 

\item[{\em (v)}]
All statements {\em (i)--(vii)} of Theorem~\ref{t:zfggzcc} hold for $\tau .$ 

\end{itemize}

\end{thm}
\begin{proof} 
%In order to prove our theorem, 
%by Lemma~\ref{l:ypwntln}, we may assume that 
%for each $\tau \in {\frak M}_{1,c}({\cal Y},\{ {\cal W}_{j}\} _{j=1}^{m})$, 
%there exists a minimal set $B_{\tau }$ in $\cap _{j=1}^{m}S({\cal W}_{j})\cap J_{\ker }(G_{\tau }).$
  Let $\tau \in {\frak M}_{1,c}({\cal Y},\{ {\cal W}_{j}\} _{j=1}^{m})$ be an element. 
There exists an element $\tau _{0}\in {\frak M}_{1,c}({\cal Y},\{ {\cal W}_{j}\} _{j=1}^{m})  $ 
with  $\sharp \mbox{supp}\, \tau _{0}<\infty $ 
arbitrarily close to $\tau .$ 
Since ${\cal Y}$ is nice with respect to $\{ {\cal W}_{j}\} _{j=1}^{m}$, 
we may assume that 
%for each $L\in \Min(G_{\tau _{0}},\CCI )$, 
for each $z\in S_{\min}(\{ {\cal W}_{j}\} _{j=1}^{m})$  and for each 
$j\in \{1,\ldots, m\} $,   
either 
\begin{itemize}
\item $Dh_{z}\neq 0$ for all $h\in \supp\, \tau _{0}\cap \{ 
f_{j,\lambda }\mid \lambda \in \Lambda _{j}\} $,  or 
\item $D(f_{j,\lambda })_{z}=0$ for all $\lambda \in \Lambda _{j}.$ 
\end{itemize}
Let ${\cal W}_{j}=\{ f_{j,\lambda }\} _{\lambda \in \Lambda _{j}}$ for each $j=1,\ldots, m.$ 
By enlarging the support of $\tau _{0}$ a little bit, we obtain an element 
$\tau _{1} \in {\frak M}_{1,c}({\cal Y},\{ {\cal W}_{j}\} _{j=1}^{m})$ arbitrarily close to $\tau $ 
such that int$(\supp\,\tau _{1} \cap \{ f_{j,\lambda }\mid \lambda \in \Lambda _{j}\} )\neq \emptyset $ 
with respect to the topology in $\{ f_{j,\lambda }\mid \lambda \in \Lambda _{j}\} $ (which is endowed with the relative topology from Rat) 
for each $j=1,\ldots, m.$ 
By enlarging the support of $\tau _{1}$ a little bit again, 
 Lemma~\ref{l:nalnotsb} implies that, we can obtain
 an element 
$\tau _{2} \in {\frak M}_{1,c}({\cal Y},\{ {\cal W}_{j}\} _{j=1}^{m})$ arbitrarily close to $\tau $ 
such that int$(\supp\,\tau _{2} \cap \{ f_{j,\lambda }\mid \lambda \in \Lambda _{j}\} )\neq \emptyset $ 
with respect to the topology in $\{ f_{j,\lambda }\mid \lambda \in \Lambda _{j}\} $ 
for each $j=1,\ldots, m$,  
such  that $J_{\ker }(G_{\tau _{2}})\subset \cap _{j=1}^{m}S({\cal W}_{j})$,  
such that $\sharp J_{\ker }(G_{\tau _{2}})<\infty $, 
such that  $\sharp \Min(G_{\tau _{2}},\CCI )<\infty $, 
and such that each $L\in \Min(G_{\tau _{2}},\CCI )$ with $L\not\subset \cap _{j=1}^{m}S({\cal W}_{j})$ 
is attracting for $\tau _{2}.$     
%Moreover, 
%for each non-attracting $L\in \Min(G_{\tau _{1}},\CCI )$ with $L\not\subset \cap _{j=1}^{m}S({\cal W}_{j})$,  
%there exists a bifurcation element $g\in \Gamma _{\tau _{1} }$ for $(L,\Gamma _{\tau _{1}})$ 
%with $\cup _{j=1}^{m}\partial (\Gamma _{\tau _{1}}\cap \{ f_{j,\lambda }\mid \lambda \in \Lambda _{j}\})$, 
%where the boundary of  $\Gamma _{\tau _{1}}\cap \{ f_{j,\lambda }\mid \lambda \in \Lambda _{j}\}$ 
%is taken with respect to the topology in $\{ f_{j,\lambda }\mid \lambda \in \Lambda _{j}\} .$ 
We now prove the following claim. \\ 
\noindent \underline{Claim 1}. There exists an element $\tau _{3}\in ({\frak M}_{1,c}({\cal Y},\{ {\cal W}_{j}\} _{j=1}^{m})$ arbitrarily close to $\tau _{2}$ such that the interior of 
$\supp\, \tau _{3}\cap \{ f_{j,\lambda}\mid \lambda \in \Lambda _{j}\} $ is not empty  
with respect to the topology in $\{ f_{j,\lambda }\mid \lambda \in \Lambda _{j}\}$ for each 
$j=1,\ldots, m$, and such that 
for each $L\in \Min(G_{\tau _{3}},\CCI )$ with $L\subset \cap _{j=1}^{m}S({\cal W}_{j})$,  
%such that 
exactly one of the following (I)-(IV) holds.
\begin{itemize}
\item[(I)] 
For each $z\in L$ and for each $g\in G_{\tau _{3}}$ with $g(z)=z$, we have 
$\| Dg_{z}\| _{s}>1.$ 
\item[(II)]
For each $z\in L$ and for each $g\in G_{\tau _{3}}$ with $g(z)=z$, we have 
$\| Dg_{z}\| _{s}<1.$
\item[(III)]
There exist a point $z_{1}\in L$ and  elements $g_{1},g_{2},g_{3}\in G_{\tau _{3}}$ 
such that $g_{1}(z_{1})=z_{1},\| D(g_{1})_{z_{1}}\| _{s}>1,$ 
$g_{2}(z_{1})=z_{1}, 0<\| D(g_{2})_{z_{1}}\| <1$, $g_{3}(z_{1})=z_{1}$, and $z_{1}$ is the center of a Siegel disk of 
$g_{3}.$ Also, there exist some elements $\alpha _{1},\ldots, \alpha _{l}\in \supp\,\tau _{3}$ 
with $\alpha _{k}\in \mbox{int}(\supp\,\tau _{3}\cap \{ f_{j_{k},\lambda }\mid \lambda \in \Lambda _{j_{k}}\})$ with 
respect to the topology in $\{ f_{j_{k},\lambda }\mid \lambda \in \Lambda _{j_{k}}\}$, $k=1,\ldots, l$, 
such that $g_{3}=\alpha _{1}\circ \cdots \circ \alpha _{l}.$ 
\item[(IV)]
There exist a point $z_{1}\in L$ and a $j\in \{ 1,\ldots, m\} $ such that 
for each $\lambda \in \Lambda _{j}$, we have $D(f_{j,\lambda })_{z_{1}}=0.$ 
Moreover, there exist a point $z_{2}\in L$ and an element $g\in G_{\tau _{3}}$ such that 
$g(z_{2})=z_{2}$ and $\| Dg_{z_{2}}\| _{s}>1.$ 
\end{itemize}
To prove this claim, 
we first remark that 
regarding the minimal set $L\in \Min(G_{\tau _{2}},\CCI )$ with 
$L\subset \cap _{j=1}^{m}S({\cal W}_{j})$   
of type (I),  by Lemmas~\ref{l:lcbwj} and \ref{l:unifea}, 
if we perturb $\tau _{2}$ a little bit to $\tau '_{3}$, then  
$L\in \Min(G_{\tau '_{3}},\CCI )$ with $L\subset \cap _{j=1}^{m}
S({\cal W}_{j})$ and $L$  is of type (I) for $\tau '_{3}.$ By Lemmas~
\ref{l:lcbwj} and \ref{l:unifea} again,  
%again, 
 the similar thing holds  
for minimal sets $L\in \Min(G_{\tau _{2}},\CCI )$ 
with $L\subset \cap _{j=1}^{m}S({\cal W}_{j})$ of type (II). 
Let $\tau _{3}\in {\frak M}_{1,c}({\cal Y},\{ {\cal W}_{j}\} _{j=1}^{m})$ be 
an element such that $\tau _{3}$ is close enough to $\tau _{2}$ and such that 
$\supp\,\tau _{2}\subset \mbox{int}(\supp\,\tau _{3}\cap 
\{ f_{j,\lambda }\mid \lambda \in \Lambda _{j}\})$ with respect to the topology in 
$\{ f_{j,\lambda }\mid \lambda \in \Lambda _{j}\} )$, for each $j=1,\ldots, m.$    
%As we see in Claim 2 below, 
Regarding the element $\tau _{3}$, suppose that we do not have  (I) or (II). 
Then there exist a point $z_{1}\in L$, an element $g_{1}\in G_{\tau _{3}}$, 
a point $z_{2}\in L$, and an element $h_{2}\in G_{\tau _{3}}$ such that 
$g_{1}(z_{1})=z_{1}$, $\| D(g_{1})_{z_{1}}\| _{s}\geq 1$, $h_{2}(z_{2})=z_{2}$, 
and $\| D(h_{2})_{z_{2}}\| _{s}\leq 1.$ Since ${\cal Y}$ is nice 
with respect to $\{ {\cal W}_{j}\} _{j=1}^{m}$,  
by enlarging the support of $\tau _{3}$ a little bit, 
we may assume that $\| D(g_{1})_{z_{1}}\| _{s}> 1$ and 
$\| D(h_{2})_{z_{2}}\| _{s}<1.$ (For, if 
$g_{1}=\gamma _{n}\circ \cdots \circ \gamma _{1}$ 
where $\gamma _{k}\in \supp\,\tau _{3}\cap \{ f_{j_{k},\lambda }\mid \lambda \in \Lambda _{j_{k}}\}$, 
$k=1,\ldots, n$,  
we may assume that $\gamma _{n}\in 
\mbox{int}(\supp\,\tau _{3}\cap 
\{ f_{j_{n},\lambda }\mid \lambda \in \Lambda _{j_{n}}\} )$. Since 
${\cal Y}$ is nice with respect to $\{ {\cal W}_{j}\} _{j=1}^{m}$, perturbing $\gamma _{n}$ a little bit if necessary, we may assume that 
$\| D(g_{1})_{z_{1}}\| >1.$ Similar argument is valid for $h_{2}.$)   
Let $\alpha ,\beta \in \supp\,\tau _{3}$ such that 
$\alpha (z_{2})=z_{1}$ and $\beta (z_{1})=z_{2}.$ 
Take a large $n\in \NN $ so that 
$\| D(\alpha h_{2}^{n}\beta )_{z_{1}}\| _{s}<1.$ Let $g_{2}=\alpha h_{2}^{n}\beta .$ 
Suppose we do not have (IV). Then we may assume that 
$0<\| D(g_{2})_{z_{1}}\| _{s}.$ In order to take an element $g_{3}$ as in (III), 
let $a=\| D(g_{1})_{z_{1}}\| _{s}>1$ and $b=\| D(g_{2})_{z_{1}}\| _{s} \in (0,1).$ 
Let 
%\vspace{2mm} 
%\hspace{15mm} 
$$\Omega := \{ m\log a+n\log b\mid (m,n)\in (\NN \cup \{ 0\} )^{2}
\setminus \{ (0,0)\} \}.$$ 
%
%\vspace{1mm} 
We now prove the following subclaim which is needed in the proof of Claim 1. \\ 
Subclaim $(\ast )$. $0\in \overline{\Omega }$ with respect to the topology in $\RR .$  

 To prove this subclaim, let $\Omega _{+}=\Omega \cap \{ x\in \RR \mid x\geq 0\} $ 
 and $\Omega _{-}:=\{ x\in \RR \mid x\leq 0\} .$ 
 Suppose that $0\not\in \overline{\Omega }.$ 
 Then $\inf \Omega _{+}>0$ and $\sup \Omega _{-}<0.$ 
 Suppose that $\inf \Omega _{+}>-\sup \Omega _{-}.$ 
 Then for each $\epsilon >0$ with 
 $\epsilon <\max \{ \inf \Omega _{+}+\sup \Omega _{-},-\sup \Omega _{-}\}$, 
 there exist an element $c_{1}\in \Omega _{+}$ with 
 $c_{1}<\inf \Omega _{+}+\epsilon $ and an element $d_{1}\in \Omega _{-}$ with 
 $d_{1}>\sup \Omega _{-}-\epsilon. $ 
 Then $c_{1}+d_{1}\geq \inf \Omega _{+}+\sup \Omega _{-}-\epsilon >0.$ Hence 
 $c_{1}+d_{1}\in \Omega _{+}.$ However, 
 $c_{1}+d_{1}\leq \inf \Omega _{+}+\sup \Omega _{-}+\epsilon <\inf \Omega _{+}.$ 
 This is a contradiction. 
 Thus we must have that $\inf \Omega _{+}\leq -\sup \Omega _{-}.$ 
 Similarly, we must have that $\inf \Omega _{+}\geq -\sup \Omega _{-}.$ 
 Hence $\inf  \Omega _{+}=-\sup \Omega _{-}.$ This implies $0\in \overline{\Omega }.$ However, 
 this is a contradiction. 
Thus we have proved subclaim $(\ast ).$ 

Going back to the proof of Claim 1, for each $i=1,2$, 
we write $g_{i}=\gamma _{1}^{i}\circ \cdots \circ \gamma _{p_{i}}^{i}$ 
where $\gamma _{k}^{i}\in \supp\,\tau _{3}\cap \{ f_{j_{i,k},\lambda }\mid \lambda \in \Lambda _{j_{i,k}}\}.$ 
By enlarging the support of $\tau _{3}$ a little bit, 
we may assume that 
$\gamma _{k}^{i}\in \mbox{int}(\supp\,\tau _{3}\cap 
\{ f_{j_{i,k},\lambda }\mid \lambda \in \Lambda _{j_{i,k}}\})$ 
with respect to the topology in $\{ f_{j_{i,k},\lambda }\mid \lambda \in \Lambda _{j_{i,k}}\}$ 
for all $i=1,2, \ k=1,\ldots, p_{i}.$ 
Then there exist an  
% a number 
$\epsilon >0$ and a neighborhood 
$V_{k,i}$ of $\gamma _{k}^{i}$ in $\mbox{int}(\supp\,\tau _{3}\cap 
\{ f_{j_{k,i},\lambda }\mid \lambda \in \Lambda _{j_{k,i}}\} )$ such that 
$(\log a-\epsilon ,\log a+\epsilon )
\subset \{ \log \| D(\tilde{\gamma }_{1}^{1}\cdots \tilde{\gamma }_{p_{1}}^{1})_{z_{1}}\| _{s}
\mid \tilde{\gamma }_{k}^{1}\in V_{k,1},k=1,\ldots, p_{1}\} $ 
and 
$(\log b-\epsilon ,\log b+\epsilon )
\subset \{ \log \| D(\tilde{\gamma }_{1}^{2}\cdots \tilde{\gamma }_{p_{2}}^{2})_{z_{1}}\| _{s}
\mid \tilde{\gamma }_{k}^{2}\in V_{k,2},k=1,\ldots, p_{2}\} $.  
We set 

\vspace{3mm} 
%\begin{eqnarray*}
$\  \  \tilde{\Omega }:=
\{ m\log \| D(\tilde{\gamma }_{1}^{1}\cdots \tilde{\gamma }_{p_{1}}^{1})_{z_{1}}\| _{s}
+n\log \| D(\tilde{\gamma }_{1}^{2}\cdots \tilde{\gamma }_{p_{2}}^{2})_{z_{1}}\| _{s}$ \\ 
\hspace{60mm} $\mid (m,n)\in (\NN \cup \{ 0\})^{2}\setminus 
\{ (0,0)\}, \tilde{\gamma }_{k}^{1}\in V_{k,1}, \tilde{\gamma }_{k}^{2}\in V_{k,2},\forall k\}.$\\ 
%\end{eqnarray*}   
%\vspace{3mm} 

\noindent Then for each $c\in \Omega, $ we have $(c-\epsilon ,c+\epsilon )\subset \tilde{\Omega }.$ 
By Subclaim $(\ast )$, it follows that $0\in \tilde{\Omega }.$ 
Therefore there exist an element $(m,n)\in (\NN \cup \{ 0\})^{2} \setminus \{ (0,0)\}$, 
$p_{1}$-elements $\tilde{\gamma }_{k}^{1}\in V_{k,1}, k=1,\ldots, p_{1}$, 
and $p_{2}$-elements $\tilde{\gamma }_{k}^{2}\in V_{k,2},k=1,\ldots, p_{2}$ such that 
setting $h_{3}=(\tilde{\gamma }_{1}^{1}\cdots \tilde{\gamma }_{p_{1}}^{1})^{m}
(\tilde{\gamma }_{1}^{2}\cdots \tilde{\gamma }_{p_{2}}^{2})^{n}$, we have 
$\| D(h_{3})_{z_{1}}\| _{s}=1.$ Perturbing $\tilde{\gamma }_{k}^{i}$ a little bit, 
we obtain an element $g_{3}$ which is close to $h_{3}$ such that 
$g_{3}'(z_{1})$ is a Brjuno number (we may assume $z_{1}\in \CC $ by 
conjugating $G_{\tau _{3}}$ by an element of Aut$(\CCI )$). Thus $g_{3}$ has a Siegel disk whose center is $z_{1}$ (\cite{Mi}).  
Thus we have proved Claim 1.  

By Lemma~\ref{l:unifea},  we have the following two claims. \\ 
\noindent \underline{Claim 2} There exists a $k\in \NN $ such that for each 
$L\in \Min(G_{\tau _{3}}, S_{\min }(\{ {\cal W}_{j}\} _{j=1}^{m}))$  of type (I), for each $z\in L$  and 
for each $(\gamma _{1},\ldots, \gamma _{k})\in 
(\supp\,\tau _{3})^{k}$, 
we have $\| D(\gamma _{k}\circ \cdots \circ \gamma _{1})_{z}\| _{s}>2.$ \\ 
\noindent \underline{Claim 3.}
There exists a $k\in \NN $ such that 
for each  $L\in \Min(G_{\tau _{3}},  S_{\min}(\{ {\cal W}\} _{j=1}^{m}))$ of type (II),  
for each $z\in L$ and for each $(\gamma _{1},\ldots, \gamma _{k})\in (\supp\,\tau _{3})^{k}$, 
we have $\| D(\gamma _{k}\circ \cdots \circ \gamma _{1})_{z}\| _{s}<\frac{1}{2}.$ 
Moreover, there exists a neighborhood $V$ of $L$ with $\sharp (\CCI \setminus V)\geq 3$ 
such that for each $(\gamma _{1},\ldots, \gamma _{k})\in 
(\supp\,\tau _{3})^{k}$, we have 
$\gamma _{k}\circ \cdots \circ \gamma _{1}(V)\subset V.$ In particular, 
$L$ is attracting for $\tau _{3}$ and 
$L\subset F(G_{\tau _{3}}).$ 

Throughout the rest of the proof, we fix an element $k\in \NN $ 
which satisfies the statements in Claims 2,3. 

We now prove the following claim.\\ 
\noindent \underline{Claim 4.} 
Let $L\in \Min(G_{\tau _{3}}, S_{\min }(\{ {\cal W}_{j}\} _{j=1}^{m}))$ be of type (III). Then $L\subset \mbox{int}(J(G_{\tau _{3}})).$ 
In particular, for each $z\in F(G_{\tau _{3}})$, we have $\overline{G_{\tau _{3}}(z)}\cap L=\emptyset .$ 
  
To prove Claim 4, let $z_{1},g_{1},g_{2},g_{3}$ be as in (III). Since $z_{1}$ is a repelling fixed point of $g_{1}$, 
we have $z_{1}\in J(G_{\tau _{3}}).$ Since $J(G_{\tau _{3}})$ is perfect (see \cite{HM}), 
there exists a point $w\in J(G_{\tau _{3}})\cap (B\setminus \{ z_{1}\} )$, where $B$ denotes the 
Siegel disk of $g_{3}$ whose center is $z_{1}.$  Therefore there exists a $g_{3}$-invariant  analytic Jordan curve $\zeta $ 
in $J(G_{\tau _{3}})\cap B$ with $w\in \zeta .$ 
 If $K$ is a compact subset in $\CCI \setminus S_{1}({\cal W}_{j})$,  
$A$ is a subset of $\{ f_{j,\lambda }\mid \lambda \in \Lambda _{j}\}$ with $\mbox{int}(A)\neq 
\emptyset $ with respect to the topology in $\{ f_{j,\lambda }\mid \lambda \in \Lambda _{j}\} $,  
and $h_{0}\in \mbox{int}(A)$, 
then  
there exists an $\epsilon >0$ such that for each  $z\in K$, 
$B(h_{0}(z),\epsilon )\subset \{ h(z)\mid h\in A\} .$ From this fact and that $\sharp (S_{1}({\cal W}_{j}))<\infty $ 
for each $j$, it follows that $\zeta \subset \mbox{int}(J(G_{\tau _{3}})).$ 
Similarly, for each $w'\in B\cap J(G_{\tau _{3}})$, if we take the $g_{3}$-invariant analytic Jordan curve $\zeta' $ 
in $B$ with $w'\in \zeta' $, then $\zeta'\subset \mbox{int}(J(G_{\tau _{3}})).$ 
From this argument, we obtain that $z_{1}\in \mbox{int}(J(G_{\tau _{3}})).$ 
Therefore $L\subset \mbox{int}(J(G_{\tau _{3}})).$ 
Thus we have proved Claim 4. 
 
 We now prove the following claim.\\ 
 \noindent \underline{Claim 5.} Let $L\in \Min(G_{\tau _{3}}, S_{\min }(\{ {\cal W}_{j}\} _{j=1}^{m}))$ be of type (IV). 
 Then $L\subset \mbox{int}(J(G_{\tau _{3}})).$ In particular, for each $z\in F(G_{\tau _{3}})$, 
 $\overline{G_{\tau _{3}}(z)}\cap L=\emptyset .$ 
 
 To prove Claim 5, let $j\in \{ 1,\ldots, m\}$, $z_{1},z_{2}\in L$ and $g\in G_{\tau _{3}}$ be as in  (IV). 
 Since $z_{2}$ is a repelling fixed point of $g$, we have $z_{2}\in J(G_{\tau _{3}}).$ 
 Moreover, let $\lambda \in \Lambda _{j}$ with $f_{j,\lambda }\in \supp\,\tau _{3}$ 
 and let $\alpha ,\beta \in G_{\tau _{3}}$ such that 
 $\alpha (z_{2})=z_{1}, \beta (f_{j,\lambda }(z_{1}))=z_{2}.$ Then 
 $\beta \circ f_{j,\lambda }\circ \alpha (z_{2})=z_{2}$ and 
 $D(\beta \circ f_{j,\lambda }\circ \alpha )_{z_{2}}=0.$ 
By \cite[Corollary 4.1]{HM2}, 
%the result ``If $G$ is generated by a compact subset of Rat and $z_{0}\in J(G)$ is a 
%superattracting fixed point of an element of $G$, then $z_{0}\in \mbox{int}(J(G))$'', 
we obtain that $z_{2}\in \mbox{int}(J(\langle g, \beta \circ f_{j,\lambda }\circ \alpha \rangle ))
\subset \mbox{int}(J(G_{\tau _{3}})).$ 
Moreover, for each $z\in L$, there exists an element $\gamma \in G_{\tau _{3}}$ such that 
$\gamma (z)=z_{2}.$ Thus $L\subset \mbox{int}(J(G_{\tau _{3}})).$ Hence we have proved 
Claim 5.  

Let ${\cal I}:=\{ L\in \Min (G_{\tau _{3}},\CCI )\mid L\subset \cap _{j=1}^{m}S({\cal W}_{j}), L\mbox{ is of type } (I)\}.$ 
Let 
$$C_{\tau _{3}}:= \{ w\in \CCI \setminus \cup _{L\in {\cal I}}L\mid 
\exists (\gamma _{1},\ldots, \gamma _{k})\in 
(\supp\,\tau _{3})^{k} 
\mbox{ s.t. } \gamma _{k}\cdots \gamma _{1}(w)\in \cup _{L\in {\cal I}}L\} .$$ 
Note that $C_{\tau _{3}}\subset J(G_{\tau _{3}}).$ 
Moreover, by Claim 2, 
\begin{equation}
\label{eq:ctau3cc}
C_{\tau _{3}}\cap \cup _{L\in \Min(G_{\tau _{3}},\CCI )}L=\emptyset  
\mbox{ and } C_{\tau _{3}} \mbox{ is compact}. 
\end{equation}  
By Lemma~\ref{l:nalnotsb}, we may assume that 
\begin{equation}
\label{eq:eachLattt3}
\mbox{each } L\in \Min(G_{\tau _{3}}, \CCI ) 
\mbox{ with } 
L\cap F(G_{\tau _{3}})\neq \emptyset \mbox{ is attracting for } \tau _{3}. 
\end{equation}
We now prove the following claim.\\ 
\noindent \underline{Claim 6.}
Let $z\in F(G_{\tau _{3}}).$ If 
$\overline{G_{\tau _{3}}(z)}\cap (\cup _{L\in \Min(G_{\tau _{3}},\CCI ), L\cap F(G_{\tau _{3}})\neq \emptyset }L)
=\emptyset $, then 
\vspace{-2mm} 
$$\overline{G_{\tau _{3}}(z)}\cap (\cup _{L\in {\cal I}, L\subset J(G_{\tau _{3}})}L)\neq \emptyset  
\mbox{\ and \ } \overline{G_{\tau _{3}}(z)}\cap C_{\tau _{3}}\neq \emptyset .$$ 
%\vspace{-2mm} 
To prove Claim 6, let $z\in F(G_{\tau _{3}})$ and suppose $\overline{G_{\tau _{3}}(z)}\cap (\cup _{L\in \Min(G_{\tau _{3}},\CCI ), L\cap F(G_{\tau _{3}})\neq \emptyset }L)
=\emptyset $. Since $\overline{G_{\tau _{3}}(z)}\cap \cup _{L\in \Min(G_{\tau _{3}},\CCI )}L\neq \emptyset $, 
 Claims 3,4,5 imply that 
 $\overline{G_{\tau _{3}}(z)}\cap (\cup _{L\in {\cal I}, L\subset J(G_{\tau _{3}})}L)\neq \emptyset $. 
 Let $\delta _{1}>0$ be a number such that 
 for each $(\gamma _{1},\ldots, \gamma _{k})\in 
 (\supp\,\tau _{3})^{k}$, 
 for each $L\in {\cal I}$  and 
 for each $x\in L$, we have $\gamma _{k}\cdots \gamma _{1}|_{B(x,\delta _{1})}$ is injective 
and we can take well-defined inverse branch $\zeta :B(\gamma _{k}\cdots \gamma _{1}(x), \delta _{1})
\rightarrow \CCI $ of $\gamma _{k}\cdots \gamma _{1}$ such that 
$\zeta (\gamma _{k}\cdots \gamma _{1}(x))=x.$  
 We may assume 
 $$\delta _{1}<(1/2)\cdot \min \{ d(a,b)\mid L\in {\cal I}, a,b\in L, a\neq b\}$$ 
 (if $\sharp L=1$ for each ${\cal I}$ then we set 
 $\min\{ d(a,b)\mid L\in {\cal I}, a, b\in L, a\neq b\}=1$).  
 Let $\delta _{2}\in (0,\delta _{1})$ be a number such that 
 for each $L\in {\cal I}$, for each $x\in L$ and for each $y\in B(x,\delta _{2})$, 
 we have $d(\gamma _{k}\cdots \gamma _{1}(y), \gamma _{k}\cdots, \gamma _{1}(x))<\delta _{1}.$ 
 Let $\epsilon \in (0, \delta _{1})$ be any small number with $\epsilon <d(z,\cup _{L\in {\cal I}}L).$  
 Then there exist an element $\gamma =(\gamma _{1},\gamma _{2},\ldots )\in 
 X_{\tau _{3}}$ 
 an element $n\in \NN $, and an element $L\in {\cal I}$ 
 such that $\gamma _{nk,1}(z)\in B(L,\epsilon ).$ We may assume that $n$ is the minimum one. 
 Suppose $\gamma _{(n-1)k,1}(z)\in \cup _{L\in {\cal I}}B(L,\delta _{2}).$ 
Then there exist an element $L_{0}\in {\cal I}$ and an element $z_{0}\in L_{0}$ 
 such that $\gamma _{(n-1)k,1}(z)\in B(z_{0},\delta _{2}).$ It implies that 
 $d(\gamma_{nk,1}(z),\ \gamma _{nk}\cdots \gamma _{(n-1)k+1}(z_{0}))<\delta _{1}.$ 
%Hence 
% $\gamma _{(n-1)k+1}(z_{0})\in B(L, 2\delta _{1}).$ By the choice of $\delta _{1}$, 
% we obtain that $\gamma _{(n-1)k+1}(z_{0})\in  L.$ 
  Let $\xi :B(\gamma _{nk}\cdots \gamma _{(n-1)k+1}(z_{0}),\delta _{1})\rightarrow \CCI $ be the well-defined inverse branch of 
  $\gamma _{nk}\cdots \gamma _{(n-1)k+1}$ such that 
  $\xi (\gamma _{nk}\cdots \gamma _{(n-1)k+1}(z_{0}))=z_{0}.$ 
By Claim 2, taking $\delta _{1}$ small enough, 
  we obtain that $$\xi (B(\gamma _{nk}\cdots \gamma _{(n-1)k+1}(z_{0}),\epsilon ))\subset B(z_{0},\frac{3}{4}\epsilon )
  \subset B(z_{0},\delta _{1}).$$ Since $\gamma _{nk}\cdots \gamma _{(n-1)k+1}|_{B(z_{0},\delta _{1})}$ 
  is injective, it follows that $$\gamma _{(n-1)k,1}(z)\in 
  \xi (B(\gamma _{nk}\cdots \gamma _{(n-1)k+1}(z_{0}),\epsilon ))\subset B(z_{0},\epsilon ).$$ 
  However, this contradicts the minimality of $n.$ Therefore we should have that 
  $\gamma _{(n-1)k,1}(z)\not\in \cup _{L\in {\cal I}}B(L,\delta _{2}).$ 
  Since the above argument is valid for arbitrarily small $\epsilon >0$, we obtain that 
  $\overline{G_{\tau _{3}}(z)}\cap C_{\tau _{3}}\neq \emptyset .$ Thus we have proved Claim 6. 
  
  Let $p\in \NN $ with $p> \sum _{j=1}^{m}\sharp S_{1}({\cal W}_{j})+1$ and 
  let $H:=\{ z\in F(G_{\tau _{3}}) \mid \overline{G_{\tau _{3}}(z)}\cap C_{\tau _{3}}\neq \emptyset \} .$ 
  For each $z\in H$ and for each $n\in \NN $, 
  there exist an element $(w_{z,n,0},\ldots, w_{z,n,p})\in (G_{\tau _{3}}(z))^{p+1}$ 
  and an element $(\gamma _{z,n,1},\ldots, \gamma _{z,n,p})\in (\supp\,\tau _{3})^{p}$ 
  such that $\gamma _{z,n,i+1}(w_{z,n,i})=w_{z,n,i+1}$ for each $i=0,\ldots, p-1$ and 
  such that $d(w_{z,n,p}, C_{\tau _{3}})<\frac{1}{n}.$ 
  We may assume that for each $i=0,\ldots, p$, there exists an element $w_{z,\infty ,i}\in 
  \overline{G_{\tau _{3}}(z)}$ such that $w_{z,n,i}\rightarrow w_{z,\infty ,i}$ as $n\rightarrow \infty. $ 
  Moreover, we may assume that for each $i=1,\ldots, p$, there exists an element 
  $\gamma _{z,\infty ,i}\in \supp\,\tau _{3}$ such that 
  $\gamma _{z,n,i}\rightarrow \gamma _{z,\infty ,i}$ as $n\rightarrow \infty .$ 
  Then we have that $\gamma _{z,\infty ,i+1}(w_{z,\infty ,i})=w_{z,\infty, i+1}$ for each 
  $i=0,\ldots, p-1.$ Since $w_{z,\infty ,p}\in C_{\tau _{3}}\subset J(G_{\tau _{3}})$, 
  we obtain that $w_{z,\infty, i}\in J(G_{\tau _{3}})$ for each $i=0,\ldots, p.$  
  For each $i=1,\ldots, p$, let $j_{z,i}\in \{ 1,\ldots, m\} $ be an element such that 
  $\gamma _{z,\infty ,i}\in \supp\,\tau _{3}\cap \{ f_{j_{z,i},\lambda }
  \mid \lambda \in \Lambda _{j_{z,i}}\} .$ We now prove the following claim. \\ 
 \noindent \underline{Claim 7.} 
 There exists a number $\epsilon >0$ such that for each $z\in H$, 
 there exists an element $i\in \NN $ with $1\leq i\leq p$ 
 such that $d(w_{z,\infty, i-1}, S_{1}({\cal W}_{j_{z,i}}))>\epsilon .$ 
  
  To prove Claim 7, suppose that the statement of Claim 7 does not hold. 
  Then for each $r\in \NN $ there exists a $z_{r}\in H$ such that 
  for each $i\in \NN $ with $1\leq i\leq p$, 
  we have 
  $d(w_{z_{r},\infty ,i-1}, S_{1}({\cal W}_{j_{z_{r},i}}))<\frac{1}{r}.$  
  We may assume that for each $i=1,\ldots, p$, there exist 
  an element $a_{i-1}\in C_{\tau _{3}} $, 
  an element $j_{i}\in \{ 1,\ldots, m\} $, 
  and an element $\gamma _{i}\in \supp\,\tau _{3}\cap 
  \{ f_{j_{i},\lambda }\mid \lambda \in \Lambda _{j_{i}}\}$ 
  such that $j_{z_{r},i}=j_{i}$ for each $r$, 
  such that $w_{z_{r},\infty ,i-1}\rightarrow a_{i-1}$ as $r\rightarrow \infty $, 
  and such that 
  $\gamma _{z_{r},\infty, i}\rightarrow \gamma _{i}$ as $r\rightarrow \infty .$ 
Also we may assume that there exists an element $a_{p}\in C_{\tau _{3}}$ such that 
$w_{z_{r},\infty ,p}\rightarrow a_{p}$ as $r\rightarrow \infty .$ 
 Then we have that $a_{i-1}\in S_{1}({\cal W}_{j_{i}})$ and 
  $\gamma _{i}(a_{i-1})=a_{i}$ for each $i=1,\ldots, p$ 
and thus $a_{i-1}\not\in \cup _{L\in \Min(G_{\tau _{3}}, \CCI )}L$ 
for each $i=1,\ldots, p$ (by (\ref{eq:ctau3cc}) and the fact $a_{p}\in C_{\tau _{3}}$).   However, this contradicts 
 to the assumption that  there exists no peripheral cycle for 
$({\cal Y}, \{ {\cal W}_{j}\} _{j=1}^{m})$. 
%${\cal  Y}.$ 
 Thus we have proved Claim 7. 
 
 Since $w_{z,\infty ,i}\in J(G_{\tau _{3}})$ for each $z\in H$ and $i=1,\ldots, p$, 
 Claim 7 implies that if $\tau _{4}\in {\frak M}_{1,c}({\cal Y},\{ {\cal W}_{j}\} _{j=1}^{m})$ 
 is an element such that $\supp\,\tau _{3}\cap \{ f_{j,\lambda }\mid 
\lambda \in \Lambda _{j}\} \subset 
\mbox{int}(\supp\,\tau _{4}\cap 
\{ f_{j,\lambda }\mid \lambda \in \Lambda _{j}\} )$ with respect to the topology in 
$\{ f_{j,\lambda }\mid \lambda \in \Lambda _{j}\} $ for each $j=1,\ldots, m$, 
then for each $z\in H$ there exists an element 
$g_{z}\in \supp\,\tau _{4}$ such that 
$$\overline{G_{\tau _{3}}(g_{z}(z))}\cap \cup _{L\in \Min(G_{\tau _{3}},\CCI ), \ 
L \mbox{ is attracting for }\tau _{3}}L
\neq \emptyset .$$ Combining this with 
Lemma~\ref{l:nalnotsb}, 
Claim 6 and (\ref{eq:eachLattt3}), we easily see that 
if we assume further that $\tau _{4}$ is close enough to $\tau _{3}$, then
for each $L\in \Min(G_{\tau _{4}},\CCI )$ with $L\not\subset J_{\ker}(G_{\tau _{4}})$,  
we have that $L$ is attracting for $\tau _{4}$ and  
\begin{equation}
\label{eq:fezfgt}
\mbox{ for each }z\in F(G_{\tau _{4}}), \mbox{ we have that } 
\overline{G_{\tau _{4}}(z)}\cap 
\cup _{L\in \Min(G_{\tau _{4}},\CCI ),\ L \mbox{ is attracting for }\tau _{4}}L\neq \emptyset .
\end{equation} 
Moreover, 
Lemma~\ref{l:nalnotsb} implies that  
there exist two non-empty open neighborhoods $V_{1,\tau _{4}}, V_{2,\tau _{4}}$ 
of the union of attracting minimal sets for $(G_{\tau _{4}},\CCI)$ 
and an element $n\in \NN $ such that 
$\overline{V_{1,\tau _{4}}}\subset V_{2,\tau _{4}}$, $\sharp (\CCI \setminus V_{2,\tau _{4}})\geq 3$ 
and for each $(\gamma _{1},\ldots, \gamma _{n})\in 
(\supp\,\tau _{4})^{n}$, 
we have $\gamma _{n}\circ \cdots \circ \gamma _{1}(V_{2,\tau _{4}})\subset 
V_{1,\tau _{4}}.$ By (\ref{eq:fezfgt}) and Lemma~\ref{l:nalnotsb} (i),  
we have 
\begin{equation}
\label{eq:dt4cap}
D_{\tau _{4}}:=\bigcap _{g\in G_{\tau _{4}}}g^{-1}(\CCI \setminus V_{2,\tau _{4}})
=J_{\ker }(G_{\tau _{4}})\subset \cap _{j=1}^{m}S({\cal W}_{j}).
\end{equation} 
Furthermore, 
%Lemma~\ref{l:nalnotsb} implies that  
for each $L\in \Min(G_{\tau _{4}},\CCI )$ with $L\subset \cap _{j=1}^{m}S_{1}({\cal W}_{j})$, 
$L$ satisfies exactly one of  (I)--(IV) in Claim 1.  
Therefore $\tau _{4}$ is weakly mean stable. 
By (\ref{eq:dt4cap}), Lemma~\ref{l:lcbwj}, Lemma~\ref{l:wmsopen} and its proof,  and Lemma~\ref{l:wmsdtjk}, 
we see that there exists an neighborhood 
$V$ of $\tau _{4}$ in $({\frak M}_{1,c}({\cal Y},\{ {\cal W}_{j}\} _{j=1}^{m}),{\cal O})$ 
such that for each $\tau _{5}\in V$, we have that statements (i)(iii)(iv)(v) in our theorem hold 
for $\tau _{5}$ and that 
$\sharp J_{\ker }(G_{\tau _{5}})<\infty $,  
$D_{\tau }=J_{\ker }(G_{\tau _{5}})$ and 
$\sharp \Min(G_{\tau _{5}}, \CCI )<\infty .$  
We now prove that there exists an open neighborhood 
$V$ of $\tau _{4}$ such that for each $\tau _{5}\in V$, 
$D_{\tau_{5}}\subset \cap _{j=1}^{m}S({\cal W}_{j}).$ 
We use the argument in the proof of 
Lemma~\ref{l:wmsopen}  with $\Gamma =\supp\,\tau _{4}.$ 
By (\ref{eq:dt4cap}), 
in the proof of Lemma~\ref{l:wmsopen}, 
if $V$ is small enough, then for each $\tau _{5}\in V$, 
for each $L\in \Min(G_{\tau _{4}}, D_{\tau _{4}})$, 
we can take $z_{L, \supp\,\tau_{5}}$ 
so that $z_{L, \supp\,\tau _{5}}=z_{L}\in L\subset D_{\tau _{4}}
\subset \cap _{j=1}^{m}S({\cal W}_{j}).$ Also, 
if we take $h_{w}\in G_{\tau _{4}}$ appropriately 
for each $w\in D_{\tau _{4}}$, then there exist 
a neighborhood $V$ of $\tau _{4}$ such that 
for each $\tau _{5}\in V$ and for each 
$w\in D_{\tau _{4}}$, 
we can take $h_{w,\supp\,\tau _{5}}\in G_{\tau _{5}}$ 
(in the proof of Lemma~\ref{l:wmsopen}) so that 
%for each $w\in D_{\tau _{4}}$, we have 
$h_{w,\supp\,\tau_{5}}(w)=z_{L_{w},\supp\,\tau _{5}}=z_{L}$ 
and the order of $h_{w,\supp\,\tau _{5}}$ 
at $w$ is equal to a constant $a_{w}\in \NN $ which 
does not depend on the choice of $\tau _{5}\in V.$ 
From this, in the argument just before (\ref{eq:capggrho}),  
for each $\tau _{5}\in V$, 
any point $z_{0}\in \cap _{g\in G_{\tau _{5}}}g^{-1}(\CCI 
\setminus V_{2,\supp\,\tau _{4}})$ must be equal to some $w\in L\subset D_{\tau _{4}}\subset \cap _{j=1}^{m}S({\cal W}_{j}).$  Hence $D_{\tau _{5}}\subset \cap _{j=1}^{m}S({\cal W}_{j})$ for each $\tau _{5}\in V$. 

%By Theorem~\ref{t:zfggzcc}, all statements (i)(ii)(iii)(iv) of Theorem~\ref{t:zfggzcc} hold 
%for $\tau _{5}.$ 
Thus we have proved Theorem~\ref{t:rcdnkmain1}. 
  \end{proof}
\begin{df}
\label{d:m1cmild}
Let ${\cal Y}$ be a weakly nice subset of $\Rat$ with respect to some holomorphic families 
$\{ {\cal W}_{j}\} _{j=1}^{m}$ of rational maps. 
We set $${\frak M}_{1,c,mild}({\cal Y}, \{ {\cal W}_{j}\} _{j=1}^{m})
:=\{ \tau \in {\frak M}_{1,c}({\cal Y}, \{ {\cal W}_{j}\} _{j=1}^{m})\mid 
 \exists L\in \Min(G_{\tau},\CCI ) \mbox{ which is attracting for }\tau\} .$$   
Also, we denote by ${\frak M}_{1,c,JF}({\cal Y}, \{ {\cal W}_{j}\} _{j=1}^{m})$ 
the set of elements $\tau \in {\frak M}_{1,c}({\cal Y},\{ {\cal W}_{j}\}_{j=1}^{m})$ 
satisfying that $J(G_{\tau })=\CCI $ and 
either $\Min(G_{\tau },\CCI )=\{ \CCI \}$ or 
$\cup _{L\in \Min(G_{\tau },\CCI )}L\subset \cap _{j=1}^{m}S({\cal W}_{j})$.
\end{df}
\begin{rem}
\label{r:m1cmild}
Let ${\cal Y}$ be a weakly nice subset of $\Rat$ with 
respect to some 
holomorphic families 
$\{ {\cal W}_{j}\} _{j=1}^{m}$ of rational maps. Then it is easy to see that 
${\frak M}_{1,c,mild}({\cal Y}, \{ {\cal W}_{j}\}_{j=1}^{m})$ is an open subset of 
$({\frak M}_{1,c}({\cal Y}, \{ {\cal W}_{j}\} _{j=1}^{m}), {\cal O})$.  
\end{rem}

We now prove a theorem in which we do not assume that ${\cal Y}$ is mild with 
$\{ {\cal W}_{j}\} _{j=1}^{m}.$ 
\begin{thm}
\label{t:nonmild}
Let ${\cal Y}$ be a strongly nice subset of $\emRatp$ with respect to 
some holomorphic families 
$\{ {\cal W}_{j}\}_{j=1}^{m}$ of rational maps. Then the set 
$$\{ \tau \in {\frak M}_{1,c,mild}({\cal Y}, \{ {\cal W}_{j}\} _{j=1}^{m})\mid \tau \mbox{ is weakly mean stable}\}$$ 
is open and dense in $({\frak M}_{1,c,mild}({\cal Y},\{ {\cal W}_{j}\} _{j=1}^{m}),{\cal O}).$ 
Moreover, there exists the largest open and dense subset ${\cal A}$ of 
$({\frak M}_{1,c,mild}({\cal Y},\{ {\cal W}_{j}\} _{j=1}^{m}),{\cal O})$ such that 
for each $\tau \in {\cal A}$, all statements {\em (i)--(v)} in Theorem~\ref{t:rcdnkmain1} hold. 
Furthermore, we have 
$$\overline{{\cal A}\cup {\frak M}_{1,c,JF}({\cal Y}, \{ {\cal W}_{j}\} _{j=1}^{m})}=
{\frak M}_{1,c}({\cal Y}, \{ {\cal W}_{j}\} _{j=1}^{m})$$ 
with respect to the topology ${\cal O}.$  
\end{thm}
\begin{proof}
By using the argument in the proof of Theorem~\ref{t:rcdnkmain1}, we obtain that 
the set of weakly mean stable elements  
$\tau \in {\frak M}_{1,c,mild}$   
%\mid \tau \mbox{ is weakly mean stable}\}$ 
is open and dense in $({\frak M}_{1,c,mild}({\cal Y},\{ {\cal W}_{j}\} _{j=1}^{m}),{\cal O})$, 
and   there exists the largest open and dense subset ${\cal A}$ of 
$({\frak M}_{1,c,mild}({\cal Y},\{ {\cal W}_{j}\} _{j=1}^{m}),{\cal O})$ such that 
for each $\tau \in {\cal A}$, all statements {\em (i)--(v)} in Theorem~\ref{t:rcdnkmain1} hold. 
To prove the last statement of the theorem, 
let $\tau \in {\frak M}_{1,c}({\cal Y},\{ {\cal W}_{j}\} _{j=1}^{m})$ and suppose that 
there exists no element in $\Min(G_{\tau },\CCI )$ which is attracting for $\tau .$ 
We want to find an element in ${\frak M}_{1,c,JF}({\cal Y}, \{ {\cal W}_{j}\} _{j=1}^{m})$ 
which is arbitrarily close to $\tau $, by using the arguments in the proof of Theorem~\ref{t:rcdnkmain1} 
with some modifications. We take $\tau _{1}$ close to $\tau $ as in the proof of Theorem~\ref{t:rcdnkmain1}. 
We may assume that there exists no element in $\Min(G_{\tau _{1}},\CCI )$ which is attracting for $\tau _{1}.$ 
We now consider the following two cases.\\ 
Case 1. $F(G_{\tau _{1}})=\emptyset. $ Case 2. $F(G_{\tau _{1}})\neq \emptyset .$ \\
Suppose we have Case 1. Let $L\in \Min(G_{\tau _{1}},\CCI )$ and 
suppose $L\neq \CCI $ and $L\not\subset \cap _{j=1}^{m}S({\cal W}_{j}).$ 
Then $\emptyset \neq \mbox{int}(L)$, $\sharp (\CCI \setminus (\mbox{int}(L)))\geq 3$  and
$G_{\tau _{1}}(\mbox{int}(L))\subset \mbox{int}(L).$ Hence by Montel's theorem, 
we obtain $\emptyset \neq \mbox{int}(L)\subset F(G_{\tau _{1}}).$ However, 
this is a contradiction. Thus $\tau _{1}\in {\frak M}_{1,c,JF}({\cal Y}, \{ {\cal W}_{j}\} _{j=1}^{m}).$ 

Suppose that we have Case 2. 
Let ${\cal W}_{j}=\{ f_{j,\lambda }\}_{\lambda \in \Lambda _{j}}$ 
for each $j.$ 
Let $\tau _{2}\in {\frak M}_{1,c}({\cal Y}, \{ {\cal W}_{j}\} _{j=1}^{m})$ such that 
$\supp\,\tau _{1}\cap \{ f_{j,\lambda }\mid \lambda \in \Lambda _{j}\} 
\subset \mbox{int}(\supp\,\tau _{2}\cap \{ f_{j,\lambda }\mid \lambda 
\in \Lambda _{j}\})$ with respect to the topology in 
$\{ f_{j,\lambda }\mid \lambda \in \Lambda _{j}\}$ which 
is endowed with the relative topology from Rat 
for each $j=1,\ldots, m$, and such that $\tau _{2}$ is close to $\tau _{1}.$ 
Then by Lemma~\ref{l:nalnotsb} (iv), we have that 
either $\Min(G_{\tau _{2}},\CCI )=\{ \CCI \} $ or 
$\cup _{L\in \Min(G_{\tau _{2}},\CCI )}L\subset \cap _{j=1}^{m}S({\cal W}_{j})$, 
and if $\Min(G_{\tau _{2}},\CCI )=\{ \CCI \}$ then $J(G_{\tau _{2}})=\CCI .$  
  Thus we may assume that $\cup _{L\in \Min(G_{\tau _{2}},\CCI )}L\subset 
\cap _{j=1}^{m}S({\cal W}_{j}).$ 
Under this condition, if $F(G_{\tau _{2}})=\emptyset $, then 
$\tau _{2}\in {\frak M}_{1,c,JF}({\cal Y}, \{ {\cal W}_{j}\} _{j=1}^{m}).$ Thus 
we may assume $F(G_{\tau _{2}})\neq \emptyset. $
By the argument in the proof of Claim 1 in the proof of Theorem~\ref{t:rcdnkmain1}, 
there exists an element $\tau _{3}$  arbitrarily close  to $\tau _{2}$ such that  
$\supp\,\tau _{2}\cap \{ f_{j,\lambda }\mid \lambda \in \Lambda _{j}\} 
\subset \mbox{int}(\supp\,\tau _{3}\cap \{ f_{j,\lambda }\mid \lambda \in \Lambda _{j}\})$ for 
each $j=1,\ldots, m$, and such that the statement in Claim 1 in the proof of Theorem~\ref{t:rcdnkmain1} 
holds for $\tau _{3}.$ By Lemmas~\ref{l:lcbwj} and \ref{l:nalnotsb}, we have 
\begin{equation}
\label{eq:clmingt3}
\cup _{L\in \Min(G_{\tau _{3}},\CCI )}L\subset \cap _{j=1}^{m}S({\cal W}_{j}).
\end{equation} 
Also, since $\supp\,\tau _{2}\subset \supp\,\tau _{3}$, 
there exists no element in $\Min(G_{\tau _{3}},\CCI )$ which is attracting for $\tau _{3}.$ 
As before, we may assume that $F(G_{\tau _{3}})\neq \emptyset.$ 
There exists a $k\in \NN $ for which the statement of Claim 2 in the proof of Theorem~\ref{t:rcdnkmain1} 
holds. We fix such an element. It is easy to see that 
statements in Claims 4,5 hold for $\tau _{3}$ even under our assumptions. 
Let ${\cal I}, C_{\tau _{3}}$ be as in the proof of Theorem~\ref{t:rcdnkmain1}.    
Then the statement of Claim 6 in the proof of Theorem~\ref{t:rcdnkmain1} holds for $\tau _{3}.$ 
More precisely, by (\ref{eq:clmingt3}) we have that 
\begin{equation}
\label{eq:ifzinf}
\mbox{ if }z\in F(G_{\tau _{3}}), \mbox{ then }
\overline{G_{\tau _{3}}(z)}\cap (\cup _{L\in {\cal I}}L)\neq \emptyset  \mbox{ and }
\overline{G_{\tau _{3}}(z)}\cap C_{\tau _{3}}\neq \emptyset .
\end{equation}
As in the proof of Theorem~\ref{t:rcdnkmain1}, 
let $p>\sum _{j=1}^{m}\sharp S({\cal W}_{j})+1 $ and let 
\begin{equation}
\label{eq:Hzin}
H:=\{ z\in F(G_{\tau _{3}})\mid \overline{G_{\tau _{3}}(z)}\cap C_{\tau _{3}}\neq \emptyset \} .
\end{equation} 
Then we have that 
\begin{equation}
\label{eq:stclaim7}
\mbox{ the statement of Claim 7 in the proof of Theorem~\ref{t:rcdnkmain1} holds 
for our } \tau _{3}.
\end{equation}
 Let $\tau _{4}$ be an element close to $\tau _{3}$ such that 
$\supp\,\tau _{3}\cap \{ f_{j,\lambda }\mid \lambda \in \Lambda _{j}\} 
\subset \mbox{int}(\supp\,\tau _{4}\cap \{ f_{j,\lambda }\mid 
\lambda \in \Lambda _{j}\})$ with respect to the topology 
in $\{ f_{j, \lambda }\mid \lambda \in \Lambda _{j}\}$  for each $j=1,\ldots, m.$ 
Then by (\ref{eq:clmingt3}) and 
Lemmas~\ref{l:lcbwj} and ~\ref{l:nalnotsb}, we have that 
\begin{equation}
\label{eq:clmgt4}
\cup _{L\in \Min(G_{\tau _{4}},\CCI )}L\subset \cap _{j=1}^{m}S({\cal W}_{j}). 
\end{equation}
%Also, each element of $\Min(G_{\tau _{4}},\CCI )$ is not attracting for $\tau _{4}.$ 
Moreover, by  (\ref{eq:stclaim7}), 
we see that for each $z\in H$ there exists an element $g_{z}\in \supp\,\tau _{4}$ such that 
$\overline{G_{\tau _{3}}(g_{z}(z))}\cap \mbox{int}(J(G_{\tau _{3}}))\neq \emptyset .$ 
In particular, $H\subset J(G_{\tau _{4}}).$ Combining this with (\ref{eq:ifzinf}) and (\ref{eq:Hzin}), 
it follows that $F(G_{\tau _{3}})\subset J(G_{\tau _{4}}).$ Hence $J(G_{\tau _{4}})=\CCI .$ 
Therefore we obtain that $\tau _{4}\in {\frak M}_{1,c,JF}({\cal Y},\{ {\cal W}_{j}\} _{j=1}^{m}). $ 
Thus we have proved our theorem. 
\end{proof}
We now prove the following theorem on the systems generated by 
weakly mean stable elements. 
% $\tau .$ 
\begin{thm}
\label{t:wmsneglyap}
Let $\tau \in {\frak M}_{1,c}(\emRat)$ be weakly mean stable. 
Suppose $\sharp J(G_{\tau })\geq 3.$ 
Suppose that for each $L\in \emMin(G_{\tau }, J_{\ker }(G_{\tau }))$, we have 
$\chi (L, \tau )\neq 0.$ 
Suppose also that for each $L\in \emMin(G_{\tau },J_{\ker }(G_{\tau }))$, 
if $\chi (L,\tau )>0$ then for each $z\in L$ and for each $g\in \mbox{supp}\,\tau$, 
we have $Dg_{z}\neq 0.$ Then all of the following hold.   
\begin{itemize}
\item[{\em (i)}] 
$\sharp J_{\ker }(G_{\tau })<\infty $.  
%$J_{\ker }(G_{\tau })\subset \cap _{j=1}^{m}S({\cal W}_{j}).$ 
\item[{\em (ii)}] 
For each $L\in \emMin(G_{\tau },\CCI )$ with $L\not\subset J_{\ker}(G_{\tau })$, we have that 
$L$ is  attracting for $\tau .$ 
\item[{\em (iii)}] 
For each $z\in F(G_{\tau })$, we have that 
$\overline{G_{\tau }(z)}\cap ((\cup _{L\in \emMin(G_{\tau },\CCI ), L\not\subset J_{\ker }(G_{\tau })}L)\neq \emptyset .$ 
\item[{\em (iv)}] 
All statements (i) --(vii) in Theorem~\ref{t:zfggzcc} hold for $\tau .$ 
\item[{\em (v)}] 
Let $H_{+,\tau }=\{ L\in \emMin(G_{\tau }, J_{\ker }(G_{\tau }))\mid  \chi (\tau, L)>0\}$ and 
let $\Omega _{\tau }$ be the set of points 
$y\in \CCI $ for which  
$\tilde{\tau }(\{ \gamma \in X_{\tau }\mid \exists n\in \NN \mbox{ s.t. }
\gamma _{n,1}(y)\in \cup _{L\in H_{+,\tau}}L\} )=0.$  
Then %There exists a subset $\Omega _{\tau }$ of $\CCI $ with 
we have 
$\Omega _{\tau }=F_{pt}^{0}(\tau ), $ 
$\sharp (\CCI \setminus \Omega _{\tau })\leq \aleph _{0}$ 
%such that 
and for each $z\in \Omega _{\tau }$, 
$\tilde{\tau }(\{ \gamma \in X_{\tau }\mid z\in J_{\gamma }\} )=0.$
Moreover, for $\tilde{\tau }$-a.e.$\gamma \in X _{\tau }$, we have 
{\em Leb}$_{2}(J_{\gamma })=0.$ Moreover, 
$\cup _{L\in H_{+,\tau }}L\subset J_{pt}^{0}(\tau )= \CCI \setminus \Omega _{\tau }$ and 
$\sharp J_{pt}^{0}(\tau )\leq \aleph_{0}.$ 

\item[{\em (vi)}] 
Let $\Omega _{\tau }$ be as in (v). Then 
$\sharp (\CCI \setminus \Omega _{\tau })\leq \aleph _{0}$ and 
there exist a constant $c_{\tau }<0$ and a constant $\rho _{\tau }\in (0,1)$ 
%and a subset $B_{\tau }$ of $\CCI $ 
%with $\sharp (\CCI \setminus B_{\tau })\leq \aleph _{0}$ 
such that for each $z\in \Omega_{\tau }$, there exists a Borel subset 
$C_{\tau ,z}$ of $X _{\tau }$ with $\tilde{\tau }(C_{\tau ,z})=1$ 
satisfying 
%the following {\em (a)} and {\em (b)}. 
that for each $\gamma =(\gamma _{1},\gamma _{2},\ldots )\in C_{\tau ,z}$ and for each 
$m\in \NN \cup \{ 0\},$  
we have the following {\em (a)} and {\em (b)}. 
\begin{itemize}
\item[{\em (a)}]  
$$\limsup _{n\rightarrow \infty }\frac{1}{n}\log \| D(\gamma _{n+m,1+m})_{\gamma _{m,1}(z)}\| _{s}\leq c_{\tau }<0.$$
\item[{\em (b)}] 
There exist a constant $\delta =\delta (\tau, z,\gamma, m)>0$, a constant 
$\zeta =\zeta (\tau, z, \gamma, m)>0$ and an element  
$L=L(\tau, z, \gamma )\in \emMin(G_{\tau },\CCI )$  
which is either {\em (i)} ``attracting for $\tau $'', or {\em (ii) } ``finite and included in $J_{\ker }(G_{\tau }) $ with 
$\chi (\tau, L)<0$'', 
  such that 
$$\mbox{{\em diam}}(\gamma _{n+m,1+m}(B(\gamma _{m,1}(z),\delta )))\leq \zeta \rho _{\tau }^{n}\ \mbox{ for all }n\in \NN, $$ 
and 
$$d(\gamma _{n+m,1+m}(\gamma _{m,1}(z)), \ L)
%\cup _{L\in \emMin(G_{\tau },\CCI ), L \mbox{ is attracting 
%for }\tau } L) 
\leq \zeta \rho _{\tau }^{n} \mbox{\ \  for all }n\in \NN .$$ 
\end{itemize}
\item[{\em (vii)}] 
For $\tilde{\tau }$-a.e. $\gamma \in X_{\tau }$, 
for $\mbox{{\em Leb}}_{2}$-a.e. $z\in \CCI $, 
there exists an element $L=L(\tau, \gamma, z)\in \emMin(G_{\tau },\CCI )$ 
which is either {\em (i)} ``attracting for $\tau $'', or {\em (ii) } ``finite and included in $J_{\ker }(G_{\tau }) $ with 
$\chi (\tau, L)<0$'', 
  such that 
$d(\gamma _{n,1}(z), L)\rightarrow \infty $ as 
$n\rightarrow \infty .$  
Also, for $\tilde{\tau }$-a.e. $\gamma \in X_{\tau }$, for each $z\in F_{\gamma }$, 
there exists an element $L=L(\tau, \gamma, z)\in \emMin(G_{\tau },\CCI )$  
which is either {\em (i)} ``attracting for $\tau $'', or {\em (ii) } ``finite and included in $J_{\ker }(G_{\tau }) $ with 
$\chi (\tau, L)<0$'', 
such that 
$d(\gamma _{n,1}(z), L)\rightarrow \infty $ as 
$n\rightarrow \infty .$

\item[{\em (viii)}] 
Let $\Omega _{\tau }$ be as in {\em (v)}. Then 
we have $\Omega _{\tau }= F_{pt}^{0}(\tau )$,  
$\sharp (\CCI \setminus F_{pt}^{0}(\tau ))\leq \aleph _{0}$ and 
for each $L\in \emMin(G_{\tau },\CCI )$, for each $j=1,\ldots, r_{L}$, where 
$r_{L}=\dim _{\CC }U_{\tau ,L}$,  and for each $y\in F_{pt}^{0}(\tau )$, we have that 
$\lim _{z\in \CCI ,z\rightarrow y}T_{L,\tau }(z)=T_{L,\tau }(y)$ and 
$\lim _{z\in \CCI ,z\rightarrow y}\alpha (L_{j},z)=\alpha (L_{j},y)$, where 
$\alpha (L_{j},\cdot )$ is the function coming from Theorem~\ref{t:zfggzcc} (iii). 
\item[{\em (ix)}] 
Let $H_{+,\tau }$ and $\Omega _{\tau }$ be as in {\em (v)}. 
%Let $L\in H_{+,\tau }$ and $r_{L}=\dim_{\CC }U_{\tau ,L}.$ 
Let $y\in J_{pt}^{0}(\tau )=\CCI \setminus \Omega _{\tau }. $ 
Then there exist an element $L\in H_{+,\tau }$ and $j\in \{1,\ldots, r_{L}\}$ 
such that all of the following hold. 
\begin{itemize}
\item[{\em (a)}] 
$T_{L,\tau }(z)$ does not tend to $T_{L,\tau }(y)$ as $z\rightarrow y$.  
\item[{\em (b)}]  
$\alpha (L_{j},z)$ does not tend to $\alpha (L_{j},y)$ as $z\rightarrow y.$ 
Here, for the notation $L_{j}$, see Theorem~\ref{t:zfggzcc} (i).
\item[{\em (c)}] 
Let $\varphi _{L}\in C(\CCI )$ be any element such that 
$\varphi _{L}|_{L}=1$ and $\varphi _{L}|_{L'}=0$ for any $L'\in \emMin(G_{\tau },\CCI )$ 
with $L'\neq L$.  Then the convergence in (\ref{eq:mtnlvy}) in 
Theorem~\ref{t:zfggzcc} for $\varphi =\varphi _{L}$ is not uniform in any neighborhood of $y.$ 
\item[{\em (d)}] 
There exists a Borel subset $E_{\tau ,y}$ of $X_{\tau }$ with 
$\tilde{\tau }(E_{\tau ,y})=T_{L,\tau }(y)>0$ such that for each $\gamma \in E_{\tau ,y}$, 
there exists an element $m\in \NN $ such that $\gamma _{m,1}(y)\in L$ and \\ 
$\lim _{n\rightarrow \infty }\frac{1}{n}\log \| D(\gamma _{n+m,1+m})_{\gamma _{m,1}(y)}\| _{s}
=\chi (\tau, L)>0.$
\end{itemize} 
\end{itemize}

\end{thm}
\begin{proof}
By Lemma~\ref{l:wmsdtjk}, 
statements (i)--(iv) hold. 
  By Theorem~\ref{t:jkfcln0}, 
  we have $\sharp (\CCI \setminus \Omega _{\tau })\leq \aleph _{0}$ 
%such that 
and for each $z\in \Omega _{\tau }$, 
$\tilde{\tau }(\{ \gamma \in (\mbox{supp}\,\tau)^{\NN }\mid z\in J_{\gamma }\} )=0.$
Moreover, for $\tilde{\tau }$-a.e.$\gamma \in X _{\tau }$, we have 
Leb$_{2}(J_{\gamma })=0.$ Moreover, 
$J_{pt}^{0}(\tau )\subset \CCI \setminus \Omega _{\tau }$ and 
$\sharp J_{pt}^{0}(\tau )\leq \aleph_{0}.$ 
In order to prove $\Omega _{\tau }=F_{pt}^{0}(\tau )$, let  
$y\in \CCI \setminus \Omega _{\tau }.$ 
Then 
by Lemma~\ref{l:chlpa0}, 
\begin{equation}
\label{eq:ttgxt}
\tilde{\tau }(\{ \gamma \in X_{\tau }\mid d(\gamma _{n,1}(y), \cup _{L\in H_{+,\tau }}L)\rightarrow 0 
\mbox{ as }n\rightarrow \infty \} )>0.
\end{equation} 
Since $\sharp (\CCI \setminus \Omega _{\tau })\leq \aleph _{0}$, there exists a 
sequence $\{ x_{m}\} _{m=1}^{\infty }$ in $\Omega _{\tau }$ 
such that $x_{m}\rightarrow y$ as $m\rightarrow \infty .$ 
  Then by Lemma~\ref{l:chlpa0} again, we have 
% \begin{equation}
%\label{eq:ttgxt2}
$\tilde{\tau }(\{ \gamma \in X_{\tau }\mid d(\gamma _{n,1}(z_{m}), \cup _{L\in H_{+,\tau }}L)\rightarrow 0 
\mbox{ as }n\rightarrow \infty \} )=0$ for each $m\in \NN.$  Combining this with 
(iv) and Theorem~\ref{t:zfggzcc} (iv), we obtain that 
\begin{equation}
\label{eq:ttgxt2}
\tilde{\tau }(\{ \gamma \in X_{\tau }\mid d(\gamma _{n,1}(z_{m}), \cup _{L\in \Min(G_{\tau },\CCI )
\setminus  H_{+,\tau }}L)\rightarrow 0 
\mbox{ as }n\rightarrow \infty \} )=1 \mbox{ for each }m\in \NN.
\end{equation} 
%\end{equation}  
By (\ref{eq:ttgxt}) and (\ref{eq:ttgxt2}), it follows that 
$y\in J_{pt}^{0}(\tau ).$ Hence $\Omega _{\tau }=F_{pt}^{0}(\tau ).$ 
Also, by the definition of $\Omega _{\tau }$, we have $\cup _{L\in H_{+,\tau }}L\subset 
\CCI \setminus \Omega _{\tau }.$ 
   Thus statement (v) holds. Moreover, by using the above argument, we can show 
   that there exist an element $L\in H_{+,\tau }$ and an element $j\in \{ 1,\ldots, r_{L}\}$ 
   such that (a) and (b) in (ix) hold. 
%   the first and the second statements in (ix). 
   By statement (a) in (ix) and the fact $\sharp \Min(G_{\tau })<\infty $, 
   statement (c) in (ix) holds.  
   Statement (d) in (ix) follows from 
   the definition of $\Omega _{\tau }$, Lemma~\ref{l:chlpa0} and  Birkhoff's ergodic theorem. 
   Hence statement (ix) holds. 

We now prove statement (vi). 
By Theorem~\ref{t:zfggzcc} (iv) and Lemma~\ref{l:chlpa0}, 
it follows that 
%there exists a subset $B_{\tau }$ of $\CCI $ with 
%$\sharp (\CCI \setminus B_{\tau })\leq \aleph _{0}$ such that 
for each $z\in \Omega_{\tau }$, there exists a Borel subset $D_{\tau ,z}$ of 
$X_{\tau }$ with $\tilde{\tau }(D_{\tau, z})=1$ 
satisfying that for each $\gamma =(\gamma _{1},\gamma _{2},\ldots )
\in D_{\tau ,z}$, we have 
\begin{equation}
\label{eq:dgn1zbig}
d\left(\gamma _{n,1}(z), \bigcup _{L\in \Min(G_{\tau },\CCI ), L \mbox{ is attracting }}L
\cup \bigcup  _{L\in \Min(G_{\tau }, J_{\ker }(G_{\tau })) \mbox{ and } \chi(L,\tau )<0}L\right)\rightarrow 0 \mbox{ as }
n\rightarrow \infty .
\end{equation} 
There exist a constant $\lambda_{\tau }\in (0,1)$ and a constant $C_{\tau }>0$ such that 
for each $\gamma =(\gamma _{1},\gamma _{2}\ldots, )\in X_{\tau }$, 
for each $z\in \cup _{L\in \Min(G_{\tau },\CCI ), L\mbox{ is attracting }} L$ and for each $n\in \NN $,  
we have $\| D(\gamma _{n,1})_{z}\| _{s}\leq C_{\tau }\lambda _{\tau }^{n}.$ 
Let $c_{\tau }:=\max \{ \log \lambda _{\tau }, 
\max _{L\in \Min(G_{\tau }, J_{\ker }(G_{\tau })), \chi (L,\tau )<0}\chi (L,\tau)\} 
<0$ (if there exists no $L\in \Min(G_{\tau }, J_{\ker }(G_{\tau }))$ with $\chi(L,\tau )<0$, then we set $c_{\tau }
=\log \lambda _{\tau }$).   
Then for each $z\in \Omega_{\tau }$, there exists a Borel subset $C_{\tau ,z}$ of 
$D_{\tau ,z}$ with $\tilde{\tau }(C_{\tau ,z})=1$ such that 
for each $\gamma =(\gamma _{1},\gamma _{2},\ldots )\in C_{\tau ,z}$ and for each 
$m\in \NN \cup \{ 0\},$ we have 
$$\limsup _{n\rightarrow \infty }\frac{1}{n}\log \| D(\gamma _{n+m,1+m})_{\gamma _{m,1}(z)}\| _{s}\leq c_{\tau }<0.$$ 
Also, by (\ref{eq:dgn1zbig}) and Lemma~\ref{l:chimune} and its proof, 
there exists an element $\rho _{\tau }\in (0,1)$ such that 
we can arrange $C_{\tau ,z}$ so that 
for any $\gamma \in C_{\tau ,z}$ and for any $m\in \NN \cup \{ 0\}$, 
statement (vi)(b) holds.  
Hence statement (vi) holds for $\tau .$ 

By statements (vi) and (v), statement (vii) holds. 

By (iv)(v) and Theorem~\ref{t:zfggzcc} (vii), statement (viii) holds. Thus we have proved our theorem. 
\end{proof}
We now prove the following theorem, which is one of the main results of this paper. 
\begin{thm}
\label{t:rcdnkmain2}
Let ${\cal Y}$ be a mild subset of $\emRatp$ and suppose that 
${\cal Y}$ is  non-exceptional and strongly nice with respect 
to some 
holomorphic families $\{ {\cal W}_{j}\} _{j=1}^{m}$ of rational maps. 
%, where 
%${\cal W}_{j}=\{ f_{j,\lambda }\} _{\lambda \in \Lambda _{j}}$ for each 
%$j=1,\ldots, m.$ 
%Suppose that $\{ f_{i,\lambda }\mid \lambda \in \Lambda _{i}\} \cap 
%\{ f_{j,\lambda }\mid \lambda \in \Lambda _{j}\} $ is nowhere dense in 
%$\{ f_{i,\lambda }\mid \lambda \in \Lambda _{i}\}$ for all $(i,j)$ with $i\neq j.$ 
Then there exists the largest open and dense subset ${\cal A}$ of 
$({\frak M}_{1,c}({\cal Y},\{ {\cal W}_{j}\} _{j=1}^{m}), {\cal O})$ such that for each $\tau \in {\cal A}$, 
all of the following hold.
\begin{itemize}
\item[{\em (i)}] 
$\tau $ is weakly mean stable. 
\item[{\em (ii)}]  
For each $L\in \emMin(G_{\tau },\CCI )$ with $L\subset \cap _{j=1}^{m}S({\cal W}_{j})$, we have 
$\chi(L,\tau )\neq 0.$ Moreover, for each 
$L\in \emMin(G_{\tau },\CCI )$ with $L\subset \cap _{j=1}^{m}S({\cal W}_{j})$, 
if $\chi (L,\tau )>0$, then for each $z\in L$ and for each $g\in \mbox{supp}\,\tau$, 
we have $Dg_{z}\neq 0.$ 
\item[{\em (iii)}] 
$\sharp J_{\ker }(G_{\tau })<\infty $ and $J_{\ker }(G_{\tau })\subset \cap _{j=1}^{m}S({\cal W}_{j}).$ 
%\item[{\em (iv)}] 
%There exists a subset $\Omega _{\tau }$ of $\CCI $ with 
%$\sharp (\CCI \setminus \Omega _{\tau })\leq \aleph _{0}$ such that 
%for each $z\in \Omega _{\tau }$, 
%$\tilde{\tau }(\{ \gamma \in \mbox{supp}\,\tau^{\NN }\mid z\in J_{\gamma }\} )=0.$
%Moreover, for $\tilde{\tau }$-a.e.$\gamma \in \mbox{supp}\,\tau^{\NN }$, we have 
%Leb$_{2}(J_{\gamma })=0.$ Moreover, $\sharp J_{pt}^{0}(\tau )\leq \aleph_{0}.$ 
\item[{\em (iv)}] 
For each $L\in \emMin(G_{\tau },\CCI )$ with $L\not\subset J_{\ker }(G_{\tau })$, we have that 
$L$ is  attracting for $\tau .$ 
\item[{\em (v)}] 
For each $z\in F(G_{\tau })$, we have that $\overline{G_{\tau }(z)}\cap ((\cup _{L\in \emMin(G_{\tau },\CCI ), 
L\not\subset J_{\ker }(G_{\tau })}L)\neq \emptyset .$ 
\item[{\em (vi)}] 
All statements (i) --(vii) in Theorem~\ref{t:zfggzcc} hold for $\tau .$ 
\item[{\em (vii)}] 
Let $H_{+,\tau }=\{ L\in \emMin(G_{\tau }, J_{\ker }(G_{\tau }))\mid  \chi (\tau, L)>0\}$ and 
let $\Omega _{\tau }$ be the set of points 
$y\in \CCI $ for which  
$\tilde{\tau }(\{ \gamma \in X_{\tau }\mid \exists n\in \NN \mbox{ s.t. }
\gamma _{n,1}(y)\in \cup _{L\in H_{+,\tau}}L\} )=0.$  
Then %There exists a subset $\Omega _{\tau }$ of $\CCI $ with 
we have 
$\Omega _{\tau }=F_{pt}^{0}(\tau ) $, 
$\sharp (\CCI \setminus \Omega _{\tau })\leq \aleph _{0}$ 
%such that 
and for each $z\in \Omega _{\tau }$, 
$\tilde{\tau }(\{ \gamma \in (\mbox{supp}\,\tau)^{\NN }\mid z\in J_{\gamma }\} )=0.$
Moreover, for $\tilde{\tau }$-a.e.$\gamma \in X _{\tau }$, we have 
{\em Leb}$_{2}(J_{\gamma })=0.$ Moreover, 
$\cup _{L\in H_{+,\tau }}L\subset J_{pt}^{0}(\tau )= \CCI \setminus \Omega _{\tau }$ and 
$\sharp J_{pt}^{0}(\tau )\leq \aleph_{0}.$ 
\item[{\em (viii)}] 
Let $\Omega _{\tau }$ be as in {\em (vii)}. Then 
$\sharp (\CCI \setminus \Omega _{\tau })\leq \aleph _{0}$ and 
there exist a constant $c_{\tau }<0$ and a constant $\rho _{\tau }\in (0,1)$ 
%and a subset $B_{\tau }$ of $\CCI $ 
%with $\sharp (\CCI \setminus B_{\tau })\leq \aleph _{0}$ 
such that for each $z\in \Omega_{\tau }$, there exists a Borel subset 
$C_{\tau ,z}$ of $X _{\tau }$ with $\tilde{\tau }(C_{\tau ,z})=1$ 
satisfying that for each $\gamma =(\gamma _{1},\gamma _{2},\ldots )\in C_{\tau ,z}$ and 
for each $m\in \NN \cup \{ 0\} $, 
we have the following {\em (a)} and {\em (b)}. 
\begin{itemize}
\item[{\em (a)}]  
$$\limsup _{n\rightarrow \infty }\frac{1}{n}\log \| D(\gamma _{n+m,1+m})_{\gamma _{m,1}(z)}\| _{s}\leq c_{\tau }<0.$$
\item[{\em (b)}] 
There exist a constant $\delta =\delta (\tau, z,\gamma, m)>0$,  a constant 
$\zeta =\zeta (\tau, z, \gamma, m)>0$ and an element 
$L=L(\tau, z,\gamma )\in \emMin(G_{\tau },\CCI )$  
which is either {\em (i)} ``attracting for $\tau $'', or {\em (ii) } ``finite and included in $J_{\ker }(G_{\tau }) $ with 
$\chi (\tau, L)<0$'', 
such that 
$$\mbox{{\em diam}}(\gamma _{n+m,1+m}(B(\gamma _{m,1}(z),\delta )))\leq \zeta \rho _{\tau }^{n}\ \mbox{ for all }n\in \NN, $$
and 
$$d(\gamma _{n+m,1+m}(\gamma _{m,1}(z)), \ L)
%\cup _{L\in \emMin(G_{\tau },\CCI ),\ L \mbox{ is attracting 
%for }\tau } L) 
\leq \zeta \rho _{\tau }^{n} \mbox{\ \  for all }n\in \NN .$$ 
\end{itemize} 
%\item[{\em (viii)}] 
%There exist a constant $c_{\tau }<0$ and a subset $B_{\tau }$ of $\CCI $ 
%with $\sharp (\CCI \setminus B_{\tau })\leq \aleph _{0}$ 
%such that for each $z\in B_{\tau }$, there exists a Borel subset 
%$C_{\tau ,z}$ of $\mbox{supp}\,\tau^{\NN }$ with $\tilde{\tau }(C_{\tau ,z})=1$ 
%satisfying that for each $\gamma =(\gamma _{1},\gamma _{2},\ldots )\in C_{\tau ,z}$, 
%we have 
%$$\limsup _{n\rightarrow \infty }\frac{1}{n}\log \| D(\gamma _{n,1})_{z}\| _{s}\leq c_{\tau }<0.$$
\item[{\em (ix)}] 
For $\tilde{\tau }$-a.e. $\gamma \in X_{\tau }$, 
for $\mbox{{\em Leb}}_{2}$-a.e. $z\in \CCI $,  
there exists an element $L=L(\tau, \gamma, z)\in \emMin(G_{\tau },\CCI )$  
which is either {\em (i)} ``attracting for $\tau $'', or {\em (ii) } ``finite and included in $J_{\ker }(G_{\tau }) $ with 
$\chi (\tau, L)<0$'', 
such that 
$d(\gamma _{n,1}(z), L)\rightarrow \infty $ as 
$n\rightarrow \infty .$  
Also, for $\tilde{\tau }$-a.e. $\gamma \in X_{\tau }$, for each $z\in F_{\gamma }$, 
there exists an element $L=L(\tau, \gamma, z)\in \emMin(G_{\tau}, \CCI )$
which is either {\em (i)} ``attracting for $\tau $'', or {\em (ii) } ``finite and included in $J_{\ker }(G_{\tau }) $ with 
$\chi (\tau, L)<0$'', 
  such that 
$d(\gamma _{n,1}(z), L)\rightarrow \infty $ as 
$n\rightarrow \infty .$ 

\item[{\em (x)}] 
Let $\Omega _{\tau }$ be as in {\em (vii)}. Then 
we have $\Omega _{\tau }= F_{pt}^{0}(\tau )$,  
$\sharp (\CCI \setminus F_{pt}^{0}(\tau ))\leq \aleph _{0}$ and 
for each $L\in \emMin(G_{\tau },\CCI )$, for each $j=1,\ldots, r_{L}$, where 
$r_{L}=\dim _{\CC }U_{\tau ,L}$,  and for each $y\in F_{pt}^{0}(\tau )$, we have that 
$\lim _{z\in \CCI ,z\rightarrow y}T_{L,\tau }(z)=T_{L,\tau }(y)$ and 
$\lim _{z\in \CCI ,z\rightarrow y}\alpha (L_{j},z)=\alpha (L_{j},y)$, where 
$\alpha (L_{j},\cdot )$ is the function coming from Theorem~\ref{t:zfggzcc} (iii). 

%\item[{\em (ix)}] 
%We have $\sharp (\CCI \setminus F_{pt}^{0}(\tau ))\leq \aleph _{0}$ and 
%for each $L\in \emMin(G_{\tau },\CCI )$, for each $j=1,\ldots, r_{L}$, where 
%$r_{L}=\dim _{\CC }U_{\tau ,L}$,  and for each $y\in F_{pt}^{0}(\tau )$, we have that 
%$\lim _{z\in \CCI ,z\rightarrow y}T_{L,\tau }(z)=T_{L,\tau })(y)$ and 
%$\lim _{z\in \CCI ,z\rightarrow y}\alpha (L_{j},z)=\alpha (L_{j},y)$, where 
%$\alpha (L_{j},\cdot )$ is the function coming from Theorem~\ref{t:zfggzcc} (iii). 
\item[{\em (xi)}] 
Let $H_{+,\tau }$ and $\Omega _{\tau }$ be as in {\em (vii)}. 
%Let $L\in H_{+,\tau }$ and $r_{L}=\dim_{\CC }U_{\tau ,L}.$ 
Let $y\in J_{pt}^{0}(\tau )=\CCI \setminus \Omega _{\tau }. $ 
Then there exist an element $L\in H_{+,\tau }$ and an element $j\in \{ 1,\ldots, r_{L}\} $ 
such that all of the following hold. 
\begin{itemize}
\item[{\em (a)}]  
$T_{L,\tau }(z)$ does not tend to $T_{L,\tau }(y)$ as $z\rightarrow y$. 
\item[{\em (b)}]  
$\alpha (L_{j},z)$ does not tend to $\alpha (L_{j},y)$ as $z\rightarrow y.$ 
Here, for the notation $L_{j}$, see Theorem~\ref{t:zfggzcc} (i).  
\item[{\em (c)}] 
Let $\varphi _{L}\in C(\CCI )$ be any element such that 
$\varphi _{L}|_{L}=1$ and $\varphi _{L}|_{L'}=0$ for any $L'\in \emMin(G_{\tau },\CCI )$ 
with $L'\neq L$. Then the convergence in (\ref{eq:mtnlvy}) in 
Theorem~\ref{t:zfggzcc} for $\varphi =\varphi _{L}$ is not uniform in any neighborhood of $y.$ 
\item[{\em (d)}] 
There exists a Borel subset $E_{\tau ,y}$ of $X_{\tau }$ with 
$\tilde{\tau }(E_{\tau ,y})=T_{L,\tau }(y)>0$ such that for each $\gamma \in E_{\tau ,y}$, 
there exists an element $m\in \NN $ such that $\gamma _{m,1}(y)\in L$ and \\ 
$\lim _{n\rightarrow \infty }\frac{1}{n}\log \| D(\gamma _{n+m,1+m})_{\gamma _{m,1}(y)}\| _{s}
=\chi (\tau, L)>0.$ 
\end{itemize}
\end{itemize}
\end{thm}
\begin{proof}
By Theorem~\ref{t:rcdnkmain1} \
%ref{t:zfggzcc} and 
and Lemma~\ref{l:wiwjnod}, there exists an open and dense subset 
${\cal A}$ of the space $({\frak M}_{1,c}({\cal Y},\{ {\cal W}_{j}\} _{j=1}^{m}), {\cal O})$ 
such that for each $\tau \in {\cal A}$,  
statements (i),  (ii) and (iii) hold. 
By Theorem~\ref{t:wmsneglyap}, 
for each $\tau \in {\cal A}$, statements (iv)--(xi) hold. Thus we have proved our theorem.  
\end{proof}

We now prove a theorem in which we do not assume that ${\cal Y}$ is mild with 
$\{ {\cal W}_{j}\} _{j=1}^{m}.$ 
\begin{thm}
\label{t:rcdnkmain2nomild}
Let ${\cal Y}$ be a  non-exceptional and strongly nice subset of $\emRatp$ with respect to some 
holomorphic families $\{ {\cal W}_{j}\} _{j=1}^{m}$ of rational maps. 
%, where 
%${\cal W}_{j}=\{ f_{j,\lambda }\} _{\lambda \in \Lambda _{j}}$ for each 
%$j=1,\ldots, m.$ 
%Suppose that $\{ f_{i,\lambda }\mid \lambda \in \Lambda _{i}\} \cap 
%\{ f_{j,\lambda }\mid \lambda \in \Lambda _{j}\} $ is nowhere dense in 
%$\{ f_{i,\lambda }\mid \lambda \in \Lambda _{i}\}$ for all $(i,j)$ with $i\neq j.$ 
Then there exists the largest open and dense subset ${\cal A}$ of 
$({\frak M}_{1,c,mild}({\cal Y},\{ {\cal W}_{j}\} _{j=1}^{m}), {\cal O})$ such that for each $\tau \in {\cal A}$, 
all statements {\em (i)--(xi)} in Theorem~\ref{t:rcdnkmain2} hold. 
Furthermore, we have 
$$\overline{{\cal A}\cup {\frak M}_{1,c,JF}({\cal Y}, \{ {\cal W}_{j}\} _{j=1}^{m})}=
{\frak M}_{1,c}({\cal Y}, \{ {\cal W}_{j}\} _{j=1}^{m})$$ 
with respect to the topology ${\cal O}.$  
\end{thm} 
\begin{proof}
By the arguments in the proof of Theorem~\ref{t:rcdnkmain2} and 
Theorem~\ref{t:nonmild}, it is easy to see that the statements of our theorem hold.  
\end{proof}
                                                                                                                                                                                                                                                                      
%\begin{df}
%\label{d:stau}
%For each $\tau \in {\frak M}_{1,c}(\Rat)$, we set 
%$S_{\tau }:=\bigcup _{L\in \Min (G_{\tau },\CCI )}L.$
%\end{df}

We now give corollaries of Theorems~\ref{t:rcdnkmain1} and \ref{t:rcdnkmain2}. 
\begin{cor}
\label{c:rcdnkmain1f}
Let ${\cal Y}$ be a mild subset of 
$\emRatp$ and suppose that ${\cal Y}$ is 
 strongly nice with respect to some holomorphic families 
$\{ {\cal W}_{j}\} _{j=1}^{m}$ of rational maps. 
Then 
the set 
$$\{ \tau \in {\frak M}_{1,c}({\cal Y},\{ {\cal W}_{j}\} _{j=1}^{m})
\mid \tau \mbox{ is weakly mean stable and } \sharp \mbox{supp}\,\tau<\infty \}$$ 
is dense in $({\frak M}_{1,c}({\cal Y},\{ {\cal W}_{j}\} _{j=1}^{m}), {\cal O})$. 
Moreover, there exists a dense subset ${\cal A}$ of 
$({\frak M}_{1,c}({\cal Y},\{ {\cal W}_{j}\} _{j=1}^{m}), {\cal O})$ such that for each 
$\tau \in {\cal A}$, we have $\sharp \mbox{supp}\,\tau <\infty $ and 
all statements {\em (i)--(v)} of Theorem~\ref{t:rcdnkmain1} hold for $\tau .$ 
\end{cor}
\begin{cor}
\label{c:rcdnkmain2f}
Let ${\cal Y}$ be a mild supset of $\emRatp$ and suppose that 
${\cal Y}$ is  non-exceptional and strongly nice 
with respect to some 
holomorphic families $\{ {\cal W}_{j}\} _{j=1}^{m}$ of rational maps. 
%, where 
%${\cal W}_{j}=\{ f_{j,\lambda }\} _{\lambda \in \Lambda _{j}}$ for each 
%$j=1,\ldots, m.$ 
%Suppose that $\{ f_{i,\lambda }\mid \lambda \in \Lambda _{i}\} \cap 
%\{ f_{j,\lambda }\mid \lambda \in \Lambda _{j}\} $ is nowhere dense in 
%$\{ f_{i,\lambda }\mid \lambda \in \Lambda _{i}\}$ for all $(i,j)$ with $i\neq j.$ 
Let ${\cal A}$ be the largest open and dense subset of 
$({\frak M}_{1,c}({\cal Y},\{ {\cal W}_{j}\} _{j=1}^{m}), {\cal O})$ given in 
Theorem~\ref{t:rcdnkmain2}. 
Let ${\cal A}^{f}:=\{ \tau \in {\cal A}\mid \sharp \mbox{supp}\,\tau<\infty \}.$ 
Then ${\cal A}^{f}$ is a dense subset of ${\cal A}$ 
and is a dense subset of 
%there exists a dense subset ${\cal A}$ of 
$({\frak M}_{1,c}({\cal Y},\{ {\cal W}_{j}\} _{j=1}^{m}), {\cal O})$ such that for each 
$\tau \in {\cal A}^{f}$, we have that 
$\sharp \mbox{{\em supp}}\, \tau<\infty $ and all statements {\em (i)--(xi)} in Theorem~\ref{t:rcdnkmain2} 
hold for $\tau .$ 
Also, let 
${\cal A}_{+}:=\{ \tau \in {\cal A}\mid 
\exists L\in \emMin(G_{\tau },\CCI ) \mbox{ s.t. }\chi (\tau, L)>0\} $ 
and let ${\cal A}_{+}^{f}:={\cal A}_{+}\cap {\cal A}^{f}.$ 
Then  ${\cal A}_{+}$ is an open subset of ${\cal A}$ 
(hence an open subset of $({\frak M}_{1,c}({\cal Y},\{ {\cal W}_{j}\} _{j=1}^{m}), {\cal O}))$ 
and ${\cal A}_{+}^{f}$ is a dense subset of ${\cal A}_{+}.$ 
Moreover, for each  $\tau \in {\cal A}_{+}^{f}$,   
%$\sharp \mbox{{\em supp}}\,\tau <\infty $ 
%and there exists a minimal set $L\in \emMin(G_{\tau },J_{\ker }
%(G_{\tau }))$ with $\chi (\tau, L)>0$, then 
we have 
$\overline{J_{pt}^{0}(\tau )}=
J(G_{\tau })$ which is a perfect set. 
 \end{cor}
\begin{proof}
It is easy to show that ${\cal A}^{f}$ is dense in ${\cal A}.$ 
Thus ${\cal A}^{f}$ is dense in $({\frak M}_{1,c}({\cal Y},\{ {\cal W}_{j}\} _{j=1}^{m}), {\cal O}).$ 
Also, by statement (ii) in Theorem~\ref{t:rcdnkmain2}, it is easy 
to show that ${\cal A}_{+}$ is open in ${\cal A}.$  
In order to prove the last statement, suppose $\tau \in {\cal A}_{+}^{f}$.   
%$\sharp \mbox{{\em supp}}\,\tau <\infty $ 
%and there exists a minimal set $L\in \Min(G_{\tau },J_{\ker }
%(G_{\tau }))$ with $\chi (\tau, L)>0$. 
Since $\cup _{L\in H_{+,\tau }}L\subset J_{pt}^{0}$, 
where $H_{+,\tau}:=\{ L\in \Min(G_{\tau}, J_{\ker }(G_{\tau }))
\mid \chi (\tau, L)>0\} $ (see Theorem~\ref{t:rcdnkmain2}), 
we have 
$J_{pt}^{0}\neq \emptyset .$ Moreover, 
since $\mbox{supp}\,\tau\subset \Ratp$, we have 
$J_{pt }^{0}\subset J(G_{\tau })\subset \CCI \setminus E(G_{\tau })$ (recall that 
$E(G_{\tau })$ denotes the exceptional set of $G_{\tau }$). 
Hence $\overline{G_{\tau }^{-1}(J_{pt }^{0}(\tau ))}\supset J(G_{\tau })$ (see \cite[Lemma 3.2]{HM}). 
Also, by the definition of $\Omega _{\tau }$, since 
$\supptau $ is finite, 
we have $\CCI \setminus \Omega =(G_{\tau }\cup \{ Id\} )^{-1}
(\cup _{L\in H_{+.\tau }}L)$ and  
$G_{\tau }^{-1}(\CCI \setminus \Omega _{\tau })\subset 
\CCI \setminus \Omega _{\tau }.$ 
Furthermore, by Theorem~\ref{t:rcdnkmain2} (vii), we have 
$J_{pt}^{0}(\tau )=\CCI \setminus \Omega _{\tau }.$ 
It follows that 
$\overline{G_{\tau }^{-1}(J_{pt }^{0}(\tau ))}\subset \overline{J_{pt }^{0}(\tau )}\subset J(G_{\tau }).$ 
 Therefore $\overline{J_{pt}^{0}(\tau )}=\overline{G_{\tau }^{-1}(J_{pt}^{0}(\tau ))}=J(G_{\tau }).$ 
 Finally, by \cite[Lemma 3.1]{HM}, $J(G_{\tau })$ is perfect. 
\end{proof}
\section{Random relaxed Newton's method}
\label{Random}
In this section we apply Theorems~\ref{t:rcdnkmain1}, \ref{t:rcdnkmain2} and the other results in 
the previous sections to 
random relaxed Newton's methods in which we find roots of given any polynomial $g$ with $\deg (g)\geq 2.$ 
\begin{df}
\label{d:randomNewton}
Let $g\in {\cal P}.$ 
Let $\Lambda :=\{ \lambda \in \CC \mid |\lambda -1|<1\} $ and 
let $N_{g,\lambda }(z)=z-\lambda \frac{g(z)}{g'(z)}$ for each $\lambda \in \Lambda .$ 
Let ${\cal W}_{g}=\{N_{g,\lambda }\} _{\lambda \in \Lambda }.$ (Note that this is a holomorphic family 
of rational maps over $\Lambda .$)  
Let ${\cal Y}_{g}:=\{ N_{g,\lambda }\in \Rat \mid \lambda \in \Lambda \}.$  
Then  ${\cal Y}_{g}$ is called the random relaxed Newton's method set for $g$ 
and  ${\cal W}_{g}$ is called the {\bf random relaxed Newton's method family for $g.$}  
Also, $({\cal Y}_{g}, {\cal W}_{g})$ is called the 
{\bf random relaxed Newton's method scheme for $g.$ }
Moreover, for each $\tau \in {\frak M}_{1,c}({\cal Y})$, the random dynamical system on $\CCI $ 
generated by $\tau $ is called a {\bf random relaxed Newton's method (or random relaxed Newton's method system) 
for $g. $} Also, let $Q_{g}:=\{ z_{0}\in \CC \mid g(z_{0})=0\}.$ 
\end{df}
We need the following lemma to investigate random relaxed Newton's methods 
and other examples to which we can apply Theorems~\ref{t:rcdnkmain1} and \ref{t:rcdnkmain2}.  
The proof is easy and it is left to the reader. 
\begin{lem}
\label{l:singfam}
Let ${\cal Y}$ be a nice subset of {\em Rat} with respect to a holomorphic family 
${\cal W}=\{ f_{\lambda }\} _{ \lambda \in \Lambda }$  of rational maps. 
Then ${\cal Y}$ is strongly nice with respect to ${\cal W}$.  %and ${\cal Y}$ satisfies the assumptions in Theorem~\ref{t:rcdnkmain1}. 
Moreover, if, in addition to the assumption of our lemma, ${\cal Y}$ satisfies that 
for each $\Gamma \in \emCpt(\{ f_{\lambda }\mid \lambda \in \Lambda \})$ and for each 
$L\in \emMin (\langle \Gamma \rangle , S({\cal W}))$, we have $\sharp L=1$, then ${\cal Y}$ is 
non-exceptional and strongly nice with respect to ${\cal W}$.  
%(hence ${\cal Y}$ 
%satisfies the assumptions in Theorem~\ref{t:rcdnkmain2}).   
\end{lem}

We now show that we can apply Theorem~\ref{t:rcdnkmain2}  to 
random relaxed Newton's methods. 
\begin{lem}
\label{l:RNMnnen} 
Let $g\in {\cal P}$ and let $({\cal Y}_{g}, {\cal W}_{g})$ be the random relaxed Newton's method scheme for $g.$ 
Then ${\cal Y}_{g}$ is a mild subset of $\emRat$ and ${\cal Y}_{g}$ 
is non-exceptional and strongly nice  with respect to 
%holomorphic family 
${\cal W}_{g}$.  
Also, for each $x\in Q_{g}$ and $\lambda \in \Lambda $, we have that 
$N_{g,\lambda }(x)=x$ and 
$N_{g,\lambda }'(x)=1-\frac{\lambda }{m_{x}}$, where $m_{x}$ denotes the 
order of $g$ at the zero $x$, and $|N_{g,\lambda }'(x)|<1.$  
Moreover, for each $\lambda \in \Lambda $, 
we have $N_{g,\lambda }(\infty )=\infty $, 
the multiplier of $N_{g,\lambda}$ at ${\infty }$ is equal to $(1-\frac{\lambda }{\deg (g)})^{-1}$,  
 and 
$\| D(N_{g,\lambda })_{\infty }\| _{s}=|1-\frac{\lambda }{\deg (g)}|^{-1}>1.$ 
Moreover, we have 
$S({\cal W}_{g})=Q_{g}\amalg \{ \infty \} \amalg \{ z_{0}\in \CC \mid g'(z_{0})=0, g(z_{0})\neq 0\} .$  
Moreover, for each 
$\Gamma \in \emCpt({\cal Y}_{g})$, we have 
$\emMin(\langle \G \rangle , S({\cal W}_{g}))=\{ \{ x\} \mid x\in Q_{g}\} \cup \{ \{ \infty \} \} .$ 
\end{lem}
\begin{proof}
It is easy to see that $S({\cal W}_{g})=Q_{g}\amalg \{ \infty \} \amalg \{ z_{0}\in \CC \mid g'(z_{0})=0, g(z_{0})\neq 0\} $ 
and for each $\Gamma \in \Cpt({\cal Y}_{g})$, we have 
$\Min(\langle \G \rangle , S({\cal W}_{g}))=\{ \{ x\} \mid x\in Q_{g}\} \cup \{ \{ \infty \} \} .$

It is easy to see that $N_{g,\lambda }(x)=x$ and 
$N_{g,\lambda }'(x)=1-\frac{\lambda }{m_{x}}$ for each 
$x\in Q_{g}$ and $\lambda \in \Lambda .$ 
Since $\{ 1-\frac{\lambda }{m_{x}}\mid \lambda \in \Lambda \} =
\{ z\in \CC \mid |z-(1-\frac{1}{m_{x}})|<\frac{1}{m_{x}}\} $, we have  
$|N_{g,\lambda }'(x)|<1$ for all $x\in Q_{g}, \lambda \in \Lambda .$  
Similarly, it is easy to see that 
for each $\lambda \in \Lambda $, 
we have $N_{g,\lambda }(\infty )=\infty $, 
the multiplier of $N_{g,\lambda}$ at ${\infty }$ is equal to 
$(1-\frac{\lambda }{\deg (g)})^{-1}$,  
 and 
$\| D(N_{g,\lambda })_{\infty }\| _{s}=|1-\frac{\lambda }{\deg (g)}|^{-1}>1.$ 
From the above arguments, we obtain that 
${\cal Y}_{g}$ is a mild subset of 
$\Rat$ and ${\cal Y}_{g}$ is  non-exceptional and nice with respect to ${\cal W}_{g}$.  
%and $({\cal Y}, {\cal W})$ is non-exceptional.   
By Lemma~\ref{l:singfam}, it follows that ${\cal Y}_{g}$ is strongly nice with respect to ${\cal W}_{g}.$ 
\end{proof} 
We now prove the following theorem on random relaxed Newton's methods. 

\begin{thm}
\label{t:RNM1}
Let $g\in {\cal P}.$ Let 
$({\cal Y}_{g}, {\cal W}_{g})$ be the random relaxed Newton's method scheme for $g.$ 
Then we have the following. 
\begin{itemize}
\item[{\em (i)}] 
There exists  the largest open and dense subset 
${\cal A}$ of $({\frak M}_{1,c}({\cal Y}_{g}, 
{\cal W}_{g}), {\cal O})$ such that 
for each $\tau \in {\cal A}$, all statements {\em (i)--(xi)} in Theorem~\ref{t:rcdnkmain2} hold.   
\item[{\em (ii)}] 
Let $\tau \in {\cal A}.$ 
Let $\Omega _{\tau }$ be the set defined in Theorem~\ref{t:rcdnkmain2}. 
Then $\sharp (\CCI \setminus \Omega _{\tau })\leq \aleph _{0}$ and 
$$\Omega _{\tau }=\{ y\in \CC 
\mid \tilde{\tau }(\{ \gamma =(\gamma _{1},\gamma _{2},\ldots )\in X_{\tau }\mid \exists n\in \NN 
\mbox{ s.t. }\gamma _{n,1}(y)=\infty \})=0\} .$$  
Moreover, there exists a constant $\rho _{\tau }\in (0,1)$ such that 
for each $z\in F_{pt}^{0}(\tau )=\Omega _{\tau }$, there exists a 
Borel subset $C_{\tau, z}$ of $X_{\tau }$ with $\tilde{\tau }(C_{\tau ,z})=1$ satisfying 
that for each $\gamma =(\gamma _{1},\gamma _{2},\ldots )\in C_{\tau ,z}$, there exists a constant $\zeta 
=\zeta (\tau, z, \gamma )>0 $ such that 
$$d(\gamma _{n,1}(z), Q_{g}\cup _{L\in \emMin(G_{\tau }, \CCI ), L\mbox{ is attracting for }\tau, 
L\cap Q_{g}=\emptyset } L)\leq \zeta \rho _{\tau }^{n}\mbox{ for all }n\in \NN .$$  
\item[{\em (iii)}] 
For each $\tau \in {\cal A}$, we have $\infty \in J_{pt}^{0}(\tau )=\CCI \setminus 
 \Omega _{\tau }$ and 
 \vspace{-1mm} 
 $$J_{\ker }(G_{\tau })=\{ z_{0}\in \CC \mid 
 g'(z_{0})=0, g(z_{0})\neq 0\} \cup \{ \infty \} .$$ 
%\vspace{-2mm} 
%, where $\Omega _{\tau }$ is the set defined in 
%Theorem~\ref{t:rcdnkmain2}. 
In particular, $\emptyset \neq J_{pt}^{0}(\tau )$ and $J_{\ker }(G_{\tau })\neq \emptyset $  
for each $\tau \in {\cal A}.$  
Moreover, if we set ${\cal A}^{f}:=$ $\{ \tau \in {\cal A}\mid \sharp \mbox{supp}\,\tau<\infty \}$, then 
${\cal A}^{f}$ is dense in 
$({\frak M}_{1,c}({\cal Y}_{g},{\cal W}_{g}), {\cal O})$.
Also, if, in addition to the assumptions of our theorem, 
$g/g'$ is not a polynomial of degree one, then 
for each $\tau \in {\cal A}^{f}$, 
we have $\overline{J_{pt}^{0}(\tau )}=J(G_{\tau })$ which is perfect.  
\item[{\em (iv)}] 
Let ${\cal A}_{conv}:=\{ \tau \in {\cal A}\mid 
 \emMin(G_{\tau }, \CCI )=\{ \{ x\}\mid x\in Q_{g}\} \cup \ \{ \infty \} \} \} .$ 
%\mbox{ which is attracting for }\tau \} .$ 
Then 
${\cal A}_{conv}$ is open in ${\cal A}.$ 
\item[{\em (v)}] 
Let $\tau \in {\cal A}_{conv}.$ Then 
%Also, for each $\tau \in {\cal A}_{conv}$, 
we have 
$\sharp (\CCI \setminus \Omega _{\tau })\leq \aleph _{0}$ and 
$\max_{x\in Q_{g}}e^{\chi (\tau, \{ x\})}<1.$ 
Moreover, 
%for each $\tau \in {\cal A}_{conv}$, 
for each $\alpha \in (\max_{x\in Q_{g}}e^{\chi (\tau, \{ x\})}, 1)$ 
%there exists a constant $\rho _{\tau }\in (0,1)$ such that 
and 
for each $z\in F_{pt}^{0}(\tau )=\Omega _{\tau }$, there exists a 
Borel subset $C_{\tau, z, \alpha }$ of $X_{\tau }$ with $\tilde{\tau }(C_{\tau ,z,\alpha })=1$ satisfying 
that for each $\gamma =(\gamma _{1},\gamma _{2},\ldots, )\in C_{\tau ,z,\alpha }$, 
there exist an element $x=x(\tau, z, \alpha, \gamma )\in Q_{g}$ and a constant $\xi  
=\xi (\tau, z, \alpha, \gamma ) >0$ such that 
\begin{equation}
\label{eq:RNMconv}
d(\gamma _{n,1}(z), x)\leq \xi \alpha ^{n}\mbox{\ \ for all }n\in \NN .
 \end{equation} 
% Moreover, for each $\tau \in {\cal A}_{conv}$,  
Also, for $\tilde{\tau }$-a.e. 
$\gamma =(\gamma _{1},\gamma _{2},\ldots, )\in X_{\tau }$, 
we have $\mbox{{\em Leb}}_{2}(J_{\gamma })=0$ and for each $z\in F_{\gamma }$, 
there exists an element $x=x(\tau, \gamma, z)\in Q_{g}$ such that 
\begin{equation}
\label{eq:dgn1zxr0}
d(\gamma _{n,1}(z), x)\rightarrow 0 \mbox{ as } n\rightarrow \infty .
\end{equation} 
Moreover, for each $x\in Q_{g}$ and 
 for each $z\in \Omega _{\tau }$, we have 
\begin{equation}
\label{eq:limwztxt}
\lim _{w\in \CCI , w\rightarrow z}T_{x,\tau }(w)=T_{x,\tau }(z).
\end{equation}  
Furthermore, 
%for each $\tau \in {\cal A}_{conv}$, 
we have 
\begin{equation}
\label{eq:tgxtme}
\tilde{\tau }(\{ \gamma \in X_{\tau }\mid \exists n\in \NN \mbox{ s.t. }\gamma _{n,1}(z)=\infty \})+\sum _{x\in Q_{g}}T_{x,\tau }(z)=1 \mbox{ for all }z\in \CCI , 
\end{equation}
and 
 we have 
\begin{equation}
\label{eq:xqgtxtz}
\sum _{x\in Q_{g}}T_{x,\tau }(z)>0  \mbox{\ \  for all } z\in \CCI \setminus J_{\ker }(G_{\tau })
=\CC \setminus \{ z_{0}\in \CC \mid g'(z_{0})=0, g(z_{0})\neq 0\} . 
%\mbox{ and } \left(\sum _{x\in Q_{g}}T_{x,\tau }(z)\right)+
%T_{\infty ,\tau }(z)=1. 
\end{equation}
In particular, for any subset $B$ of $\CC $ with $\sharp B\geq \deg (g)$, there exists 
an element $z\in B$ such that $\sum _{x\in Q_{g}}T_{x,\tau }(z)>0.$ 

\item[{\em (vi)}] 
Let $\tau \in {\frak M}_{1,c}({\cal Y}_{g}, 
{\cal W}_{g})$ and suppose that  
{\em int}$(\mbox{supp}\,\tau)\supset \{ N_{g,\lambda }\in 
{\cal Y}_{g}\mid \lambda \in \CC,  |\lambda -1|\leq \frac{1}{2}\} $ with respect to 
the topology in ${\cal Y}_{g}$. Then 
$\tau \in {\cal A}_{conv}.$ In particular,  
%for such $\tau $, 
the statements regarding {\em  (\ref{eq:RNMconv})},  {\em (\ref{eq:dgn1zxr0})}, 
{\em  (\ref{eq:limwztxt})}, {\em (\ref{eq:tgxtme})} and {\em (\ref{eq:xqgtxtz}) } hold for $\tau .$   
\item[{\em (vii)}] 
Let $\tau \in {\frak M}_{1,c}({\cal Y}_{g}, 
{\cal W}_{g})$ and suppose that  
 {\em int}$(\mbox{supp}\,\tau)\supset \{ N_{g,\lambda }
 \in {\cal Y}_{g}\mid \lambda \in \CC,  |\lambda -1|\leq \frac{1}{2}\} $ with respect to 
the topology in ${\cal Y}_{g}$ and 
$\tau $ is absolutely continuous with respect to the $2$-dimensional Lebesgue measure 
on ${\cal Y}_{g}\cong \Lambda $  
(e.g.,  let $\tau $ be the normalized $2$-dimensional Lebesgue measure 
on the set $\{ N_{g,\lambda }\in 
{\cal Y}_{g}\mid \lambda \in \CC, |\lambda -1|\leq r\} $ where $\frac{1}{2}<r<1$, 
under the identification ${\cal Y}_{g}\cong \Lambda $).  
 Then $\tau \in {\cal A}_{conv}$ and 
  the statements regarding {\em  (\ref{eq:RNMconv})},  {\em (\ref{eq:dgn1zxr0})}, 
{\em  (\ref{eq:limwztxt})}, {\em (\ref{eq:tgxtme})} and {\em (\ref{eq:xqgtxtz}) } hold for $\tau .$ 
Moreover, we have 
\begin{equation}
\label{eq:Omegatauf} 
\Omega _{\tau }=\CC \setminus \{ z_{0}\in \CC \mid 
g'(z_{0})=0, g(z_{0})\neq 0\} .
\end{equation}
 In particular, we have $\sharp (\CCI \setminus \Omega _{\tau })\leq  \deg (g)-1$, 
and for any subset $B$ of $\CC $ with $\sharp B\geq \deg (g)$, there exists 
an element $z\in B$ such that 
\begin{equation}
\label{eq:sumxqgt1}
\sum _{x\in Q_{g}}T_{x,\tau }(z)=1.
\end{equation}
Furthermore, for each $\varphi \in C(\CCI )$ and for each $z\in \Omega _{\tau }$, 
we have 
\begin{equation}
\label{eq:mtnxquc}
M_{\tau }^{n}(\varphi )(z)\rightarrow \sum _{x\in Q_{g}}T_{x,\tau }(z)\varphi (x) \mbox{ as }
n\rightarrow \infty 
\end{equation}
and this convergence is uniform on any compact subset of $\Omega _{\tau }.$ 
\end{itemize}
\end{thm}
\begin{proof}
When $g/g'$ is a polynomial of degree one (i.e., $g(z)$ is of the form $a(z-c)^{m}$), then it is easy to see that statements 
(i)-(vii) hold. Thus we may assume that $g/g'$ is not a polynomial. 
By Lemma~\ref{l:RNMnnen}, Lemma~\ref{l:wmsopen} and its proof, Theorem~\ref{t:rcdnkmain2}, the proof of Lemma~\ref{l:chimune} and Corollary~\ref{c:rcdnkmain2f}, 
statements (i)--(v) hold. 

We now prove (vi). 
Let $\Theta :=\{ L\in \Min(\langle N_{g,1}\rangle ,\CCI )\mid L\subset \CC \setminus Q_{g}, 
L \mbox{ is attracting for } \delta _{N_{g,1}}\} .$ 
Then each $L$ is an attracting periodic cycle of 
$N_{g,1}.$ 
Let $L\in \Theta .$ If the period $p_{L}$ of 
$(N_{g,1},L)$ is equal to $1$, then 
there exists an element $x\in Q_{g}$ with $L=\{ x\}.$ However, this is a contradiction. 
Hence, we have $p_{L}\geq 2.$ 
In particular, two different points of $L$ never belong to the same connected 
component of $F(N_{g,1}). $ 

We now let $\tau \in {\frak M}_{1,c}({\cal Y}_{g}, {\cal W}_{g})$ and suppose 
int$(\mbox{supp}\,\tau)\supset \G_{0}$ with respect to the topology in ${\cal Y}_{g}$, where 
$\G_{0}:=\{ N_{g,\lambda }\in {\cal Y}_{g}\mid \lambda \in \CC, |\lambda -1|\leq \frac{1}{2}\} $. 
%with respect to the 
%topology in ${\cal Y}.$ 
Let $\Gamma \in \Cpt({\cal Y}_{g})$ be an element such that 
int$(\G ) \supset \G _{0}$, 
%\{ f_{\lambda }\mid |\lambda -1|\leq \frac{1}{2}\} $, 
int$(\mbox{supp}\,\tau)\supset \G $, and $\G $ is close enough to 
%$\{ f_{\lambda }\mid |\lambda -1|\leq \frac{1}{2}\} $ 
$\G _{0}$ 
with respect to the Hausdorff metric. 
 We now use the arguments in the proof of Theorem~\ref{t:rcdnkmain2} and we modify them a little. 
  By Lemma~\ref{l:RNMnnen}, we have the following claim. \\ 
 \underline{Claim 1.} 
Let $h\in \G $ and $x\in Q_{g}$. Then we have 
$h(x)=x$ and 
 $\| Dh_{x}\| _{s}<1$. Also,  $h(\infty )=\infty $ and $\| Dh_{\infty }\| _{s}>1.$ 
 
 We now prove the following claim.\\ 
 \underline{Claim 2.} 
 Let $L\in \Min(\langle \G_{0}\rangle ,\CCI ) $ and suppose $L\subset \CC \setminus Q_{g}.$
 Then $L$ is not attracting for $\G_{0}.$ 
 
 To prove this claim, suppose that 
 there exists an element $L\in \Min(\langle \G_{0}\rangle ,\CCI )$ with 
 $L\subset \CC \setminus Q_{g}$ which is attracting for $\G_{0}.$ 
 Then there exists an element $L_{0}\in \Theta $ with $L_{0}\subset L.$ 
 We have that the period of 
 $(N_{g,1}, L_{0})$ is not $1.$ 
 Let $B:=\max \{ |N_{g,1}(x)-x|\mid x\in L_{0}\} >0.$ 
 Let $x_{0}\in L_{0}$ be an element such that 
 $|N_{g,1}(x_{0})-x_{0}|=B.$ Then we have 
 \begin{equation}
 \label{eq:flx0ml} 
 \{ N_{g,\lambda }(x_{0})\mid \lambda \in \CC, |\lambda -1|\leq \frac{1}{2}\} 
 =\{ x_{0}-\lambda \frac{g(x_{0})}{g'(x_{0})}\mid \lambda \in \CC, | \lambda -1|\leq \frac{1}{2}\} 
 =\{ z\in \CC \mid |z-N_{g,1}(x_{0})|\leq \frac{1}{2}B\} . 
% =\{ z\in \CC \mid |z-f_{1}(x_{0})|\leq \frac{1}{2}B\} . 
 \end{equation}
% Similarly, we have 
% \begin{equation}
% \label{eq:flf1x0}
% \{ f_{\lambda }(f_{1}(x_{0}))\mid \lambda \in \CC, |\lambda -1|\leq \frac{1}{2}\} 
% =\{ z\in \CC \mid |z-f_{1}^{2}(x_{0})|\leq \frac{1}{2}|f_{1}(x_{0})-f_{1}^{2}(x_{0})|\} .
% \end{equation}
Let $N_{g,0}=Id.$ 
By (\ref{eq:flx0ml}) and the fact $|N_{g,1}(x_{0})-
N_{g,1}^{2}(x_{0})|\leq B$, we obtain that 
\begin{eqnarray*}
\ & \ & \{  N_{g,\lambda }(N_{g,1}(x_{0}))\mid \lambda \in [0,1]\} \\ 
  & = &\{ N_{g,\lambda }(N_{g,1}(x_{0}))\mid \lambda \in [0,\frac{1}{2}]\} 
         \cup \{ N_{g,\lambda }(N_{g,1}(x_{0}))\mid \lambda \in [\frac{1}{2}, 1]\} \\ 
  & \subset & \{ z\in \CC \mid |z-N_{g,1}(x_{0})|\leq \frac{1}{2}|N_{g,1}(x_{0})-N_{g,1}^{2}(x_{0})|\} 
                 \cup \{ N_{g,\lambda }(N_{g,1}(x_{0}))\mid \lambda \in \CC, |\lambda -1|\leq \frac{1}{2}\} \\ 
  & \subset &\{ N_{g,\lambda }(x_{0})\mid | \lambda \in \CC , | \lambda -1|\leq \frac{1}{2}\} 
                \cup \{ N_{g,\lambda }(N_{g,1}(x_{0}))\mid \lambda \in \CC, |\lambda -1|\leq \frac{1}{2}\} 
                \subset L.                                         
\end{eqnarray*}
 Moreover, since two different points $N_{g,1}(x_{0})$ and $N_{g,1}^{2}(x_{0})$ in $L_{0}$ cannot 
 belong to the same connected component of 
 $F(N_{g,1})$, we have that 
 $\{ N_{g,\lambda }(N_{g,1}(x_{0}))\mid \lambda \in [0,1]\} \cap J(N_{g,1})\neq \emptyset .$ 
 From these arguments, it follows that $L\cap J(\langle \G_{0}\rangle )\neq \emptyset .$ 
 However, this contradicts the assumption that $L$ is attracting for $\G_{0}.$ 
 Thus we have proved Claim 2. 
    
 We now prove the following claim. \\ 
 \underline{Claim 3.} 
We have $\Min(\langle \G \rangle, \CCI )=\{ \{ x\} \mid x\in Q_{g}\} \cup \{ \{ \infty \}\} .$ 
 
 This claim is proved by combining Claims 1, 2,  Lemma~\ref{l:nalnotsb} and Zorn's lemma. 
 
 By using Claim 3, Lemma~\ref{l:RNMnnen} and the arguments in the part from Claims 6, 7 and the last 
 in the proof of Theorem~\ref{t:rcdnkmain1}, we obtain that 
 $\tau $ is weakly mean stable,  $\tau $ satisfies the assumptions of 
 Theorem~\ref{t:wmsneglyap}, and there is no $L\in \Min(G_{\tau },\CCI )$ with 
 $L\subset \CC \setminus Q_{g}.$  By Theorem~\ref{t:wmsneglyap}, it follows that 
 $\tau \in {\cal A}_{conv}.$ Thus we have proved statement (vi) in our theorem.  

Statement (vii) follows from statements (i), (ii), (iv), (v), (vi) and Theorem~\ref{t:zfggzcc}. (Note that if $\tau $ is absolutely continuous with respect to the $2$-dimensional Lebesgue 
measure on ${\cal Y}_{g}\cong \Lambda $, then 
the formula for $\Omega _{\tau }$ in (ii) gives us 
(\ref{eq:Omegatauf}).)   

Thus we have proved our theorem. 
\end{proof}

\begin{rem}
\label{r:normalizedpolynomial} 
Let $g$ be a non-constant polynomial. We say that $g$ is {\bf normalized} if the set 
$\{ z_{0}\in \CC \mid g(z_{0})=0\}$ is contained in   
$\Bbb{D}:=\{ z\in \CC \mid |z|<1\}.$ 
Note that if $g\in {\cal P}$ is normalized, then $g'$ is also a normalized polynomial
 (see \cite[page 29]{Ah}). Thus, 
 for a normalized polynomial $g\in {\cal P}$, for a 
 random relaxed Newton's method scheme  
 $({\cal Y}_{g}, {\cal W}_{g})$ for $g$, 
 if $\tau \in {\cal M}_{1,c}({\cal Y}_{g}, {\cal W}_{g})$ is an element such that 
 $\mbox{int}(\mbox{supp}\,\tau)\supset 
 \{ N_{g,\lambda }\in {\cal Y}_{g}\mid \lambda \in \CC \mid |\lambda -1|\leq \frac{1}{2}\} $ with respect to 
 the topology in ${\cal Y}_{g}$ 
and $\tau $ is absolutely continuous with respect to the $2$-dimensional 
Lebesgue measure on ${\cal Y}_{g}\cong \{ \lambda \in \CC \mid |\lambda -1|<1\} $, 
then 
 for any $z_{0}\in \CC \setminus \Bbb{D}$, 
 for $\tilde{\tau }$-a.e. $\gamma =
 (\gamma _{1},\gamma _{2},\ldots )\in 
 (\Rat)^{\NN }$, 
 $\{ \gamma _{n,1}(z_{0})\} _{n=1}^{\infty }$ converges to a root of $g$ as $n\rightarrow \infty .$  
\end{rem}
\section{Examples}
\label{Examples}
In this section, we give some examples to which we can apply our main theorems. 

\begin{ex}
\label{ex:wninfty}
Let ${\cal Y}$ be a weakly nice subset of ${\cal P}$ with respect to some holomorphic families 
$\{ {\cal W}_{j}\} _{j=1}^{m}$ of polynomial maps. 
%Suppose ${\cal Y}\subset {\cal P}.$ 
Suppose that $S({\cal W}_{j})=\{ \infty \}$ for each $j=1,\ldots, m$.  Then 
${\cal Y}$ is nice with respect to $\{ {\cal W}_{j}\} _{j=1}^{m}$ 
 and $({\cal Y}, \{ {\cal W}_{j}\} _{j=1}^{m})$ satisfies the assumptions of 
Lemma~\ref{l:ypwntln}. Thus by Lemma~\ref{l:ypwntln}, 
the set ${\cal A}:=\{ \tau \in {\frak M}_{1,c}({\cal Y}, \{ {\cal W}_{j}\} _{j=1}^{m}) 
\mid \tau \mbox{ is mean stable}\}$ is open and dense in 
${\frak M}_{1,c}({\cal Y}, \{ {\cal W}_{j}\} _{j=1}^{m}) $ with respect to the topology  ${\cal O}.$ 
In particular, all statements (i)--(xi) of Theorem~\ref{t:rcdnkmain2} hold for any $\tau \in {\cal A}$ 
and the set $\Omega_{\tau }$ in Theorem~\ref{t:rcdnkmain2} is equal to $\CCI $. 
%for 
%each $\tau \in {\cal A}.$  

\end{ex}
We give some examples of ${\cal Y}$ which are mild, non-exceptional and strongly nice and satisfies 
the assumptions in Theorem~\ref{t:rcdnkmain2}. 
\begin{ex}

%\begin{itemize}
%\item 
For each $q\in \NN $ with $q\geq 2$, let ${\cal P}_{q}:= \{ f\in {\cal P}\mid \deg (f)=q\} .$ 
Let $(q_{1},\ldots, q_{m})\in \NN ^{m}$ with $q_{1}<q_{2}<\cdots <q_{m}$ and 
let ${\cal W}_{j}=\{ f\} _{f\in {\cal P}_{q_{j}}}, j=1,\ldots, m$ and let 
${\cal Y}=\cup _{j=1}^{m}{\cal P}_{q_{j}}.$ In this case, $S({\cal W}_{j})=\{ \infty \} $. 
Thus  by Example~\ref{ex:wninfty}, 
the set ${\cal A}:=\{ \tau \in {\frak M}_{1,c}({\cal Y}, \{ {\cal W}_{j}\} _{j=1}^{m}) 
\mid \tau \mbox{ is mean stable}\}$ is open and dense in 
${\frak M}_{1,c}({\cal Y}, \{ {\cal W}_{j}\} _{j=1}^{m}) $ with respect to the topology ${\cal O}$  
and 
the set $\Omega _{\tau }$ in 
Theorem~\ref{t:rcdnkmain2} is equal to $\CCI $. 
%(see Example~\ref{ex:wninfty}).   
\end{ex}
\begin{ex}
Let $q\in \NN $ with $q\geq 2$ and let ${\cal W}=\{ z^{q}+c\} _{c\in \CC } .$ 
Let ${\cal Y} =\{ z^{q}+c\mid c\in \CC \} .$ 
In this case, $S({\cal W})=\{ \infty \} $. 
%we can take $B _{\tau }$ in 
%Theorem~\ref{t:rcdnkmain2} as $\CCI $ (see Example~\ref{ex:wninfty}).  
Thus  by Example~\ref{ex:wninfty}, 
the set ${\cal A}:=\{ \tau \in {\frak M}_{1,c}({\cal Y},  {\cal W}) 
\mid \tau \mbox{ is mean stable}\}$ is open and dense in 
${\frak M}_{1,c}({\cal Y}, {\cal W}) $ 
with respect to the topology  ${\cal O}$  
and 
the set $\Omega _{\tau }$ in 
Theorem~\ref{t:rcdnkmain2} is equal to $\CCI $. 
% (see Example~\ref{ex:wninfty}). 
\end{ex}
We now give an important example of ${\cal Y}$ to which we can apply Theorems~\ref{t:rcdnkmain1} and 
\ref{t:rcdnkmain2} but in which $\Omega _{\tau }\neq \CCI $ for any $\tau $ in 
an open subset of 
${\cal A}$, where ${\cal A}$ is the set in Theorems~\ref{t:rcdnkmain1} and 
\ref{t:rcdnkmain2}.   
\begin{ex}
\label{ex:wlz1z}
Let ${\cal W}=\{ \lambda z(1-z)\} _{\lambda \in \CC \setminus \{ 0\} } $ and 
let ${\cal Y}=\{ \lambda z(1-z)\in {\cal P}_{2}\mid \lambda \in \CC \setminus \{ 0\} \} .$ 
In this case, $S({\cal W})=\{ 0,1,\infty \}$ and 
$S({\cal W})\setminus \{ \infty \} =\{ 0,1\} \neq \emptyset .$ 
It is easy to see that ${\cal Y}$ is a mild subset of 
${\cal P}$ and ${\cal Y}$ is  non-exceptional and strongly nice 
with respect to holomorphic family ${\cal W} $. 
%and ${\cal Y}\subset {\cal P}.$ 
Thus the statements of Theorems~\ref{t:rcdnkmain1}, \ref{t:rcdnkmain2} hold.  
Let ${\cal A}$ be the largest  open and dense subset of 
$({\frak M}_{1,c}({\cal Y},{\cal W}), {\cal O})$ such that 
for each $\tau \in {\cal A}$, all statements (i)-(v) of Theorem~\ref{t:rcdnkmain1} and 
all statements (i)--(ix) of Theorem~\ref{t:rcdnkmain2} hold. 
Since each element of ${\cal Y}$ is a quadratic polynomial, 
for each $\tau\in {\cal A}$, exactly one of the followings holds. 
\begin{itemize}
\item 
Type (I).  $\Min(G_{\tau },\CCI )=\{ \{ 0 \} ,\{ \infty \}\}.$ 
\item 
Type (II). $\Min(G_{\tau },\CCI )=\{ \{ 0\}, \{ \infty \}, L_{\tau }\}$, 
where $L_{\tau }$ is an attracting minimal set with $L_{\tau }\neq \{0 \}, \{ \infty \} .$ 
\end{itemize}
If $\tau \in {\cal A}$ is of type (I), then Theorem~\ref{t:rcdnkmain1} (v) and Theorem~\ref{t:zfggzcc} imply that 
%if $\tau $ is of type (I), then 
\begin{equation}
\label{eq:mtnvyt0}
M_{\tau }^{n}(\varphi )(y)\rightarrow T_{0,\tau }(y)\varphi (0)
+T_{\infty ,\tau }(y) \varphi (\infty ) \mbox{ as }n\rightarrow \infty , \mbox{ for each }y\in \CCI ,\varphi \in C(\CCI)
\end{equation} 
i.e., $(M_{\tau }^{\ast })^{n}(\delta _{y})\rightarrow T_{0,\tau }(y)\delta _{0}+T_{\infty ,\tau }(y)\delta _{\infty }$ 
as $n\rightarrow \infty $. If $\tau \in {\cal A}$ is of type (II), then 
 Theorem~\ref{t:rcdnkmain1} (v) and Theorem~\ref{t:zfggzcc} imply that 
\begin{equation}
\label{eq:mtnltau}
M_{\tau }^{nr_{\tau }}(\varphi )(y)\rightarrow 
T_{0,\tau }(y)\varphi (0)+T_{\infty ,\tau }(y)\varphi (\infty )+
\sum _{j=1}^{r_{\tau }}\alpha ((L_{\tau })_{j},y)\int \varphi \ d\omega _{L_{\tau },j} 
\mbox{ as }n\rightarrow \infty 
%\mbox{ for each }y\in \CCI , \varphi \in C(\CCI ), 
\end{equation}
for each $y\in \CCI $ and for each $\varphi \in C(\CCI )$, where 
$r_{\tau }=\dim _{\CC }(U_{\tau ,L_{\tau }})$ (the period of $(\tau, L_{\tau })$, 
see Lemma~\ref{l:pwjkf} and Definition~\ref{d:period}), and 
$\{ (L_{\tau })_{j}\} _{j=1}^{r_{\tau }}, $ 
$\{ \omega _{L_{\tau },j}\} _{j=1}^{r_{\tau }}$ are elements coming from Theorem~\ref{t:zfggzcc}. 

Note that there exists an element $\tau \in {\cal A}$ of type (I). For example, 
let $g_{0}(z)=\lambda _{0}z(1-z)\in {\cal Y}$ where $0<|\lambda _{0}|<1$ and 
let $\tau _{0}= \delta _{g_{0}}.$ Then any element $\tau \in {\cal A}$ which is close enough to $\tau _{0}$ is of 
type (I). Also, there exists an element $\tau \in {\cal A}$ of type (II). For example, 
let $g_{1}\in {\cal Y}$ be an element which has an attracting periodic cycle with 
period $p\geq 2.$ Let $\tau _{1}=\delta _{g_{1}}.$ Then any element 
$\tau \in {\cal A}$ which is close enough to $\tau _{1}$ is of type (II) 
with $r_{\tau }=p.$   

We now classify elements $\tau \in {\cal A}$ of type (I) into the following three types. 
\begin{itemize}
\item 
Type (Ia). $0\in F(G_{\tau })$ and $\{ 0\}\in \Min(G_{\tau },\CCI )$ is attracting for $\tau .$  
\item 
Type (Ib). $0\in J_{\ker }(G_{\tau })$ and $\chi (\tau ,\{ 0\} )<0.$
\item 
Type (Ic). $0\in J_{\ker }(G_{\tau })$ and $\chi (\tau, \{ 0\} )>0.$
\end{itemize} 
We first remark that for each type $(\ast )$ above, there exists an element $\tau \in {\cal A}$ 
of type $(\ast ).$ In fact, for the above $\tau _{0}$, any element $\tau \in {\cal A}$ 
which is close enough to $\tau _{0}$ is of type (Ia). 
Also, let $g_{3}(z)=\frac{1}{2}z(1-z)\in {\cal Y}, g_{4}(z)=6z(1-z)\in {\cal Y}$ and 
let $\tau_{2}:=p_{1}\delta _{g_{3}}+p_{2}\delta _{g_{4}}$, where 
$(p_{1},p_{2})\in (0,1)^{2}$ with $p_{1}+p_{2}=1$, $p_{1}\log \frac{1}{2}+p_{2}\log 6<0.$ 
Then any element $\tau \in {\cal A}$ which is close enough to $\tau _{3}$ is of type 
(Ib). Moreover, let $\tau _{3}:=q_{1}\delta _{g_{3}}+q_{2}\delta _{g_{4}}$, where 
$(q_{1},q_{2})\in (0,1)^{2}$ with $q_{1}+q_{2}=1$, $q_{1}\log \frac{1}{2}+q_{2}\log 6>0.$ Then 
any element $\tau \in {\cal A}$ which is close enough to $\tau _{3}$ is of type (Ic). 
Hence for each type ($\ast $), there exists an element $\tau \in {\cal A}$ of type ($\ast $). 

For each type $(\ast )$=(Ia),  (Ib),  (Ic),  (II), we set ${\cal A}  _{\ast }$ the set of  element 
$\tau \in {\cal A}$ of type ($\ast $). We show the following claim. \\ 
\underline{Claim 1.} For each $(\ast )$=(Ia), (Ib), (Ic), (II), the set ${\cal A}_{\ast }$ is 
a non-empty open subset of ${\cal A}.$ Also, ${\cal A}=\amalg _{\ast }{\cal A}_{\ast }$, 
where $\amalg $ denotes the disjoint union.  

To show this claim, we first remark that we have already shown that 
each ${\cal A}_{\ast }$ is non-empty and ${\cal A}=\cup _{\ast }{\cal A}_{\ast }.$ 
By \cite[Lemma 5.2]{Sadv}, the sets ${\cal A}_{Ia}, {\cal A}_{II}$ are open in ${\cal A}.$ 
Also, since $\chi (\tau, \{ 0\} )$ is continuous with respect to $\tau \in {\cal A}$, 
we see that ${\cal A}_{Ic}$ is open in ${\cal A}.$ 
Finally, since each $\tau \in {\cal A}$ is weakly mean stable, 
for each $\tau \in {\cal A}_{Ib}$, there exists an element 
$g\in \mbox{supp}\,\tau$ with $|g'(0)|>1.$ From this, we obtain that 
${\cal A}_{Ib}$ is open in ${\cal A}.$ Thus we have proved Claim 1. 

We now show the following claim. \\ 
\underline{Claim 2.} 
For each $\tau \in {\cal A}_{II}$ and for each $g\in \mbox{supp}\,\tau$, we have 
$|g'(0)|>1.$ In particular, $0\in J_{\ker }(G_{\tau })$ and $\chi (\tau, \{ 0\})>0.$ 

To show this claim, let $\tau \in {\cal A}_{II}$ and $g\in \mbox{supp}\,\tau.$ 
Then $g$ has an attracting periodic cycle in $L_{\tau }$, which 
does not meet $0.$ Thus $|g'(0)|>1.$ Hence we have proved Claim 2.

For each $\tau \in {\cal A}$, we have that $\tau $ is weakly mean stable. 
We now show the following claim.\\
\underline{Claim 3.} 
Each element $\tau \in {\cal A}_{Ia}$ is mean stable. 
However, each element  $\tau \in {\cal A}_{Ib}\cup {\cal A}_{Ic}\cup 
{\cal A}_{II}$ is weakly mean stable but not mean stable. 

To show this claim, let $\tau \in {\cal A}_{Ia}.$ Since 
each minimal set is attracting, $\tau $ is mean stable. 
We now let $\tau \in {\cal A}_{Ib}\cup {\cal A}_{Ic}\cup 
{\cal A}_{II}$. Then $0\in J_{\ker }(G_{\tau }).$ Thus 
$\tau $ is not mean stable. Hence we have proved Claim 3. 

For each $\tau \in {\cal A}_{Ia}$, the convergence in (\ref{eq:mtnvyt0}) is 
uniform on $y\in \CCI $ since $\tau $ is mean stable.  
 
Let  ${\cal U}:=\{ \tau \in {\cal A}\mid 
\chi (\tau ,\{ 0\})>0\} ={\cal A}_{Ic}\cup {\cal A}_{II}$. Then 
${\cal U}$ is a non-empty   open subset  of ${\cal A}.$ 
%$({\frak M}_{1,c}({\cal Y}, {\cal W}), {\cal O})$. 
%Also, this is non-empty. For, 
% let $a_{1},\ldots, a_{n}\in \CC $ with 
%$|a_{1}|,\ldots, |a_{n}|>1$, let 
%$(p_{1},\ldots, p_{n})\in (0,1)^{n}$ with $\sum _{j=1}^{n}p_{j}=1$,  
%let 
%$\tau _{0}=\sum _{j=1}^{n}\delta _{g_{j}}$, where $g_{j}(z)=a_{j}z(1-z)$, 
%Then $\tau _{0}\in {\cal U}.$ Thus 
%${\cal U}$ is a non-empty open subset of 
%$({\frak M}_{1,c}({\cal Y}, {\cal W}), {\cal O})$. 
%For each $\tau \in {\cal U}$, $\tau $ is not mean stable 
%since $ 0 \in J_{\ker }(G_{\tau }) . $  
% 
%However, by Theorem~\ref{t:rcdnkmain2}, 
%the set of weakly mean stable elements is open and dense in 
%$({\frak M}_{1,c}({\cal Y}, {\cal W}), {\cal O}).$ In fact, 
%there exists an open dense subset ${\cal A}$ of 
%$({\frak M}_{1,c}({\cal Y}, {\cal W}), {\cal O}).$ such that for each 
%$\tau \in {\cal A}$, all statements (i)--(ix) of Theorem~\ref{t:rcdnkmain2} hold for $\tau .$ 
We now prove the following claim. \\ 
\underline{Claim 4.} 
For each $\tau \in {\cal U}$, the set $\Omega_{\tau }$ 
(with $\sharp (\CCI \setminus \Omega_{\tau })\leq \aleph _{0}$)
 in Theorem~\ref{t:rcdnkmain2} is not equal to  
$\CCI $. In particular, $\emptyset \neq  J_{pt}^{0}(\tau )=\CCI \setminus \Omega _{\tau }.$ 
% for each $\tau \in {\cal U}$.

To prove this claim,     
 let $\tau \in {\cal U}.$ Then 
$\chi (\tau ,\{ 0\})>0$. By the definition of 
$\Omega _{\tau }$, we obtain that $0\in \CCI \setminus \Omega _{\tau }.$ 
Also, by Theorem~\ref{t:rcdnkmain2}  (vii), we have $J_{pt}^{0}(\tau )=\CCI \setminus \Omega _{\tau }.$ 
Hence we have proved Claim 4. 

We now prove the following claim. Note that the set   
$\{ \tau \in {\cal U}\mid \sharp \mbox{supp}\,\tau<\infty \}$ is dense in ${\cal U}.$  
\\ 
\underline{Claim 5.} 
Let  $\tau \in {\cal U}$ with $\sharp \mbox{supp}\,\tau<\infty $. Then  
we have $\overline{J_{pt}^{0}(\tau )}=J(G_{\tau })$ and this is a perfect set. 
Also, there exists an element $L\in \Min(G_{\tau },\CCI )$ such that 
letting $\varphi _{L}\in C(\CCI )$ be any element such that 
$\varphi _{L}|_{L}=1$ and $\varphi _{L}|_{L'}=0$ for any $L'\in \Min(G_{\tau },\CCI )$ 
with $L'\neq L$, the convergence in (\ref{eq:mtnlvy}) in 
Theorem~\ref{t:zfggzcc} for $\varphi =\varphi _{L}$ is not uniform in any open subset $V$ of 
$\CCI $ with $V\cap J(G_{\tau })\neq \emptyset .$ 

This claim follows from Claim 4,  Corollary~\ref{c:rcdnkmain2f} and Theorem~\ref{t:rcdnkmain2}. We have proved Claim 5. 

%  Birkhoff's ergodic theorem implies that  
%for $\tilde{\tau }$-a.e.$\gamma $, we have 
%$\frac{1}{n}\log \| D(\gamma _{n,1})_{0}\| _{s}\rightarrow \chi (\tau, \{0\} )>0$ as $n\rightarrow \infty .$  
%More precisely, for each $\tau \in {\cal A}$, let $${\cal N}_{\tau }:=\{ z\in \CCI \mid 
%\exists C_{\tau, z}\subset X_{\tau } \mbox{ with }\tilde{\tau }(C_{\tau .z})=1 \mbox{ s.t. }
%\limsup _{n\rightarrow \infty }\frac{1}{n}\log \| D(\gamma _{n,1})_{z}\| _{s}<0\} .$$ 
%Then by Theorem~\ref{t:sjkfcln}, for each $\tau \in {\cal A}_{Ia}\cup {\cal A}_{Ib}$, we have 
%${\cal N}_{\tau }=\CCI $ and there exists a constant $c_{\tau }<0$ such that 
%for each $z\in \CCI $ there exists a Borel subset $C_{\tau ,z}$ of $X_{\tau }$ 
%with $\tilde{\tau }(C_{\tau ,z})=1$ such that 
%for each $\gamma \in C_{\tau ,z}$, we have 
%$\limsup _{n\rightarrow \infty }\frac{1}{n}\log \| D(\gamma _{n,1})_{z}\| _{s}\leq c_{\tau }<0$.  
%(for the case $\tau \in {\cal A}_{Ib}$, see the proof of    
We now prove the following claim.\\ 
\underline{Claim 6.} 
For each $\tau \in {\cal A}_{Ia}$, 
the functions $T_{0 ,\tau }, T_{\infty, \tau}$ are continuous on $\CCI $ and 
there exists a neighborhood $V$ of $0$ such that 
$T_{0,\tau }|_{V}\equiv 1$ and $T_{\infty ,\tau }|_{V}\equiv 0.$ 
Also, for each $\tau \in {\cal A}_{Ib}$, the functions 
$T_{0,\tau }$ and $T_{\infty ,\tau }$ are continuous on $\CCI $ 
and $T_{0.\tau }(0)=1$, $T_{\infty ,\tau }(0)=0$,  
but 
for any neighborhood $V$ of $0$, we have $T_{0,\tau }|_{V}\not\equiv 1$ and 
$T_{\infty ,\tau }|_{V}\not\equiv 0.$ 

To prove this claim, let $\tau \in {\cal A}_{Ia}\cup {\cal A}_{Ib}$. 
Then by Theorem~\ref{t:sjkfcln} 
(or Theorem~\ref{t:rcdnkmain2}), the functions $T_{0,\tau },T_{\infty }$ are 
continuous. If $\tau \in {\cal A}_{Ia}$, then $0\in F(G_{\tau }) $  and since 
the functions $T_{0,\tau }$ and $T_{\infty ,\tau }$ are locally constant on $F(G_{\tau })$ 
(see \cite[Theorem 3.15]{Splms10} or Theorems~\ref{t:rcdnkmain1} 
and \ref{t:zfggzcc} (vi)),  
there exists a neighborhood $V$ of $0$ such that 
$T_{0,\tau }|_{V}\equiv 1$ and $T_{\infty ,\tau }|_{V}\equiv 0.$  
We now suppose $\tau \in {\cal A}_{Ib}.$ 
Let $F_{\infty }(G_{\tau })$ be the connected component of $F(G_{\tau })$ 
with $\infty \in F_{\infty }(G_{\tau }).$ Then $T_{\infty ,\tau }|_{F_{\infty }(G_{\tau })}\equiv 1$.  
%(see \cite[]{Splms10}). 
Let $V$ be any neighborhood of $0.$ 
Since $0\in J(G_{\tau })$, 
%for any neighborhood $V$ of $0$, 
there exist an element $z\in V$ 
and an element $g\in G_{\tau }$ 
such that $g(z)\in F_{\infty }(G_{\tau }).$ 
Let $(\gamma _{1} \ldots, \gamma _{n})\in 
(\mbox{supp}\,\tau)^{n}$ be an element such that 
$g=\gamma _{n}\circ \cdots \circ \gamma _{1}.$ Then 
there exists a neighborhood $\Lambda $ of $(\gamma _{1},\ldots, \gamma _{n})$ 
in $(\mbox{supp}\,\tau)^{n}$ such that for each $(\alpha _{1},\ldots, \alpha _{n})\in \Lambda $, 
$\alpha _{n}\circ \cdots \circ \alpha _{1}(z)\in F_{\infty }(G_{\tau }).$ 
It implies that $T_{\infty ,\tau }(z)\geq (\otimes _{j=1}^{n}\tau )(\Lambda )>0.$ 
Therefore  $T_{\infty ,\tau }|_{V}\not\equiv 0.$ 
Since $T_{0,\tau }+T_{\infty ,\tau }=1$, it follows that 
$T_{0,\tau }|_{V}\not\equiv 1.$ Thus we have proved Claim 6. 

We now prove the following claim.\\ 
\underline{Claim 7.} 
Let $\tau \in {\cal A}_{Ic}.$ Then 
for each $z\in \Omega _{\tau }$, where $\Omega _{\tau }$ is the subset 
of $\CCI $  defined in 
Theorem~\ref{t:rcdnkmain2}, 
we have $T_{\infty ,\tau }(z)=1.$ Also, $\sharp (\CCI \setminus \Omega _{\tau })\leq \aleph _{0}.$ 
%there exists a Borel subset 
%$B_{z}$ of $X_{\tau }$ with $\tilde{\tau }(B_{z})=1$ such that 
%for each $\gamma \in B_{z}$, we have 
%$\gamma _{n,1}(z)\rightarrow \infty $ as $n\rightarrow \infty .$ 
%Moreover, for $\tilde{\tau }$-a.e.$\gamma \in X_{\tau }$, 
%we have Leb$_{2}(K_{\gamma })=0$, where 
%$K_{\gamma }$ denotes the filled-in Julia set of $\gamma$, i.e., 
%$K_{\gamma }:=\{ z\in \CC \mid \{ \gamma _{n,1}(z)\} _{n=1}^{\infty } \mbox{ is bounded in }\CC \} .$ 

To prove this claim, by Theorem~\ref{t:rcdnkmain2}, we have 
$\sharp (\CCI \setminus \Omega _{\tau })\leq \aleph _{0}.$ 
Also, by the definition of $\Omega _{\tau }$, the result 
$T_{0,\tau }+T_{\infty }=1$ on $\CCI $ and Lemma~\ref{l:chlpa0}, 
 we see that $T_{\infty ,\tau}(y)=1$ for each $y\in \Omega _{\tau }.$ 
 Thus we have proved Claim 7. 
%Since Leb$_{2}(\CCI \setminus \Omega _{\tau })-0$, the Fubini theorem implies that 
%for $\tilde{\tau }$-a.e.$\gamma \in X_{\tau }$, we have 
%Leb$_{2}(K_{\gamma })=0.$ Thus we have proved Claim 7. 

We now prove the following claim.\\ 
\underline{Claim 8.} 
Let $\tau \in {\cal A}.$ Then 
$\tau \in {\cal A}_{Ic}$ if and only if 
for $\tilde{\tau }$-a.e.$\gamma \in X_{\tau }$, we have Leb$_{2}(K_{\gamma })=0$, 
where $K_{\gamma }$ denotes the filled-in Julia set of $\gamma $, i.e., 
$K_{\gamma }:=\{ z\in \CC \mid \{ \gamma _{n,1}(z)\} _{n=1}^{\infty } \mbox{ is bounded in }\CC \} .$ 

To prove this claim, 
let $\tau \in {\cal A}_{Ic}.$ Then 
by Claim 7 and the Fubini theorem, 
%for any $\tau \in {\cal A}_{Ic}$, 
for $\tilde{\tau }$-a.e.$\gamma \in X_{\tau }$, we have Leb$_{2}(K_{\gamma })=0$. 
We now suppose that $\tau \in {\cal A}$ and 
for $\tilde{\tau }$-a.e.$\gamma \in X_{\tau }$, we have Leb$_{2}(K_{\gamma })=0$. 
Then by the Fubini theorem, we obtain that 
for Leb$_{2}$-a.e. $z\in \CCI $, we have $T_{\infty ,\tau }(z)=1.$ 
Therefore by Claim 6, $\tau \not\in {\cal A}_{Ia}\cup {\cal A}_{Ib} \cup {\cal A}_{II}. $ 
Hence by Claim 1, we obtain $\tau \in {\cal A}_{Ic}.$ Thus we have proved Claim 8.

\end{ex}
We also give some further examples to which we can apply Theorems~\ref{t:rcdnkmain1} 
and \ref{t:rcdnkmain2}. 
\begin{ex}
\label{ex:qx1any}
Let $Q=\{ x_{1},\ldots ,x_{n}\}$ be any non-empty finite subset of $\CC $, 
where 
$x_{1},\ldots, x_{n}$ are mutually distinct points. 
Let  
$f(z)=a\prod _{j=1}^{n}(z-x_{j})\in {\cal P}$, where $a\in \CC \setminus \{ 0\} $. Then  we have 
$\{ z_{0}\in \CC \mid f(z_{0})=0\} =Q$ and 
if $z_{0}\in \CC,  f(z_{0})=0,$ then $f'(z_{0})\neq 0.$ 
%
%Let $f\in {\cal P}$ such that if $z_{0}\in \CC, f(z_{0})=0$, then $f'(z_{0})\neq 0.$ 
Let ${\cal W}=\{ z+\lambda f(z)\} _{\lambda \in \CC \setminus \{ 0\} }$ and 
let ${\cal Y}=\{ z+\lambda f(z)\in {\cal P}\mid \lambda \in \CC \setminus \{ 0\} \} .$ 
In this case, $S({\cal W})=Q\cup \{ \infty \} $ and 
$S({\cal W})\cap \CC =\{ z_{0}\in \CC \mid f(z_{0})=0\} =Q\neq \emptyset .$  
By Lemma~\ref{l:singfam}, we obtain that 
${\cal Y}$ is a mild subset of ${\cal P}$, the set ${\cal Y}$ is 
 strongly nice and non-exceptional
 with respect to holomorphic family ${\cal W}$ and 
$({\cal Y}, {\cal W})$ satisfies the assumptions of Theorems~\ref{t:rcdnkmain1}, 
\ref{t:rcdnkmain2}. Thus there exists the largest open and dense subset 
${\cal A}$ of $({\frak M}_{1,c}({\cal Y},{\cal W}),{\cal O})$ such that 
for each $\tau \in {\cal A}$, all statements (i)--(xi) in Theorem~\ref{t:rcdnkmain2} 
hold for $\tau .$ In particular, each $\tau \in {\cal A}$ is weakly mean stable. 
Let $f_{\lambda }(z)=z+\lambda f(z).$ Then we have 
\begin{equation}
\label{eq:fldiff}
f_{\lambda }'(z)=1+\lambda f'(z). 
\end{equation}
%We have the following claim. \\ 
%\underline{Claim 1.} 
%For each $x\in Q$, we have $\{ x\} \in \Min(G_{\tau },\CCI )$ for 
%each $\tau \in {\frak M}_{1,c}({\cal Y},{\cal W}).$ 
%Also, let 
%${\cal A}_{MS}:=\{ \tau \in {\cal A}\mid 
%Q\subset F(G_{\tau })\}.$ Then ${\cal A}_{MS}$ is a non-empty open subset of 
%${\cal A}.$ Moreover, ${\cal A}_{MS}=\{ \tau \in {\cal A}\mid 
%\tau \mbox{ is mean stable}\}. $ 
%
%To show this claim , by the definition of $f_{\lambda }$, it is easy to see that 
%$\{ x\}\in \Min (G_{\tau },\CCI )$ for each $\tau \in {\frak M}_{1,c}({\cal Y},{\cal W}).$ 
%By (\ref{eq:fldiff}) and the fact $f'(x)\neq 0$ for each $x\in Q$, there exists an element 
% $\lambda _{0}\in \CC \setminus \{ 0\}$  such that 
%for each $x\in Q$, we have $|f_{\lambda _{0}}'(x)|<1.$ 
%Let $\tau _{0}=\delta _{f_{\lambda _{0}}}$ and let $\tau _{1}\in {\cal A}$ 
%be an element close enough to $\tau _{0}.$ Then 
%$\tau _{1}\in {\cal A}_{MS}.$ Hence ${\cal A}_{MS}\neq \emptyset .$  
%Therefore ${\cal A}_{MS}$ is a non-empty and open subset of ${\cal A}.$ 
%We now suppose  $\tau \in {\cal A}_{MS}.$ Then 
%since $\tau $ is weakly mean stable and since $Q\subset F(G_{\tau })$, 
%$\{ x\} $ is attracting for $\tau $ for each $x\in Q.$ Therefore 
%Each $L\in \Min(G_{\tau },\CCI )$ is attracting. It implies that 
%$\tau $ is mean stable. We now suppose $\tau \in {\cal A}$ is mean stable. 
%Since each $L\in \Min(G_{\tau },\CCI )$ is attracting, we have $Q\subset F(G_{\tau }).$ 
%Thus $\tau \in {\cal A}_{MS}.$ Hence we have proved Claim 1. 
Let 
${\cal A}_{+}:=\{ \tau \in {\cal A}\mid 
\exists L\in \Min(G_{\tau },Q) \mbox{ s.t. }\chi (\tau ,L)>0\} .$ 
Also, let ${\cal A}_{+,all}:=\{ \tau \in {\cal A}\mid 
\mbox{ for all }L\in \Min(G_{\tau },Q) \mbox{ we have }\chi (\tau ,L)>0\} .$ 
Moreover, let ${\cal A}^{f}:=\{ \tau \in {\cal A}\mid \sharp \mbox{supp}\,\tau<\infty \} $, 
${\cal A}_{+}^{f}:= {\cal A}_{+}\cap {\cal A}^{f}$, and 
${\cal A}_{+.all}^{f}:={\cal A}_{+,all}\cap {\cal A}^{f}.$ 

We now show the following claim. \\ 
\underline{Claim 1.} 
The sets ${\cal A}_{+}$ and ${\cal A}_{+,all}$ are  non-empty open subsets of 
${\cal A}$ (and thus they are  non-empty open subsets of 
$({\frak M}_{1,c}({\cal Y},  {\cal W}), {\cal O})$).  
Also, ${\cal A}_{+}^{f}$ is dense in ${\cal A}_{+}$ and ${\cal A}_{+,all}^{f}$ 
is dense in ${\cal A}_{+,all}.$ Moreover, 
for each $\tau \in {\cal A}_{+}$, we have 
$\emptyset \neq \cup _{L\in H_{+,\tau }}L\subset J_{pt}^{0}(\tau )=\CCI \setminus \Omega _{\tau }$, where 
$\Omega _{\tau }$ and $H_{+,\tau }$ are the sets defined in Theorem~\ref{t:rcdnkmain2}, and 
for each $\tau \in {\cal A}_{+,all}$, we have 
$Q\subset J_{pt}^{0}(\tau ).$ Furthermore, 
for each $\tau \in {\cal A}_{+}^{f}$, 
we have $\overline{J_{pt}^{0}(\tau )}=J(G_{\tau })$ which is a perfect set. 

%Also, there exists a non-empty  open subset ${\cal U}$ of ${\cal A}$ 
%(thus ${\cal U}$ is a non-empty open subset of ${\frak M}_{1,c}({\cal Y}, {\cal W})$) 
%such that  
%for each $\tau \in {\cal U}$ and  
%for each $x\in Q$, we have $\chi (\tau , \{ x\})>0.$ 
%In particular, for each $\tau \in {\cal U}$, we have $Q\subset J_{pt}^{0}(\tau )$. 
%Also, if we set ${\cal U}^{f}:=\{ \tau \in {\cal U}\mid \sharp \mbox{supp}\,\tau<\infty \} $ 
%then  
%${\cal U}^{f}$ is dense in ${\cal U}$ and for each $\tau \in {\cal U}^{f}$, 
%we have $\overline{J_{pt}^{0}(\tau )}=J(G_{tau })$ which is a perfect set. 

To prove this claim, 
it is easy to see that ${\cal A}_{+}$ and ${\cal A}_{+.all}$ are open in ${\cal A}.$ 
%For each $x\in Q$, 
%we have that $f_{\lambda }'(x)=1+\lambda f'(x).$ Since $f'(x)\neq 0$ for each $x\in Q$,  
By (\ref{eq:fldiff}) and the fact $f'(x)\neq 0$ for each $x\in Q$, 
if $|\lambda _{0}|$ is large enough, then letting $\tau _{0}
:=\delta _{f_{\lambda _{0}}}$, we have 
$\chi (\tau _{0}, \{ x\})>0$ for each $x\in Q.$  
Therefore 
%there exists a non-empty  open subset ${\cal U}$ of ${\cal A}$ 
%such that  
for each $\tau \in {\cal A}$ which is close enough to $\tau _{0}$ and   
for each $x\in Q$, we have $\chi (\tau , \{ x\})>0.$ 
Thus ${\cal A}_{+}\supset {\cal A}_{+,all}\neq \emptyset.$  
The rest statements follow from Theorem~\ref{t:rcdnkmain2} and Corollary~\ref{c:rcdnkmain2f}. 
Thus we have proved Claim 1. 

Let ${\cal A}_{-, all}:=
%\{ \tau \in {\cal A}\mid 
%${\cal A}_{+}:=
\{ \tau \in {\cal A}\mid 
\mbox{for all } L\in \Min(G_{\tau },Q) \mbox{ we have }\chi (\tau ,L)<0\} .$ 
We now prove the following claim.\\ 
\underline{Claim 2.} 
The set ${\cal A}_{-, all}$ is a non-empty open subset of ${\cal A}$ and 
${\cal A}_{-, all}\cap {\cal A}_{+, all}=\emptyset .$ 

To prove this claim, it is easy to see ${\cal A}_{-,all}\cap {\cal A}_{+,all}=\emptyset $ 
and ${\cal A}_{-, all}$ is open in ${\cal A}.$  
For each  $x\in Q$,  combining (\ref{eq:fldiff}), the fact $f'(x)\neq 0$ and the method above, 
we see that  
there exists an element $\lambda _{x}\in \CC \setminus \{ 0\} $ such that 
$f_{\lambda _{x}}'(x)=0.$ Let $\tau _{1}=\sum _{x\in Q}\frac{1}{n}\delta _{f_{\lambda _{x}}}.$  
Then $\chi (\tau _{1}, \{x\})=-\infty $ for each $x\in Q.$  
Hence for each $\tau \in {\cal A}$ which is close enough to 
$\tau _{1}$, we have $\chi (\tau, \{ x\} )<0$ for all $x\in Q.$ Thus 
${\cal A}_{-, all}\neq \emptyset .$ 
Hence we have proved Claim 2. 
%In particular, for each $\tau \in {\cal U}$, we have $Q\subset J_{pt}^{0}(\tau )$.  
%By Corollary~\ref{c:rcdnkmain2f}, we obtain that 
%${\cal U}^{f}$ is dense in ${\cal U}$ and for each $\tau \in {\cal U}^{f}$, 
%we have $\overline{J_{pt}^{0}(\tau )}=J(G_{\tau })$ which is a perfect set. 

We now prove the following claim.\\ 
\underline{Claim 3.} 
Let $\tau \in {\cal A}_{-,all}.$ Then for each $L\in \Min(G_{\tau },J_{\ker}(G_{\tau }))$, 
we have $\chi (\tau ,L)<0$, and each $L\in \Min(G_{\tau },\CCI)$ with $L\not\subset J_{\ker }(G_{\tau })$ 
is attracting $\tau .$ Thus $\tau $ satisfies all assumptions of Theorem~\ref{t:sjkfcln} and 
all conclusions in Theorem~\ref{t:sjkfcln} hold. In particular, 
$J_{pt}^{0}(\tau )=\emptyset $ and $F_{meas}(\tau )={\frak M}_{1}(\CCI).$   

This claim follows from Theorem~\ref{t:rcdnkmain2}, the fact $\tau $ is weakly mean stable and Theorem~\ref{t:sjkfcln}. 
%Thus we have proved Claim 3. 
\end{ex}
\begin{ex}
\label{ex:wizlzn} 
Let $n\in \NN $ with $n\geq 2$ and let $w=e^{2\pi i/n}\in \CC .$ 
For each $i=1,\ldots, n$, let ${\cal W}_{i}=
\{ w^{i}(z+\lambda (z^{n}-1))\} _{\lambda \in \CC \setminus \{ 0\} }.$ 
Let $i_{1},\ldots, i_{m}\in \{ 1,\ldots, n\} $ with $i_{1}<i_{2}\cdots <i_{m}.$ 
Let ${\cal Y}=\cup _{j=1}^{m}\{ w^{i_{j}}(z+\lambda _{j}(z^{n}-1))\in {\cal P}\mid 
\lambda _{j}\in \CC \setminus \{ 0\} \} .$ 
For each $j=1,\ldots, m$, let  $\Lambda _{j}:=
\CC \setminus \{ 0\} $ and let  $f_{j,\lambda _{j}}(z)=w^{i_{j}}(z+\lambda _{j}(z^{n}-1))$ 
for each $z\in \CCI , \lambda _{j}\in \Lambda _{j}.$  
Let ${\cal W}_{j}=\{ f_{j,\lambda _{j}}\} _{\lambda _{j}\in \Lambda _{j}}$ for each 
$j=1,\ldots, m.$ We show the following claim.\\ 
\underline{Claim 1.} 
${\cal Y}$ is a mild subset of ${\cal P}$ and ${\cal Y}$ is  
strongly nice and non-exceptional with respect to 
holomorphic families $\{ {\cal W}_{j}\} _{j=1}^{m}$ 
of polynomial maps and $({\cal Y}, \{ {\cal W}_{j}\} _{j=1}^{m})$ satisfies the 
assumptions of Theorem~\ref{t:rcdnkmain2}. 

To prove this claim,  
we first note that   
$S({\cal W}_{j})=\{ w^{k}\mid k=1,\ldots, n\} \cup \{ \infty \} $ for each $j=1,\ldots, m$. 
Hence we have 
$\cap _{j=1}^{m}S({\cal W}_{j})\cap \CC 
=\{ w^{k}\mid k=1,\ldots, n\} \neq \emptyset .$  
For each  
$w^{k}\in \cap _{j=1}^{m}S({\cal W}_{j})$ and for each 
$j=1,\ldots, m$ and for each $\lambda _{j}\in \CC \setminus \{ 0\}$, 
we have $f_{j,\lambda _{j}}(w^{k})=w^{k+i_{j}}\in \cap _{j=1}^{m}
S({\cal W}_{j}).$ Thus for each 
$\tau \in {\frak M}_{1,c}({\cal Y}, \{ {\cal W}_{j}\} _{j=1}^{m})$, 
we have $G_{\tau }(\cap _{j=1}^{m}S({\cal W}_{j})\cap \CC)
\subset \cap _{j=1}^{m}S({\cal W}_{j})\cap \CC.$ 
Let $Q:= \cap _{j=1}^{m}S({\cal W}_{j})\cap \CC =\{ w^{k}\mid 
k=1,\ldots, m\} .$ 
For each $j=1,\ldots, m$, let 
$\alpha _{j}:Q\rightarrow Q$ be the map defined by 
$\alpha _{j}(z)=w^{i_{j}}\cdot z, z\in Q.$ 
Then for each $j=1,\ldots, m$ and for each 
$\lambda_{j}\in \Lambda _{j}$, 
we have $f_{j,\lambda _{j}}|_{Q}=\alpha _{j}.$ 
Since the semigroup 
$\{ \alpha _{j}^{n}\mid n\in \NN \}$ is a cyclic group generated by 
$\alpha _{j}$, there exists an element $n_{j}\in \NN $ such that 
$\alpha _{j}^{-1}=\alpha _{j}^{n_{j}}.$ 
Therefore we obtain that 
$Q$ is equal to the union of minimal sets of the semigroup generated 
by $\{ \alpha _{1},\ldots, \alpha _{m}\} .$ Thus 
$Q=\cup _{L\in \Min(G_{\tau },Q)}L$ for each 
$\tau \in {\frak M}_{1,c}({\cal Y}, \{ {\cal W}_{j}\} _{j=1}^{m}).$ 
 Hence there is no peripheral cycle for $({\cal Y}, \{ {\cal W}_{j}\})$. 
Moreover, for each $z\in Q$, for each $j=1,\ldots, m$ and for each 
$\lambda _{j}\in \Lambda _{j}$, we have 
\begin{equation}
\label{eq:fjljz}
f_{j,\lambda _{j}}'(z)=w^{i_{j}}(1+\lambda _{j}nz^{n-1}).
\end{equation}   
Hence ${\cal Y}$ is strongly nice with respect to holomorphic families 
$\{ {\cal W}_{j}\} _{j=1}^{m}$ of polynomial maps. Moreover, 
by (\ref{eq:fjljz}), 
%the formula $f_{j,\lambda _{j}}'(z)=w^{i_{j}}(1+\lambda _{j}nz^{n-1}).$  above, 
it is easy to see that ${\cal Y}$ is non-exceptional 
with respect to $\{ {\cal W}_{j}\} _{j=1}^{m}.$ 
Thus we have proved Claim 1. 

Let ${\cal A}$ be the open and dense subset of 
$({\frak M}_{1,c}({\cal Y}, \{ {\cal W}_{j}\} _{j=1}^{m}), {\cal O})$ given in 
Theorem~\ref{t:rcdnkmain2}. 
Then for each $\tau \in {\cal A}$, all statements (i)--(xi) in Theorem~\ref{t:rcdnkmain2} hold. 
In particular, any $\tau \in {\cal A}$ is weakly mean stable. 
Let 
${\cal A}_{+}:=\{ \tau \in {\cal A}\mid 
\exists L\in \Min(G_{\tau },Q) \mbox{ s.t. }\chi (\tau ,L)>0\} .$ 
Also, let $${\cal A}_{+,all}:=\{ \tau \in {\cal A}\mid 
\mbox{ for all }L\in \Min(G_{\tau },Q) \mbox{ we have }\chi (\tau ,L)>0\} .$$ 
Moreover, let ${\cal A}^{f}:=\{ \tau \in {\cal A}\mid \sharp \mbox{supp}\,\tau<\infty \} $, 
${\cal A}_{+}^{f}:= {\cal A}_{+}\cap {\cal A}^{f}$, and 
${\cal A}_{+.all}^{f}:={\cal A}_{+,all}\cap {\cal A}^{f}.$ 

We now show the following claim. \\ 
\underline{Claim 2.} 
The sets ${\cal A}_{+}$ and ${\cal A}_{+,all}$ are  non-empty open subsets of 
${\cal A}$ (and thus they are  non-empty open subsets of 
$({\frak M}_{1,c}({\cal Y}, \{ {\cal W}_{j}\} _{j=1}^{m}), {\cal O})$).  
Also, ${\cal A}_{+}^{f}$ is dense in ${\cal A}_{+}$ and ${\cal A}_{+,all}^{f}$ 
is dense in ${\cal A}_{+,all}.$ Moreover, 
for each $\tau \in {\cal A}_{+}$, we have 
$\emptyset \neq \cup _{L\in H_{+,\tau }}L\subset J_{pt}^{0}(\tau )=\CCI \setminus \Omega _{\tau }$, where 
$\Omega _{\tau }$ and $H_{+,\tau }$ are the sets defined in Theorem~\ref{t:rcdnkmain2}, and 
for each $\tau \in {\cal A}_{+,all}$, we have 
$Q\subset J_{pt}^{0}(\tau ).$ Furthermore, 
for each $\tau \in {\cal A}_{+}^{f}$, 
we have $\overline{J_{pt}^{0}(\tau )}=J(G_{\tau })$ which is a perfect set. 

To prove this claim, by (\ref{eq:fjljz}), we obtain that 
${\cal A}_{+}$ and ${\cal A}_{+,all}$ are non-empty. It is easy to see that 
${\cal A}_{+}$ and ${\cal A}_{+,all}$ are open in ${\cal A}.$ The rest statements follow 
from Theorem~\ref{t:rcdnkmain2} and Corollary~\ref{c:rcdnkmain2f}. 
%\end{itemize}   

Let ${\cal A}_{-, all}:=
%\{ \tau \in {\cal A}\mid 
%${\cal A}_{+}:=
\{ \tau \in {\cal A}\mid 
\mbox{for all } L\in \Min(G_{\tau },Q) \mbox{ we have. }\chi (\tau ,L)<0\} .$ 
We now prove the following claims.\\ 
\underline{Claim 3.} 
The set ${\cal A}_{-, all}$ is a non-empty open subset of ${\cal A}$ and 
${\cal A}_{-, all}\cap {\cal A}_{+, all}=\emptyset .$ \\ 
\underline{Claim 4.} 
Let $\tau \in {\cal A}_{-,all}.$ Then for each $L\in \Min(G_{\tau },J_{\ker}(G_{\tau }))$, 
we have $\chi (\tau ,L)<0$, and each $L\in \Min(G_{\tau },\CCI)$ with $L\not\subset J_{\ker }(G_{\tau })$ 
is attracting $\tau .$ Thus $\tau $ satisfies all assumptions of Theorem~\ref{t:sjkfcln} and 
all conclusions in Theorem~\ref{t:sjkfcln} hold. In particular, $J_{pt}^{0}(\tau )=\emptyset $ 
and $F_{meas}(\tau )={\frak M}_{1}(\CCI ).$

These claims 3,4 can be shown by (\ref{eq:fjljz}) and the method in Example~\ref{ex:qx1any}. 
\end{ex}
\begin{ex}
\label{ex:pzljgz}
Let $x_{1},\ldots ,x_{u}\in \CC $ be mutually distinct points with $u\geq 2.$  
Let $a\in \CC \setminus \{ 0\} $ and let $g(z)=a\prod _{j=1}^{u}(z-x_{j}).$ 
Let $Q=\{ x_{1},\ldots, x_{u}\}.$ 
Let $P_{1}, \ldots, P_{m}$ be mutually distinct non-constant polynomials and 
suppose that $P_{j}(Q)\subset Q$ for each $j=1,\ldots. m.$ Also, 
suppose that $Q=\cup _{L\in \Min(\langle P_{1},\ldots, P_{m}\rangle ,Q)}L.$ 
Note that we have the following claim.\\ 
\underline{Claim 1.} 
For any finite subset $Q$ of $\CC $, we can take such elements $P_{1},\ldots, P_{m}$. 

 To prove this claim, we  remark that  
for any map $\varphi :Q\rightarrow Q$, there exists a polynomial $P$ 
such that $P|_{Q}=\varphi $ on $Q.$ This fact can be shown by using van der Monde determinant argument. 
Thus the statement of Claim 1 holds. 

For each $j=1,\ldots, m$, 
let $\Lambda _{j}:=\CC \setminus \{ 0\} $ and 
for each $\lambda _{j}\in \Lambda _{j}$, let 
$f_{j,\lambda _{j}}(z)=P_{j}(z+\lambda _{j}g(z)).$ 
Let ${\cal W}_{j}=\{ f_{j,\lambda _{j}}\} _{\lambda _{j}\in \Lambda _{j}} $ and 
let ${\cal Y}=\cup _{j=1}^{m}\{ f_{j,\lambda _{j}}\mid \lambda _{j}\in \Lambda _{j}\} .$  
Then ${\cal Y}$ is a weakly nice subset 
of ${\cal P}$ with respect to   holomorphic families 
$\{ {\cal W}_{j}\} _{j=1}^{m}$ of polynomials.  We now prove the following claim.\\ 
\underline{Claim 2.} We have 
$S({\cal W}_{j})=Q\cup \{ \infty \} $ for each $j=1,\ldots, m.$ 
Moreover, ${\cal Y}$ is a mild subset of ${\cal P}$ and 
${\cal Y}$ is  non-exceptional  and strongly nice 
 with respect to holomorphic families 
$\{ {\cal W}_{j}\} _{j=1}^{m}$ of polynomial maps. Hence, 
there exists the largest open and dense subset ${\cal A}$ of 
$({\frak M}_{1,c}({\cal Y}, \{ {\cal W}_{j}\} _{j=1}^{m}), {\cal O})$ such that for  each 
$\tau \in  {\cal A}$, all statements (i)--(xi) in Theorem~\ref{t:rcdnkmain2} hold. 
 In particular, any $\tau \in {\cal A}$ is weakly mean stable. 

We give the proof of this claim.  Since ${\cal Y}\subset {\cal P}$, 
the set ${\cal Y}$ is mild. 
For each $x\in Q$, for each $j=1,\ldots, m$ and for each 
$\lambda _{j}\in \Lambda _{j}$, we have 
$f_{j,\lambda _{j}}(x)=P_{j}(x). $
Thus for each $j=1,\ldots, m$, we have 
$S({\cal W}_{j})=Q\cup \{ \infty \} $. Hence  
$\cup _{j=1}^{m}S({\cal W}_{j})=Q\cup \{ \infty \} .$ 
Also, by the property of $\{ P_{j}\} _{j=1}^{m}$, we have that 
for each $\tau \in {\frak M}_{1,c}({\cal Y}, \{ {\cal W}_{j}\} _{j=1}^{m})$, we have 
$Q=\cup _{L\in \Min(G_{\tau }, Q)}L.$ Hence there is no peripheral cycle for 
$({\cal Y}, \{ {\cal W}_{j}\} _{j=1}^{m}).$ Also, 
we have 
\begin{equation}
\label{eq:fjljdiffz}
f_{j, \lambda _{j}}'(x)=P_{j}'(x)(1+\lambda _{j}g'(x)) \mbox{ for all }x\in Q, j=1,\ldots, m, \lambda _{j}\in \Lambda _{j}.
\end{equation}  
Therefore ${\cal Y}$ is strongly nice  with 
respect to 
$\{ {\cal W}_{j} \} _{j=1}^{m}.$ 
By (\ref{eq:fjljdiffz}), it is easy to see that ${\cal Y}$ is non-exceptional with respect to $\{ {\cal W}_{j}\} _{j=1}^{m}.$ 
  By Theorem~\ref{t:rcdnkmain2}, the statement of Claim 2 holds. Thus we have proved Claim 2.  
  
  We define subsets ${\cal A}_{+}, {\cal A}_{+,all}, {\cal A}^{f}, $ 
  $ {\cal A}^{f}_{+}, {\cal A}^{f}_{+,all}, {\cal A}_{-, all}$ of ${\cal A}$ in the same way 
  as that of Example~\ref{ex:wizlzn}. Then by (\ref{eq:fjljdiffz}) and 
  the arguments in Examples~\ref{ex:wizlzn} and \ref{ex:qx1any}, we obtain the following 
  claims. \\ 
 \underline{Claim 3.} 
The set ${\cal A}_{-, all}$ is a non-empty open subset of ${\cal A}$ and 
${\cal A}_{-, all}\cap {\cal A}_{+, all}=\emptyset .$ \\ 
\underline{Claim 4.} 
Let $\tau \in {\cal A}_{-,all}.$ Then for each $L\in \Min(G_{\tau },J_{\ker}(G_{\tau }))$, 
we have $\chi (\tau ,L)<0$, and each $L\in \Min(G_{\tau },\CCI)$ with $L\not\subset J_{\ker }(G_{\tau })$ 
is attracting $\tau .$ Thus $\tau $ satisfies all assumptions of Theorem~\ref{t:sjkfcln} and 
all conclusions in Theorem~\ref{t:sjkfcln} hold. In particular, $J_{pt}^{0}(\tau )=\emptyset $ 
and $F_{meas}(\tau )={\frak M}_{1}(\CCI ).$  \\ 
  \underline{Claim 5.} 
  Suppose that $P_{j}'(x)\neq 0$ for any $j=1,\ldots, m$ and for any $x\in Q.$ Then   
the sets ${\cal A}_{+}$ and ${\cal A}_{+,all}$ are  non-empty open subsets of 
${\cal A}$ (and thus they are  non-empty open subsets of 
$({\frak M}_{1,c}({\cal Y}, \{ {\cal W}_{j}\} _{j=1}^{m}), {\cal O})$).  
Also, ${\cal A}_{+}^{f}$ is dense in ${\cal A}_{+}$ and ${\cal A}_{+,all}^{f}$ 
is dense in ${\cal A}_{+,all}.$ Moreover, 
for each $\tau \in {\cal A}_{+}$, we have 
$\emptyset \neq \cup _{L\in H_{+,\tau }}L\subset J_{pt}^{0}(\tau )=\CCI \setminus \Omega _{\tau }$, where 
$\Omega _{\tau }$ and $H_{+,\tau }$ are the sets defined in Theorem~\ref{t:rcdnkmain2}, and 
for each $\tau \in {\cal A}_{+,all}$, we have 
$Q\subset J_{pt}^{0}(\tau ).$ Furthermore, 
for each $\tau \in {\cal A}_{+}^{f}$, 
we have $\overline{J_{pt}^{0}(\tau )}=J(G_{\tau })$ which is a perfect set.

\end{ex}
\begin{rem}
As in Example~\ref{ex:pzljgz}, we can embed many finite 
irreducible Markov chains 
into $\CC $ as weak attractors (i.e. minimal sets with negative Lyapunov exponents) of one 
random complex polynomial dynamical system generated by $\tau \in 
{\frak M}_{1,c}({\cal P})$ which is weakly mean stable and satisfies 
all statements in Theorem~\ref{t:sjkfcln} (e.g. $F_{meas}(\tau )={\frak M}_{1}(\CCI )$).     
\end{rem}
\section{List of notations}
\label{s:list}
In this section we give the list of notations of this paper.  
\begin{itemize}
\item $B(A, r), D(C, r), \mbox{int}(A), \mbox{Con}(A)$. 
Notation in section~\ref{Pre}. 
\item $\mbox{CM}(Y), \mbox{OCM}(Y), C(Y)$ for metric spaces 
$Y$. Definition~\ref{d:cmx}   
\item rational semigroup, polynomial semigroup, Rat, $\kappa$, Rat$_{+}$, ${\cal P}$. Definition~\ref{d:rational}. 
\item Rat$_{m}$. ${\cal P}_{m}$. Remark~\ref{r:ratm}. 
\item $F(G), J(G), \langle g_{1}, \ldots, g_{m}\rangle, $ 
$ \langle \Gamma \rangle , G(A), G^{-1}(A), G^{\ast }$. 
Definition~\ref{d:FGJG}. 
\item $E(G)$ (exceptional set of a rational semigroup $G$). 
Remark after Lemma~\ref{ocminvlem}.
\item $J_{\ker }(G)$ (kernel Julia set). Definition~\ref{d:kernelJ}. 
\item $\gamma _{m, n}, F_{\gamma ,0}, F_{\gamma }$, 
$J_{\gamma ,0}, J_{\gamma },$ 
$F^{\gamma, 0}, F^{\gamma }, J^{\gamma ,0}, 
J^{\gamma }$. Definition~\ref{d:gammamn}.  
\item ${\frak M}_{1}(Y), {\frak M}_{1,c}(Y), 
\mbox{supp}\,\tau, d_{0}, {\cal B}^{\ast }$. 
Definition~\ref{d:d0}. 
\item $X_{\tau }, \tilde{\tau }, G_{\tau }, M_{\tau }, 
M_{\tau }^{\ast }, F_{meas}(\tau ), J_{meas}(\tau )$, 
$ F_{meas}^{0}(\tau ), J_{meas}^{0}(\tau ).$ 
Definition~\ref{d:ytau}. 
\item $\Phi : Y\rightarrow {\frak M}_{1}(Y).$ 
Definition~\ref{d:Phi}.
\item $F_{pt}(\tau), J_{pt}(\tau ), F_{pt}^{0}(\tau ), 
J_{pt}^{0}(\tau ).$ Definition~\ref{d:manyFJ}. 
\item Cpt$(Y)$. Definition~\ref{d:CptY}. 

\item Minimal set for $(G, Y)$, Min$(G, Y)$. 
Definition~\ref{d:minimal}. 
\item period of irreducible finite Markov chain. Definition~\ref{d:chainperiod}. 
\item $G_{\tau }^{r}$. Definition~\ref{d:Gtaur}. 

\item $D\varphi _{z}, T_{z}U, \| D\varphi _{z}\|_{s}$ 
subsection~\ref{ss:invariant}.  
 \item $f: \Gamma ^{\NN} \times \CCI \rightarrow 
 \Gamma ^{\NN} \times \CCI $ (skew product associated 
 with the generator system $\Gamma $), 
 $\pi :\Gamma ^{\NN }\times \CCI \rightarrow \Gamma ^{\NN}$, $\pi _{Y}: \Gamma ^{\NN}\times Y 
 \rightarrow Y$, $f_{\gamma }^{n}, f_{\gamma, n}$, 
 $\tilde{J}(f), \tilde{F}(f)$, 
 $\hat{J}^{\gamma, \Gamma }, \hat{F}^{\gamma, \Gamma }, 
 \hat{J}_{\gamma, \Gamma }, \hat{F}_{\gamma, \Gamma}, $ 
 $Df_{z}$ for skew product map $f$ at a point 
 $z=(\gamma, y)$. Definition~\ref{d:sp}. 
 
 \item $\tau$-invariant measure, $\tau$-ergodic measure, 
 $\chi (\tau, \rho )$ (Lyapunov exponent of $(\tau, \rho )$). Definition~\ref{d:invmeaserg}. 
 
 \item $\omega _{L,j}, \omega _{L}$(canonical $\tau$-ergodic measure on $L$),  
 $\chi (\tau, L)$ (Lyapunov exponent of $(\tau, L)$). 
 Definition~\ref{d:celyap}. 
 
 \item holomorphic family of rational maps, holomorphic family of polynomial maps, non-constant holomorphic family of rational maps, $S_{n}({\cal W})$, $S({\cal W})$ (singular 
 set of a holomorphic family ${\cal W}$ of rational maps), 
 singular point of holomorphic family ${\cal W}$ of 
 rational maps. Definition~\ref{d:singdf}. 
 
\item weakly nice subset of Rat, 
${\frak M}_{1}({\cal Y}, \{ {\cal W}_{j}\} _{j=1}^{m}),$ 
${\frak M}_{1,c}({\cal Y}, \{ {\cal W}_{j}\} _{j=1}^{m}),$ 
$\cap _{j=1}^{m} S({\cal W}_{j})$ (singular set of 
$({\cal Y}, \{ {\cal W}_{j}\} _{j=1}^{m})$. 
Definition~\ref{d:weaklynice}.
 
 \item ${\cal O}$ (wH-topology in ${\frak M}_{1,c}({\cal Y})$ 
 where ${\cal Y}$ is a closed subset of an open subset of Rat). Definition~\ref{d:topologyO}. 
 
 \item $S_{\min }(\{ {\cal W}_{j}\}_{j=1}^{m})$. 
 Definition~\ref{d:smin}.  
 
\item attracting minimal set. Definition~\ref{d:attminset}. 

\item mild subset of Rat. Definition~\ref{d:mild}.  

\item mean stable. Definition~\ref{d:meanstable}. 

\item exceptional with respect to $\{ {\cal W}_{j}\} _{j=1}^{m}$, non-exceptional with respect to $\{ {\cal W}_{j}\} _{j=1}^{m}$. Definition~\ref{d:exceptional}. 

\item unitary eigenfunction, unitary eigenvalue, 
$U_{\tau, L}, U_{\tau, L, \ast }$. Definition~\ref{d:unitary}. 
\item limit function. Definition~\ref{d:limitfunction}. 

\item $r_{L}$, period of $(\tau, L)$. Definition~\ref{d:period}. 
\item $T_{A,\tau}(z)$ (the probability of tending to $A$ 
regarding the random orbits starting with the initial value 
$z\in \CCI $), $T_{a, \tau}$. Definition~\ref{d:probabilityoftending}. 

\item nice subset of Rat, peripheral cycle, strongly nice 
subset of Rat. Definition~\ref{d:stronglynice}. 

\item strict bifurcation element for $(\Gamma, L)$ with 
corresponding suffix $j$. Definition~\ref{d:strbf}. 

\item weakly mean stable. Definition~\ref{df:wms}. 

\item ${\frak M}_{1,c,mild}({\cal Y}, \{ {\cal W}_{j}\}_{j=1}^{m})$, 
 ${\frak M}_{1,c,JF}({\cal Y}, \{ {\cal W}_{j}\}_{j=1}^{m})$. 
 Definition~\ref{d:m1cmild}. 
 
\item $N_{g, \lambda }(z)$, ${\cal W}_{g}$ (random relaxed 
Newton's method family for $g$), ${\cal Y}_{g}$, 
$({\cal Y}_{g}, {\cal W}_{g})$ (random relaxed Newton's method 
scheme for $g$), random relaxed Newton's method (or random relaxed Newton's method system) for $g$, 
$Q_{g}$. Definition~\ref{d:randomNewton}.   

\item normalized polynomial. Remark~\ref{r:normalizedpolynomial}. 
\end{itemize}

\noindent 
{\bf Acknowledgement.} 
The author thanks Rich Stankewitz for valuable comments. 
The author would like to thank the referees and the corresponding editor for valuable and important comments and suggestions. 
This research was partially supported by JSPS Kakenhi 
15K04899, 19H01790.

\end{document}